\documentclass[11pt]{amsart}

\usepackage{amsmath,amsthm,amssymb,amsfonts,ifpdf}
\usepackage{xcolor}

\usepackage{xargs}  
\usepackage[colorinlistoftodos,prependcaption,textsize=tiny,obeyFinal]{todonotes}
\newcommandx{\ebltodo}[2][1=]{\todo[linecolor=red,backgroundcolor=red!25,bordercolor=red,#1]{#2}}
{
  \color{olive}%
}%
{}
\PassOptionsToPackage{hyphens}{url}
\ifpdf
\usepackage[
]{hyperref}
\else
\usepackage[hypertex]{hyperref}
\fi
\usepackage{enumitem}
\setenumerate{leftmargin=*}
\setitemize{leftmargin=*}
\usepackage{amscd}
\usepackage[latin1]{inputenc}
\DeclareMathAlphabet{\mathpzc}{OT1}{pzc}{m}{it}
\usepackage{bbm}

\numberwithin{equation}{section}

\swapnumbers
\newtheorem{thm}[subsection]{Theorem}

\newtheorem*{cor*}{Corollary}
\newtheorem{lemma}[subsection]{Lemma}

\newtheorem{propos}[subsection]{Proposition}
\newtheorem{conjecture}[subsection]{Conjecture}

\newtheorem*{thm*}{Theorem}
\newtheorem*{thma*}{Theorem A}
\newtheorem*{thmb*}{Theorem B}
\newtheorem*{thmc*}{Theorem C}

\newtheorem*{claim}{Claim}

\swapnumbers
\theoremstyle{definition}

\newtheorem{defin}[subsection]{Definition}

\newcounter{consta}

\newcounter{constk}
\renewcommand{\theconstk}{{\kappa_{\arabic{constk}}}}
\newcommand{\constk}{\refstepcounter{constk}\theconstk}

\newcounter{constc}

\newcounter{constE}
\renewcommand{\theconstE}{{{C}_{\arabic{constE}}}}
\newcommand{\constE}{\refstepcounter{constE}\theconstE}

\newcounter{constd}

\makeatletter
\newcommand*\bigcdot{\mathpalette\bigcdot@{.5}}
\newcommand*\bigcdot@[2]{\mathbin{\vcenter{\hbox{\scalebox{#2}{$\m@th#1\bullet$}}}}}
\makeatother

\def\XXint#1#2#3{{\setbox0=\hbox{$#1{#2#3}{\int}$ }
\vcenter{\hbox{$#2#3$ }}\kern-.6\wd0}}

\DeclareMathOperator{\Lip}{Lip}
\DeclareMathOperator{\supp}{supp}
\DeclareMathOperator{\Haar}{Haar}
\DeclareMathOperator{\Leb}{Leb}

\DeclareMathOperator{\diff}{d}

\DeclareMathOperator\Ad{Ad}

\newcommand\Aut{\rm{Aut}}

\newcommand\vol{{\rm{vol}}}
\newcommand\SL{{\rm{SL}}}

\newcommand\SO{{\rm{SO}}}
\newcommand\Lie{{\rm Lie}}

\def\sl{{\mathfrak{sl}}}

\def\bbz{\mathbb{Z}}
\def\bbq{\mathbb{Q}}

\def\bbr{\mathbb{R}}

\def\bbc{\mathbb{C}}

\def\bbn{\mathbb{N}}
\def\Gbf{\mathbb{G}}

\def\Z{\bbz}
\def\R{\bbr}
\def\C{\bbc}

\def\Scal{\mathcal{S}}

\def\hfrak{\mathfrak{h}}

\def\rfrak{\mathfrak{r}}

\def\gfrak{\mathfrak{g}}

\def\Gbf{\mathbf{G}}
\def\Hbf{\mathbf{H}}
\def\Lbf{\mathbf{L}}

\def\H{\Hbf}
\def\G{\Gbf}

\def\bpz{\mathpzc{b}}

\def\vare{\varepsilon}

\def\zg0{Z_{G_\omega}(s)}

\def\zg{Z_G(s)}

\def\be{\begin{equation}}
\def\ee{\end{equation}}

\def\dist{{\rm dist}}

\def\lf{\mathfrak l}
\def\Sob{{\mathcal S}}

\def\dist{d}
\def\mixexp{\kappa_0}

\def\boxH{\mathsf B^H}
\def\boxHs{\mathsf B^{s,H}}
\def\boxG{\mathsf B^{G}}

\def\rwm{\nu}

\def\rws{t}

\def\conv{\ast}

\def\injr{\eta}

\def\rel{r}

\def\ave{\int_{0}^1}
\def\uvk{u_\rel}
\def\uvkd{\diff\!\rel}

\def\mfht{f}

\def\mfm{f}

\def\margI{I}
\def\noI{\psi}

\def\inj{{\rm inj}}
\def\lf{\mathbb R}
\def\qlf{\mathbb C}

\def\qi{i}

\def\nuni{e}
\def\coneH{\mathsf E}
\def\cone{\mathcal E}

\def\umt{\mathsf Q}

\newcommand{\rhsc}{b}

\newcommand{\scmf}{b}
\newcommand{\mfsc}{b}
\newcommand{\sfh}{\mathsf h}

\newcommand{\boxU}{U}

\newcommand{\sfs}{\mathsf s}

\newcommand{\balpha}{\alpha} 
\newcommand{\trct}{\mathsf R}
\newcommand\mfbd{\Upsilon}
\newcommand\eng{\mathcal G}
\newcommand\egbd{\Upsilon}
\newcommand\pvare{\mathsf{c}}

\newcommand\gdh{L}
\newcommand{\icone}{\cone'}
\newcommand{\iconeH}{\coneH'}
\newcommand{\ccone}{\hat{\cone}}
\newcommand{\cconeH}{\hat{\coneH}}
\newcommand{\cmu}{\hat{\mu}}
\newcommand{\density}{\rho}
\newcommand{\ddensity}{\varrho}
\newcommand{\mbhs}{V}
\newcommand\adm{M}
\newcommand{\adl}{\lambda}
\newcommand\fub{m_0}
\newcommand{\czeta}{\check{\zeta}}
\newcommand{\cZ}{\check{Z}}
\newcommand{\sfk}{\mathsf k}

\newcommand{\exceptional}{{\rm Exc}}

\newcommand{\hide}[1]{}

\begin{document}
\title[Effective equidistribution for unipotent flows]{Effective equidistribution for some one parameter unipotent flows}

\author{E.~Lindenstrauss}
\address[E.L.]{Institute for Advanced Study, 1 Einstein Drive, Princeton, NJ 08540, USA\newline
\emph{and}\newline
The Einstein Institute of Mathematics, Edmund J. Safra Campus, Givat Ram,
The Hebrew University of Jerusalem, Jerusalem 91904, Israel
}
\email{elonl@ias.edu}
\thanks{E.L.\ acknowledges support by ERC 2020 grant HomDyn (grant no.~833423).}

\author{A.~Mohammadi}
\address{A.M.: Department of Mathematics, University of California, Berkeley, CA 94720}
\email{amirmo@math.berkeley.edu}
\thanks{A.M.\ acknowledges support by the NSF, grants DMS-2055122 and 2350028.}

\author{Z.~Wang}
\address{Z.W.: Department of Mathematics, Johns Hopkins University,
Baltimore, MD 21218}
\email{zhirenw@jhu.edu}
\thanks{Z.W.\ acknowledges support by the NSF grant  DMS-1753042.}

\begin{abstract}
We prove effective equidistribution theorems, with polynomial error rate, for orbits of the unipotent subgroups of $\SL_2(\lf)$ 
in arithmetic quotients of $\SL_2(\qlf)$ and $\SL_2(\lf)\times\SL_2(\lf)$. 

The proof is based on the use of a Margulis function, tools from incidence geometry, 
and the spectral gap of the ambient space. 
\end{abstract}

\maketitle

\setcounter{tocdepth}{1}
\tableofcontents

\section{Introduction}

A landmark result of Ratner~\cite{Ratner-topological} states that if $G$ is a Lie group, $\Gamma$ a lattice in $G$ and if $u_t$ is a one-parameter $\Ad$-unipotent subgroup of $G$, then for {\bf any} $x \in G/\Gamma$ the orbit $u_t.x$ is equidistributed in a periodic orbit of some subgroup $L <G$ that contains both the one parameter group $u_t$ and the initial point $x$. We say an orbit $L.x$ of a  group $L$ in some space $X$ is periodic if the stabilizer of $x$ in $L$ is a lattice in $L$, equivalently that the stabilizer of $x$ in $L$ is discrete and $L.x$ supports a unique $L$-invariant probability measure $m_{L.x}$;
and $u_t.x$ is equidistributed in $L.x$ in the sense that 
\begin{equation}\label{eq:equidist}
\frac1T \int_{0}^T f(u_t.x) dt \to \int f dm_{L.x} \qquad\text{for any $f \in C_0(G/\Gamma)$.}
\end{equation}
In order to prove this equidistribution result, Ratner first classified the $u_t$-invariant probability measures on $G/\Gamma$ \cite{Ratner-Acta, Ratner-measure}; the proof also uses the non-divergence properties of unipotent flows established by Dani and Margulis \cite{Marg-Nondiv, Dani-Nondiv-1, Dani-Nondiv-2}.

\medskip

In this paper we prove a quantitative equidistribution result for orbits of a one parameter unipotent group on quotients $G/\Gamma$ where $G$ is either $\SL_2(\qlf)$ or $\SL_2(\lf)\times\SL_2(\lf)$ with a polynomial error rate, which is the first quantitative equidistribution statement for individual orbits of unipotent flows on quotients of semi-simple groups beyond the horospherical case. Our approach builds on the paper \cite{LM-PolyDensity} by the first two authors, where an effective density result with a polynomial rate for orbits of a Borel subgroup of a subgroup $H \simeq \SL_2(\bbr)$ of $G$ was proved.

\medskip

Recall that a group $N<G$ is horospheric if there is some $g \in G$ so that 
\[
N = \{h \in G: g^{-n} h g^{n} \to 1 \text{ as $n\to\infty$}\}.
\]
For instance, the one parameter unipotent group
\[
\left\{\begin{pmatrix}1&r\\0&1\end{pmatrix}: r\in\bbr\right\}
\]
is horospheric in $\SL_2(\R)$ as are the groups
\[
\left\{\begin{pmatrix}1&r+is\\0&1\end{pmatrix}:r,s\in\bbr \right \}
\quad\text{and}\quad
\left\{\left(\begin{pmatrix}1&r\\0&1\end{pmatrix}, \begin{pmatrix}1&s\\0&1\end{pmatrix}\right): r,s\in\bbr\right\}
\]
in $\SL_2(\C)$ and $\SL_2(\R)\times\SL_2(\R)$, respectively.
The classification of invariant measures and orbit closures for horospherical flows was established prior to Ratner's work by Hedlund, Furstenberg, Dani, Veech and others, and this has been understood for some time also quantitatively since one can relate the distribution properties of individual $N$ orbits to the ergodic theoretic properties of the action of $g$ on $G/\Gamma$ (cf.~\S\ref{sec: horospheric} for more details).

\medskip

The non-horospheric case, on the other hand, is much more delicate, and proving a quantitative form of Ratner's theorem regarding equidistribution of unipotent orbits has been a major challenge. We survey below in \S\ref{sec: history} what was known before our work as well as some very recent developments that have taken place after these results have been announced.

To state our main results we first fix some notations.
Let
\[
G=\SL_2(\qlf)\quad\text{ or }\quad G=\SL_2(\lf)\times\SL_2(\lf).
\] 
Let $\Gamma\subset G$ be a lattice, and put $X=G/\Gamma$. We let $m_X$ denote the  $G$-invariant probability measure on $X$.   
Throughout the paper, we will denote by $H$ a subgroup of $G$ isomorphic to $\SL_2(\lf)$, namely
\[
\SL_2(\lf)\subset\SL_2(\qlf)\quad\text{or}\quad\{(g,g): g\in\SL_2(\lf)\}\subset \SL_2(\lf)\times\SL_2(\lf).
\]

For all $t, r\in\bbr$, let $a_t$ and $u_r$ denote the image of 
\[
\begin{pmatrix} e^{t/2} & o \\ 0 & e^{-t/2}\end{pmatrix}\quad\text{and}\quad \begin{pmatrix} 1 & r \\ 0 & 1\end{pmatrix},
\]
in $H$, respectively.

We fix maximal compact subgroups ${\rm SU}(2)\subset \SL_2(\bbc)$ and $\SO(2)\times\SO(2)\subset\SL_2(\bbr)\times\SL_2(\bbr)$.
Let $\dist$ be the right invariant metric on $G$ which is defined using the Killing form and the aforementioned maximal compact subgroups.\label{d definition page} This metric induces a metric $\dist_X$ on $X$, 
and natural volume forms on $X$ and its submanifolds. We define the injectivity radius of a point $x\in X$ using this metric. 
In the sequel, $\|\;\|$ denotes the maximum norm on ${\rm Mat}_2(\qlf)$ or 
${\rm Mat}_{2}(\lf)\times {\rm Mat}_{2}(\lf)$ with respect to the standard basis. 

\medskip

Our main result is the following:

\begin{thm}\label{thm:main}
Assume $\Gamma$ is an arithmetic lattice. 
For every $x_0\in X$, and large enough $R$ (depending explicitly on $X$ and the injectivity radius of $x_0$), for any $T\geq R^{A}$, at least one of the following holds.
\begin{enumerate}
\item For every $\varphi\in C_c^\infty(X)$, we have 
\[
\biggl|\int_0^1 \varphi(a_{\log T}u_rx_0)\diff\!r-\int \varphi\diff\!m_X\biggr|\leq \Sob(\varphi)R^{-\ref{k:main-1}}
\]
where $\Sob(\varphi)$ is a certain Sobolev norm. 
\item There exists $x\in X$ such that $H.x$ is periodic with $\vol(H.x)\leq R$, and 
\[
\dist_X(x,x_0)\leq R^{A}(\log T)^AT^{-1}.
\] 
\end{enumerate} 
The constants $A$ and $\constk\label{k:main-1}$ are positive and depend on $X$ but not on $x_0$.
\end{thm}

Theorem~\ref{thm:main} can be viewed as an effective version of \cite[Thm.~1.4]{Shah-Expanding}. Combining Theorem~\ref{thm:main} and the Dani--Margulis linearization method \cite{DM-Linearization} (cf.\ also Shah \cite{Shah-MathAnn}),
that allows to control the amount of time a unipotent trajectory spends near
invariant subvarieties of a homogeneous space, 
we also obtain an effective equidistribution theorem for long pieces of unipotent orbits (more precisely, we use a sharp form of the linearization method taken from~\cite{LMMS}).

\begin{thm}\label{thm:main unipotent}
Assume $\Gamma$ is an arithmetic lattice.
For every $x_0\in X$ and large enough $R$ (depending explicitly on $X$), for any $T\geq R^{A_1}$, at least one of the following holds.
\begin{enumerate}
\item For every $\varphi\in C_c^\infty(X)$, we have 
\[
\biggl|\frac{1}{T}\int_0^{T} \varphi(u_rx_0)\diff\!r-\int \varphi\diff\!m_X\biggr|\leq \Sob(\varphi)R^{-\ref{k:main uni}}
\]
where $\Sob(\varphi)$ is a certain Sobolev norm. 
\item There exists $x \in G/\Gamma$ with $\vol(H.x)\leq R^{A_1}$, 
and for every $r\in [0,T]$ there exists $g\in G$ with $\|g\|\leq R^{A_1}$ so that  
\[
\dist_X(u_{s}x_0, gH.x)\leq R^{A_1}\left(\frac{|s-r|}{T}\right)^{1/A_2}\quad\text{for all $s\in[0,T]$.}
\] 
\item For every $r\in[0,T]$ and $t \in [\log R, \log T]$, the injectivity radius at $a_{-t}u_r x_0$ is at most $R^{A_1}e^{-t}$.
\end{enumerate} 
The constants $A_1$, $A_2$, and $\constk\label{k:main uni}$ are positive, and depend on $X$ but not on $x_0$. 
\end{thm}

The assumption in Theorem~\ref{thm:main}, that $\Gamma$ is arithmetic, may be relaxed. 
Let us say $\Gamma$ has {\em algebraic entries} if the following is satisfied:
there is a number field $F$, a semisimple $F$-group $\G$ of adjoint type, 
and a place $v$ of $F$ so that $F_v=\bbr$ and $\G(F_v)$ and $G$ are locally isomorphic --- in which case there is a surjective homomorphism from $G$ onto the connected component of the identity in $\G(F_v)$ ---
and the image of $\Gamma$ in $\G(F_v)$ (possibly after conjugation) is contained in $\G(F)$. Every arithmetic lattice has algebraic entries, but there are lattices with algebraic entries that are not arithmetic.

Note that the condition that  $\Gamma$ has {\em algebraic entries} is automatically satisfied if $\Gamma$ is an irreducible lattice in $\SL_2(\bbr)\times \SL_2(\bbr)$ or if $G=\SL_2(\C)$. Indeed, by arithmeticity theorems of Selberg and Margulis, irreducible lattices in $\SL_2(\bbr)\times \SL_2(\bbr)$ are arithmetic~\cite[Ch.~IX]{Margulis-Book}. Moreover, by local rigidity, lattices in $\SL_2(\bbc)$ always have algebraic entries \cite[Thm.~0.11]{GarRag-FD} (see also~\cite{Selb, Weil-1, Weil-2}).

\begin{thm}\label{thm: main arithmeticity relaxed}
Assume $\Gamma$ is a lattice which has algebraic entries.
For every $0<\delta<1/4$, every $x_0\in X$ and large enough $T$ (depending explicitly on $X$, $\delta$ and the injectivity radius of $x_0$) 
at least one of the following holds.
\begin{enumerate}
\item For every $\varphi\in C_c^\infty(X)$, we have 
\[
\biggl|\int_0^1 \varphi(a_{\log T}u_rx_0)\diff\!r-\int \varphi\diff\!m_X\biggr|\leq \Sob(\varphi)T^{-\delta^2\ref{k:main-1-relax}}
\]
where $\Sob(\varphi)$ is a certain Sobolev norm. 
\item There exists $x\in X$ with
\[
\dist_X(x,x_0)\leq T^{-1/A'},
\]
satisfying the following: there are elements $\gamma_1$ and $\gamma_2$ in ${\rm Stab}_H(x)$
with $\|\gamma_i\|\leq T^{\delta}$ for $i=1,2$ so that the group generated by $\{\gamma_1,\gamma_2\}$ is Zariski dense in $H$.
\end{enumerate} 
The constants $A'$ and $\constk\label{k:main-1-relax}$ are positive, and depend on $X$ but not on $\delta$ and $x_0$.
\end{thm}

The obstacle to effective equidistribution in Theorem~\ref{thm:main} is much cleaner and simpler than in Theorem~\ref{thm:main unipotent}. This is not an artifact of the proof but a reflection of reality; a unipotent orbit may fail to equidistribute at the expected rate without it staying near a single period orbit of some subgroup $\{u_t\}< L< G$: one must allow a slow drift of the periodic orbit in the direction of the centralizer of $u_t$. Unlike the work of Shah in \cite{Shah-Expanding}, where (in particular) a non-effective version of Theorem~\ref{thm:main} is proved relying on Ratner's measure classification theorem for unipotent flows, our proof goes the other way, first establishing Theorem~\ref{thm:main}, and then deduce Theorem~\ref{thm:main unipotent} from it using a linearization and non-divergence argument. 

\medskip

These results have been announced in \cite{LMW-announce}, as well as in a series of three talks at the IAS in Princeton in February 2022\footnote{\url{https://www.ias.edu/video/effective-equidistribution-some-one-parameter-unipotent-flows-polynomial-rates-i-ii}}. The announcement \cite{LMW-announce} also contains an overview of the argument; the reader may find it useful to consult \cite{LMW-announce} before (or while) reading the full version.

\subsection{Background and further discussion}\label{sec: history}
Ratner's equidistribution theorem implies a corresponding orbit closure classification theorem. Answering a conjecture of Raghunathan, Ratner deduced from the equidistribution theorem a classification of orbit closures: if $G$ is a Lie group, $\Gamma$ a lattice in $G$, and if $H<G$ is generated by one parameter Ad-unipotent subgroups of $G$, then for any $x \in G/\Gamma$ one has that $\overline{H.x}=L.x$ where $H \leq L \leq G$ and $L.x$ is periodic.
Important special cases of Raghunathan's conjecture were proven earlier by Margulis and by Dani and Margulis using a different more direct approach, which in particular gave a proof of a rather strong form of the longstanding Oppenheim conjecture \cite{Margulis-Oppenheim,DM-Oppenheim,DM-MathAnn}. The rigidity properties of unipotent flows have had many other surprising applications to number theory, from equidistribution to counting integer points and even regarding nonvanishing of central values of L-functions, as well as many other areas. Already the cases we study here, e.g., the action of $u_t$ on $\SL_2(\R)\times\SL_2(\R)/ \SL_2(\Z)\times\SL_2(\Z)$ is of interest to some number theoretic implications (e.g.~\cite{Sarnak-Ubis, B-S-Z}). 

Both because of its intrinsic interest, but especially in view of the applications, obtaining quantitative versions of equidistribution results for unipotent flows has been a well known open problem (cf.\ \cite[\S1.3]{Maconj}, in particular problem 7 there, or \cite[Ques.~17]{Gorodnik-conjectures}).

As mentioned above, the equidistribution of orbits of horospheric groups is by now well understood, in part using the relation between studying individual orbits of horospheric groups and mixing properties of a corresponding diagonalizable group. 
The first work in this direction we are aware of is Sarnak \cite{Sarnak-Thesis} who studied periodic orbits of the horocycle flow. Burger \cite{Burger-Horo} gave a general effective treatment for quotients of $\SL_2(\R)$ (even in some infinite volume cases). In \cite{KMnonquasi}, Kleinbock and Margulis use a quantitaive equidistribution result for expanding translates of orbits of horospheric groups \cite[Proposition 2.4.8]{KMnonquasi}. More recent papers in the topic include the work of Flaminio and Forni \cite{FlFor}, Str\"{o}mbergsson \cite{Strom-Horo}, and Sarnak and Ubis \cite{Sarnak-Ubis}. Quantitative horospheric equidistribution has now been established in much greater generality e.g.\ by Kleinbock and Margulis in \cite{KM-Drich}, McAdam in \cite{McAdam} and by Asaf Katz \cite{Katz-Quantitative}. Moreover a quantitative equidistribution estimate twisted by a character was proved by Venkatesh \cite{Venkatesh-Sparse} and further developed by Tanis
and Vishe as well as Flaminio, Forni, and Tanis \cite{Tanis-Vishe, FFT-twisted}; this was generalized to a disjointness result with a general nil-system by Asaf Katz in \cite{Katz-Quantitative}.  
Closely related is the case of translates of periodic orbits of subgroups $L\subset G$ which are fixed by an involution by Duke, Rudnick and Sarnak, Eskin and McMullen, and Benoist and Oh in \cite{DRS-Counting, Eskin-McMullen-1993, BO-Counting}.

Unipotent dynamics have a very different flavour when the ambient group $G$ itself is a unipotent group (in which case the study of these flows, e.g. the classification of invariant measures, dates back to work by Leon Green, Parry and others from the late 1960s) on the one extreme and when $G$ is a semisimple group on the other. 
The case when $G$ is a skew product $G' \ltimes N$ with $G'$ semisimple and $N$ unipotent, with the acting group $U$ projecting to a horospheric subgroup of $G'$, can be viewed as intermediate between these two cases.

\begin{itemize}
    \item 
Even when $G$ is unipotent (and $G/\Gamma$ a nilmanifold) the quantitative behaviour of unipotent flows has only been understood relatively recently by Green and Tao \cite{GrTao-Nil}. 

\item
In the case of quotients of the skew product $G=\SL_2(\R)\ltimes \R^2$, Strombergsson \cite{Strom-Semi} has an effective equidistribution result for one parameter unipotent orbits (which are not horospheric in $G$, but project to a horospheric group on $\SL_2(\R)$), and this has been generalized by several authors, in particular by Wooyeon Kim \cite{ASLn-Kim} (using a completely different argument) to $\SL_n(\R)\ltimes\R^n$. 
The case where $G$ is a direct product $G=G' \times N$ and $U$ projects to a horospheric subgroup of $G'$ is discussed in Katz paper \cite{Katz-Quantitative}.

\item
Not quite in this framework, but also somewhat of an intermediate case between the case of $G$ semisimple and nilpotent is the study of random walks by automorphisms of the torus or nilmanifold $X$ driven by a probability measure on $\Aut(X)$ whose support generates a group with sufficiently large Zariski closure. Here there is a quantitative
equidistribution result by Bourgain, Furman, Mozes and the first named author \cite{BFLM}, which was extended by Weikun He and de Saxce \cite{He-Saxce}.
Elements from this proof were used by Wooyen Kim in \cite{ASLn-Kim}.

\item
When $G$ is semisimple, there have been some results regarding effective \emph{density} of non-horospherical unipotnet flows. Specifically, for $G/\Gamma=\SL_3(\R)/\SL_3(\Z)$ and $u_t$ is the generic one parameter unipotent subgroup a result towards effective density with a logarithmic error term was proved by Margulis and the first named author \cite{LM-Oppenheim} in order to give an effective and quantitative proof of the Oppenheim Conjecture. A more general result in this direction, with iterated logarithmic rate\footnote{I.e.\ very far from the right kind of dependence which should be polynomial.}, was announced by Margulis, Shah and two of us (E.L. and A.M.) with the first installment of this work appearing in \cite{LMMS}. An effective density result for $G=\SL_2(\C)$ or $\SL_2(\R)\times\SL_2(\R)$ and $u_t$ a one-parameter unipotent (i.e.\ the case we consider in this paper), with a polynomial rate, was established by the first two named authors~\cite{LM-PolyDensity}.

\item
When $G$ is semisimple, there have been some results regarding effective \emph{equidistribution} of \emph{special} orbits of non-horospherical groups generated by unipotents. In particular we note the work of Einsiedler, Margulis and Venkatesh \cite{EMV} showing that periodic orbits of semisimple subgroups $H$ of a semisimple group $G$ are quantitatively equidistributed in an appropriate homogeneous subspace of $G/\Gamma$ if $\Gamma$ is a congruence lattice and $H$ has finite centralizer in $G$. Subsequently Einsiedler, Margulis, Venkatesh and the second named author by using Prasad's volume formula and a more adelic view point were able to prove such an equidistribution result for periodic orbits of maximal semisimple subgroups of $G$ when the subgroup is allowed to vary \cite{EMMV} with arithemetic applications. The equidistribution of periodic orbits of semisimple groups is also closely connected to the equidistribution of Hecke points; a quantitative treatment of such equidistribution was given by Clozel, Oh and Ullmo in \cite{Clozel-Oh-Ullmo}.

In a different direction, but also under this general heading we note the paper of Chow and Lei Yang \cite{Chow-Yang} which deals with expanding translates of special 1-parameter unipotent orbits, with applications to Diophantine approximations.

\item For $G$ semisimple and $U$ a  nonhorospheric unipotent group there were no quantitative equidistribution results known, with any rate, before our work (certainly not for a one parameter group $U$; but see e.g.~\cite{Ubis-translates} for a related result in an ``almost horospheric'' situation).  Our work was announced in \cite{LMW-announce}. While we were working on finishing this paper Lei Yang posted a very interesting preprint treating another nonhorospheric case \cite{Yang-SL3} --- the case of trajectories of a non-generic one-parameter unipotent group on $\SL_3(\R)/\SL_3(\Z)$. That paper uses some elements common with our approach (e.g.\ a similar closing lemma as a starting point and a similar last stage), but the critical dimension increment phase seems to be done quite differently. We note that the case treated by Lei Yang in that paper is the same case for which Chow and Yang proved equidistribution for translates of special orbits in \cite{Chow-Yang}.

\end{itemize}

An extremely interesting analogue to unipotent flows on homogeneous spaces is given by the action of $\SL_2(\R)$ and its subgroups on strata of abelian differentials. 
Let $g \geq 1$, and 
let $\alpha = (\alpha_1,\dots, \alpha_n)$ be a partition of $2g-2$. Let $\mathcal H(\alpha)$ be the corresponding stratum of abelian differentials, i.e., the space of pairs $(M,\omega)$ where $M$ is a compact Riemann surface with genus $g$ and $\omega$ is a holomorphic $1$-form on $M$ whose zeroes have
multiplicities $\alpha_1, \dots, \alpha_n$. The form $\omega$ defines a canonical flat metric on $M$ with conical singularities and a natural area from. Let $\mathcal H_1(\alpha)$ be the space of unit area surfaces in $\mathcal H(\alpha)$. The space $\mathcal H(\alpha)$
admits a natural action of $\SL_2(\bbr)$; this action preserves the unit area hyperboloid $\mathcal H_1(\alpha)$.

A celebrated theorem of Eskin and Mirzakhani~\cite{EMir-Meas} shows that any $P$-invariant ergodic measure is $\SL_2(\bbr)$-invariant and is supported on an affine invariant manifold, where $P$ denotes the group of upper triangular matrices in $\SL_2(\bbr)$. We shall refer to these measures as {\em affine invariant measures}. 
Moreover, if we define, for any interval $I \subset \R$ and $x \in \mathcal H_1(\alpha)$, the probability measure $\mu_I^x$ on $\mathcal H_1(\alpha)$ by
\[
\mu_{I}^{x}=|I|^{-1}\int_I \delta_{u_s x}\diff\! s,
\]
then Eskin, Mirzakhani and the second named author \cite{EMM-Orbit} showed that for any  $x \in \mathcal H_1(\alpha)$ the limit
\begin{equation}\label{eq:EMM-Orbit}
    \lim_{T\to\infty} \frac{1}{T}\int_{t=0}^T a_t\mu_{[0,1]}^x \diff\! t \qquad\text{exists in weak${}^*$ sense}
\end{equation}
and is equal to an ($\SL_2(\R)$-invariant) affine invariant probability measure with $x$ in its support.
On the other hand, there are several results, in particular by Chaika, Smillie and B.~Weiss in~\cite{CSW-Tremors}, that show that an analogue of Ratner's equidistribution theorem (or our Theorem~\ref{thm:main unipotent}) fails to hold in this setting, for instance for some $x$ the sequence of measure $\mu_{[0,T]}^{x}$ may fail to converge as $T\to\infty$, or may converge to a non-ergodic measure. However the following conjecture of Forni seems to us very plausible:

\begin{conjecture}[{\cite[Conj.~1.4]{Forni-limits}}]
Let $\mathcal H_1(\alpha)$ be the space of unit area surfaces in stratum of abelian differentials on a genus $g$ surface whose zeros have multiplicities given by $\alpha=(\alpha_1, \dots, \alpha_n)$, and let $x\in\mathcal H_1(\alpha)$.
Then $\lim_{t\to\infty} a_t\mu^{x}_{[0,1]}$ exists in the weak${}^*$ sense and is equal to an affine invariant measure with $x$ in its support.  
\end{conjecture}

\noindent Of course, once one establishes that $\lim_{t\to\infty} a_t\mu^{x}_{[0,1]}$ exists, the rest follows from \cite{EMM-Orbit}. In this context again obtaining quantitative equidistribution results would be very interesting.

\subsection*{Acknowledgment}
A.M. and E.L. would like to thank the Hausdorff Institute for its hospitality during the winter of 2020. The three authors thank the Institute for Advanced Study for its hospitality while working on this project; indeed, we first started discussing this project when the three of us were visiting the IAS. In particular  
A.M.\ would like to thank the Institute for Advanced Study for its hospitality during the fall of 2019 and Z.W.\ would like to thank the Institute for Advanced Study for its hospitality during the fall of 2022. The authors would like to thank Gregory Margulis and Nimish Shah for many discussions about effective density, and Joshua Zahl for helpful communications regarding projections theorems. We would also like to thank Lei Yang for alerting us to his work and for several related discussions.

\section{The main steps of the proofs}\label{sec: outline}
As mentioned above, Theorem~\ref{thm:main unipotent} is proved by combining 
Theorem~\ref{thm:main} and the linearization techniques~\cite{DM-Linearization} in their quantitative form~\cite{LMMS}, see \S\ref{sec: proof main unip} for details. 
We note that the idea of using equidistribution of expanding translates of a fixed piece of a $U$ orbit of the type $\{a_t u_s .x: 0\leq s \leq 1\}$ to deduce equidistribution of a large segment of a non-translated $U$ orbit $\{u_s.x: 0\leq s \leq T\}$ is quite classical.

Let us now highlight some of the main ingredients used in the proof of Theorem~\ref{thm:main}. 
Assume that part~(2) in Theorem~\ref{thm:main} fails for $x_0$, $T$, and $R$ as the proof is complete otherwise. 
We begin with a version of avoidance principle \'a la linearization techniques of Dani--Margulis albeit for random walks.

Roughly speaking, the following proposition asserts that failure of part~(2) in Theorem~\ref{thm:main} may be 
upgraded to a Diophantine estimate with a polynomial rate (whose degree is absolute) in terms of $R$. 
We will let $\inj(x)$ denote (our slightly modified) injectivity radius of $x$, see~\S\ref{sec: notation} and \S\ref{sec: non-div}.  

\begin{propos}\label{prop: linearization translates 1}
There exist $D_0$ (absolute) and $\constE\label{c: linear trans},s_0$ (depending on $X$) so that the following holds. 
Let $R, S\geq 1$. Suppose $x_0\in X$ is so that 
\[
\dist_X(x_0,x)\geq (\log S)^{D_0}S^{-1}
\] 
for all $x$ with $\vol(Hx)\leq R$. Then for all 
\[
s\geq \max\Bigl\{\log S, 2|\log(\inj(x_0))|\Bigr\}+s_0
\] 
and all $0<\eta\leq 1$, we have 
\[
\biggl|\biggl\{r\in [0,1]\!:\!\! \begin{array}{c}\inj(a_su_rx_0)\leq \eta \text{ or there is $ x$ with }\\ 
\vol(Hx)\leq R \text{ s.t. }\dist_X(a_{s}u_rx_0,x)\leq \frac{1}{\ref{c: linear trans}R^{D_0}}\end{array}\!\!\biggr\}\biggr|\!\!\leq \ref{c: linear trans}(\eta^{1/2}+R^{-1}).
\]
\end{propos}

The proof of this proposition uses {\em Margulis functions} for periodic $H$-orbits and is completed in Appendix~\ref{sec: proof linearization}, see also \S\ref{sec: Dioph inheritance} for more details.  

We will apply this proposition with $\eta=R^{-\star}$ where $\star$ is a small constant. 
In view of this proposition and the fact that part~(2) in Theorem~\ref{thm:main} does not hold,
for all but a set with measure $\ll R^{-\star}$ of $r\in [0,1]$, the point $x_1=a_{s}u_rx_0$ (where $s=\log T-C\log R$ for appropriate choice of $C$) satisfy 
\be\label{eq: improved dioph}
\inj(x_1)\geq \eta\quad\text{and}\quad d(x,x_1)\geq R^{-D_0}\; \text{for every $x$ with $\vol(Hx)\leq R$.}
\ee

Thus, in order to show that $\int_0^1\varphi(a_{\log T}u_rx_0)\diff\!r$ 
is within $R^{-\star}$ of $\int\varphi\diff\!m$, it suffices to show that $\int_0^1\varphi(a_{C\log R}u_r x_1)\diff\!r$
is within $R^{-\star}$ of $\int\varphi\diff\!m$ where $x_1$ satisfies~\eqref{eq: improved dioph}.
We will show this statement in three phases. 

\subsection*{A closing lemma and the initial dimension}
In this phase, we show that the improved Diophantine condition~\eqref{eq: improved dioph} for $x_1$ implies that 
points in $\{a_{\star \log R}u_rx_0: r\in[0,1]\}$ 
(possibly after removing an exceptional set of measure $\ll R^{-\star}$) are {\em separated transversal to} $H$.  

Let $t>0$ be a large parameter, and fix some $\nuni^{-0.01t}<\beta=\nuni^{-\kappa t}$ (in our application, $\kappa$ will be chosen to be $\ll 1/D_0$ where the implied constant depends on $X$ and $D_0$ is as in Proposition~\ref{prop: linearization translates}, moreover, we will assume $\beta=\eta^2$ in that proposition).  

For every $\tau\geq0$, put
\[
\coneH_{\tau}=\boxHs_{\beta}\cdot a_\tau\cdot \{u_r: r\in [0,1]\} \subset H,
\]
where $\boxHs_{\beta}:=\{u_s^-:|s|\leq {\beta}\}\cdot\{a_t: |t|\leq \beta\}$ and $u_s^-$ is the transpose of $u_s$. 

Let $\gfrak=\Lie(G)$, that is, $\gfrak=\mathfrak {sl}_2(\qlf)$ or $\gfrak=\mathfrak{sl}_2(\lf)\oplus \mathfrak{sl}_2(\lf)$.
Let $\rfrak=\qi\mathfrak{sl}_2(\bbr)$ if $\gfrak=\mathfrak{sl}_2(\qlf)$ and 
$\rfrak=\mathfrak{sl}_2(\bbr)\oplus\{0\}$
if $\gfrak=\mathfrak{sl}_2(\bbr)\oplus \mathfrak{sl}_2(\bbr)$.
In either case $\mathfrak g=\mathfrak h\oplus \mathfrak r$ where $\hfrak=\Lie(H)\simeq\mathfrak{sl}_2(\bbr)$, and both $\hfrak$ and $\rfrak$ are $\Ad(H)$-invariant.

Let $\tau\geq 0$ and $y\in X$. Assume that $\mathsf h\mapsto\mathsf hy$ is injective over $\coneH$. 
For every $z\in\coneH_\tau.y$, put
\[
\margI_\tau(z):=\Bigl\{w\in \rfrak: \|w\|<\inj(z) \text{ and } \exp(w) z\in \coneH_\tau.y\Bigr\};
\]
this is a finite subset of $\rfrak$ since $\coneH_\tau$ is bounded ---  
we will define $\margI_\cone(h,z)$ 
for all $h\in H$ and more general sets $\cone$ in the bootstrap phase below.

Let $0<\alpha<1$. Define the function $f_{\tau}:\coneH_\tau.y\to [1,\infty)$ as follows
\[
f_{\tau}(z)=\begin{cases} \sum_{0\neq w\in I_\tau(z)}\|w\|^{-\alpha} & \text{if $\margI_\tau(z)\neq\{0\}$}\\
\inj(z)^{-\alpha}&\text{otherwise}
\end{cases}.
\]

\begin{propos}\label{prop: closing lemma intro 1}
Assume $\Gamma$ is arithmetic. 
There exists $D_1$ (which depends on $\Gamma$ explicitly) satisfying the following. Let $D\geq D_1$ 
and $x\in X$. Then for all large enough $t$ at least one of the following holds.
 
\begin{enumerate}
\item  There is a subset $I(x)\subset [0,1]$ with $|[0,1]\setminus I(x)|\ll_X \beta^{1/4}$ 
such that for all $r\in I(x)$ we have the following  
\begin{enumerate}
\item $\inj(a_{8t}u_rx)\geq \beta^{1/2}$.
\item $\sfh\mapsto \sfh.a_{8t}u_rx$ is injective over $\coneH_{\rws}$.
\item For all $z\in\coneH_\rws.a_{8t}u_rx$, we have 
\[
f_{t}(z)\leq  \nuni^{D\rws}.
\]
\end{enumerate}

\item There is $x'\in X$ such that $Hx'$ is periodic with
\[
\vol(Hx')\leq \nuni^{D_1\rws}\quad\text{and}\quad\dist_X(x',x)\leq \nuni^{(-D+D_1)\rws}.
\] 
\end{enumerate} 
\end{propos}

This proposition will be proved in \S\ref{sec: closing lemma}. 
We also refer to that section for discussions regarding the assumption that $\Gamma$ is arithmetic.

For the rest of the argument, let $t=\frac{1}{D_1}\log R$, where $R$ is as in Theorem~\ref{thm:main}, and let $x_1$ be as in~\eqref{eq: improved dioph}.
Apply Proposition~\ref{prop: closing lemma intro 1} with the point $x_1$. Then for every $r_1\in I(x_1)$, 
the conclusions in part~(1) of that proposition holds for $x_2=a_{8t}u_{r_1}x_1$.
That is,
$h\mapsto hx_2$ is injective over $\coneH_t$ and the transverse dimension of $\coneH_t.x_2$ is $\geq 1/D$ for all 
\begin{equation}\label{eq:x2}
    x_2 \in\Bigl\{a_{8t} u_{r_1}  x_1 : r_1 \in I(x_1)\Bigr\}
\end{equation}
where $D=D_0D_1+2D_1$.
Therefore, in order to show that $\int_0^1 \varphi(a_{C\log R}u_rx_1)\diff\!r$ is within $R^{-\star}$ of $\int \phi$, 
it is enough to show a similar estimate for
\[
\int_0^1 \varphi(a_{C\log R-8t}u_rx_2)\diff\!r
\] 
for all $x_2$ as in \eqref{eq:x2}.

\subsection*{Improving the dimension}
Roughly speaking, Proposition~\ref{prop: closing lemma intro 1} states that the set $\{a_{8t}u_rx_1: r\in[0,1]\}$ 
has transversal dimension $1/D$. In this step, we will improve this dimension to reach at dimension $\alpha$, close to $1$. 

We need some notation. Recall that $t=\frac{1}{D_1}\log R$. Let $\beta=\nuni^{-\kappa t}$ for some small $\kappa>0$. 
(More explicitly, we will fix some $0<\vare\leq10^{-8}$ to be explicated later, and let  $\kappa=10^{-6}\vare/(2D)$, $D= D_0D_1+2D_1$). 
Let 
\[
\coneH=\boxHs_\beta\cdot\Bigl\{u_r: |r|\leq \eta_0\Bigr\}.
\]
It will be more convenient to {\em approximate} translations 
\[
\{a_{\bigcdot} u_rx_0: r\in[0,1]\}
\] 
with sets which are a disjoint union of local $\coneH$-orbits as we now define.  
Let $F\subset B_\rfrak(0,\beta)$ be a finite set with $\#F\geq \nuni^{t/2}$, and let $y\in X$ with $\inj(y)\geq \beta^{1/2}$. 
Put 
\be\label{eq: define cone intro}
\cone=\bigcup\coneH.\{\exp(w)y: w\in F\}.
\ee
For every $w\in F$, we let $\mu_{w}$ 
be a measure which is absolutely continuous with respect to the pushforward of the Haar measure $m_H|_{\coneH}$ 
to $\coneH.\exp(w)y$ whose density satisfies certain Lipschitz condition, see \S\ref{sec: cone and mu cone} for more details.
We equip $\cone$ with the probability measure $\mu_\cone$ proportional to $\sum_w\mu_{w}$.
  
\medskip

Let $\theta$ be a small constant; in our application, the exact choice of $\theta$ will depending on the decay of matrix coefficients in $G/\Gamma$, see~\eqref{eq: choose theta intro}. Let 
\[
\text{$\balpha=1-\theta\quad$ and $\quad\vare=\theta^2$}.
\]
Let $\ell=0.01\vare t$, and let $\rwm_\ell$ be the probability measure on $H$ defined by 
\[
\rwm_\ell(\varphi)=\ave\varphi(a_{\ell}\uvk)\uvkd\qquad\text{for all $\varphi\in C_c(H)$;}
\] 
let $\rwm_\ell^{(n)}=\rwm_\ell\conv\cdots\conv\rwm_\ell$ denote the $n$-fold convolution of $\rwm_\ell$ for all $n\in\bbn$.

The following proposition is one of main steps in the proof.

\begin{propos}\label{propos: main bootstrap intro}
Let $x_1\in X$, and assume that Proposition~\ref{prop: closing lemma intro 1}(2) does not hold for $D$, $x_1$, and $t$. 
 Let  
\[
J:=[d_2, d_{1}]\cap\bbn,
\] 
where $d_{1}=100\lceil\tfrac{4D-3}{2\vare}\rceil$ and $d_2=d_{1}-\lceil{10^4}\vare^{-1/2}\rceil$. 

Let $r_1\in I(x_1)$, see Proposition~\ref{prop: closing lemma intro 1}(1), and put $x_2=a_{8t}u_{r_1}x_1$.
For every $d\in J$, there is a collection $\Xi_d=\{\cone_{d,i}: 1\leq i\leq N_d\}$
of sets 
\[
\cone_{d,i}=\coneH.\{\exp(w)y_{d,i}: w\in F_{d,i}\},
\] 
with $F_{d,i}\subset B_\rfrak(0,\beta)$ and $\inj(y_{d,i})\geq \beta^{1/2}$, and admissible measures $\mu_{\cone_{d,i}}$, see~\S\ref{sec: cone and mu cone}, so that both of the following hold: 
\begin{enumerate}
    \item Put $\scmf=\nuni^{-\sqrt\vare t}$. Let $d\in J$, $1\leq i\leq N_d$, and let $w_0\in B_\rfrak(0,\beta)$. 
    Then for every $w\in B_\rfrak(w_0,\scmf)$ and all $\delta\geq \nuni^{-t/2}$, we have
\be\label{eq: energy estimate final intro-1}
\frac{\#\Bigl(B_\rfrak(w,\delta) \cap B_\rfrak(w_0, \scmf)\cap F_{d,i}\Bigr)}{\#\Bigl(B(w_0, \scmf)\cap F_{d,i}\Bigr)}\leq \nuni^{\vare t} (\delta/\scmf)^{\alpha}. 
\ee
\item For all $s\leq t$ and all $r\in[0,2]$, we have
\begin{multline}\label{eq: nud1 and nud1-d 1}
\int \varphi(a_su_rz)\diff\!\rwm_\ell^{(d_{1})}\conv\mu_{\coneH_{t}.x_2}(z)=\\
\sum_{d,i}c_{d,i}\int \varphi(a_su_rz)\diff\!\rwm_\ell^{(d_{1}-d)}\conv\mu_{\cone_{d,i}}(z) + O(\operatorname{Lip}(\varphi)\beta^{\ref{k: bootstrap beta exp}})
\end{multline}
where $\varphi\in C_c^{\infty}(X)$, $c_{d,i}\geq0$ and $\sum_{d,i}c_{d,i}=1-O(\beta^{\ref{k: bootstrap beta exp}})$, $\operatorname{Lip}(\varphi)$ is the Lipschitz norm of $\varphi$, and $\constk\label{k: bootstrap beta exp}$ and the implied constants depend on $X$.
\end{enumerate}
\end{propos}

Roughly speaking, the proposition states 
that up to an exponentially small error, 
$\rwm_\ell^{(d_1)}\conv\mu_{\coneH_t.x_1}$ may be decomposed as 
$\sum_{d,i}c_{d,i}\rwm_\ell^{(d_1-d)}\conv\mu_{\cone_{d,i}}$ where $\sum_{d,i} c_{d,i}=1-O(\beta^{\ref{k: bootstrap beta exp}})$ (see~\eqref{eq: nud1 and nud1-d 1}) and for all $d\in J$ and $1\leq i\leq N_d$ the dimension of $\cone_{d,i}$ transversal to $H$ at controlled scales is $\geq \alpha$ (see~\eqref{eq: energy estimate final intro-1}). See Proposition~\ref{propos: imp dim main} for a more precise formulation which relies on a {\em Modified Margulis function}. The proof of Proposition~\ref{propos: imp dim main} (and hence of Proposition~\ref{propos: main bootstrap intro}) will be completed in \S\ref{sec: improve dim}--\ref{sec: proof main prop}. 

\medskip
Using this proposition we further reduce the analysis to equidistribution of sets $\cone$ 
satisfying part~(1) in Proposition~\ref{propos: main bootstrap intro}: 
Let $s=2\sqrt\vare t$ (note that this is much larger than $\ell=0.01\vare t$ but much smaller than $t$). Then
\[
\int_0^1\varphi(a_{s+ d_1\ell+t}u_rx_2)\diff\!r
\]
is within $R^{-\star}$ of  
\[
\int_0^1\int \varphi(a_{s}u_rz)\diff\!\rwm_\ell^{(d_1)}\conv\mu_{\coneH_{t}.x_2}(z)\diff\!r.
\]

We now use Proposition~\ref{propos: main bootstrap intro} to improve the small transversal dimension from $1/D$ to $\alpha$.
More precisely, Proposition~\ref{propos: main bootstrap intro} shows that 
\[
\int_0^1\int \varphi(a_{s}u_rz)\diff\!\rwm_\ell^{(d_{1})}\conv\mu_{\coneH_{t}.x_2}(z)\diff\!r
\]
is within $R^{-\star}$ of a convex combination of integrals of the form 
\be\label{eq: reduction to int wrt mu-cone}
\int_0^1\int \varphi(a_su_rz)\diff\!\rwm_\ell^{(n)}\conv\mu_{\cone}(z)\diff\!r
\ee
where $0\leq n=d_1-d\leq 10^4\vare^{-1/2}$ and $\cone=\cone_{d,i}$ has dimension at least $\alpha$ transversal to $H$ at controlled scales, see~\eqref{eq: energy estimate final intro-1}.

\subsection*{From large dimension to equidistribution}
In this final step of the argument, we will show that~\eqref{eq: reduction to int wrt mu-cone} equidistributes so long as $\theta$ (recall that $\alpha=1-\theta$) is chosen carefully.   

Let begin with the following quantitative decay of correlations for the ambient space $X$: 
There exists $0<\mixexp\leq1$ so that  
\be\label{eq:actual-mixing-intro}
\biggl|\int \varphi(gx)\psi(x)\diff\!{m_X}-\int\varphi\diff\!{m_X}\int\psi\diff\!{m_X}\biggr|\ll \Sob(\varphi)\Sob(\psi) \nuni^{-\mixexp d(e,g)}
\ee
for all $\varphi,\psi\in C^\infty_c(X)+\bbc\cdot 1$, where $m_X$ is the $G$-invariant probability measure on $X$ and $d$ is our fixed right $G$-invariant metric on $G$. See, e.g., \cite[\S2.4]{KMnonquasi} and references there for~\eqref{eq:actual-mixing-intro}; we note that $\mixexp$ is absolute if $\Gamma$ is a congruence subgroup. This is known in much greater generality, but the cases relevant to our paper are due to Selberg and Jacquet-Langlands~\cite{Selberg-ThreeSixteenth, Jacquet-Langlands}.

The quantitative decay of correlation can be used to establish quantitative results regarding the equidistribution of translates of pieces of an $N$-orbit. Specifically we employ the results in~\cite{KMnonquasi}, but there is rich literature around the subject; a more complete list can be found in \S\ref{sec: history}.

Now let $\xi:[0,1]\to \rfrak$ be a smooth non-constant curve. Then using the quantitative results regarding equidistribution of translates of pieces of an $N$-orbit such as~\cite{KMnonquasi}, one can show that for every $x\in X$,
\[
a_\tau\Bigl\{u_r\exp(\xi(s)).x:r,s\in [0,1]\Bigr\}
\]
is equidistributed in $X$ as $\tau\to\infty$ (with a rate which is polynomial in $\nuni^{-\tau}$).  
The key point in the deduction of this equidistribution result from the equidistribution of shifted $N$ orbits is that conjugation by $a_\tau$ moves $u_r\exp(\xi(s))$ to the direction of 
$N$, hence the above average essentially reduces to an average on a $N$ orbit.

\medskip 

Roughly speaking, the following proposition states that one may replace the curve $\{\xi(s):s\in[0,1]\}$ with a measure on 
$\rfrak$ so long as the measure has dimension $\geq 1-\theta$, for an appropriate choice of $\theta$ depending on $\mixexp$. 

\medskip

The precise formulation is the following.

\begin{propos}\label{prop: equi translates of cone intro}
For any $\theta>0$ and $c>0$ there is a $\constk\label{k: rho/varrho exp}$ so that the following holds:
Let $0<\mfsc_0<10^{-6}$, and let $F\subset B_\rfrak\Bigl(0,\mfsc_0\Bigr)$ 
be a finite set satisfying  
\[
\frac{\#(F\cap B_\rfrak(0,\delta))}{\#F}\leq \mfsc_1^{-c}\Bigl(\delta/\mfsc_0\Bigr)^{1-\theta}\quad\text{for all $\delta\geq \mfsc_1$}
\]
where $\mfsc_1 <\mfsc_0^{10}$.

Then for all $x\in X$ with $\inj(x)\geq \mfsc_0^{1/20}$, all 
$|\log(\mfsc_0)|\leq \tau\leq \frac{1}{10} |\log(\mfsc_1)|$, and every $\varphi\in C_c^\infty(X)$, we have 
\begin{multline*}
   \biggl|\int_0^1\frac{1}{\#F}\sum_{w\in F}\varphi(a_\tau u_r\exp(w)x)\diff\!r-\int\varphi\diff\!m_X\biggr|\ll_X \\
   \Sob(\varphi)\max\left((\mfsc_1/\mfsc_0)^{\ref{k: rho/varrho exp}},\mfsc_1^{-2c} e^{2\tau\theta}\mfsc_0^{\mixexp^2/M}\right), 
\end{multline*}
where $\Sob(\varphi)$ is a certain Sobolev norm and $M$ an absolute constant. 
\end{propos}

The proof of this proposition is significantly more delicate than that of the ``toy version" of a shifted curve, and relies on an adaptation of a projection theorem due to K\"{a}enm\"{a}ki, Orponen, and Venieri~\cite{kenmki2017marstrandtype}, based on the works of 
Wolff~\cite{Wolff}, Schlag~\cite{Schlag}, and~\cite{Zahl}, in conjunction with a sparse equidistribution argument due to Venkatesh~\cite{Venkatesh-Sparse}. 
These elements also played a crucial role in previous work by E.L.\ and A.M.~\cite{LM-PolyDensity} regarding quantitative density for the action of $AU$ on the spaces we consider here. A slightly modified statement and the proof are given in \S\ref{sec: equidistribution}, see in particular Proposition~\ref{prop: high dim to equid}.

We now use this proposition and outline the last step in the proof of Theorem~\ref{thm:main}: 
Using the above notation, fix $\theta$ and $\vare$ as follows
\be\label{eq: choose theta intro}
0<\theta < 10^{-8}\mixexp^2/M\quad\text{and}\quad \vare=\theta^2.
\ee

Recall that $s=2\sqrt\vare t$. In view of~\eqref{eq: reduction to int wrt mu-cone}, it now suffices to show that 
\[
\int_0^1\int \varphi(a_su_rz)\diff\!\rwm_\ell^{(n)}\conv\mu_{\cone}(z)\diff\!r
\] 
is within $R^{-\star}$ of $\int\varphi\diff\!m_X$ for all $\cone$ and $n$ as above. 
We will use Proposition~\ref{prop: equi translates of cone intro} to show this. 
First note that   
\[
\int_0^1\int \varphi(a_su_rz)\diff\!\rwm_\ell^{(n)}\conv\mu_{\cone}(z)\diff\!r
\]
is within $R^{-\star}$ of 
\[
\int\int_0^1\varphi(a_{s+n\ell}u_rz)\diff\!r\diff\mu_{\cone}(z).
\]
Moreover, we have 
\[
2\sqrt\vare t\leq s+n\ell\leq 2\sqrt\vare t+\frac{10^4\ell}{\sqrt\vare}=102\sqrt\vare t;
\]
in view of our choice of $\theta$ the right most term in the above series of inequalities is $\leq (10^{-5}\mixexp^2/M) t$.
Thus, Proposition~\ref{prop: equi translates of cone intro}, applied with $\theta=\sqrt\vare=1-\alpha$, $c=2\vare$, $\mfsc_0=\nuni^{-\sqrt\vare t}$, $\mfsc_1=\nuni^{-t/2}$, and $\tau=s+n\ell$, gives
\be\label{eq: using prop equi trans}
\biggl|\iint \varphi(a_{s+n\ell}u_rz)\diff\!\mu_{\cone}(z)\diff\!r-\int\varphi \diff\!m_X\biggr|\ll\Sob(\varphi)e^{-\star t}=\Sob(\varphi)R^{-\star}
\ee
where the implied constants depend on $X$. 

Note that the total time required for these three phases is $s+d_1\ell+9t$ which in view of the choices of $s$, $\ell$ and $t$ is indeed a (large) constant times $\log R$. Theorem~\ref{thm:main} follows.

\section{Notation and preliminary results}\label{sec: notation}
Throughout the paper     
\[
\text{$G=\SL_2(\qlf)\quad$ or $\quad G=\SL_2(\lf)\times\SL_2(\lf)$}.
\] 
Let $\Gamma\subset G$ be a lattice, and put $X=G/\Gamma$.

Let $A=\{a_t: t\in\bbr\}\subset H$. Let $U\subset N$ denote the group of upper triangular unipotent matrices in $H\subset G$, respectively. 
More explicitly, if $G=\SL_2(\qlf)$, then  
\[
N=\left\{n(\rel,s)=\begin{pmatrix} 1 & \rel+is\\ 0 &1\end{pmatrix}: (r,s)\in\lf^2\right\}
\] 
and $U=\{n(\rel,0):\rel\in \lf\}$; note that $n(r,0)=u_\rel$ for $\rel\in\bbr$. Let 
\[
V=\{n(0, s)=v_s: s\in \lf\};
\]
if $G=\SL_2(\lf)\times \SL_2(\lf)$, then  
\[
N=\left\{n(\rel,s)=\left(\begin{pmatrix} 1 & \rel+s\\ 0 &1\end{pmatrix},\begin{pmatrix} 1 & \rel\\ 0 &1\end{pmatrix}\right): (r,s)\in\lf^2\right\}
\] 
and $U=\{n(\rel,0):\rel\in \lf\}$.
As before, $n(\rel,0)=u_r$ for $\rel\in\bbr$. Let \[V=\{n(0,s)=v_s: s\in\lf\}.\] In both cases, we have $N=UV$.
Let us denote the transpose of $U$ by $U^-$ and its elements by $u^-_r$.

\subsection*{Lie algebras and norms}
Let $|\;|$ denote the usual absolute value on $\qlf$ (and on $\lf$).  
Let $\|\;\|$ denotes the maximum norm on ${\rm Mat}_2(\qlf)$ and ${\rm Mat}_2(\lf)\times {\rm Mat}_2(\lf)$,
with respect to the standard basis.

Let $\gfrak=\Lie(G)$, that is, $\gfrak=\mathfrak {sl}_2(\qlf)$ or $\gfrak=\mathfrak{sl}_2(\lf)\oplus \mathfrak{sl}_2(\lf)$.
We write $\mathfrak g=\mathfrak h\oplus \mathfrak r$ where $\hfrak=\Lie(H)\simeq\mathfrak{sl}_2(\bbr)$,
$\rfrak=\qi\mathfrak{sl}_2(\bbr)$ if $\gfrak=\mathfrak{sl}_2(\qlf)$ and 
$\rfrak=\mathfrak{sl}_2(\bbr)\oplus\{0\}$
if $\gfrak=\mathfrak{sl}_2(\bbr)\oplus \mathfrak{sl}_2(\bbr)$.

Note that $\rfrak$ is a {\em Lie algebra} in the case $G=\SL_2(\bbr)\times\SL_2(\bbr)$, but not when $G=\SL_2(\bbc)$.

Throughout the paper, we will use the uniform notation
\[
w=\begin{pmatrix}
w_{11} & w_{12}\\
w_{21} & w_{22}
\end{pmatrix}
\]
for elements $w\in\rfrak$, where $w_{ij}\in i\bbr$ if $G=\SL_2(\bbc)$ and $w_{ij}\in\bbr$ if $G=\SL_2(\bbr)\times\SL_2(\bbr)$.

We fix a norm on $\hfrak$ by taking the maximum norm where the coordinates are given by $\Lie(U)$, $\Lie(U^-)$, and $\Lie(A)$; similarly fix a norm on $\rfrak$. 
By taking maximum of these two norms we get a norm on $\gfrak$. These norms will also be denoted by $\|\;\|$.

Let $\constE\label{E:2ball}\label{E:dist-sheet}\geq 1$ be so that 
\be\label{eq:2ball}
\text{$\|hw\|\leq \ref{E:2ball}\|w\|$ for all $\|h-I\|\leq 2$ and all $w\in \gfrak$.} \ee

For all ${\beta}>0$, we define
\be\label{eq:def-BoxH}
\boxH_{\beta}:=\{u_s^-:|s|\leq {\beta}\}\cdot\{a_t: |t|\leq \beta\}\cdot\{u_\rel:|\rel|\leq {\beta}\}
\ee 
for all $0<{\beta}<1$.  Note that for all $h_i\in(\boxH_{\beta})^{\pm1}$, $i=1,\ldots,5$, we have  
\be\label{eq:B-beta-almost-group}
 h_1\cdots h_5\in \boxH_{100\beta}.
\ee

We also define $\boxG_{\beta}:=\boxH_\beta\cdot\exp(B_\rfrak(0,\beta))$ where $B_\rfrak(0,{\beta})$ denotes the ball of radius $\beta$ in $\rfrak$ with respect to $\|\;\|$. 

Similarly, using $\|\;\|$ we define $\mathsf B_\delta^L$ for $\delta>0$ and $L=U^\pm, A, AU, H, N$. Given an open subset $\mathsf B\subset L$, and $\delta>0$, $\partial_\delta\mathsf B=\{\sfh\in\mathsf B: \mathsf B_\delta^L.\sfh\not\subset\mathsf B\}$.

We deviate slightly from the notation in the introduction, and 
define the injectivity radius of $x\in X$ using $\boxG_{\beta}$ instead of the metric $d$ on $G$. 
Put
\be\label{eq:def-inj}
\inj(x)=\min\Big\{0.01, \sup\Big\{\beta: \text{ $g\mapsto gx$ is injective on $\boxG_{100\beta}$}\Big\}\Big\}.
\ee
Taking a further minimum if necessary, we always assume that the injectivity radius of $x$ defined using the metric $d$ dominates $\inj(x)$.  

For every $\injr>0$, let  
\[
X_\injr=\Bigl\{x\in X: \inj(x)\geq \injr\Bigr\}.
\]

\subsection*{The set $\coneH_{\eta, t,\beta}$}
For all $\eta,t, \beta>0$, set 
\be\label{eq:def-Ct}
\coneH_{\eta, t,\beta}:=\boxHs_{\beta}\cdot a_t\cdot \big\{u_r: r\in[0,\eta]\big\} \subset H.
\ee
Then $m_H(\coneH_{\eta,t,\beta})\asymp \eta{\beta}^2\nuni^{t}$ where $m_H$ denotes our fixed Haar measure on $H$.

Throughout the paper, the notation $\coneH_{\eta, t,\beta}$ will be used only for $\eta,t, \beta>0$ 
which satisfy $\nuni^{-0.01t}<\beta\leq\eta^2$ even if this is not explicitly mentioned.

For all $\eta,\beta, m>0$, put 
\be\label{eq: def B ell beta}
\umt^H_{\eta,\beta,m}=\Bigl\{u^-_s: |s|\leq \beta \nuni^{- m}\Bigr\}\cdot\{a_t: |t|\leq \beta\}\cdot\Bigl\{u_r: |r|\leq \eta\Bigr\}.
\ee
Roughly speaking, $\umt_{\eta, \beta,m}^H$ is a {\em small thickening} of the $(\beta,\eta)$-neighborhood of the identity in $AU$. 
We write $\umt^H_{\beta,m}$ for $\umt^H_{\beta,\beta,m}$.

The following lemma will also be used in the sequel.

\begin{lemma}[\cite{LM-PolyDensity}, Lemma 2.3]
\label{lem:commutation-rel}
\begin{enumerate}
\item Let ${m}\geq 1$, and let $0<\eta,\beta<0.1$. Then 
\[
\bigg(\Bigl(\umt_{0.01\eta,0.01\beta,m}^H\Bigr)^{\pm1}\bigg)^3\subset \umt_{\eta,\beta, m}^H.
\]
\item For all $0\leq \beta,\eta\leq 1$, $t,m>0$, and all $|r|\leq 2$, we have
\be\label{eq:well-rd-tau-1}
\Bigl(\umt_{\eta, \beta^2, m}^H\Bigr)^{\pm1}\cdot a_m u_r \coneH_{\eta', t,\beta'}\subset a_m u_r\coneH_{\eta,t, \beta},
\ee
where $\eta'= \eta(1-100\nuni^{-t})$ and $\beta'= \beta(1-100\beta)$.
\end{enumerate}
\end{lemma}

\subsection*{Constants and the $\star$-notation}
In our analysis, the dependence of the exponents on $\Gamma$ are via the application of results in \S\ref{sec: horospheric}, see~\eqref{eq: exp mixing}, and~\S\ref{sec: closing lemma}.

We will use the notation $A\asymp B$ when the ratio between the two lies in $[C^{-1}, C]$
for some constant $C\ge 1$ which depends at most on $G$ and $\Gamma$ in general. 
We write $A\ll B^\star$ (resp.\ $A\ll B$) to mean that $A\le C B^\kappa$ (resp.\ $A\leq CB$) 
for some constant $C>0$ depending on $G$ and $\Gamma$, and $\kappa>0$ which follows the above convention about exponents.

\subsection*{Commutation relations} We also record the following two lemmas.

\begin{lemma}[\cite{LM-PolyDensity}, Lemma 2.1]
\label{lem: BCH}
There exist absolute constants $\beta_0$ and $\constE\label{E:BCH}$ so that the following holds.  
Let $0<\beta\leq \beta_0$, and let $w_1,w_2\in B_\rfrak(0,\beta)$. There are $h\in H$ and $w\in\rfrak$ which satisfy 
\[
\text{$\tfrac23\|w_1-w_2\|\leq \|w\|\leq \tfrac32\|w_1-w_2\|\quad$ and $\quad\|h-I\|\leq \ref{E:BCH}\beta\|w\|$}
\] 
so that $\exp(w_1)\exp(-w_2)=h\exp(w)$. More precisely,
\[
\|w-(w_1-w_2)\|\leq \ref{E:BCH}\beta\|w_1-w_2\|
\]
\end{lemma}

\begin{lemma}[\cite{LM-PolyDensity}, Lemma 2.2]
\label{lem:dist-sheet} 
There exists $\beta_0$ so that the following holds for all $0<\beta\leq \beta_0$.
Let $x\in X_{10\beta}$ and $w\in B_\rfrak(0,\beta)$. If there are $h,h'\in \boxH_{2\beta}$
so that $\exp(w')hx=h'\exp(w)x$, then 
\[
\text{$h'=h\quad$ and $\quad w'=\Ad(h)w$}.
\]
Moreover, we have $\|w'\|\leq 2\|w\|$.
\end{lemma}

\section{Avoidance principles in homogeneous spaces}\label{sec: avoidance}

In this section we will collect statements concerning avoidance principles for unipotent flows and random walks on homogeneous spaces.    

\subsection{Nondivergence results}\label{sec: non-div}\label{sec:SiegelSet}
This subsection, is devoted to non-divergence results for unipotent flows. The results in this section are known to the experts and were also proved in details in~\cite[\S3]{LM-PolyDensity}.

The results of this subsection are trivial when $\Gamma$ a uniform lattice.

\begin{propos}[Prop.~3.1,\cite{LM-PolyDensity}]
\label{prop:Non-div-main}\label{prop:one-return}\label{lem:one-return}\label{prop: Non-div main}
There exist $\constE\label{E:non-div-main}\geq 1$ with the following property. 
Let $0<\delta, \vare<1$ and $x\in X$.
Let $I\subset [-10,10]$ be an interval with $|I|\geq\delta$. Then
\[
\Bigl|\Bigl\{r\in I:\inj(a_t\uvk x)< \vare^2\Bigr\}\Bigr|<\ref{E:non-div-main}\vare |I|
\]
so long as $t\geq |\log(\delta^2\inj(x))|+\ref{E:non-div-main}$.
\end{propos} 

The following is a direct corollary of Proposition~\ref{prop:Non-div-main}.  

\begin{propos}[Prop.~3.4,\cite{LM-PolyDensity}]
\label{prop:non-div}
There exists $0<\eta_X<1$, depending on $X$, so that the following holds. 
Let $0<\eta<1$ and let $x\in X$. Let $I\subset\bbr$ be an interval of length at least $\eta$. Then 
\[
|\{r\in I: a_t u_rx\in X_{\eta_X}\}|\geq 0.9|I|
\]
for all $t\geq |\log(\eta_X^2\inj(x))|+\ref{E:non-div-main}$. 
\end{propos}

\begin{proof}
Apply Proposition~\ref{lem:one-return} with $\vare=0.1\ref{E:non-div-main}^{-1}$. The claim thus holds with $\eta_X=\vare^2$. 
\end{proof}

\subsection*{The subsets $X_{\rm cpt}$ and $\mathfrak S_{\rm cpt}$} 
If $X$ is compact, let $X_{\rm cpt}=X$; otherwise, let $X_{\rm cpt}=\{gx: x\in X_{\eta_X}, \|g-I\|\leq 2\}$   
where $X_{\eta_X}$ is given by Proposition~\ref{prop:non-div}. 
Note that by~\cite[Lemma 3.6]{LM-PolyDensity}, we have 
\be\label{eq: periodic orbit X cpt}
\mu_{Hx}(X_{\rm cpt})>0.9
\ee 
for every periodic orbit $Hx$. 

We also fix once and for all a compact subset with piecewise smooth boundary 
$\mathfrak S_{\rm cpt}\subset G$ which projects onto $X_{\rm cpt}$. 

More generally, we have the following lemma which is a consequence of reduction theory. In this form, the lemma is a spacial case of~\cite[Lemma 2.8]{LMMS}. 

\begin{lemma}\label{lem: reduction theory}
There exist $D_2$ (absolute) and $\constE\label{c: red th}$ (depending on $X$) so that the following holds for all $0<\eta\leq \eta_X$. Let $g\in G$ be so that $g\Gamma\in X_\eta$. Then there is some $\gamma\in\Gamma$ so that 
\[
\|g\gamma\|\leq \ref{c: red th}\eta^{-D_2}.
\]
\end{lemma}

\subsection{Inheritance of the Diophantine property}\label{sec: Dioph inheritance}
As it was mentioned in the outline given in \S\ref{sec: outline}, assuming part~(2) in Theorem~\ref{thm:main} does not hold, the first step in the proof is to improve this Diophantine condition. The following proposition (which was also stated in \S\ref{sec: outline}) is tailored for this purpose.  

\begin{propos}\label{prop: linearization translates}
There exist $D_0$ (absolute) and $\ref{c: linear trans}, s_0$ (depending on $X$) so that the following holds. 
Let $R, S\geq 1$. Suppose $x_0\in X$ is so that 
\[
\dist_X(x_0,x)\geq (\log S)^{D_0}S^{-1}
\] 
for all $x$ with $\vol(Hx)\leq R$. Then for all 
\[
s\geq \max\Bigl\{\log S, 2|\log(\inj(x_0))|\Bigr\}+s_0
\] 
and all $0<\eta\leq 1$, we have 
\[
\biggl|\biggl\{r\in [0,1]\!:\!\! \begin{array}{c}\inj(a_su_rx_0)\leq \eta \text{ or there is $ x$ with }\\ 
\vol(Hx)\leq R \text{ s.t. }\dist_X(a_{s}u_rx_0,x)\leq \frac{1}{\ref{c: linear trans}R^{D_0}}\end{array}\!\!\biggr\}\biggr|\!\!\leq \ref{c: linear trans}(\eta^{1/2}+R^{-1}).
\]
\end{propos}

In the proof of Proposition~\ref{prop: linearization translates}, 
which is given in Appendix~\ref{sec: proof linearization}, we use Margulis functions for periodic $H$-orbits similar to those which were used in~\cite[\S9]{LM-PolyDensity}, see also~\cite[Prop.\ 2.13]{EMM-Orbit} and the original paper~\cite{EMM-Upp}. 
This will then be combined with the fact that the number of periodic $H$-orbits with volume $\leq R$ 
in $X$ is $\ll R^6$, see e.g~\cite[\S10]{MO-MargFun}, to conclude.  We also refer the reader to~\cite[\S2]{ELMV-1} for results concerning isolation of periodic orbits.

It is also worth mentioning that even though~\cite[Thm.~1.4]{LMMS} concerns long pieces of $U$-orbits and Proposition~\ref{prop: linearization translates} deals with translates of pieces of $U$-orbits, similar tools are applicable here as well. In particular, a version of Proposition~\ref{prop: linearization translates} can be proved using the methods of~\cite{LMMS}.

\subsection{Closing lemma}\label{sec: closing lemma}
Let $t>0$ be a large parameter. Fix some 
\[
\nuni^{-0.01t}<\beta=\eta^2<\eta_X^2;
\]
in our application, we will let $\beta=\nuni^{-\kappa t}$ where $\kappa\ll 1/D_0$ with $D_0$ as in Proposition~\ref{prop: linearization translates} and the implied constant depending on $X$. 

For every $\tau\geq0$, put
\[
\coneH_{\tau}=\boxHs_{\beta}\cdot a_\tau\cdot \{u_r: r\in [0,1]\} \subset H.
\]

If $y\in X$ is so that the map $\sfh\mapsto \sfh y$ is injective over $\coneH_{\tau}$, then $\mu_{\coneH_\tau.y}$ 
denotes the pushforward of the normalized Haar measure on $\coneH_\tau$ to $\coneH_\tau.y\subset X$.

Let $\tau\geq 0$ and $y\in X$. For every $z\in\coneH_\tau.y$, put
\[
\margI_\tau(z):=\Bigl\{w\in \rfrak: \|w\|<\inj(z) \text{ and } \exp(w) z\in \coneH_\tau.y\Bigr\};
\]
this is a finite subset of $\rfrak$ since $\coneH_\tau$ is bounded ---  
we will define $\margI_\cone(h,z)$ 
for all $h\in H$ and more general sets $\cone$ in the bootstrap phase below.

Let $0<\alpha<1$. Define the function $f_{\tau}:\coneH_\tau.y\to [1,\infty)$ as follows
\[
f_{\tau}(z)=\begin{cases} \sum_{0\neq w\in I_\tau(z)}\|w\|^{-\alpha} & \text{if $\margI_\tau(z)\neq\{0\}$}\\
\inj(z)^{-\alpha}&\text{otherwise}
\end{cases}.
\]

The following proposition supplies an initial dimension which we will bootstrap in the next phase.
Roughly speaking, it asserts that points in $\{a_{8t}u_rx_0: r\in[0,1]\}$ 
(possibly after removing an exponentially small set of exceptions) 
are {\em separated transversal to} $H$, unless $x_0$ is extremely close to a periodic $H$ orbit.    

\begin{propos}\label{prop:closing lemma intro}
Assume $\Gamma$ is arithmetic. 
There exists $D_1$ (which depends on $\Gamma$ explicitly) satisfying the following. Let 
$D\geq D_1$ and $x_1\in X$.
Then for all large enough $t$ (depending on $\inj(x_1)$) at least one of the following holds.
 
\begin{enumerate}
\item  There is a subset $I(x_1)\subset [0,1]$ with $|[0,1]\setminus I(x_1)|\ll_X \eta^{1/2}$ 
such that for all $r\in I(x_1)$ we have the following  
\begin{enumerate}
\item $a_{8t}u_rx_1\in X_{\eta}$.
\item $\sfh\mapsto \sfh.a_{8t}u_rx_1$ is injective on $\coneH_{\rws}$.
\item For all $z\in\coneH_\rws.a_{8t}u_rx_1$, we have 
\[
f_{t}(z)\leq  \nuni^{D\rws}.
\]
\end{enumerate}

\item There is $x\in X$ such that $Hx$ is periodic with
\[
\vol(Hx)\leq \nuni^{D_1\rws}\quad\text{and}\quad\dist_X(x,x_1)\leq \nuni^{(-D+D_1)\rws}.
\] 
\end{enumerate} 
\end{propos}

The proof of this proposition is a minor modification of the proof of~\cite[Prop.~6.1]{LM-PolyDensity}. The details are provided in Appendix~\ref{sec: proof closing}. 

Proposition~\ref{prop:closing lemma intro} is where the arithmeticity assumption on $\Gamma$ is used. If we replace the assumption that $\Gamma$ is arithmetic with 
the weaker requirement that $\Gamma$ has algebraic entries, we get a version of this proposition where part~(2) is replaced with the following. 

\medskip

{\em 
\begin{enumerate}
    \item[(2')]  There is $x\in X$ with
    \[
\dist_X(x,x_1)\leq \nuni^{(-D+D_1)\rws},
\] 
satisfying the following: there are elements $\gamma_1$ and $\gamma_2$ in ${\rm Stab}_H(x)$
with $\|\gamma_i\|\leq\nuni^{D_1t}$ for $i=1,2$ so that the group generated by $\{\gamma_1,\gamma_2\}$ is Zariski dense in $H$.
\end{enumerate}
}

See Appendix \ref{sec: proof closing} for more details.

\section{Equidistribution of translates of horospheres}\label{sec: horospheric}

We begin by recalling the following quantitative decay of correlations for the ambient space $X$: 
There exists $0<\mixexp\leq1$ so that  
\be\label{eq: exp mixing}
\biggl|\int \varphi(gx)\psi(x)\diff\!{m_X}-\int\varphi\diff\!{m_X}\int\psi\diff\!{m_X}\biggr|\ll \Sob(\varphi)\Sob(\psi) \nuni^{-\mixexp d(e,g)}
\ee
for all $\varphi,\psi\in C^\infty_c(X)+\bbc\cdot 1$, where $m_X$ is the $G$-invariant probability measure on $X$ and $d$ is the right $G$-invariant metric on $G$ defined on p.~\pageref{d definition page}. See, e.g., \cite[\S2.4]{KMnonquasi} and references there for~\eqref{eq: exp mixing}. 

Here $\Scal(\cdot)$ is a certain Sobolev norm on $C_c^\infty(X)+\bbc\cdot1$ 
which is assumed to dominate $\|\cdot\|_\infty$ and the 
Lipschitz norm $\|\cdot\|_{\rm Lip}$. Moreover, $\Scal(g.f)\ll\|g\|^\star\Scal(f)$ where the implied constants are absolute.

We note that by the works of Selberg and Jacquet-Langlands~\cite{Selberg-ThreeSixteenth, Jacquet-Langlands}, the constant $\mixexp$ is absolute if $\Gamma$ is a congruence subgroup, with the best known constant\footnote{To give a numerical value one needs to fix a normalization for $d$.} given by Kim and Sarnak \cite{KimSarnak} (this phenomenon, sometimes called \emph{property $(\tau)$} of congruence lattices, also holds in much greater generality).

Recall that $N=\{u_rv_s: r,s\in\bbr\}$ is a maximal unipotent subgroup of $G$, see \S\ref{sec: notation}. 
For $\delta_1,\delta_2>0$, put $B^N_{\delta_1,\delta_2}=\Bigl\{u_rv_s: 0\leq r\leq \delta_1, 0\leq s\leq\delta_2\Bigr\}$. 
We will denote $B^N_{1,1}$ by
$B^N_1$. Let $\diff\!n=\diff\!r\diff\!s$; in particular, $|B^N_{\delta_1,\delta_2}|=\delta_1\delta_2$.

It follows from Proposition~\ref{prop: Non-div main}, that 
for every $\vare>0$ and all $x\in X$, 
\[
\Bigl|\Bigl\{s\in [0,1]:\inj(a_tv_s x)< \vare^2\Bigr\}\Bigr|<\ref{E:non-div-main}\vare 
\]
so long as $t\geq |\log(\inj(x))|+\ref{E:non-div-main}$. 
Indeed Proposition~\ref{prop: Non-div main} is stated with $u_r$ instead of $v_s$, but the proof applies to this case as well 
--- note that $a_t,v_s\in H'$ where $H'=gHg^{-1}$ where $g={\rm diag}(i, 1)$.

\begin{propos}[cf.\ \cite{KMnonquasi}, Prop.~2.4.8]\label{prop: equid trans horo}
There exists $\constk\label{k:thick-mixing-prop}\gg \mixexp$ (where the implied constant is absolute) so that the following holds.  
Let $0<\eta,\delta\leq 1$ and $x\in X_{\eta}$. Then for every $t\geq 4|\log\eta|+2\ref{E:non-div-main}$ we have 
\[
\biggl|\frac{1}{|B^N_{\delta, 1}|}\int_{B^N_{\delta, 1}} f(a_tn.x)\diff\!n-\int f\diff\!m_X\biggr|\ll\Scal(f)(e^{t}\delta)^{-\ref{k:thick-mixing-prop}}  
\]
here $f\in C_c^\infty(X)+\bbc\cdot1$ and the implied constant depends on $X$. 
\end{propos}

\begin{proof}
We may assume $e^t\delta> 1$ or else the statement holds trivially.  
Put $d_1=\frac12\log(e^t\delta)$ and  
\begin{align*}
d_2=t-d_1&=\tfrac12(t+|\log\delta|)=|\log\delta|+ \tfrac12\log (e^t\delta)\\
&\geq 2|\log\eta|+\ref{E:non-div-main},
\end{align*}
where we used $t\geq 4|\log\eta|+2\ref{E:non-div-main}$. 

Now, for every $u_rv_s\in B^N_1$, we have 
\[
a_{d_1}u_rv_sa_{d_2}=a_tu_{e^{-d_2}r}v_{e^{-d_2}s};
\]
moreover, for every $u_rv_s\in B^N_1$, we have 
\[
\frac{|u_{e^{-d_2}r}v_{e^{-d_2}s}B^N_{\delta, 1}\triangle B^N_{\delta, 1}|}{|B^N_{\delta, 1}|}\ll (e^{d_2}\delta)^{-1}=(e^t\delta)^{-1/2}.
\]
We conclude that 
\begin{align*}
&\frac{1}{|B^N_{\delta, 1}|}\int_{B^N_{\delta, 1}} f(a_tn.x)\diff\!n=\frac{1}{|B^N_{\delta, 1}|}\int_{B^N_1}\diff\!n_1\int_{B^N_{\delta, 1}} f(a_tn_2.x)\diff\!n_2=\\
&\frac{1}{|B^N_{\delta, 1}|}\int_{B^N_1}\!\int_{B^N_{\delta, 1}} f(a_{d_1}n_1a_{d_2}n_2.x)\diff\!n_2\diff\!n_1+O\Bigl((e^t\delta)^{-1/2}\Scal(f)\Bigr).
\end{align*}
The above and the definition of $d_1$, thus, reduce the proof to showing that 
\[
\biggl|\frac{1}{|B^N_{\delta, 1}|}\int_{B^N_1}\int_{B^N_{\delta, 1}} f(a_{d_1}n_1a_{d_2}n_2.x)\diff\!n_2\diff\!n_1-\int f\diff\!m_X\biggr|
\ll \Scal(f)(e^{t}\delta)^{-\ref{k:thick-mixing-prop}}.
\]

We now turn to the proof of the above.  
Let $\vare$ be a constant which will be optimized and will be chosen to be $(e^t\delta)^{-\star}$. 
Since $d_2\geq 2|\log\eta|+\ref{E:non-div-main}$,  
Proposition~\ref{prop:Non-div-main}, applied to $u_rx$ for any $0\leq r\leq \delta$, implies that  
\[
\{s\in [0,1]: \inj(a_{d_2}v_su_rx)\leq \vare^2\}\leq \vare.
\]
This in particular implies the following: Put 
\[
\mathsf B:=\{n_2\in B^N_{\delta, 1}:\inj(a_{d_2}n_2x)\leq \vare^2 \},
\] 
then $|B^N_{\delta, 1}\setminus \mathsf B|\ll \vare |B^N_{\delta,1}|$.

In consequence, the following holds
\begin{multline*}
\frac{1}{|B^N_{\delta, 1}|}\int_{B^N_1}\int_{B^N_{\delta, 1}} f(a_{d_1}n_1a_{d_2}n_2.x)\diff\!n_2\diff\!n_1=\\
\frac{1}{|\mathsf B|}\int_{B^N_1}\int_{\mathsf B} f(a_{d_1}n_1a_{d_2}n_2.x)\diff\!n_2\diff\!n_1+ O(\vare \Scal(f)).
\end{multline*}
This reduces the investigations to the study of
\[
\frac{1}{|\mathsf B|}\int_{B^N_1}\int_{\mathsf B} f(a_{d_1}n_1a_{d_2}n_2.x)\diff\!n_2\diff\!n_1.
\]

Recall that $d_1=\frac12\log(e^t\delta)$. 
For every $n_2\in\mathsf B$, we have $z_{n_2}=a_{d_2}n_2x\in X_{\vare^2}$. Therefore, using e.g.~\cite[Prop.~4.1]{LM-PolyDensity}, 
we have 
\[
\biggl|\int_{B^N_1}f(a_{d_1}n_1.z_{n_2})\diff\! n_1-\int f\diff\!m_X\biggr|\ll \vare^{-\star}\Scal(f) e^{-\star d_1}=\vare^{-\star}\Scal(f)(e^t\delta)^{-\star}.
\]
Hence, if we choose $\vare$ to be a small negative power of $e^t\delta$, 
the above is $\ll \Scal(f)(e^t\delta)^{-\star}$. Averaging this over $\mathsf B$ finishes the proof.    
\end{proof}

Using Proposition~\ref{prop: equid trans horo} and an argument due to Venkatesh~\cite{Venkatesh-Sparse}, we obtain the following.

\begin{propos}\label{prop: 1-epsilon N}
There exist $\constk\label{k:mixing}\gg\mixexp^2$ so that the following holds. 
Let $0\leq \theta,\theta'<1$ and $0<\bpz\leq 0.1$. Let $\rho$ be a probability measure on $[0,1]$ which satisfies the following: there exists $C\geq 1$ so that 
\be\label{eq:C-rho-reg-N}
\rho(J)\leq C \bpz^{1-\theta}
\ee
for every interval $J$ of length $\bpz$.

Let $|\log\bpz|/4\leq t\leq (1-\theta')|\log \bpz|$, $0<\eta,\delta\leq 1$. Let $x\in X_{\eta}$, and assume 
\be\label{eq: comapring bpz and eta}
|\log\bpz|\geq 16|\log\eta|+8\ref{E:non-div-main}.
\ee
Then for all $f\in C_c^\infty(X)+\bbc\cdot1$, we have 
\begin{multline}
\biggl|\frac{1}{\delta}\int_0^1\int_0^\delta f(a_tu_rv_s.x)\diff\!r\diff\!\rho(s)-\int f\diff\!\mu_X\biggr|\\
\ll\Scal(f) \max\Bigl\{(C\bpz^{-\theta})^{1/2}(e^{t}\delta)^{-\ref{k:mixing}},\bpz^{\theta'}\Bigr\}\label{eq: in Prop 5.2}.
\end{multline}
where the implied constant depends on $X$.
\end{propos}

\begin{proof}
We will prove this for the case $G=\SL_2(\bbr)\times\SL_2(\bbr)$; 
the proof in the case $G=\SL_2(\bbc)$ is similar. 

Without loss of generality, we may assume $\int_Xf\diff\!\mu_X=0$.

Let $M\in\bbn$ be so that $1/M\leq \bpz\leq 1/(M-1)$.
For every $1\leq j\leq M$, let $I_j=\Big[\frac{j-1}{M}, \frac{j}{M}\Big)$; also put $s_j=\frac{2j-1}{2M}$ and $c_j=\rho(I_j)$ for all $j$.
Since $I_j$'s are disjoint, we have $\sum_j c_j=1$.

For all such $j$, let
\[
\mathsf B_j=\Bigl\{u_r v_s: 0\leq r\leq \delta , 0\leq s-s_j\leq \tfrac{\bpz}{4}\Bigr\}.
\]
In view of the choice of $M$, we have $\mathsf B_j\cap\mathsf B_{j'}=\emptyset$ for all $j\neq j'$.
Let $\varphi=\textstyle\sum_j (\delta\bpz/4)^{-1}c_j\mathbbm 1_{\mathsf B_j}$. 
Then $\int_N \varphi(r,s)\diff\!r\diff\!s=1$.

In view of~\eqref{eq:C-rho-reg-N}, we have $c_j\leq C\bpz^{1-\theta}$ for all $j$. This and the fact that $\mathsf B_j$'s are disjoint imply that  
\be\label{eq:phi-bound}
\varphi(n(z))\leq \max\{(\delta\bpz/4)^{-1}c_j: 1\leq j\leq M\}\ll C\bpz^{-\theta}\delta^{-1}
\ee
for all $n(z)\in N$; here and in what follows, $z=(r,s)$ and $\diff\!z=\diff\!r\diff\!s$. 

Using the fact that $I_j$'s are disjoint, we have 
\[
\int_0^1\int_0^\delta f(a_tu_rv_s.x)\diff\!r\diff\!\rho(s)=\sum_j\int_{I_j}\int_0^\delta f(a_tu_rv_s.x)\diff\!r\diff\!\rho(s);
\]
thus, we conclude that
\begin{align}
\label{eq:iint-rho-sj}&\bigg|\delta^{-1}\int_0^1\int_0^\delta f(a_tu_rv_s.x)\diff\!r\diff\!\rho(s)-\sum_{j}c_j\delta^{-1}\int_0^\delta f(a_tu_rv_{s_j}.x)\diff\!r\bigg|\\
\notag&\leq \sum_{j}\int_{I_j}\delta^{-1}\int_0^\delta \Big|f(a_tu_rv_s.x)-f(a_tu_rv_{s_j}.x)\Big|\diff\!r\diff\!\rho(s)\ll  \Scal(f)\bpz^{\theta'}
\end{align}
where we used the facts that $|s-s_j|\leq \bpz$ and $t\leq (1-\theta')|\log \bpz|$ in the last inequality.

In view of~\eqref{eq:iint-rho-sj}, thus, we need to bound $\sum_j\delta^{-1}\int c_jf(a_tu_rv_{s_j}x)\diff\!r$. Similar to~\eqref{eq:iint-rho-sj}, we can now make the following computation. 
 \begin{align}
\label{eq:mix-1st}&\biggl|\sum_{j} \delta^{-1}\int_0^\delta \!\!c_jf(a_tn(s_j,r).x)\diff\!r-\int_{N}\varphi(n(z))f(a_tn(z).x)\diff\!z\biggr|\\
\notag &\leq\sum_{j} \int_0^\delta\!\!(\bpz\delta/4)^{-1}c_j\int_{s_j}^{s_j+\frac{\bpz}{4}} \Bigl|f(a_tn(s_j,r).x)-f(a_tn(s,r).x)
\diff\!s\Bigr|\diff\!r\\
\notag&\ll\Scal(f)\bpz^{\theta'}
 \end{align}
where again we used the facts that $|s-s_j|\leq \bpz$ and $t\leq (1-\theta')|\log \bpz|$.
 
Thus, it suffices to investigate 
\[
 A_1=\int \varphi(n(z))f(a_tn(z).x)\diff\!z.
\] 

To that end, let $N\geq 1$ be so that $\Sob(g.f)\leq \|g\|^N\Sob(f)$. 
Let
\be\label{eq: choose ell for extra average}
\tau=\delta \cdot (\nuni^{t}\delta )^{-1+\frac{\ref{k:thick-mixing-prop}}{2N}},
\ee
and define
 \[
 A_2:=\tau^{-1}\int_0^\tau\int \varphi(n(z))f(a_t u_rn(z).x)\diff\!z\diff\!r.
 \]
Roughly speaking, we introduce an extra averaging in the direction of $U$. 
 
For every $0\leq r\leq \tau$, we have 
$|(\mathsf B_j+r)\Delta \mathsf B_j|\ll |\mathsf B_j|\tau/\delta $. Hence, 
\begin{align*}
\biggl|\int \varphi(z)&f(a_tu_rn(z).x)\diff\!z-\int \varphi(z)f(a_tn(z).x)\diff\!z\biggr|\\
&\leq\sum_j(\bpz\delta/4)^{-1}c_j\int_{(\mathsf B_j+r)\Delta \mathsf B_j}|f(a_tn(z)x)|\diff\!z\\
&\leq \sum_j(\bpz\delta/4)^{-1}c_j|\mathsf B_j|(\tau/\delta )\|f\|_\infty\\
&\leq \|f\|_\infty\cdot(\tau/\delta )\ll \Scal(f)\cdot(\tau/\delta );
\end{align*}
we used $|\mathsf B_j|=\bpz\delta/4$ for every $j$ and $\sum c_j=1$, in the second to the last inequality. 
Averaging the above over $[0,\tau]$, we conclude that 
\be\label{eq:mix-2nd}
|A_1-A_2|\ll \Scal(f)\tau/\delta \leq \Scal(f)(\nuni^t\delta )^{-1/2};
\ee 
where we used~\eqref{eq: choose ell for extra average}.
 
In consequence, we have reduced the proof to the study of $A_2$ to which we now turn. By the Cauchy-Schwarz inequality, we have
\[
|A_2|^2\leq \int \varphi(z)\biggl(\tau^{-1}\int_0^\tau f(a_tu_rn(z).x)\diff\!r\biggr)^2\diff\!z.
\]
Now using $\biggl(\tau^{-1}\int_0^\tau f(a_tu_rn(z).x)\diff\!r\biggr)^2\geq0$,~\eqref{eq:phi-bound}, and the above estimate, 
we conclude  
\begin{align}
 \notag|A_2|^2&\ll\frac{C\bpz^{-\theta}}{|B^N_{\delta ,1}|}\int_{B^N_{\delta ,1}}\biggl(\tau^{-1}\int_0^\tau f(a_tu_rn(z).x)\diff\!r\biggr)^2\diff\!z\\
\label{eq:mix-A2-est}&=\frac{1}{\tau^2}\int_0^\tau\int_0^\tau\frac{C\bpz^{-\theta}}{|B^N_{\delta ,1}|}\int_{B^N_{\delta ,1}}\hat f_{r_1,r_2}(a_tn(z).x)\diff\!z\diff\!r_1\diff\!r_2
 \end{align}
where $B^N_{\delta ,1}=\{u_rv_s: 0\leq r\leq \delta , 0\leq s\leq 1\}$ and for all 
$r_1,r_2\in [0,\tau]$ 
\[
 \hat f_{r_1,r_2}(y)=f(a_tu(r_1)a_{-t}.y)f(a_tu(r_2)a_{-t}.y).
\]
 
By~\eqref{eq: choose ell for extra average}, we have 
\be\label{eq:Sob-fhat}
 \Scal(\hat f_{r_1,r_2})\ll \Scal(f)^2 (e^t\tau)^N\ll \Scal(f)^2 (e^t\delta )^{\ref{k:thick-mixing-prop}/2}.
\ee
 
Now since $t\geq 4|\log\eta|+2\ref{E:non-div-main}$, by Proposition~\ref{prop: equid trans horo}, we have  
\[
 \biggl|\frac{1}{|B^N_{\delta ,1}|}\int_{B^N_{\delta ,1}}\hat f_{r_1,r_2}(a_tn(z)x)\diff\!z\biggr|=\int_X\hat f_{r_1,r_2}\diff\!\mu_X
 +  O(\Scal(\hat f_{r_1,r_2})(e^t\delta )^{-\ref{k:thick-mixing-prop}}).
\]
Recall from~\eqref{eq:Sob-fhat} that 
$\Scal(\hat f_{r_1,r_2})(e^t\delta )^{-\ref{k:thick-mixing-prop}}\leq \Scal(f)^2(e^t\delta )^{-\ref{k:thick-mixing-prop}/2}$. 
Altogether, we conclude that 
\begin{multline}\label{eq:mixing-ineq-main}
 \biggl|\frac{1}{|B^N_{\delta ,1}|}\int_{B^N_{\delta ,1}}\hat f_{r_1,r_2}(a_tn(z)x)\diff\!z\biggr|=\int_X\hat f_{r_1,r_2}\diff\!\mu_X \\
 + O(\Scal(f)^2 (e^t\delta )^{-\ref{k:thick-mixing-prop}/2}).
\end{multline}
 
 We now use estimates on the decay of matrix coefficients,~\eqref{eq: exp mixing}, and obtain the following: If $|r_1-r_2|>\tau \cdot(e^t\delta )^{-\frac{\ref{k:thick-mixing-prop}}{4N}}$, then 
 \be\label{eq:int-fhat}
 \biggl|\int_X\hat f_{r_1,r_2}(x)\diff\!\mu_X\biggr|\ll \Scal(f)^2(e^t\delta )^{-\ref{k:thick-mixing-prop}\mixexp/4N}
 \ee
 where we used $e^t\tau=(e^t\delta )^{\frac{\ref{k:thick-mixing-prop}}{2N}}$.
 
 Divide now the integral $\int_0^{\tau}\int_0^\tau$ in~\eqref{eq:mix-A2-est}
 into terms: one with $|r_1-r_2|\leq \tau \cdot(e^t\delta )^{-\frac{\ref{k:thick-mixing-prop}}{4N}}$ and the other its complement.
 We thus get from~\eqref{eq:mix-A2-est},~\eqref{eq:mixing-ineq-main}, and~\eqref{eq:int-fhat} that
 \[
 |A_2|\ll (C\bpz^{-\theta})^{1/2} \Scal(f)
 \biggl((e^t\delta )^{-\ref{k:thick-mixing-prop}\mixexp/4N}+(e^t\delta )^{-\ref{k:thick-mixing-prop}/4N}\biggr)^{1/2}.
 \]
This, together with~\eqref{eq:iint-rho-sj},~\eqref{eq:mix-1st}, and ~\eqref{eq:mix-2nd}, implies that the proposition holds with $\ref{k:mixing}=\ref{k:thick-mixing-prop}\mixexp/8N$.    
 \end{proof}

\section{Discretized dimension}\label{sec: discret dim}

Let $0<\alpha\leq 1$. We begin by defining a modified (and localized) $\alpha$-dimensional energy for finite subsets of $\mathbb R^d$.

Fix some norm $\|\;\|$ on $\bbr^d$ (below we will apply this for the cases $d=3$ and $d=1$). 
Let $0<\mfsc_0\leq 1$, and let $\Theta\subset\{w\in\bbr^d: \|w\|< \mfsc_0\}$ be a finite set. For $\trct\geq1$, define 
$\eng_{\Theta, \trct}: \Theta\to (0,\infty)$ as follows: If $\#\Theta\leq \trct$, put 
\[
\eng_{\Theta, \trct}(w)=\mfsc_0^{-\alpha},\quad\text{for all $w\in \Theta$,}
\]
and if $\#\Theta>\trct$, put 
\[
\eng_{\Theta,\trct}(w)=\min\left\{\sum_{\Theta'}\|w-w'\|^{-\alpha}: \begin{array}{c}\Theta'\subset \Theta\text{ and }\\ \#(\Theta\setminus\Theta')=\trct\end{array}\right\}.
\] 
We will also use this notation for finite subsets of $\rfrak$, which as a vector space is $\simeq \bbr^3$. 

\subsection{A projection theorem}\label{sec: proj thm}

We now state a projection theorem which plays a crucial role in our argument. Indeed, this theorem (as stated here) will be used in improving the dimension phase, \S\ref{sec: MF and IG}--\S\ref{sec: proof main prop}; a modified version of it (Theorem~\ref{thm: proj thm 2}) will also be used in the endgame phase, \S\ref{sec: equidistribution}.

\begin{thm}\label{thm: proj thm}
Let $0<\alpha\leq1$, and let $0<\pvare<0.01\alpha$. 
Let $\egbd\geq 1$ be large enough depending on $\pvare$, and let $\Theta\subset B_\rfrak(0,\mfsc_0)$ be a finite set
satisfying  
\be\label{eq: energy bd proj thm main}
\eng_{\Theta,\trct}(w)\leq \egbd\quad\text{for every $w\in\Theta$ and some $\trct\geq 1$.}
\ee
Consider the one-parameter family of projections $\xi_r : \rfrak \to \R$ given by
\[
\xi_r(w)=(\Ad(u_r)w)_{12}=-w_{21}r^2-2w_{11}r+w_{12}.
\]

Let $J\subset [0,1]$ be an interval with $|J|\geq 10^{-6}$. 
There exists a subset $J'\subset J$ with $|J\setminus J'|\leq  L_1 \egbd^{-\pvare^2}$, 
where $L_1=L\pvare^{-L}$ for an absolute constant $L$, so that the following holds. 
Let $r\in J'$, then there exists a subset $\Theta_{r}\subset \Theta$ with 
\[
\#(\Theta\setminus \Theta_{r})\leq L_1 \egbd^{-\pvare^2}\cdot (\#\Theta)
\]
such that the projected set $\xi_{r}(\Theta)$ satisfies that
\[
\eng_{\xi_{r}(\Theta),\trct_1}(w)\leq \egbd_1 \qquad\textup{for all $w\in \xi_r(\Theta_{r})$}
\] 
where $\trct_1=L_1\egbd^{7\pvare}\trct$, $\egbd_1=L_1 \egbd^{1+8\pvare}$.
\end{thm}

This theorem will be proved in Appendix~\ref{sec: proof proj}. We also refer to that section for references and historic comments.

\subsection{Regularization lemmas}\label{sec: regularization lemmas}
It will be more convenient to work with finite sets which have more {\em regular} structure, see~\cite[Lemma 5.2]{BFLM} and~\cite[\S2]{Bour-Proj}. In this section we recall this construction, tailored to the applications in our paper. 

\medskip

Let $t, m_0\geq 1$ and $0<\vare<1$ be three parameters: $t$ is large and arbitrary, $m_0$ is moderate and fixed, and $\vare$ is small and fixed; in particular, our estimates are allowed to depend on $m_0$ and $\vare$, but not on $t$. Let $e^{-0.01\vare t}\leq \eta\leq 1$ and let $\mfsc_0=e^{-\sqrt\vare t}\eta$.

Let $F\subset B_\rfrak(0,1)$ with 
\[
e^{t/2}\leq \#F\leq e^{\fub t}.
\]
For all $w\in F$, let $F_w=B_\rfrak(w, \mfsc_0)\cap F$, and assume that 
\be\label{eq: tree dec eng bd}
\eng_{F_w,\trct}(w')\leq \egbd \qquad\text{for all $w'\in I_w$}
\ee
where $1\leq \trct\leq e^{0.01\vare t}$ and $\egbd>0$ satisfying the following   
\be\label{eq: a priori bound on egbd}
\egbd\leq e^{(m_0+1)t}.
\ee
Note that there is $w\in F$ so that $\#F_w\geq e^{0.5t-4\sqrt\vare}>e^{9t/20}$. Thus~\eqref{eq: tree dec eng bd} and the the fact that 
$1\leq \trct\leq e^{0.01\vare t}$ imply that indeed, $\egbd\geq e^{0.4t}$.

\medskip

Let $\beta=e^{-\kappa t}$ for some $\kappa$ satisfying $0<\kappa (m_0+1)\leq 10^{-6}\vare$. Fix ${\mathsf M}\in\bbn$, large enough, so that both of the following hold  
\be\label{eq: condition on M}
\text{$2^{-{\mathsf M}}(m_0+1)< \kappa/100\quad$ and $\quad 6{\mathsf M}<2^{\kappa {\mathsf M}/100}$}.
\ee
Define $k_0:=\lfloor (-\log_2\mfsc_0)/{{\mathsf M}}\rfloor$ and 
$k_1:=\lceil(1+\alpha^{-1}\log_2\egbd) /{\mathsf M} \rceil+1$; note that    
\be\label{eq: F cap the smallest cube k1}
2^{({\mathsf M}k_1-1)\alpha}>\Upsilon.
\ee
In view of~\eqref{eq: tree dec eng bd} and~\eqref{eq: F cap the smallest cube k1}, we have 
\be\label{eq: F cap the smallest cube}
\#\Bigl(B_\rfrak(w,2^{-{\mathsf M}k_1})\cap F\Bigr)\leq \trct \qquad\text{for all $w\in\rfrak$.}
\ee

For every $k_0\leq k\leq k_1$, let $\mathcal Q_{{\mathsf M}k}$ 
denote the collection of $2^{-{\mathsf M}k}$-cubes
\[
\{w\in\rfrak: w_{rs}\in [\tfrac{n_{rs}}{2^{{\mathsf M}k}},\tfrac{n_{rs}+1}{2^{{\mathsf M}k}}), r,s=1,2\}
\]
for some trace zero $(n_{ij})\in \operatorname{Mat}_2 (\bbz)$ if $G=\SL_2(\bbr)\times\SL_2(\bbr)$ and with the obvious modification when $G=\SL_2(\bbc)$.

\begin{lemma}\label{lem: regular tree decomposition}\label{lem: trimming cone 2}
For all large enough $t$,
we can write $F=F'\bigcup (\bigcup_{i=1}^{N}F_i)$ (a disjoint union) with 
\[ 
\text{$\#F'<\beta^{1/4}\cdot(\#F)\quad$ and $\quad\#F_i\geq \beta^2\cdot(\#F)$}
\] 
so that the following holds. For every $i$ and every $k_0-10\leq k\leq k_1$, there exists some $\tau_{ik}$ so that 
for every cube $Q \in \mathcal Q_{{\mathsf M}k}$ we have 
\be\label{eq: regular tree}
2^{{\mathsf M}(\tau_{ik}-2)}\leq \#F_i\cap Q\leq 2^{{\mathsf M}\tau_{ik}}\quad\text{or}\quad F_i\cap Q=\emptyset.
\ee
Moreover, for every $i$ and every cube $Q\in \mathcal Q_{{\mathsf M}k_0}$, we have 
\be\label{eq: trivial lower bd}
\#F_i\cap Q\geq e^{-4\sqrt\vare t}\cdot(\#F_i)\quad\text{or}\quad F_i\cap Q=\emptyset.
\ee
\end{lemma}

\begin{proof}
This lemma is essentially proved in~\cite[Lemma 5.2]{BFLM}. We explicate this construction for completeness. Let us begin with a preparatory step before applying the construction in loc.\ cit.; this step is also present in~\cite[Lemma 5.2]{BFLM}.  

\begin{claim}
We may write $F=F''\bigcup (\cup\hat F_j)$ (disjoint union) satisfying that $\#F''\leq \beta^{1/2}\cdot (\#F)$ and for each $\hat F_j$, there exists some $w_j\in \rfrak$ so that if $Q, Q'\in\mathcal Q_{\mathsf Mk}$ intersect $\hat F_j+w_j$ non-trivially, the distance between $Q\cap(\hat F_j+w_j)$ and $Q'\cap(\hat F_j+w_j)$ is at least $2^{-\mathsf Mk-\mathsf M}$. 
\end{claim}

\begin{proof}[Proof of the Claim]
For every $k_0-10\leq k\leq k_1$, the density of 
\[
D_k=\Bigl\{w\in \rfrak: \exists r,s, \text{ such that } w_{rs}\in 2^{-\mathsf k}(\bbz+[0, 2^{-{\mathsf M}}])\Bigr\}
\]
in $\rfrak$ is $\leq 3\times 2^{-{\mathsf M}}$. Using the definition, we conclude that the density of $D:=\bigcup_k D_k$ in $\rfrak$ is 
$\geq 1-(1-3\times 2^{-{\mathsf M}})^{k_1-k_0+1}$. 

Hence there exists some $w_1$ so that 
\[
\#(F+w_1\setminus D)\geq (1-3\times 2^{-{\mathsf M}})^{k_1-k_0+1}\cdot (\#F)\gg \beta^{0.1}\cdot (\#F),
\]
where we used $k_1-k_0\leq 2(m_0+1)t$ and the fact that $2^{-{\mathsf M}}(m_0+1)\leq \kappa/100$.

Note that $F+w_1\subset B_\rfrak(0,10)$, and put 
\[
\hat F_1:=(F+w_1\setminus D)-w_1.
\]
Cover $B_\rfrak(0,10)$ with dyadic cubes $\{Q_r\}$ 
in $\mathcal Q_{{\mathsf M}k_1}$, and set
\[
\hat Q_{1}^r=\Bigl((F+w_1\setminus D)\cap Q_r\Bigr)-w_1
\]
for any $r$ so that $(F+w_1\setminus D)\cap Q_r\neq \emptyset$.

Assuming $\hat F_1,\ldots, \hat F_n$ are defined, repeat the above with $F\setminus (\cup_{i=1}^n \hat F_i)$ if this set has $\geq \beta^{1/2}\cdot(\#F)$ many elements. 
Each set thus obtain satisfies 
\[
\#\hat F_j\gg \beta^{0.6}\cdot (\#F).
\]
In consequence, this process terminates after $N'\ll\beta^{-0.6}$ many steps and yields
sets $\hat F_1,\ldots, \hat F_{N'}$. Define $\{\hat Q_{j}^r\}$ similarly for each $\hat F_j$.

Let $F''=F\setminus (\bigcup \hat F_j)$, then $\#F''\leq \beta^{1/2}\cdot (\#F)$. The claim follows.
\end{proof}

We now further subdivide the sets $\hat F_j$ so that the resulting sets satisfy~\eqref{eq: regular tree} and~\eqref{eq: trivial lower bd}.
Fix some $j$. We will begin trimming $\hat F_j$ from the smallest cells, i.e., $2^{-{\mathsf M}k_1}$-cubes. 
In view of~\eqref{eq: F cap the smallest cube}, 
$\# \hat Q_{j}^r\ll \trct$
for all $r$. For $\ell\in\bbn$, let 
\[
\hat F_{j\ell}=\bigcup\{\hat Q_{j}^r: 2^{-\ell-1}\trct\leq \#\hat Q_{j}^r\leq 2^{-\ell}\trct\}.
\]
Let $\hat F_j'=\bigcup_\ell \Bigl\{\hat F_{j\ell}: \#\hat F_{j\ell}\leq \beta\cdot (\# \hat F_j)\Bigr\}$.

Recall that $1\leq \trct\leq e^{0.01\vare t}$ and $\beta=e^{-\kappa t}$. Therefore,  
\[
\#\Bigl(\bigcup F'_j\Bigr)\ll \sum\#F'_j\ll N'\cdot \beta\cdot (\# \hat F_j)\cdot\log\trct< \beta^{0.3}\cdot(\#F),
\] 
so long as $t$ is large enough. 
Put $\bar F=F''\bigcup(\bigcup F'_j)$, then $\#\bar F<  2\beta^{0.3}\cdot(\#F)$. 

Thanks to this and the claim we can now apply the construction in~\cite[p.~246]{BFLM}, with $\hat F_{j\ell}$ and dyadic cubes $2^{-{\mathsf M}k}$ with $k_0-10\leq k\leq k_1$, and write
\[
\hat F_{j\ell}=F'_{j\ell}\bigcup\Bigl(\textstyle\bigcup_q \hat F_{j\ell}^q\Bigr)
\]
so that $\#F'_{j\ell}\ll \beta\cdot (\#\hat F_{j\ell})$, Moreover, for every $q$, $F_{j\ell}^q$ satisfies~\eqref{eq: regular tree} and  
\[
\#\hat F_{j\ell}^q\gg (6{\mathsf M})^{-k_1}\cdot (\#\hat F_{j\ell})\gg 2^{-\kappa {\mathsf M}k_1/10}\cdot (\#\hat F_{j\ell})\gg \beta^{0.1}\cdot (\#\hat F_{j\ell});
\]
we used $6{\mathsf M}\leq 2^{\kappa {\mathsf M}/10}$, see~\eqref{eq: condition on M}, in the second inequality, and used the definitions of $k_1$ and $\beta$ together with~\eqref{eq: a priori bound on egbd} in the last inequality. 

Recall now that $\#\hat F_{j\ell}\geq \beta\cdot(\#\hat F_j)\geq \beta^{1.6}\cdot (\#F)$. Hence,
\[
\#\hat F_{j\ell}^q\geq \beta^2\cdot (\#F)
\]
if we assume $t$ is large enough to account for implied multiplicative constant. 

In view of~\eqref{eq: regular tree}, if for some $j,\ell,q$ and $2^{-{\mathsf M}k_0}$ 
cube $Q$ with $F_{j\ell}^q\cap Q\neq \emptyset$ we have $\#(F_{j\ell}^q\cap Q)\leq e^{-4\sqrt\vare t}\cdot(\#F_{j\ell}^q)$, then~\eqref{eq: regular tree}, applied with $k_0$, implies  
\[
\#F_{j\ell}^q\ll e^{-\sqrt\vare t}\cdot (\#F_{j\ell}^q),
\]
which is a contradiction if $t$ is large enough. 

Finally, note that as it was done 
\[
\#\bigcup_{j,\ell} F'_{j\ell}\leq N'\cdot \log R\cdot\beta\cdot (\#F)< \beta^{0.3}\cdot(\#F).
\]
The lemma thus holds with $F'=\bar F\bigcup (\bigcup_{j,\ell} F'_{j\ell})$ and $\{\hat F_{j\ell}^q: j, \ell, q\}$.
\end{proof}

Recall that for all $w\in F$, we put $F_w=B_\rfrak(w, \mfsc_0)\cap F$.
Assume now that for some $\mathsf C\leq e^{10\vare t}$ for all $w'\in F_w$, we have 
\be\label{eq: tree dec eng bd'}
\eng_{F_w,\trct}(w')\leq \mathsf C\cdot \mfsc_0^{-\alpha}\cdot (\#F_w).
\ee
Since $e^{t}\leq \#F\leq e^{\fub t}$ and $\mfsc_0=e^{-\lfloor \sqrt\vare t\rfloor}\eta$ where $\eta>e^{-0.01\vare t}$,~\eqref{eq: tree dec eng bd'} implies 
\[
\eng_{F_w,\trct}(w')\leq  e^{(m_0+2\sqrt\vare)t}.
\]
In particular,~\eqref{eq: tree dec eng bd} holds with $\egbd=e^{(m_0+2\sqrt\vare)t}$, and Lemma~\ref{lem: regular tree decomposition} is applicable.

\begin{lemma}\label{lem: regular tree and local eg}
Let $F=F'\bigcup (\bigcup_{i=1}^{N}F_i)$ be a decomposition of $F$ as in Lemma~\ref{lem: regular tree decomposition}.
Then for every $i$ and all $w\in F_i$ we have 
\[
\eng_{F_{i,w}, \trct}(w')\leq \mathsf C\beta^{-4}\mfsc_0^{-\alpha}\cdot (\#F_{i,w})
\]
for all $w'\in F_{i,w}:=F_i\cap B_\rfrak(w,\mfsc_0)$. 
\end{lemma}

\begin{proof}
Let $k_0\leq k\leq k_1$ and let $w\in F_i$.  Then using~\eqref{eq: tree dec eng bd'} 
and the fact that $\trct\leq 2^{0.01\vare t}$, we conclude that  
\be\label{eq: measure of cubes tree}
\begin{aligned}
\#\Bigl(B\Bigl(w,2^{-{\mathsf M}k}\Bigr)\cap F_i\Bigr)&\leq \#(B\Bigl(w,2^{-{\mathsf M}k}\Bigr)\cap F)\\
&\leq 2^{10{\mathsf M}}\mathsf C\cdot (2^{-{\mathsf M}k}/\mfsc_0)^\alpha\cdot (\#F_w).
\end{aligned}
\ee

Let $Q_0\in\mathcal Q_{{\mathsf M}k_0}$ be so that $Q_0\cap F_i\neq \emptyset$, and  
let $w\in F_i$. Then $B(w, 2^{-{\mathsf M}k_0})$ can be covered by at most $8$ cubes in $\mathcal Q_{{\mathsf M}k_0}$, moreover, it contains at least one cube in $\mathcal Q_{{\mathsf M}(k_0+1)}$ which also contains $w$. Thus by~\eqref{eq: regular tree}, 
\be\label{eq: I-i-w almost constant}
2^{-3-4{\mathsf M}}(\#Q_0\cap F_i)\leq \# F_{i,w}\leq 2^{3+2{\mathsf M}}(\#Q_0\cap F_i)
\ee

We claim that there exists $w_i\in F_i$ so that 
\be\label{eq: a typical w in Fi}
\begin{aligned}
\#F_{w_i}=\#(B_\rfrak(w_i,\mfsc_0)\cap F)&\leq \beta^{-3}\cdot (\#(B_\rfrak(w,\mfsc_0)\cap F_i)\\
&=\beta^{-3}\cdot(\#F_{i,w_i}).
\end{aligned}
\ee

Let us assume~\eqref{eq: a typical w in Fi} and finish the proof. 
Note that~\eqref{eq: measure of cubes tree} applied with $w=w_i$, together with~\eqref{eq: a typical w in Fi}, implies that
\be\label{eq: localized dim at w_i}
\#\Bigl(B_\rfrak\Bigl(w_i,2^{-{\mathsf M}k}\Bigr)\cap F_i\Bigr)\leq 2^{\star\mathsf M}\beta^{-3}\mathsf C\cdot (2^{-{\mathsf M}k}/\mfsc_0)^\alpha\cdot (\#F_{i,w_i}),
\ee
where we assumed $t$ is large.

Let now $k_0+2\leq k'\leq k_1$. Then 
\[
\#\Bigl(B_\rfrak\Bigl(w,2^{-{\mathsf M}k'}\Bigr)\cap F_i\Bigr)\leq \#(Q\cap F_i)
\] 
where $Q$ is a $2^{-{\mathsf M}(k'-1)}$ cube which contains $B_\rfrak\Bigl(w,2^{-{\mathsf M}k'}\Bigr)$. 
Let $Q'$ be a cube of same size which contains $w_i$, then using~\eqref{eq: regular tree}, we have 
\[
\#(Q\cap F_i)\leq 2^{2{\mathsf M}}\cdot(\#(Q'\cap F_i)).
\]
Since $Q'\subset B_\rfrak(w_i, 2^{-{\mathsf M}(k'-2)})$, 
using~\eqref{eq: localized dim at w_i} with $k=k'-2$, we conclude that 
\[
\begin{aligned}
\#\Bigl(B_\rfrak\Bigl(w,2^{-{\mathsf M}k'})\cap F_i\Bigr)&\leq 2^{2{\mathsf M}}\#(B_\rfrak(w_i, 2^{-{\mathsf M}(k'-2)})\cap F_i)\\
&\leq 2^{\star{\mathsf M}}\beta^{-3}\mathsf C\cdot (2^{-{\mathsf M}(k'-2)}/\mfsc_0)^\alpha\cdot (\#F_{i,w_i}).
\end{aligned}
\]
This and~\eqref{eq: I-i-w almost constant} (whic is used to replace $F_{i,w_i}$ with $F_{i,w}$) imply that 
\be\label{eq: num in Bwk' vs I-i-w}
\#\Bigl(B_\rfrak\Bigl(w,2^{-{\mathsf M}k'})\cap F_i\Bigr)\leq 2^{\star{\mathsf M}}\beta^{-3}\mathsf C\cdot(2^{-{\mathsf M}k'}/\mfsc_0)^\alpha\cdot (\#F_{i,w}).
\ee
Since $\#(B_\rfrak(w, 2^{-{\mathsf M}k_1})\cap F_i)\leq \#(B_\rfrak(w, 2^{-{\mathsf M}k_1})\cap F)\leq \trct$, see~\eqref{eq: F cap the smallest cube}, 
from~\eqref{eq: num in Bwk' vs I-i-w} we conclude that 
\[
\begin{aligned}
\eng_{F_{i,w}, \trct}(w)&\leq k_1 2^{\star{\mathsf M}}\beta^{-3}\mathsf C\cdot(\mfsc_0)^{-\alpha}\cdot (\#F_{i,w})\\
&\leq \beta^{-4}\mathsf C\cdot(\mfsc_0)^{-\alpha}\cdot (\#F_{i,w}),
\end{aligned}
\]
so long at $t$ is large enough. 
This completes the proof assuming~\eqref{eq: a typical w in Fi}.

We now prove~\eqref{eq: a typical w in Fi}. 
Let $\mathcal B=\{B_\rfrak(v, \mfsc_0): v\in F_i\}$ be a covering of $F_i$ with multiplicity $\leq K$.
Then  
\begin{align*}
\sum\#(B(v)\cap F)&\leq K\cdot \Bigl(\#\bigcup (B(v)\cap F)\Bigr)\leq K\cdot (\#F)\\
&\leq K\beta^{-2}\cdot(\#F_i)\leq K\beta^{-2}\sum\#(B(v)\cap F_i), 
\end{align*} 
where we write $B(v)$ for $B_\rfrak(v, \mfsc_0)$. We conclude that 
for some $w_i\in F_i$,
\begin{align*}
  \#F_{w_i}=\#(B(w_i)\cap F)&\leq K\beta^{-2}\cdot \Bigl(\#(B(w_i)\cap F_i)\Bigr)\\
  &\leq \beta^{-3}\cdot\Bigl(\#(B(w_i)\cap F_i)\Bigr)=\beta^{-3}(\#F_{i,w_i})  
\end{align*}
as was claimed in~\eqref{eq: a typical w in Fi}. 
\end{proof}


\section{Boxes, complexity and the Folner property}\label{sec: folner}

For every $\ell>0$, let $\rwm_\ell$ be the probability measure on $H$ defined by 
\be\label{eq: def rwm ell}
\rwm_\ell(\varphi)=\ave\varphi(a_{\ell}\uvk)\uvkd\qquad\text{for all $\varphi\in C_c(H)$.}
\ee

Our goal in this section and the next is to show that $\rwm_\ell^{(d)}$ (the $d$-fold convolution of $\rwm_\ell$) can be approximated with a convex combination of certain natural measures supported on a finite union of local $H$ orbits, see \S\ref{sec: cone and mu cone}. 

This section will lay the groundwork for this decomposition. 
In particular, we will prove a covering lemma, Lemma~\ref{lem: E good h0}, 
define the notion of an admissible measure, \S\ref{sec: cone and mu cone}, and
prove a certain almost invariance property for a class of measures appearing in our analysis, Lemmas~\ref{lem: Folner property} and~\ref{lem: Folner property 2}.

\subsection*{Covering lemmas}
We will fix $0<\injr\leq0.01\eta_X$ and $\beta=\eta^2$ throughout this section. 
For $m\geq0$, we introduce the shorthand notation $\umt^H_m$ for 
\be\label{eq: def BH}
\umt_{\eta,\beta^2,m}^H=\Bigl\{u^-_s: |s|\leq \beta^2 \nuni^{-m}\Bigr\}\cdot\{a_\tau: |\tau|\leq \beta^2\}\cdot U_\eta,
\ee
where for every $\delta>0$, let $U_\delta=\{u_r: |r|\leq \delta\}$, see~\eqref{eq: def B ell beta}.

Define $\umt^G_{m}\subset G$ by thickening $\umt^H_m$ 
in the transversal direction as follows: 
\be\label{eq: def O ell C}
\umt^G_{m}:=\umt^H_m\cdot \exp(B_\rfrak(0,2\beta^2)).
\ee

We begin by fixing a particular covering of $X_{2\eta}$.   

\begin{lemma}\label{lem: E good h0}
For every $m\geq 0$, there exists a covering 
\[
\Big\{\umt^G_{m}.y_{j}: j\in \mathcal J_m, y_j\in X_{3\eta/2}\Big\}
\] 
of $X_{2\injr}$ with multiplicity $K$, depending only on $X$.
In particular, $\#\mathcal J_m\ll\eta^{-1}\beta^{-10}\nuni^{m}$.
\end{lemma}

\begin{proof}
We first prove the following.
There exists a covering 
\[
\Big\{\Bigl(\boxHs_{\beta^2}\cdot\boxU_{\eta}\cdot\exp\bigl(B_\rfrak(0,\beta^2)\bigr)\Bigr).\hat y_k: k\in \mathcal K, \hat y_k\in X_{2\eta}\Big\}
\] 
of $X_{2\injr}$ with multiplicity $O(1)$ depending only on $X$.

Let us write $\bar{\mathsf B}^G_{\eta,\beta^2}=\boxHs_{\beta^2}\cdot \boxU_{\eta}\cdot\exp(B_\rfrak(0,\beta^2))$. Then 
\be\label{eq: doubling propert Bbar}
\Bigl(\bar{\mathsf B}^G_{0.1\eta, 0.1\beta^2}\Bigr)^{-1}\cdot\Bigl(\bar{\mathsf B}^G_{0.1\eta,0.1\beta^2}\Bigr)\subset 
\Bigl(\bar{\mathsf B}^G_{10\eta,10\beta^2}\Bigr),
\ee
see Lemma~\ref{lem: BCH}. 

Let $\{\hat y_k\in X_{2\eta}: k\in\mathcal K\}$ be maximal with the following property
\[
\bar{\mathsf B}^G_{0.01\eta,0.01\beta^2}.\hat y_i\cap \bar{\mathsf B}^G_{0.01\eta,0.01\beta^2}.\hat y_j=\emptyset\quad\text{ for all $i\neq j$.}
\]
In view of~\eqref{eq: doubling propert Bbar} thus $\{\bar{\mathsf B}^G_{\eta,\beta^2}.\hat y_k:k\in\mathcal K\}$ covers $X_{2\eta}$ with multiplicity $O(1)$. Since $m_G(\bar{\mathsf B}^G_{\eta,\beta^2})\asymp \eta\beta^{10}$, we also conclude that $\mathcal K\ll \eta^{-1}\beta^{-10}$.

The following generalization will also be used: for any $1\leq c\leq 100$, 
\be\label{eq: up-graded cov step 0}
\{\bar{\mathsf B}^G_{c\eta,c\beta^2}.\hat y_k:k\in\mathcal K\}
\ee
covers $X_{2\eta}$ with multiplicity $\leq K_1$, depending only on $X$.

Let now $m\geq 0$, and recall that we write $\umt^H_m$ for $\umt^H_{\eta,\beta^2, m}$. 
Fix a subset $\mathcal H\subset \umt_0^H$ which is maximal with the following property 
\[
\umt^H_{0.01\eta, 0.01\beta^2, m}h\cap \umt^H_{0.01\eta, 0.01\beta^2, m}h'=\emptyset,
\]
for all $h\neq h'\in \mathcal H$. Since 
\[
m_H(\umt^H_{0.01\eta, 0.01\beta^2, m})\asymp \nuni^{-m}m_H(\umt_0^H),
\]
we have $\#\mathcal H\ll\nuni^{m}$ where the implied constants are absolute. 
Furthermore,    
\[
\Bigl(\umt^H_{0.01\eta, 0.01\beta^2, m}\Bigr)^{\pm1}\cdot \umt^H_{0.01\eta, 0.01\beta^2, m}
\subset \umt^H_{0.1\eta, 0.1\beta^2, m}.
\]
Thus $\{\umt^H_mh_j: h_j\in\mathcal H\}$ covers $\umt_0^H=\boxHs_{\beta^2}\cdot \boxU_{\eta}$ with multiplicity $\ll K_2$. 

Combining these two coverings, we obtain a covering 
\[
\{\umt^H_mh_j\exp(B_\rfrak(0,\beta^2)).\hat y_k: h_j\in \mathcal H, k\in\mathcal K\}.
\]
of $X_{2\eta}$. Note further that 
\[
\umt^H_mh_j\exp(B_\rfrak(0,\beta^2))=\umt^H_m\exp\Big(\Ad(h_j)B_\rfrak(0,\beta^2)\Big)h_j\subset \umt^G_m h_j;
\]
where we used the fact that $\Ad(h_j)B_\rfrak(0,\beta^2)\subset B_\rfrak(0,2\beta^2)$ in the final inclusion above --- this holds since $\|h_j-I\|\leq2\beta^2$ and $\beta$ is small. 

Finally note that since $\hat y_k\in X_{2\eta}$ and $\|h_j-I\|\leq2\beta^2$, we have $h_j\hat y_k\in X_{19\eta/10}$, for every $j,k$. Altogether, we obtain a covering 
\[
\{\umt^G_m.y_j:j\in \mathcal J, y_j\in X_{19\eta/10}\}=\{\umt^G_m.h_j\hat y_k:h_j\in\mathcal H, k\in\mathcal K\}
\]
of $X_{2\eta}$. 

We claim: the multiplicity of this covering is $\leq K_1K_2$. Suppose $z\in X$ belongs to $M>K_1K_2$ sets 
$\umt^G_m.h_j\hat y_k$. That is, for $i=1,\ldots, M$, we have  
\[
z=\sfh_i\exp(w_i)h_{j_i}\hat y_{k_i}\in \umt^G_m.h_{j_i}\hat y_{k_i}.
\]  
Note that $\umt^G_mh_{j_i}\subset \bar{\mathsf B}^G_{10\eta,10\beta^2}$. 
Thus in view of~\eqref{eq: up-graded cov step 0} and the fact that for all $\hat y_k$, $g\mapsto g\hat y_{k}$ 
is injective over $\boxG_{10\eta}$, we conclude that for at least 
$M/K_1>K_2$ many choices of $i$ we have $\sfh_i\exp(w_i)h_{j_i}=\sfh\exp(w)h$. This implies 
\[
\sfh_ih_{j_i}\exp(\Ad(h_{j_i}^{-1})w_i)=\sfh h\exp(\Ad(h^{-1})w).
\]
Since the map $(h,w)\mapsto h\exp(w)$ is injective on $\boxH_{100\eta}\times B_\rfrak(0,100\eta)$, 
for more than $K_2$ choices of $i$ we have $\sfh_ih_{j_i}=\sfh h$. This contradicts 
the choice of $K_2$ and completes the proof. 
\end{proof}

\subsection*{A density function}
For every $m\geq 0$, we fix a covering 
\[
\{\umt^G_m y_j: y_j\in X_{3\eta/2}, j\in \mathcal J_m\}
\] 
as in Lemma~\ref{lem: E good h0}. For every $z\in X$, let $\mathsf k_m(z)=\#\{j: z\in \umt^G_m.y_j\}$.
Then $1\leq \mathsf k_m(z)\leq K$. Define 
\[
\density_{m}: X\to \{1/d: d=1,\ldots, K\} \qquad\text{by $\;\density_{m}(z):={1}/{\mathsf k_m(z)}$}.
\]
For every $j\in\mathcal J_m$, put 
\[
\density_{m,j}=\density_{m}|_{\umt^G_m.y_j}.
\]
Note that $\sum_{j}\density_{m,j}(z)=1$ for all $z\in X$.

\medskip

\subsection{Boxes and complexity}\label{sec: box complexity} 
Let $\mathsf{prd}:\mathbb R^3\to H$ be the map
\[
\mathsf{prd}(s,\tau,r)= u^-_sa_\tau u_r.
\]
A subset $\mathsf D\subset H$ will be called a {\em box} if there exist intervals 
$I^{\bigcdot}\subset\bbr$ (for $\bigcdot=\pm, 0$) so that 
\[
\mathsf D=\mathsf{prd}(I^-\times I^0\times I^+).
\] 

We say $\Xi\subset H$ has complexity bounded by $L$ (or at most $L$) if 
$\Xi=\bigcup_{1}^L \Xi_i$
where each $\Xi_i$ is a box.

For every interval $I\subset\bbr$, let $\partial I=\partial_{100\eta|I|}I$ (recall that $\eta=\beta^{1/2}$), and put $\mathring I=I\setminus\partial I$.
Given a box $\mathsf D=\mathsf{prd}(I^-\times I^0\times I^+)$, we let 
\begin{subequations}
\begin{align}
\label{eq: def mathring D}&\mathring{\mathsf D}=\mathsf{prd}\Bigl(\mathring{I^-}\times\mathring{I^0}\times\mathring{I^+}\Bigr) \quad\text{and}\\ 
\label{eq: def partial D}&\partial\mathsf D=\mathsf D\setminus \mathring{\mathsf D}.
\end{align}
\end{subequations}

More generally, if $\mathsf D=\mathsf{prd}(I^-\times I^0\times I^+)$ is a box, and 
$\Xi\subset \mathsf D$ has complexity bounded by $L$, we define 
$\partial\Xi:=\bigcup\partial\Xi_i$ and  
\be\label{eq: def mathring Xi}
\mathring\Xi_{\mathsf D}:=\bigcup\mathring\Xi_i
\ee
where the union is taken over those $i$ so that $\Xi_i=\mathsf{prd}(I_i^-\times I_i^0\times I_i^+)$ with 
$|I_i^{\bigcdot}|\geq 100\eta|I^{\bigcdot}|$ for $\bigcdot=\pm,0$.

\begin{lemma}\label{lem: density level sets}
There exists $K'$ depending only on $X$ so that the following holds. 
Let $j\in\mathcal J_m$ and $w\in B_\rfrak(0,2\beta^2)$. 
Then for every $1\leq \sfk\leq K$, there is $\Xi^\sfk=\Xi^\sfk(j,w)\subset \umt^H_m$ with complexity at most $K'$ so that 
\[
\begin{aligned}&\density_{m,j}(z)=1/\sfk\quad\text{for all $z\in\Xi^\sfk.\exp(w)y_j$ and}\\
&\Bigl|\{z\in\umt^H_m.\exp(w)y_j: \density_{m,j}(z)=1/\sfk\}\setminus \Bigl(\Xi^\sfk.\exp(w)y_j\Bigr)\Bigr|\ll \eta |\umt^H_m|
\end{aligned}
\]
where the implied constant depends only on $X$. 
\end{lemma}

\begin{proof}
We will use that $(h,v)\mapsto h\exp(v)y$ is injective over 
$\boxH_{10\eta}\times B_\rfrak(0,10\eta)$ for all $y\in X_\eta$, and that 
\[
(\umt^H_m)^{\pm1}\cdot (\umt^H_m)^{\pm1}\cdot (\umt^H_m)^{\pm1}\subset \umt^H_{10\eta, 10\beta^2, m} \quad\text{for all $m\geq 0$}.
\]

Let $\mathcal Y_j=\{y_{k_i}: \umt^G_m.y_j\cap \umt^G_m.y_{k_i}\}\neq \emptyset$. We now find the local $H$-leaves in $\umt^G_m.y_{k_i}$ ($y_{k_i}\in\mathcal Y_j$) which intersect $\umt^H_m.\exp(w)y_j$. Let 
\[
\mathcal Y_j^w=\Bigl\{(w_i, y_{k_i})\in B_\rfrak(0,2\beta^2)\times \mathcal Y_j: \Bigl(\umt^H_m.\exp(w)y_j\Bigr)\cap \Bigl(\umt^H_m.\exp(w_i)y_{k_i}\Bigr)\neq \emptyset\Bigr\}.
\]
 
Note that if $w_i, w'_i\in B_\rfrak(0,2\beta^2)$ are so that $\sfh\exp(w)y_j=\bar\sfh\exp(w_i)y_{k_i}$ 
and $\sfh'\exp(w)y_j=\bar\sfh'\exp(w_i')y_{k_i}$. Then 
\[
\sfh^{-1}\bar\sfh\exp(w_i)y_{k_i}=\sfh'^{-1}\bar\sfh'\exp(w_i')y_{k_i},
\]
which implies $w_i=w'_i$. Thus $\#\mathcal Y_j^w=n\leq \#\mathcal Y_j\leq K$. 

For every $(w_i, y_{k_i})\in\mathcal Y_j^w$, let 
$\sfh_i\in \mathsf B=(\umt^H_m)^{-1}\cdot(\umt^H_m)$ be so that 
\[
\exp(w_i)y_{k_i}=\sfh_i\exp(w)y_j.
\] 
Let us list these elements as $\{\sfh_{cd}\}$ where $1\leq c\leq l$ and for every such $c$ we have $1\leq d\leq n_c$, moreover, $\sfh_{c_1d_1}=\sfh_{c_1d_2}$ and if and only if $c_1=c_2$ and $d_1=d_2$. 

Let $\mathcal N_\sfk$ denote the set of $L\subset\{1,\ldots, l\}$ so that $\sum_{c\in L} n_c=\sfk$. Then 
\[
z\in\umt^H_m.\exp(w)y_j
\]
satisfies $\density_{m,j}(z)=1/\sfk$ if and only if there exists an 
$L\in\mathcal N_\sfk$ so that 
\[
z\in \umt^H_m\sfh_{cd}.\exp(w)y_j
\]
for all $c\in L$ and all $1\leq d\leq n_{c}$, and 
$z\not\in \umt^H_m\sfh_{cd}.\exp(w)y_j$ for any $(c,d)$ with $c\not\in L$. 
Therefore, $\{z\in\umt^H_m.\exp(w)y_j: \density_{m,j}(z)=1/\sfk\}$ is the image under the map $g\mapsto g\exp(w)y_j$ of the set  
\be\label{eq: density = d has bounded cplx}
\bigcup_{L\in\mathcal N_\sfk} \biggl(\textstyle\bigcap_{c\in L}\Bigl(\umt^H_m\cap\umt^H_m\sfh_{cd}\Bigr)\biggr)\bigcap \biggl(\textstyle\bigcap_{c\not\in L}\Bigl(\umt^H_m\setminus\umt^H_m\sfh_{cd}\Bigr)\biggr).
\ee

We now study the set appearing in~\eqref{eq: density = d has bounded cplx}. 
Let us begin with the following computation. 
Suppose $h\in H$ can be written as $h=u^-_{s_0}a_{\tau_0} u_{r_0}$.
Then 
\[
u^-_sa_\tau u_rh= u_{\hat s}a_{\hat\tau}u_{\hat r}
\] 
where $(\hat s, \hat\tau, \hat r)$ are given by
\be\label{eq: commutation in convex comb sec}
\begin{aligned}
&\hat r=\hat r_h(r)=\frac{r}{e^{\tau_0}(1+rs_0)}+r_0=r+r_0+\tilde r_h(r)r,\\
&\hat \tau=\hat \tau_h(r, \tau)=\tau+\tau_0+\tfrac{1}{2}\log(1+rs_0)=\tau+\tau_0+\tilde \tau_h(r)r,\\
& \hat s=\hat s_h(r,\tau, s)=s+\frac{s_0}{e^\tau(1+rs_0)}=s+s_0+\tilde s_{h,1}(r) r+ \tilde s_{h,2}(r,\tau)\tau,
\end{aligned}
\ee
so long as these parameters are defined (which is always the case near the identity). 

Apply the above with $u^-_sa_\tau u_r\in\umt^H_m$ and $h=\sfh_{cd}$ with $1\leq c\leq l$. Then $|s_0|\leq 10e^{-m}\beta^2$ and $|\tau_0|\leq 10\beta^2$, see~\eqref{eq: def BH}, and the functions $\tilde r_h$, $\tilde\tau_h$, $\tilde s_{h,1}$, and $\tilde s_{h,2}$ are analytic functions satisfying the following  
\[
\begin{aligned}
&|\tilde r_h(r)|\leq 10|\tau_0|\leq 100\beta^2,\\
&|\tilde\tau_h(r)|\leq 10|s_0|\leq 100e^{-m}\beta^2,\\
&|\tilde s_{h,1}(r,\tau)|, |\tilde s_{h,2}(r,\tau)|\leq 10|s_0|\leq 100 e^{-m}\beta^2.
\end{aligned}
\]
Therefore, there exists a box $\Xi_{cd}\subset \umt^H_m\sfh_{cd}$ so that
\[
|\umt^H_m\sfh_{cd}\setminus \Xi_{cd}|\ll \eta|\umt^H_m|. 
\]
Repeat this for all $c\in L$ and all $1\leq d\leq n_c$; let $\Xi(L)=\bigcap_{L}(\Xi_{cd}\cap\umt^H_m)$. 
Then
\[ 
\Bigl|\Bigl(\textstyle\bigcap_{L}\Bigl(\umt^H_m\sfh_{cd}\cap \umt^H_m\Bigr)\Bigr)\setminus\Xi(L)\Bigr|\ll \eta|\umt^H_m|.
\]
Similarly, there is $\Xi(L^\complement)$ of complexity $\ll1$ so that 
\[ 
\Bigl|\Bigl(\textstyle\bigcap_{L^\complement}\Bigl(\umt^H_m\setminus \umt^H_m\sfh_{cd}\Bigr)\Bigr)\setminus\Xi(L^\complement)\Bigr|\ll \eta|\umt^H_m|.
\]

The claim in the lemma thus holds with $\Xi^\sfk=\bigcup_{\mathcal N_\sfk}\Bigl(\Xi(L)\cap\Xi(L^\complement)\Bigr)$.
\end{proof}

\subsection*{Thickening in the stable direction}
We now record two lemmas whose proofs are essentially based on almost invariance (under small translations) 
of the measures in question, and on commutation relations in $H$. 
Let $\sigma$ denotes the uniform measure on $\boxHs_{\beta+100\beta^2}$, where as before, 
\[
\boxHs_{\delta}=\{u_s^-:|s|\leq {\delta}\}\cdot\{a_\tau: |\tau|\leq \delta\}
\]
for all $\delta>0$. 

We will write $\mbhs=m_{U^-A}(\boxHs_{\beta})$ where $m_{U^-A}$ denotes the left invariant measure. 
Recall also the definition of $\rwm_t$ from~\eqref{eq: def rwm ell}:
\[
\rwm_t(\varphi)=\ave\varphi(a_{t}\uvk)\uvkd\qquad\text{for all $\varphi\in C_c(H)$.}
\]
We fixed $0<\injr\leq0.01\eta_X$ and $\beta=\eta^2$. 
In the discussion below, we will work with $\rwm_t$ with large enough $t$ so that $e^{-t}\leq \beta^2$.

Let us begin with the following lemma. 

\begin{lemma}\label{lem: thickening stable}
Let $x\in X$. Let $t_1,t_2>0$, and assume that $\nuni^{-t_1}\leq \beta^2$. Put
$\mu=\sigma\conv\rwm_{t_2}\conv\sigma\conv\rwm_{t_1}$.
For every $\varphi\in C_c^\infty(X)$, we have
\[
\biggl|\int\varphi(hx)\diff\!\rwm_{t_2+t_1}(h)-\int\varphi(hx)\diff\!\mu(h)\biggr|\ll \beta\Lip(\varphi) 
\]
where the implied constant is absolute. 
\end{lemma}

\begin{proof}
Let us recall the the following: for $c,d>0$, $a_{d}\boxHs_c a_{-d}\subset\boxHs_c$ and 
$u_ra_{d}=a_{d}u_{e^{-d}r}$. Moreover, for every $r\in[0,1]$ and $\sfh\in \boxHs_c$,
we have $u_r\sfh= \sfh' u_{r'}$ where $\sfh'\in\boxHs_{10c}$ and $|r'|\leq 2$. 
Altogether, we conclude that for every $\sfh\in\boxHs_{\beta+100\beta^2}$ and $r\in[0,1]$
we have 
\[
a_{t_2}u_{r}\sfh a_{t_1} = \sfh' a_{t_1+t_2} u_{e^{-t_1}r'}
\] 
where $|r'|\leq 2$. Since $\Bigl|[0,1]\triangle(e^{-t}r'+[0,1])\Bigr|\ll \beta$, we conclude that 
\[
\biggl|\int\varphi(hx)\diff\!\rwm_{t_2+t_1}(h)-\int\varphi(hx)\diff\!\rwm_{t_2}\conv\sigma\conv\rwm_{t_1}(h)\biggr|\ll \beta\Lip(\varphi).
\]
The lemma follows.
\end{proof}

\begin{lemma}\label{lem: Folner property}
Let $x\in X$ and $t>0$. Assume that $\nuni^{-t}\leq \beta^2$ 
and that $h\mapsto hx$ is injective on $\boxHs_\beta \cdot a_t\cdot U_1$. 
Let $j\in\mathcal J_0$ and $w\in B_\rfrak(0,2\beta^2)$ be so that
\[
\umt_0^H.\exp(w)y_j\subset \supp(\sigma\conv\rwm_{t}\conv\delta_x) \cap \umt^G_0.y_j.
\] 
Put $\bar\mu_{j,w}=(\sigma\conv\rwm_{t}\conv\delta_x)|_{\umt^H_0.\exp(w)y_j}$ and put 
\[
\diff\!\mu_{j,w}(z)=\density_{0,j}(z)\diff\!\bar\mu_{j,w}(z).
\] 
Then for all $\varphi\in C_c^\infty(X)$, all $\mathsf d\geq 0$, and $|r_1|, |r_2|\leq 2$ with $|r_1-r_2|\leq c\beta$, 
\[
\biggl|\int\varphi(a_{\mathsf d} u_{r_1}z)\diff\!\mu_{j,w}(z)-\int\varphi(a_{\mathsf d} u_{r_2}z)\diff\!\mu_{j,w}(z)\biggr|\ll \eta\Lip(\varphi)\mu_{j,w}(X)
\]
where the implied constant depends on $X$ and $c$. 
\end{lemma} 

\begin{proof}
Write $r_2=r_1+r'$ where $|r'|\leq c\beta$, and let $\sfh u_s\in \umt_0^H=\boxHs_{\beta^2}U_\eta$. 
Then 
\be\label{eq: ur' sfh us}
u_{r'}\sfh u_s=\sfh\sfh'u_{s+r''} \quad\text{where $|r''|\leq 10c\beta$ and $\|\sfh'-I\|\ll\beta^3$},
\ee
see~\eqref{eq: commutation in convex comb sec}.

Write  
$\umt_0^H.\exp(w)y_j=\bigcup_{\sfk=1}^K\{z\in \umt_0^H.\exp(w)y_j: \density_{0,j}(z)=1/\sfk\}$, and 
let 
\[
\Xi^\sfk.\exp(w)y_j\subset \{z\in \umt_0^H.\exp(w)y_j: \density_{0,j}(z)=1/\sfk\}
\] 
be as in Lemma~\ref{lem: density level sets}. By that lemma, 
there are collections of intervals $\mathcal J^-=\{J^-\subset[-\beta^2,\beta^2]\}$, $\mathcal J^0=\{J^0\subset[-\beta^2,\beta^2]\}$, and 
$\mathcal J^+=\{J^+\subset[-\eta,\eta]\}$ with $\#{\mathcal J}^{\bigcdot}\leq K'$, and $\mathcal J\subset \mathcal J^-\times\mathcal J^0\times\mathcal J^+$ so that 
\[
\Xi^\sfk=\bigcup_{\mathcal J}\mathsf{prd}(J^-\times J^0\times J^+),
\]
where $\mathsf{prd}(s,\tau,r)=u^-_sa_\tau u_r$. 

Let $\mathring\Xi^\sfk$ denote $\mathring\Xi^\sfk_{\umt^H_0}$, see~\eqref{eq: def mathring Xi}. We will write 
$\Xi^\sfk_{j,w}$ and $\mathring\Xi^\sfk_{j,w}$ for $\Xi^\sfk.\exp(w)y_j$ and $\mathring\Xi^\sfk.\exp(w)y_j$, respectively. 
Using~\eqref{eq: ur' sfh us} and the definition of $\mathring\Xi^\sfk$, we conclude that  
\be\label{eq: ur' pushes z to partila Xi sfk}
u_{r'}\mathring\Xi^\sfk_{j,w}\subset \Xi_{j,w}^\sfk
\ee 
so long as $\beta$ is small enough compared to $c$, see~\S\ref{sec: box complexity}.

Recall now that 
\[
\supp(\sigma\conv\rwm_{t})=\boxHs_{\beta+100\beta^2} \cdot a_t\cdot \{u_r: r\in [0,1]\}
\] 
and that $\mbhs=m_{U^-A}(\boxHs_{\beta+100\beta^2})$, where 
$m_{U^-A}$ is the left invariant measure. For $|s|, |\tau|\leq \beta+100\beta^2$ and $r\in[0,1]$, 
\be\label{eq: conv abs cont}
\diff\!\sigma\conv\rwm_{t}(u^-_sa_{\tau+t} u_r)= \frac{e^\tau}{\mbhs}\diff\! s\diff\!\tau\diff\! r.
\ee
 
Note also that $\umt_0^H.\exp(w)y_j\subset \supp(\sigma\conv\rwm_{t}\conv\delta_x) \cap \umt^G_0.y_j$.
Thus the definition of $\bar\mu_{j,w}$, and the fact $1/K\leq \density_{0,j}\leq 1$, imply that 
\be\label{eq: measure of partial Xi sfk}
\mu_{j,w}(\Xi^\sfk_{j,w}\setminus\mathring\Xi^\sfk_{j,w})\ll \eta\mu_{j,w}(X).
\ee

Using~\eqref{eq: measure of partial Xi sfk}, Lemma~\ref{lem: density level sets} and the definition of $\mu_{j,w}$ again, we have 
\[
\biggl|\int\varphi(a_{\mathsf d} u_{r_i}z)\diff\!\mu_{j,w}(z)-\sum_{\mathcal N}\int_{\mathring\Xi^\sfk_{j,w}}\varphi(a_{\mathsf d} u_{r_i}z)\diff\!\mu_{j,w}(z)\biggr|\ll\eta\Lip(\varphi)\mu_{j,w}(X),
\]
for $i=1,2$, where $\mathcal N=\{1\leq \sfk\leq K: \mathring\Xi^\sfk\neq\emptyset\}$.

In view of this, and since $r_2=r_1+r'$, we need to estimate the following  
\be\label{eq: mu jw folner prop m=0 sfk 1}
\biggl|\int_{\mathring\Xi_{j,w}^\sfk}\varphi(a_{\mathsf d} u_{r_1}z)\diff\!\mu_{j,w}(z)-\int_{\mathring\Xi_{j,w}^\sfk}\varphi(a_{\mathsf d} u_{r_1}u_{r'}z)\diff\!\mu_{j,w}(z)\biggr|
\ee
for all $\sfk\in\mathcal N$. 

Recall that $\diff\!\mu_{j,w}=\density_{0,j}\diff\!\bar\mu_{j,w}$. 
Thus~\eqref{eq: mu jw folner prop m=0 sfk 1} may be written as 
\[
\biggl|\int_{\mathring\Xi_{j,w}^\sfk}\varphi(a_{\mathsf d} u_{r_1}z)\density_{0,j}(z)\diff\!\bar\mu_{j,w}(z)-\int_{\mathring\Xi_{j,w}^\sfk}\varphi(a_{\mathsf d} u_{r_1}u_{r'}z)\density_{0,j}(z)\diff\!\bar\mu_{j,w}(z)\biggr|.
\] 

In view of~\eqref{eq: ur' pushes z to partila Xi sfk}, $\density_{0,j}(z)=\sfk$ and $\density_{0,j}(u_{r'}z)=\sfk$ 
for all $z\in\mathring\Xi_{j,w}^\sfk$.
Recall also that $h\mapsto hx$ is injective on $\supp(\sigma\conv\rwm_t)\subset \boxHs_\beta\cdot a_t\cdot U_1$.  
Thus, $\diff\!\bar\mu_{j,w}$ is the restriction to $\umt_0^H.\exp(w)y_j$ 
of the pushforward of the measure $\frac{e^\tau}{\mbhs}\diff\!s\diff\!\tau\diff\!r$ under the map $h\mapsto hx$. 
Moreover, by~\eqref{eq: measure of partial Xi sfk} and~\eqref{eq: ur' pushes z to partila Xi sfk}, we have 
$\bar\mu_{j,w}(u_{r'}\mathring\Xi_{j,w}^\sfk\triangle\mathring\Xi_{j,w}^\sfk)\ll \eta\mu_{j,i}(X)$.
Altogether, we conclude that    
\[
\biggl|\int_{\mathring\Xi_{j,w}^\sfk}\varphi(a_{\mathsf d} u_{r_1}z)\diff\!\mu_{j,w}(z)-\int_{\mathring\Xi_{j,w}^\sfk}\varphi(a_{\mathsf d} u_{r_1}u_{r'}z)\diff\!\mu_{j,w}(z)\biggr|\ll \eta\|\varphi\|_\infty\mu_{j,w}(X).
\] 
The proof is complete.
\end{proof}

\subsection{The set $\cone$ and the measure $\mu_\cone$}\label{sec: cone and mu cone}
Recall that $0<\injr\leq0.01\eta_X$ and $\beta=\eta^2$. Define
\be\label{eq: def cone}
\coneH=\boxHs_\beta\cdot\{u_r: |r|\leq \eta\},
\ee
where $\boxHs_{\beta}:=\{u_s^-:|s|\leq {\beta}\}\cdot\{a_t: |t|\leq \beta\}$ for all $\beta>0$.

Let $F\subset B_\rfrak(0,\beta)$ be a finite set, and let $y\in X_{2\eta}$. 
Then $\exp(w)y\in X_\eta$ for all $w\in F$, moreover
$\sfh\mapsto \sfh\exp(w)y$ is injective on $\coneH$. 
For every subset $\coneH'\subset\cone$, put  
\be\label{eq: def cone'}
\cone_{\coneH'}=\bigcup\coneH'.\{\exp(w)y: w\in F\};
\ee
we will denote $\cone_\coneH$ simply by $\cone$. 

Let $\adl,\adm>0$. 
Let $\cone=\coneH.\{\exp(w)y: w\in F\}$. 
A probability measure $\mu_\cone$ on $\cone$ is said to be $(\adl,\adm)$-{\em admissible} if 
\[
\mu_\cone=\frac1{\sum_{w\in F}\mu_w(X)}\sum_{w\in F}\mu_w
\]
where for every $w\in F$, $\mu_w$ is a measure on $\coneH.\exp(w)y$ satisfying that if $\sfh\exp(w)y$ is in the support of $\mu_w$
\[
\diff\!\mu_w(\sfh\exp(w)y)=\adl\ddensity_w(\sfh)\diff\!m_H(\sfh)\quad\text{where $1/\adm\leq \ddensity_w(\bigcdot)\leq \adm$;}
\] 
moreover, there is a subset $\coneH_w=\bigcup_{p=1}^{\adm}\coneH_{w,p}\subset \coneH$
so that 
\begin{enumerate}
\item $\mu_w\Bigl((\coneH\setminus \coneH_w).\exp(w)y\Bigr)\leq \adm\beta \mu_w(\coneH.\exp(w)y)$,
\item The complexity of $\coneH_{w,p}$ is bounded by $\adm$ for all $p$, and 
\item $\Lip(\ddensity_w|_{\coneH_{w,p}})\leq \adm$ for all $p$.
\end{enumerate}

Using the notation in~\eqref{eq: def mathring Xi}, let $(\mathring\coneH_w)_{\coneH}=\bigcup_p(\mathring\coneH_{w,p})_{\coneH}$. Put
\[
\text{$\mathring\cone=\bigcup_{w}(\mathring\coneH_w)_{\coneH}\;\;$ and $\;\;\mathring\mu_\cone=\mu_\cone|_{\mathring\cone}$},
\] 
for $\cone$ and an admissible measure $\mu_\cone$ as above.

The following lemma is an analogue of Lemma~\ref{lem: Folner property}.

\begin{lemma}\label{lem: Folner property 2}
Let $\ell>0$, and let $r\in[0,1]$. Assume that $\nuni^{-\ell}\leq \beta^2$.
Let $\mu_\cone$ be an admissible measure on $\cone=\coneH.\{\exp(w)y: w\in F\}$ 
for some $F\subset B_\rfrak(0,\beta)$, see~\eqref{eq: def cone}.
Let $j\in\mathcal J_\ell$ and $v\in B_\rfrak(0,2\beta^2)$ be so that
\[
\umt_\ell^H.\exp(v)y_j\subset \supp(a_\ell u_r\mathring\mu_\cone) \cap \umt^G_\ell.y_j.
\] 
Put $\bar\mu_{r,j}^{v}=(a_\ell u_r\mathring\mu_\cone)|_{\umt^H_\ell.\exp(v)y_j}$, and let 
$\diff\!\mu_{r,j}^v(z)=\density_{\ell,j}(z)\diff\!\bar\mu_{r,j}^v(z)$. 
Then for all $\varphi\in C_c^\infty(X)$, all $\mathsf d\geq 0$, and all $|r_1-r_2|\leq c\beta$, we have 
\[
\biggl|\int\varphi(a_{\mathsf d} u_{r_1}z)\diff\!\mu_{r,j}^v(z)-\int\varphi(a_{\mathsf d} u_{r_2}z)\diff\!\mu_{r,j}^v(z)\biggr|\ll \eta\Lip(\varphi)\mu_{r,j}^v(X)
\]
where the implied constant depends on $X$ and $c$.  
\end{lemma}

\begin{proof}
The proof is similar to the proof of Lemma~\ref{lem: Folner property}. 

Since $r$, $v$, and $j$ are fixed throughout the proof, we will denote $\mu_{r,j}^v$ and $\bar\mu_{r,j}^v$ 
simply by $\mu$ and $\bar\mu$. 

Write $r_2=r_1+r'$ where $|r'|\leq c\beta$. Let $\sfh u_{\hat r}\in \umt_\ell^{H}$, 
then 
\be\label{eq: ur' sfh us'}
u_{r'}\sfh u_{\hat r}=\sfh u^-_sa_\tau u_{\hat r+r''} \quad\text{where $|r''|\ll \beta$ and $\nuni^{\ell}|s|, |\tau|\ll \nuni^{-\ell}\beta^2$},
\ee
see~\eqref{eq: commutation in convex comb sec}.

Let $I^-=[-e^{-\ell}\beta^2, e^{-\ell}\beta^2]$, $I^0=[-\beta^2,\beta^2]$, and $I^+=[-\eta,\eta]$.
As it was done in the proof of Lemma~\ref{lem: Folner property}, write  
\[
\umt_\ell^H.\exp(v)y_j=\bigcup_{\sfk=1}^K\{z\in \umt_\ell^H.\exp(v)y_j: \density_{\ell,j}(z)=1/\sfk\},
\] 
and let $\Xi^\sfk.\exp(v)y_j\subset \{z\in \umt_\ell^H.\exp(v)y_j: \density_{\ell,j}(z)=1/\sfk\}$ 
be as in Lemma~\ref{lem: density level sets}. There are collections of intervals 
$\mathcal J^-=\{J^-\subset[-\beta^2,\beta^2]\}$, $\mathcal J^0=\{J^0\subset[-\beta^2,\beta^2]\}$, and 
$\mathcal J^+=\{J^+\subset[-\eta,\eta]\}$ with $\#{\mathcal J}^{\bigcdot}\leq K'$, and $\mathcal J\subset \mathcal J^-\times\mathcal J^0\times\mathcal J^+$ so that 
\[
\Xi^\sfk=\bigcup_{\mathcal J}\mathsf{prd}(J^-\times J^0\times J^+),
\]
where $\mathsf{prd}(s,\tau,r)=u^-_sa_\tau u_r$. 

Let $\mathring\Xi^\sfk$ denote $\mathring\Xi^\sfk_{\umt^H_\ell}$, see~\eqref{eq: def mathring Xi}. We will write 
$\Xi^\sfk_{j,v}$ and $\mathring\Xi^\sfk_{j,v}$ for $\Xi^\sfk.\exp(v)y_j$ and $\mathring\Xi^\sfk.\exp(v)y_j$, respectively. 
Using~\eqref{eq: ur' sfh us'} and the definition of $\mathring\Xi^\sfk$, we conclude that  
\be\label{eq: ur' pushes z to partila Xi sfk'}
u_{r'}\mathring\Xi^\sfk_{j,v}\subset \Xi_{j,v}^\sfk
\ee 
so long as $\beta$ is small enough compared to $c$, see~\S\ref{sec: box complexity}.

In view of the definitions of $\bar\mu$ and $\mu$, there exists some $w$ and $p$ so that 
$\bar\mu$ is the restriction of the measure 
\[
a_\ell u_r\mu_{w}|_{\mathring\coneH_{w,p}.\exp(w)y}
\] 
to $\umt^H_\ell.\exp(v)y_j$. 
Note that $a_\ell u_r\mu_{w}|_{\mathring\coneH_{w,p}.\exp(w)y}$ is supported on 
$a_\ell u_r\coneH.\exp(w)y$, moreover, for every $\sfh\in\mathring\coneH_{w,p}$, we have 
\be\label{eq: conv abs cont 2}
\diff\!\mu_{w}(\sfh\exp(w)y)|=\adl\ddensity_w(\sfh)\diff\!m_H(\sfh),
\ee
and $\Lip(\ddensity_w|_{\mathring\coneH_{w,p}})\leq \adm$.

Recall that $1\ll \density_{\ell,j},\ddensity\ll1$. 
In view of the definitions of $\bar\mu$ and $\mu$, thus, 
the above implies 
\be\label{eq: measure of partial Xi sfk'}
\mu(\Xi^\sfk_{j,v}\setminus\mathring\Xi^\sfk_{j,v})\ll \eta\mu(X)
\ee
the implied constant depends on $\adl$, $\adm$, and $X$ (via $K$ and $K'$). 

Using~\eqref{eq: measure of partial Xi sfk'}, Lemma~\ref{lem: density level sets}, and the definition of $\mu$ again, we have 
\[
\int\varphi(a_{\mathsf d} u_{r_i}z)\diff\!\mu(z)=\sum_{\mathcal N}\int_{\mathring\Xi^\sfk_{j,v}}\varphi(a_{\mathsf d} u_{r_i}z)\diff\!\mu(z)+O(\eta\Lip(\varphi)\mu(X)),
\]
for $i=1,2$, where $\mathcal N=\{1\leq \sfk\leq K: \mathring\Xi^\sfk\neq\emptyset\}$.

In view of this, and since $r_2=r_1+r'$, we need to estimate the following  
\be\label{eq: mu jw folner prop m=0 sfk 1'}
\biggl|\int_{\mathring\Xi_{j,v}^\sfk}\varphi(a_{\mathsf d} u_{r_1}z)\diff\!\mu(z)-\int_{\mathring\Xi_{j,v}^\sfk}\varphi(a_{\mathsf d} u_{r_1}u_{r'}z)\diff\!\mu(z)\biggr|
\ee
for all $\sfk\in\mathcal N$. 

Recall that $\diff\!\mu=\density_{\ell,j}\diff\!\bar\mu$. 
Thus~\eqref{eq: mu jw folner prop m=0 sfk 1'} may be written as 
\[
\biggl|\int_{\mathring\Xi_{j,v}^\sfk}\varphi(a_{\mathsf d} u_{r_1}z)\density_{\ell,j}(z)\diff\!\bar\mu(z)-\int_{\mathring\Xi_{j,v}^\sfk}\varphi(a_{\mathsf d} u_{r_1}u_{r'}z)\density_{\ell,j}(z)\diff\!\bar\mu(z)\biggr|.
\] 
First note that by~\eqref{eq: ur' pushes z to partila Xi sfk'}, $\density_{\ell,j}(z)=\sfk$ and $\density_{\ell,j}(u_{r'}z)=\sfk$ 
for all $z\in\mathring\Xi_{j,v}^\sfk$.

Now let $\mathsf C^\sfk\subset\coneH$ be so that $a_\ell u_r\mathsf C^\sfk\exp(w)y=\Xi^\sfk_{j,v}$; similarly, define $\mathring{\mathsf C}^\sfk$. 
Then 
\be\label{eq: C and Xi}
u_r\mathring{\mathsf C}^\sfk\exp(w)y=(a_{-\ell}\mathring\Xi^\sfk a_\ell) .a_{-\ell}\exp(v)y_j,
\ee
similarly for ${\mathsf C}^\sfk$ with $\Xi^\sfk$ on the right side. 

In view of~\eqref{eq: C and Xi},~\eqref{eq: conv abs cont 2}, and the definition of $\bar\mu$, 
$\diff\!\bar\mu|_{(u_{r'}\mathring\Xi)\cap\mathring\Xi}$ is a constant multiple
of the pushforward of $\ddensity_w\cdot\diff\!\mu_w^{\Haar}$ restricted to 
\[
\Bigl((u_{e^{-\ell}r'}\mathring{\mathsf C}^\sfk)\cap\mathring{\mathsf C}^\sfk\Bigr).\exp(w)y.
\]
Thus, using~\eqref{eq: measure of partial Xi sfk'} and~\eqref{eq: ur' pushes z to partila Xi sfk'}, we conclude that 
$\bar\mu(u_{r'}\mathring\Xi_{j,v}^\sfk\triangle\mathring\Xi_{j,v}^\sfk)\ll \eta\mu(X)$. Altogether, we get 
\[
\biggl|\int_{\mathring\Xi^\sfk_{j,v}}\varphi(a_{\mathsf d} u_{r_1}z)\diff\!\mu(z)-\int_{\mathring\Xi^\sfk_{j,v}}\varphi(a_{\mathsf d} u_{r_1}u_{r'}z)\diff\!\mu(z)\biggr|\ll \eta\Lip(\varphi)\mu(X).
\] 
The proof is complete.
\end{proof}

\section{A convex combination decomposition}\label{sec: convex comb}

Recall that for every $\ell>0$, we defined  
\be\label{eq: def rwm ell new section}
\rwm_\ell(\varphi)=\ave\varphi(a_{\ell}\uvk)\uvkd\qquad\text{for all $\varphi\in C_c(H)$.}
\ee
In this section, we will show that if $\rwm_\ell^{(d)}$ is the $d$-fold convolution of $\rwm_\ell$ and~$x\in G/\Gamma$, then the measure $\rwm_\ell^{(d)}.x$ can be approximated by a convex combination $\sum c_i\mu_{\cone_i}$, where $\mu_{\cone_i}$ is an admissible measure for all $i$; see~\S\ref{sec: cone and mu cone}. Since $\rwm_\ell^{(d)}$ and $\rwm_{d\ell}$ stay close to each other, see Lemma~\ref{lem: thickening stable}, we thus conclude that averages of the form appearing in Theorem~\ref{thm:main} (albeit for $a_{d\ell}$) can be approximated by a convex combination of measures supported on sets which are a finite union of local $H$ orbits.
The main results are Lemma~\ref{lem: convex comb 1} and Lemma~\ref{lem: convex comb 2}; the proofs are based on Lemmas~\ref{lem: Folner property} and~\ref{lem: Folner property 2}. 

The results of this section will be combined with Lemma~\ref{lem: truncated MF main estimate} in the proof of Proposition~\ref{propos: imp dim main}; see, in particular, part~(2) in that proposition.

\subsection*{Convex combination: the base case}
Let $x\in X$, and let $t>0$. 
Assume that $\nuni^{-t}\leq \beta$ and that $h\mapsto hx$ is injective on $\coneH\cdot a_t\cdot U_1$.

By Proposition~\ref{prop: Non-div main}, for every  interval $I\subset [0,1]$ with $|I|\geq \delta$, we have 
\be\label{eq: non-div use}
\Bigl|\Bigl\{r\in I:\inj(a_t\uvk x)< \vare^2\Bigr\}\Bigr|<\ref{E:non-div-main}\vare |I|,
\ee
so long as $t\geq |\log(\delta^2\inj(x))|+\ref{E:non-div-main}$.

In order to deal with boundary effects, we will consider {\em interior} points for the supports of $\rwm_t$ and $\sigma$. 
Let $\rwm_{t,1}'$ be the restriction of $\rwm_t$ to $\{a_tu_r: r\in [e^{-t}, 1-e^{-t}]\}$, note that for every $h\in\supp(\rwm_{t,1}')$, 
we have $U_{1}.h\subset\supp(\rwm_{t})$. 
Applying~\eqref{eq: non-div use}, with $\vare=(2\eta)^{1/2}$ and $I=[e^{-t}, 1-e^{-t}]$, we may write 
\[
\rwm_t=\rwm_{t,1}+\rwm_{t,2}
\] 
where $\supp(\rwm_{t,1}.x)\subset X_{2\eta}$, for every $h\in\supp(\rwm_{t,1})$ 
we have $U_{1}.h\subset\supp(\rwm_{t})$, and $\rwm_{t,2}(H)\ll \nuni^{-t}\ll \eta^{1/2}$. 

Recall that $\sigma$ is the uniform measure on $\boxHs_{\beta+100\beta^2}$, write $\sigma=\sigma_1+\sigma_2$ where 
\[
\sigma_1=\sigma|_{{\mathsf B}_{\beta-100\beta^2}^{s,H}}.
\]

Similarly, write $\rwm_t=\mathring\rwm_{t}+\partial\rwm_t$ where $\supp(\mathring\rwm_{t}.x)\subset X_{2\eta}$, for every $h\in\supp(\mathring\rwm_{t})$ we have $U_{1-100\eta}.h\subset \supp(\rwm_{t})$ and $\partial\rwm_t(H)\ll\eta^{1/2}$; also write 
$\sigma=\mathring\sigma+\partial\sigma$ where $\mathring\sigma=\sigma|_{\mathsf B_\beta^{s,H}}$. Note that 
\[
\text{$\supp(\nu_{t,1})\subset\supp(\mathring{\nu_{t}})\quad$ and $\quad\supp(\sigma_1)\subset\supp(\mathring\sigma)$.}
\]

For every $j\in \mathcal J_0$ and every $z\in\supp(\sigma_1\conv\rwm_{t,1}).x \cap \umt^G_0.y_j$,
we have $z=\sfh\exp(w)y_j$ where $w\in B_\rfrak(0,2\beta^2)$ and 
\[
\sfh\in\umt^H_0=\Bigl\{u^-_sa_\tau: |s|,|\tau|\leq \beta^2 \Bigr\}\cdot U_\eta.
\] 
In consequence, $\umt_0^H.\exp(w)y_j\subset \supp\Bigl((\mathring\sigma\conv\mathring\rwm_t).x\Bigr) \cap \umt^G_0.y_j$. 
This observation, in particular, implies that for every $j\in\mathcal J_0$, we have 
\[
((\sigma\conv\rwm_t).x)|_{\umt^G_0.y_j}=\mu'_j+\sum_{i=1}^{N_j} \bar\mu_{j,i}
\]
where for all $i$ there exists $w_i$ so that $\bar\mu_{j,i}=(\mathring\sigma\conv\mathring\rwm_t.x)|_{\umt^H_0.\exp(w_i)y_j}$ and
\[
\mu'_j(\umt^G_0.y_j)\leq ((\sigma_2\conv\rwm_t).x)(\umt^G_0.y_j).
\] 
For all $j\in\mathcal J_0$, put 
\be\label{eq: def Fj for J0}
F_j=\Bigl\{w_i: \bar\mu_{j,i}=(\mathring\sigma\conv\mathring\rwm_t.x)|_{\umt^H_0.\exp(w_i)y_j}\Bigr\}.
\ee

\begin{lemma}\label{eq: numb Fj 1}
We have 
\[
\#F_j\ll \beta^{-3}e^{t}. 
\]
\end{lemma}

\begin{proof}
The proof is similar to~\cite[Lemmas 6.4 and 7.5]{LM-PolyDensity}, 
we reproduce the argument for the convenience of the reader. 

Recall from~\eqref{eq:def-inj} that 
\[
\inj(z)=\min\Big\{0.01, \sup\Big\{\delta: \text{ $g\mapsto gz$ is injective on $\boxG_{100\delta}$}\Big\}\Big\},
\]
where for every $0<\delta\leq0.1$ we put $\boxG_{\delta}:=\boxH_\delta\cdot\exp(B_\rfrak(0,\delta))$.

Therefore, for every $z\in X_\eta$, the map $(\sfh, w)\mapsto \sfh\exp(w)z$ is injective over 
$\boxH_{4\eta}\times B_\rfrak(0,4\eta)$. Hence, for all distinct $w,w'\in B_\rfrak(0,2\eta)$, we have
\[
\boxH_{4\eta}\exp(w)z\cap \boxH_{4\eta}\exp(w')z=\emptyset.
\]

This, and the fact that $\umt^H_0.\exp(w_i)y_j\subset \supp(\sigma\conv\rwm_t.x)\cap X_\eta$ for every $w_i\in F_j$,
implies that 
\[
(\#F_j)\cdot (\beta^4\eta)\ll\beta^2e^t.
\] 
We obtain $\#F_j\ll \beta^{-2}\eta^{-1}e^t\ll \beta^{-3}e^t$, as it was claimed. 
\end{proof}

For any $j\in\mathcal J_0$ and $1\leq i\leq N_j$, define 
$\diff\!\mu_{j,i}(z)=\density_{0,j}(z)\diff\!\bar\mu_{j,i}(z)$.
Altogether, we obtain
\be\label{eq: sigma mu decom}
\sigma\conv\rwm_t.x=\mu'+\sum_{j\in\mathcal J_0}\sum_{i=1}^{N_j}\mu_{j,i}
\ee
where $\mu'(X)\ll \eta^{1/2}$. Let 
\be\label{eq: def cj}
c_j=\sum_{i=1}^{N_j}\mu_{j,i}(X).
\ee

\begin{lemma}\label{lem: popular conej's 1}
If $c_j\geq \beta^{11}$, then $\#F_j=N_j\geq \beta^{9}e^t$. Moreover, 
\[
\sum_{\;\;c_j\geq \beta^{11}} c_j\geq 1-O(\eta^{1/2})
\]
\end{lemma}

\begin{proof}
Recall that $\diff\!\mu_{j,i}(z)=\density_{0,j}(z)\diff\!\bar\mu_{j,i}(z)$, where 
\[
\bar\mu_{j,i}=(\mathring\sigma\conv\mathring\rwm_t.x)|_{\umt^H_0.\exp(w_i)y_j}\quad\text{and}\quad 1/K\leq \density_{0,j}\leq 1. 
\]
Therefore, $c_j\asymp N_je^{-t}\beta^{-2}\beta^4\eta=N_je^{-t}\beta^{2}\eta$. 
Hence if $c_j\geq \beta^{11}$, we have 
\[
N_j\gg \beta^{9}e^{t}
\]
where we also used $0<\eta\leq1$.

To see the second claim, recall from Lemma~\ref{lem: E good h0} 
that $\#\mathcal J_0\ll \eta^{-1}\beta^{-10}$. Using $\beta= \eta^2$, thus, we conclude
\[
\sum_{\;\; c_j< \beta^{11}}c_j\leq \beta\eta^{-1}\leq \eta.
\]
This and the fact that $\mu'(X)\ll\eta^{1/2}$ imply the claim.   
\end{proof}

For every $j$ so that $c_j\geq \beta^{11}$, define 
\be\label{eq: def conej}
\cone_j=\coneH.\{\exp(w_i)y_j: w_i\in F_j\}.
\ee 
Let $\mu_{\cone_j}$ be the restriction of  
\be\label{eq: def mu conj}
\sum_{i=1}^{N_j}\sigma\conv\mu_{j,i}
\ee
to $\cone_j$, normalized to be a probability measure.

\begin{lemma}\label{lem: mu cone is admissible}
The measure $\mu_{\cone_j}$ is a $(1/\mbhs, \adm)$-admissible measure on $\cone_j$ where $\mbhs=m_{{U^-}A}(\boxHs_{\beta+100\beta^2})$ and $\adm$ depends only on $X$. 
\end{lemma}

\begin{proof}
For every $w_i\in F_j$, let $\mu_{w_i}$ denote the restriction of $\sigma\conv\mu_{j,i}$ to $\coneH.\exp(w_i)y_j$. 
Then $\mu_{\cone_j}=\frac{1}{\sum_i\mu_{w_i}(X)}\sum\mu_{w_i}$. We will show that
\[
\diff\!\mu_{w_i}=V^{-1}\ddensity_i\cdot\diff\!m_H|_{\coneH.\exp(w_i)y_j}
\]
where $\ddensity_i$ satisfies the desired properties for all $i$. 

Recall that $\sigma$ is the uniform measure on $\boxHs_{\beta+100\beta^2}$.
Moreover, $\mu_{j,i}=\density_{0,j}\cdot\bar\mu_{j,i}$ where 
\[
\bar\mu_{j,i}=(\mathring\sigma\conv\mathring\rwm_t)|_{\umt^H_0.\exp(w_i)y_j}
\]
and $\umt^H_0=\boxHs_{\beta^2}\cdot U_\eta$. 
These, together with $1/K\leq \density_{0,j}\leq 1$, 
imply 
\[
\diff\!\mu_{w_i}=V^{-1}\ddensity_i\cdot\diff\!m_H
\]
where $1\ll \ddensity_i(\sfh)\ll1$.

Let $\Xi_{j,i}^\sfk$ be as in the proof of Lemma~\ref{lem: Folner property} (and Lemma~\ref{lem: Folner property 2}) applied with $v=w_i$, 
write $\mathring\Xi_{j,i}^\sfk$ for $(\mathring\Xi_{j,i}^\sfk)_{\umt^H_0}$. We will show that the claim holds with   
\[
\text{$\coneH_{w_i}=\bigcup_{\sfk}\coneH_{w_i, \sfk}\quad$ where $\quad\coneH_{w_i, \sfk}={\mathsf B}_{\beta-100\beta^2}^{s,H}\cdot\mathring\Xi_{j,i}^\sfk$}.
\]
First note that the complexity of $\coneH_{w_i, \sfk}$ is $\ll1$ by its definition. 
Moreover, 
\[
\mu_{j,i}\Bigl((\Xi_{j,i}^\sfk\setminus \mathring\Xi_{j,i}^\sfk).\exp(w_i)y_j\Bigr)\ll \eta\mu_{j,i}(\coneH.\exp(w_i)y_j).
\] 
This and Lemma~\ref{lem: density level sets} imply that 
\[
\mu_{w_i}\Bigl((\coneH\setminus\coneH_{w_i}).\exp(w_i)y_j\Bigr)\ll \eta\mu_{w_i}(\coneH.\exp(w_i)y_j). 
\]
Finally, since $\density_{0,j}$ is constant on $\mathring\Xi_{j,i}^\sfk$, we have $\Lip(\ddensity_i|_{\coneH_{w_i, \sfk}})\ll 1$.  
\end{proof}

The following lemma is the base case of our inductive argument. 

\begin{lemma}\label{lem: convex comb 1}
Let $x\in X$, and let $t>0$. 
Assume that $\nuni^{-t}\leq \beta$ and that $h\mapsto hx$ is injective on $\coneH\cdot a_t\cdot U_1$.
Let $\{c_j\}$ and $\{\mu_{\cone_j}\}$ be as in~\eqref{eq: def cj} and~\eqref{eq: def mu conj}, respectively.
Then for every $\varphi\in C_c^\infty(X)$, every $\mathsf d>0$, and all $|s|\leq 2$, 
\[
\biggl|\int\varphi(a_{\mathsf d} u_sz)\diff((\sigma\conv\rwm_t).x)(z)-\sum_j c_j\int\varphi(a_{\mathsf d}u_sz)\diff\!\mu_{\cone_j}(z)\biggr|\ll \eta^{1/2}\Lip(\varphi)
\]
where the implied constant depends only on $X$. 
\end{lemma}

\begin{proof}
We begin with the following observation.
For every $|r|\leq 2$ and all $\sfh\in\boxHs_\beta$, we have $u_r\sfh=\sfh'u_{r_\sfh}$ where $|r_\sfh-r|\ll\beta|r|$
and $\sfh'\in \boxHs_{10\beta}$, see~\eqref{eq: commutation in convex comb sec}. 
Moreover, $a_{\mathsf d}\boxHs_{\bigcdot}a_{-\mathsf d}\subset \boxHs_{\bigcdot}$. 
Therefore,
\be\label{eq: hat mu j sigma' 1}
\biggl|c_j\!\int\!\!\varphi(a_{\mathsf d} u_r z)\diff\!\mu_{\cone_j}(z)-\iint\!\!\varphi(a_{\mathsf d} u_{r_\sfh}z)\diff\!\hat\mu_j(z)\diff\!\sigma(\sfh)\biggr|\ll_X c_j\beta\Lip(\varphi)
\ee
where $\hat\mu_j=\sum_{i=1}^{N_j}\mu_{j,i}$.

Moreover, by Lemma~\ref{lem: Folner property} applied with $r_\sfh$ and $r$ and $c=2$, we have 
\be\label{eq: hat mu j sigma'}
\biggl|\int\varphi(a_{\mathsf d} u_{r_\sfh}z)\diff\!\hat\mu_j(z)-\int\varphi(a_{\mathsf d} u_rz)\diff\!\hat\mu_j(z)\biggr|\ll_X c_j\beta\Lip(\varphi).
\ee
In view of~\eqref{eq: sigma mu decom} and since $\sum c_j=1-O(\eta^{1/2})$, see Lemma~\ref{lem: popular conej's 1}, 
the claim follows from~\eqref{eq: hat mu j sigma' 1} and~\eqref{eq: hat mu j sigma'}. 
\end{proof}

\subsection{Convex combination: the inductive step}\label{sec: conv comb induction}
Let $x\in X$, and let $t$ and $\ell$ be positive. 
Assume that $\nuni^{-t},\nuni^{-\ell}<\beta$ and that $h\mapsto hx$ is injective on $\coneH\cdot a_t\cdot U_1$.
We also assume fixed some $\mathsf d_0\geq t,\ell$. 

For any $n\in\bbn$, define 
\be\label{eq: def mu tn...t1}
\mu_{t,\ell, n}=\rwm_{\ell}\conv\cdots\conv\nu_{\ell}\conv\sigma\conv\rwm_{t}
\ee
where $\rwm_\ell$ appears $n$-times. Put $\mu_{t,\ell,0}=\sigma\conv\rwm_{t}$.

Let $n\geq1$. Assume there are $0\leq c'_i\leq 1$ and 
$(\adl_{n-1},\adm_{n-1})$-admissible measures $\{\mu_{\cone_i'}\}$ supported on 
\[
\cone_i'=\coneH.\{\exp(w_q')y_i': w_q'\in F_i'\}\subset X_\eta
\] 
so that for every $0<\mathsf d\leq \mathsf d_{0}$ and all $|s|\leq 2$, we have 

\begin{multline}\label{eq: conv comb ind hyp}
\int\varphi(a_{\mathsf d}u_shx)\diff\!\mu_{t,\ell,n-1}(h)=\\
\sum_i c_i'\int\varphi(a_{\mathsf d} u_sz)\diff\!\mu_{\cone_i'}(z)+O(\delta_{n-1}\Lip(\varphi))
\end{multline}
for some $0<\delta_{n-1}\leq1$. 

Our goal in this section is to construct a collection of admissible measures $\mu_{\cone_j}$ and constants $0\leq c_j\leq 1$ 
so that~\eqref{eq: conv comb ind hyp} holds for $\mu_{t,\ell,n}$. 

\medskip

We begin with the following non-divergence result.  

\begin{lemma}\label{lem: a-ell u-r mu cone non-div}
For every $r\in[0,1]$ we have 
\[
\mu_{\cone'_i}\Bigl(\{z\in \cone'_i: a_\ell u_rz\not\in X_{2\eta}\}\Bigr)\ll\eta^{1/2} 
\]
so long as $\ell\geq 3|\log\eta|+\ref{E:non-div-main}$.
\end{lemma}

\begin{proof}
Recall that $\coneH=\boxHs_\beta\cdot\{u_{r'}: |r'|\leq \eta\}$.
We will show that for every $\sfh\in\boxHs_{\beta}$ and every $w'_q\in F'_i$,  
\be\label{eq: a-ell u-r mu cone non-div}
|\{r'\in[-\eta,\eta]: a_{\ell}u_r\sfh u_{r'}\exp(w'_q)y'_i\not\in X_{2\eta}\}|\ll \eta^{1/2}
\ee
Since $\diff\!\mu_{w'_q}=\adl_{n-1}\ddensity \diff\!m_H$ and 
$\frac{1}{\adm_{n-1}}\leq \ddensity\leq \adm_{n-1}$,~\eqref{eq: a-ell u-r mu cone non-div} implies the lemma. 

To see~\eqref{eq: a-ell u-r mu cone non-div}, note that $u_r\sfh=\sfh' u_{\hat r}$, for some $\sfh'\in \boxHs_{10\beta}$ and $|\hat r|\leq 2$.
Since $a_\ell \boxHs_{10\beta}a_{-\ell}\subset \boxHs_{10\beta}$, we conclude that
\be\label{eq: contraction non-div convex comb}
a_{\ell}u_r\sfh u_{r'}\exp(w'_q)y'_i\subset \boxHs_{10\beta} a_\ell u_{\hat r+r'}\exp(w'_q)y'_i.
\ee
Apply Proposition~\eqref{prop: Non-div main} with $I=\hat r+[-\eta,\eta]$ and $\vare=3\eta$. 
Then 
\[
|\{r'\in[-\eta,\eta]: a_{\ell}u_{\hat r+r'}\exp(w'_q)y'_i\not\in X_{3\eta}\}|\ll \eta^{1/2}.
\]
This and~\eqref{eq: contraction non-div convex comb} imply~\eqref{eq: a-ell u-r mu cone non-div} and finish the proof.  
\end{proof}

In view of this lemma, for the remainder of this section, we will assume that 
$\ell\geq 3|\log\eta|+\ref{E:non-div-main}$.

Recall that $\cone_i'=\coneH.\{\exp(w_q')y_i': w_q'\in F_i'\}$ is equipped with the admissible measure $\mu_{\cone_i'}$. 
For every $w'_q\in F_i'$, let $\ddensity_{w'_q}$ and $\coneH_{w_q'}=\bigcup_p\coneH_{w_q', p}$ 
be as in the definition of an admissible measure, \S\ref{sec: cone and mu cone}. 

Using the notation in~\eqref{eq: def mathring Xi}, let $\mathring\coneH_{w'_q}:=\bigcup_p(\mathring\coneH_{w'_q,p})_{\coneH}$. Put
\[
\text{$\mathring\cone'_i=\bigcup_{w'_q}\mathring\coneH_{w'_q}\quad$ and $\quad\mathring\mu_{\cone'_i}=\mu_{\cone'_i}|_{\mathring\cone'_i}$}.
\] 

For every $i$ and $r\in[0,1]$, put $\mu_{i,r}=a_\ell u_r\mu_{\cone'_i}$. 
In view of the definition of $\mathring\mu_{\cone'_i}$ and Lemma~\ref{lem: a-ell u-r mu cone non-div}, 
we will write $\mu_{i,r}=\mu_{i,r,1}+\mu_{i,r,2}$ where $\mu_{i,r,2}(X)\ll \max\{\adm_{n-1}\beta,\eta^{1/2}\}$ and  
\begin{align*}
\supp(\mu_{i,r,1})&\subset \supp(a_\ell u_r\mathring\mu_{\cone'_i})\cap X_{2\eta}\\
&=a_\ell u_r\Bigl(\bigcup\mathring\coneH_{w_q'}.\{\exp(w'_q)y'_i: w'_q\in F_i'\}\Bigr)\cap X_{2\eta},
\end{align*} 
moreover, for every $z\in\supp(\mu_{i,r,1})$ there are $q$ and $p$ so that 
\[
\hat\umt^H_\ell.z\subset a_\ell u_r\coneH_{w'_q,p}\exp(w'_q)y'_i,
\]
where $\hat\umt^H_{\ell}=\Bigl\{u^-_sa_\tau: \nuni^\ell|s|,|\tau|\leq 100\beta^2 \Bigr\}\cdot U_{10\eta}$.

For every $j\in \mathcal J_\ell$ as in Lemma \ref{lem: E good h0} and every $z\in\supp(\mu_{i,r,1}) \cap \umt^G_\ell.y_j$,
we have $z=\sfh\exp(v)y_j$ where $v\in B_\rfrak(0,2\beta^2)$ and 
$\sfh\in\umt^H_\ell=\Bigl\{u^-_sa_\tau: \nuni^{\ell}|s|,|\tau|\leq\beta^2 \Bigr\}\cdot U_{\eta}$.
Thus, 
\be\label{eq: interior mu r-i}
\begin{aligned}
\umt_\ell^H.\exp(v)y_j&\subset (a_\ell u_r\coneH_{w'_q,p}\exp(w'_q)y'_i) \cap \umt^G_\ell.y_j\\
&\subset \supp(\mu_{i,r}) \cap \umt^G_\ell.y_j.
\end{aligned}
\ee
This observation, in particular, implies that for every $j\in\mathcal J_\ell$, we have 
\[
\mu_{i,r}|_{\umt^G_\ell.y_j}=\mu'_{i,r}+\sum_{\varsigma=1}^{N_{i,r}^j} \bar\mu_{i,r}^{j,\varsigma}
\]
where for all $\varsigma$ there exists $v_\varsigma$ so that $\bar\mu_{i,r}^{j,\varsigma}=\mu_{i,r}|_{\umt^H_\ell.\exp(v_\varsigma)y_j}$ and
\[
\mu'_{i,r}(\umt^G_\ell.y_j)\leq \mu_{i,r,2}(\umt^G_\ell.y_j).
\] 

For all $j\in\mathcal J_\ell$, put 
\be\label{eq: define Frij}
F_{i,r}^j=\Bigl\{v_\varsigma: \bar\mu_{i,r}^{j,\varsigma}=(\mu_{i,r})|_{\umt^H_\ell.\exp(v_\varsigma)y_j}\Bigr\}.
\ee

For any $j\in\mathcal J_\ell$ and $1\leq \varsigma\leq N_{i,r}^j$, define 
$\diff\!\hat\mu_{i,r}^{j,\varsigma}(z)=\density_{\ell,j}(z)\diff\!\bar\mu_{i,r}^{j,\varsigma}(z)$. 
Then
\be\label{eq: mu r i decomp}
\mu_{i,r}=\mu'+\sum_{j\in\mathcal J_\ell}\sum_{\varsigma=1}^{N_{i,r}^j}\hat\mu_{i,r}^{j,\varsigma}
\ee
where $\mu'(X)\ll\max\{\eta^{1/2},\adm_{n-1}\beta\}$. For all $j\in\mathcal J_\ell$, put 
\be\label{eq: def cj 2}
c_{i,r}^j=\sum_{\varsigma=1}^{N_{i,r}^j}\hat\mu_{i,r}^{j,\varsigma}(X).
\ee

We have the following analogue of Lemma~\ref{lem: popular conej's 1}.

\begin{lemma}\label{lem: popular conej's 2}
Assume $\eta$ is small enough compare to $\adm_{n-1}$. 
If $c_{i,r}^j\geq \beta^{12}e^{-\ell}$, then $\#F_{i,r}^j=N_{i,r}^j\geq \beta^{8}\cdot(\#F'_i)$. Moreover, 
\[
\sum_{\;\;c_{i,r}^j\geq \beta^{12}e^{-\ell}} c_{i,r}^j\geq 1-O\Bigl(\max\{\eta^{1/2},\adm_{n-1}\beta\}\Bigr)
\]
\end{lemma}

\begin{proof}
Recall that $\diff\!\hat\mu_{i,r}^{j,\varsigma}(z)=\density_{\ell,j}(z)\diff\!\bar\mu_{i,r}^{j,\varsigma}(z)$ where 
\[
\bar\mu_{i,r}^{j,\varsigma}=\mu_{i,r}|_{\umt^H_\ell.\exp(v_\varsigma)y_j}
\]
and $1/K\leq \density_{0,j}\leq 1$. 

Since $\mu_{\cone'_i}$ is admissible, see \S\ref{sec: cone and mu cone}, we have 
$c_{i,r}^j\asymp N_{i,r}^j\Bigl(e^{-\ell}\beta^4\eta\Bigr)\cdot (\#F'_i)^{-1}$.
Therefore, if $c_{i,r}^j\geq \beta^{12}e^{-\ell}$, then  
\[
N_{i,r}^j\geq \beta^{8}\cdot(\#F'_i)
\]
where we assume $0<\eta\leq1$ is small enough to account for the implied constant which depends on $\adm_{n-1}$.

To see the second claim, recall from Lemma~\ref{lem: E good h0} 
that $\#\mathcal J_\ell\ll \eta^{-1}\beta^{-10}e^{\ell}\leq \beta^{-11}e^{\ell}$, therefore,
\[
\sum_{\;\; c_{i,r}^j< \beta^{12}e^{-\ell}}c_j\leq \beta.
\]
This and the fact that $\mu'(X)\ll\max\{\eta^{1/2},\adm_{n-1}\beta\}$ imply the claim.   
\end{proof}

Let $j$ be so that $c_{i,r}^j\geq \beta^{12}e^{-\ell}$. Then by Lemma~\ref{lem: popular conej's 2}, 
we have $\#F^j_{i,r}\geq \beta^8\cdot(\#F'_i)$. We write
\[
F^j_{i,r}=\tilde F^j_{i,r}\bigcup\,\biggl(\textstyle\bigcup_{m=1}^{M_{i,r}^j}F^{j,m}_{i,r}\biggr)
\] 
where $\#\tilde F^j_{i,r}< \beta^9\cdot(\#F'_i)$ and 
\be\label{eq: num in F-j-m-i-r almost const'}
\beta^9\cdot(\#F'_i)\leq \#F^{j,m}_{i,r}\leq \beta^8\cdot(\#F'_i)
\ee 
for every $m$.

Let the notation be as in~\eqref{eq: mu r i decomp}.
As it was observed in the proof of Lemma~\ref{lem: popular conej's 2}, we have 
$\hat\mu_{i,r}^{j,\varsigma}(X)\asymp \hat\mu_{i,r}^{j,\varsigma'}(X)$ for all $\varsigma,\varsigma'$. Thus, we may write 
\be\label{e: mu r i j m}
\sum_{\varsigma=1}^{N_{i,r}^j}\hat\mu_{i,r}^{j,\varsigma}=\mu_j'+\sum_{m=1}^{M_{i,r}^j}\sum_{k=1}^{N_{i,r}^{j,m}}\mu_{i,r}^{j,m,k}
\ee
where $\mu_j'(X)\ll \beta c_{i,r}^j$. 
Note that for every $k$, there is some $\varsigma$ so that 
\[
\mu_{i,r}^{j,m,k}=\hat\mu_{i,r}^{j,\varsigma}.
\]
Recall that $\diff\!\hat\mu_{i,r}^{j,\varsigma}(z)=\density_{\ell,j}(z)\diff\!\bar\mu_{i,r}^{j,\varsigma}(z)$, we will write $\bar\mu_{i,r}^{j,m,k}=\bar\mu_{i,r}^{j,\varsigma}$.

For every $1\leq m\leq M_{i,r}^j$, put 
\[
\mu_{i,r}^{j,m}:=\sum_{k=1}^{N_{i,r}^{j,m}}\mu_{i,r}^{j,m,k},\ c_{i,r}^{j,m}:=\mu_{i,r}^{j,m}(X).
\] 
Then~\eqref{e: mu r i j m} and~\eqref{eq: mu r i decomp} yield  
\be\label{eq: mu r i decomp 2}
\mu_{i,r}=\mu''+\sum_{c_{i,r}^j\geq \beta^{12}e^{-\ell}}\;\;\sum_{m=1}^{M_{i,r}^j}\mu_{i,r}^{j,m}
\ee
where $\mu''(X)\ll\max\{\eta^{1/2},\adm_{n-1}\beta\}$.

For every $j$ so that $c_{i,r}^j\geq \beta^{12}e^{-\ell}$ and all $1\leq m\leq M_{i,r}^j$, define 
\be\label{eq: def conej 2}
\cone_{i,r}^{j,m}=\coneH.\{\exp(v_k)y_j: v_k\in F_{i,r}^{j,m}\}.
\ee 
Let $\mu_{\cone_{i,r}^{j,m}}$ be the restriction of  
\be\label{eq: def mu conj 2}
\sigma\conv\mu_{i,r}^{j,m}
\ee
to $\cone_{i,r}^{j,m}$, normalized to be a probability measure. 

We will refer to 
$(\cone_{i,r}^{j,m},\mu_{\cone_{i,r}^{j,m}})$ as an {\em offspring} of $a_\ell u_r\mu_{\cone'_i}$.

\begin{lemma}\label{lem: mu cone is admissible 2}
The measure $\mu_{\cone_{i,r}^{j,m}}$ is a $(\adl_n,\adm_n)$-admissible measure, where $\adm_n$ depends only on $X$ and $\adm_{n-1}$.  
\end{lemma}

\begin{proof}
The proof is similar to Lemma~\ref{lem: mu cone is admissible}.  
Since $r$, $i$, $j$, and $m$ are fixed throughout the argument, we will drop them from the notation whenever there is no confusion, 
e.g., we denote $\cone'_i$ by $\cone'$, $\mu_{i,r}^{j,m,k}$ by $\mu^k$, and $\cone_{i,r}^{j,m}$ by $\cone$. 

Recall that for every $k$, $\diff\!\mu^{k}=\density_{\ell,j}\diff\!\bar\mu^{k}$ where $\bar\mu^{k}=\mu_{i,r}|_{\umt^H_\ell.\exp(v_k)y_j}$ and $1/K\leq \density_{\ell,j}(z)\leq 1$. Also recall that there are $w'_q$ and $p$ so that 
\[
\supp(\bar\mu^k)\subset a_\ell u_r\Bigl(\coneH_{w_q',p}.\exp(w'_q)y'_i\Bigr).
\] 
Moreover, $\ddensity_{w'_q}$ (in the definition of $\mu_{w'_q}$) 
is $\adm_{n-1}$-Lipschitz on $\coneH_{w_q',p}$.

For every $v_k\in F$, let $\mu_{v_k}$ denote the restriction of $\sigma\conv\mu^k$ to $\coneH.\exp(v_k)y_j$. 
Thus $\mu_\cone=\frac{1}{\sum_i\mu_{v_k}(X)}\sum\mu_{v_k}$, and we have  
\[
\diff\!\mu_{v_k}(\cdot)=\adl_n\ddensity_k(\cdot)\diff\!m_H(\cdot). 
\]
We will show that $\ddensity_k$ satisfies the desired properties for all $k$. 

Recall that 
$\umt_\ell^H=\{u^-_s:|s|\leq e^{-\ell}\beta^2\}\cdot\{a_\tau:|\tau|\leq \beta^2\}\cdot U_\eta$,
and that $\sigma$ is the uniform measure on $\boxHs_{\beta+100\beta^2}$.
For every 
\[
\sfh\exp(v_k)y_j\in \umt_\ell^H.\exp(v_k)y_j=\supp(\bar\mu^k),
\] 
there exists a unique $\sfh'\in\coneH_{w'_q,p}$ so that $a_\ell u_r\sfh'\exp(w'_q)y'_i=\sfh\exp(v_k)y_j$. 
Let us define $\hat\ddensity_k$ on $\umt_\ell^H$ by 
\[
\hat\ddensity_k(\sfh)=\density_{\ell,j}(\sfh\exp(v_k)y_j)\ddensity_{w'_q}(\sfh'\exp(w'_q)y_j)
\]
We note that $\ddensity_k= \sigma\conv\hat\ddensity_k$. Thus 
$(K\adm_{n-1})^{-1}\ll \ddensity_k\ll \adm_{n-1}$. 

For every $1\leq \mathsf f\leq K$, 
let $\Xi_{j,i}^{\mathsf f}$ be as in the proof of Lemma~\ref{lem: Folner property} (and Lemma~\ref{lem: Folner property 2}) applied with $v=v_k$, 
and write $\mathring\Xi_{j,k}^{\mathsf f}$ for $(\mathring\Xi_{j,k}^{\mathsf f})_{\umt^H_\ell}$. In particular, $\density_{\ell,j}$ equals $1/\mathsf f$ on $\mathring\Xi_{j,k}^{\mathsf f}$. We will show that the claim holds with   
\[
\text{$\coneH_{v_k}=\bigcup_{d}\coneH_{v_k, \mathsf f}\quad$ where $\quad\coneH_{v_k, \mathsf f}=\boxHs_{\beta-100\beta^2}\cdot\mathring\Xi_{j,k}^{\mathsf f}$}.
\]
To see this note that the complexity of $\coneH_{v_k, \mathsf f}$ is $\ll1$ by its definition. Moreover, $\density_{\ell,j}$ 
is constant on $\mathring\Xi_{j,k}^{\mathsf f}$.
Thus in order to control $\Lip(\ddensity_k)$ on $\coneH_{v_k, \mathsf f}$, we may drop $\density_{\ell,j}$ from the definition of $\hat\ddensity_k$ above. 
Now $u_{r'}a_\ell u_r=a_\ell u_{r+e^{-\ell}r'}$, $\Lip(\ddensity_{w'_q}|{\coneH_{w'_q,p}})\leq \adm_{n-1}$, furthermore, 
\[
\boxHs_{\beta-100\beta^2}\subset\supp(\sigma)\setminus\partial_{100\beta^2}\supp(\sigma).
\] 
Altogether, we conclude that $\Lip(\sigma\conv\hat\ddensity_k)\ll\adm_{n-1}$ on $\coneH_{v_k, \mathsf f}$ for every $\mathsf f$. 

The proof is complete. 
\end{proof}

\begin{lemma}\label{lem: convex comb 2}
Let $x\in X$, and let $\ell$ and $t$ be positive. 
Assume that $\nuni^{-\ell},\nuni^{-t}<\beta$ and that $h\mapsto hx$ is injective on $\coneH\cdot a_t\cdot U_1$.

Suppose that for every $i$, we have fixed $L_i\subset [0,1]$ with $|[0,1]\setminus L_i|\leq\delta$, 
and let $\{r_{i,q}: q=1,\ldots, N_i\}$ be a maximal $e^{-3\mathsf d_0}$-separated subset of $L_i$. 
Let $\varphi\in C_c^\infty(X)$, $0<\mathsf d\leq \mathsf d_0-\ell$, and $|s|\leq 2$.
Then for every $r_{i,q}$ we have 
\begin{multline}\label{eq: convex combination one r}
\biggl|\int\varphi(a_{\mathsf d}u_sz)\diff(a_\ell u_{r_{i,q}}\mu_{\cone_{i}'})(z)-\sum c_{i, r_{i,q}}^{j,m}\int\varphi(a_{\mathsf d} u_sz)\diff\!\mu_{\cone_i^{j,m}}(z)\biggr|\\
\ll\max\Bigl\{\eta^{1/2}, \adm_{n-1}\beta, \beta\Bigr\}\Lip(\varphi),
\end{multline}
where $\sum=\sum_j\sum_m$. Moreover, we have  
\begin{multline}\label{eq: convex combination sum}
\biggl|\int\varphi(a_{\mathsf d}u_shx)\diff \mu_{t,\ell,n}(h)-\sum c_{i, r_{i,q}}^{j,m}\int\varphi(a_{\mathsf d} u_sz)\diff\!\mu_{\cone_{i, r_{i,q}}^{j,m}}(z)\biggr|\\
\ll\max\Bigl\{\eta^{1/2}, \adm_{n-1}\beta, \delta,\delta_{n-1}\Bigr\}\Lip(\varphi),
\end{multline}
where $\sum=\sum_i\sum_q\sum_j\sum_m$. 

The implied constants depend only on $X$ and $\adm_{n-1}$. 
\end{lemma}

\begin{proof}
The proof is similar to the proof of Lemma~\ref{lem: convex comb 1}.
Indeed loc.\ cit.\ will be used as case $n=0$ in our inductive proof of this lemma.

We will first reduce~\eqref{eq: convex combination sum} to~\eqref{eq: convex combination one r}: 
\begin{align*}
\int\varphi(a_{\mathsf d} u_shx)&\diff\!\mu_{t,\ell,n}(h)=\iint\varphi(a_{\mathsf d} u_sa_\ell u_{r}hx)\diff\!\mu_{t,\ell,n-1}(h)\diff\!r\\
&=\iint\varphi(a_{\mathsf d+\ell} u_{r+se^{-\ell}}hx)\diff\!\mu_{t,\ell,n-1}(h)\diff\!r\\
&=\sum_ic'_i\iint \varphi(a_{\mathsf d+\ell} u_{r+se^{-\ell}}z)\diff\!\mu_{\cone'_i}(z)\diff\!r+O(\delta_{n-1}\Lip(\varphi));
\end{align*}
in the last equality we used~\eqref{eq: conv comb ind hyp}, and $0<d+\ell\leq \mathsf d_0$ and $|r+se^{-\ell}|\leq 2$.
 
Since $|[0,1]\setminus L_i|\leq \delta$ and $\{r_{i,q}: q=1,\ldots, N_i\}\subset L_i$ 
is a maximal $e^{-3\mathsf d_0}$-separated subset, we have 
\begin{multline*}
\sum_ic'_i\iint \varphi(a_{\mathsf d+\ell} u_{r+se^{-\ell}}z)\diff\!\mu_{\cone'_i}(z)\diff\!r=\\
\sum_{i}\sum_q\int \varphi(a_{\mathsf d+\ell} u_{r_{i,q}+se^{-\ell}}z)\diff\!\mu_{\cone'_i}(z)+O\Bigl(\max\{\delta,\beta\}\Lip(\varphi)\Bigr),
\end{multline*}
where we again used $\mathsf d+\ell\leq \mathsf d_0$.

In view of this, let us fix some $i$ and $q$, and investigate 
\[
\int \varphi(a_{\mathsf d+\ell} u_{r_{i,q}+se^{-\ell}}z)\diff\!\mu_{\cone'_i}(z)=\int \varphi(a_{\mathsf d} u_sa_\ell u_{r_{i,q}}z)\diff\!\mu_{\cone'_i}(z),
\]
which also completes the reduction of~\eqref{eq: convex combination sum} to~\eqref{eq: convex combination one r}. 

For simplicity, let us write $r=r_{i,q}$. Using~\eqref{eq: mu r i decomp}, we have 
\[
\int\varphi(a_{\mathsf d} u_sa_\ell u_{r}z)\!\diff\!\mu_{\cone'_i}(z)=\sum_{j}\int\varphi(a_{\mathsf d} u_sz)\!\diff\biggl(\sum_\varsigma\hat\mu_{i,r}^{j,\varsigma}\biggr)(z)+O(\beta\Lip(\varphi)).
\] 

In view of~\eqref{eq: mu r i decomp 2}, see also Lemma~\ref{lem: popular conej's 2}, 
it suffices to consider $j$'s so that $c_{i,r}\geq \beta^{12}e^{-\ell}$, we will however need to add
\[
O\Bigl(\max\{\eta^{1/2},\adm_{n-1}\beta\}\Lip(\varphi)\Bigr)
\]
to the error. Moreover, using~\eqref{e: mu r i j m}, 
we may replace $\sum_\varsigma\hat\mu_{i,r}^{j,\varsigma}$ with $\sum_m\mu_{i,r}^{j,m}$.
Fix one such $j\in\mathcal J_\ell$ and let $1\leq m\leq M_{i,r}^j$. Then $\mu_{i,r}^{j,m}=\sum_k\mu_{i,r}^{j,m,k}$.

We now compare 
\[
\int\varphi(a_{\mathsf d} u_sz)\!\diff\biggl(\sum_k\mu_{i,r}^{j,m,k}\biggr)(z)
\] 
with $\int \varphi(a_{d} u_sz)\diff\!\mu_{\cone_{i,r}^{j,m}}(z)$. Recall from~\eqref{eq: def mu conj 2} that 
\[
\int c_{i,r}^j\varphi(a_{\mathsf d} u_sz)\diff\!\mu_{\cone_{i,r}^{j,m}}(z)=\sum_k\iint\varphi(a_{\mathsf d} u_s\sfh z)\diff\!\mu_{i,r}^{j,m,k}(z)\diff\!\sigma^s(\sfh).
\] 

For every $\sfh\in\boxHs_\beta$ and all $|s|\leq 2$, we have $u_s\sfh=\sfh' u_{s+s_\sfh}$ where $|s_{\sfh}|\ll \beta$ 
and $\sfh'\in\boxHs_{10\beta}$, moreover, $a_d\boxHs_{10\beta}a_{-d}\subset \boxHs_{10\beta}$ for all $d>0$. 
Therefore, for every $k$ and all $\sfh\in\boxHs_\beta$, we have 
\[
\biggl|\int\varphi(a_{\mathsf d} u_s\sfh z)\diff\!\mu_{i,r}^{j,m,k}(z)-\int\varphi(a_{\mathsf d} u_{s+s_\sfh} z)\diff\!\mu_{i,r}^{j,m,k}(z)\biggr|\ll\beta\Lip(\varphi)\mu_{i,r}^{j,m,k}(X).
\]
Finally by Lemma~\ref{lem: Folner property 2}, we have 
\begin{multline*}
\biggl|\int\varphi(a_{\mathsf d} u_{s+s_\sfh} z)\diff\!\mu_{i,r}^{j,m,k}(z)-\int\varphi(a_{\mathsf d} u_{s+s_\sfh} z)\diff\!\mu_{i,r}^{j,m,k}(z)\biggr|\\\ll\adm_{n-1}\beta\Lip(\varphi)\mu_{i,r}^{j,m,k}(X)
\end{multline*}
which completes the proof. 
\end{proof}


\section{Margulis functions and Incidence geometry}\label{sec: MF and IG}

In this section, we will prove Lemma~\ref{lem: truncated MF main estimate} which is one of the main ingredients in the proof of Proposition~\ref{propos: imp dim main}, see also Proposition~\ref{propos: main bootstrap intro}.

\subsection*{The set $\cone$ and the measure $\mu_{\cone}$}
Let $0<\injr\leq0.01\eta_X$ and $\beta=\eta^2$. Recall that 
\[
\coneH=\boxHs_\beta\cdot\{u_r: |r|\leq \eta\}
\]
where $\boxHs_{\beta}:=\{u_s^-:|s|\leq {\beta}\}\cdot\{a_t: |t|\leq \beta\}$.

Let $F\subset B_\rfrak(0,\beta)$ be a finite set, and let $y\in X_{2\eta}$.
Then $\exp(w)y\in X_\eta$ for all $w\in F$, moreover,
$\sfh\mapsto \sfh\exp(w)y$ is injective over $\coneH$. 
For every subset $\coneH'\subset\cone$, put  
\be\label{eq: def cone''}
\cone_{\coneH'}=\bigcup\coneH'.\{\exp(w)y: w\in F\};
\ee
we will denote $\cone_\coneH$ by $\cone$. Throughout this section, we will 
assume fixed an admissible measure $\mu_\cone$ on $\cone$
whose definition we now recall from \S\ref{sec: cone and mu cone}. 

Let $\adl, \adm>0$. 
A probability measure $\mu_\cone$ on $\cone$ is said to be $(\adl, \adm)$-admissible if 
\[
\mu_\cone=\frac1{\sum_{w\in F}\mu_w(X)}\sum_{w\in F}\mu_w
\]
where for every $w\in F$, $\mu_w$ is a measure on $\coneH.\exp(w)y$ satisfying that  
\be\label{eq: adm meas again}
\diff\!\mu_w(\sfh\exp(w)y)=\adl\ddensity_w(\sfh)\diff\!m_H(\sfh)\quad\text{where $1/\adm\leq \ddensity_w(\bigcdot)\leq \adm$;}
\ee
moreover, there is a subset $\coneH_w=\bigcup_{p=1}^{\adm}\coneH_{w,p}\subset \coneH$
so that 
\begin{enumerate}
\item $\mu_w\Bigl((\coneH\setminus \coneH_w).\exp(w)y\Bigr)\leq \adm\beta \mu_w(\coneH.\exp(w)y)$,
\item The complexity of $\coneH_{w,p}$ is bounded by $\adm$ for all $p$, and 
\item $\Lip(\ddensity_w|_{\coneH_{w,p}})\leq \adm$ for all $p$.
\end{enumerate}

\subsection*{Regularity of $\cone$}
Let $0<\delta\leq\inj(z)$ for all $z\in\cone$. 
We will say $\cone$ is $(c,\delta)$-regular if for all $w\in F$ 
\be\label{eq: mfsc regular}
\#(F\cap B_\rfrak(w, \delta/100))\geq c\cdot \Bigl(\#(F\cap B_\rfrak(w, \delta))\Bigr),
\ee
see \S\ref{sec: regularization lemmas} where similar (and finer) regularity properties are discussed. 

\medskip

Our goal is to show that the discretized dimension of $\cone$ at controlled scales will improve under a certain random walk. 
We begin by defining a function which encodes this discretized transversal dimension.  

Let $0<\scmf\leq1/10$. For every $(h,z)\in H\times \cone$, define 
\be\label{eq: def margI h,z}
\margI_{\cone,\scmf}(h,z):=\Bigl\{w\in \rfrak: \|w\|<\scmf\,\inj(hz),\, \exp(w) h z\in h\cone.x\Bigr\}.
\ee
Note that $\margI_{\cone,\scmf}(h,z)$ contains $0$ for all $z\in\cone$. 
Moreover, since $\mathsf E$ is bounded, $\margI_{\cone,\scmf}(h,z)$ is a finite set for all $(h,z)\in H\times\cone$. 

Fix some $0<\alpha<1$.
For every $\trct\geq 1$, define the modified 
and localized Margulis function 
$
\mfm_{\cone,\scmf,\trct}: H\times \cone\to [1,\infty)
$
as follows: if $\#\margI_{\cone,\scmf}(h,z)\leq \trct$, put 
\[
\mfm_{\cone,\scmf,\trct}(h,z)=(\scmf\,\inj(hz))^{-\alpha};
\]
and if $\#\margI_{\cone,\scmf}(h,z)>\trct$, put
\[
\mfm_{\cone,\scmf,\trct}(h,z)=\min\left\{\sum_{w\in I}\|w\|^{-\alpha}: \begin{array}{c}I\subset I_{\cone,\scmf}(h,z)\text{ and }\\ \#(I_{\cone,\scmf}(h,z)\setminus I)=\trct\end{array}\right\}.
\]

Let us also define $\noI_{\cone,\scmf}$ on $H\times \cone$ by  
\be\label{eq: def noI h,z}
\noI_{\cone,\scmf}(h,z):=(\scmf\,\inj(hz))^{-\alpha}\cdot(\#\margI_{\cone,\mfsc}(h,z)\Bigr).
\ee  
If $\coneH'\subset\coneH$, we define $\margI_{\cone_{\coneH'},\mfsc}$, $\noI_{\cone_{\coneH'},\scmf}$, and $\mfm_{\cone_{\coneH'},\scmf,\trct}$ accordingly.

\medskip

Recall also the definition of $\eng$ from \S\ref{sec: discret dim}. Let $0<\mfsc_0\leq 1$, and let $I\subset B_\rfrak(0,\mfsc_0)$. For $\trct\geq1$, define $\eng_{I, \trct}: I\to (0,\infty)$ as follows: If $\#I\leq \trct$, put 
\[
\eng_{I, \trct}(w)=\mfsc_0^{-\alpha},\quad\text{for all $w\in I$,}
\]
and if $\#I>\trct$, put 
\[
\eng_{I,\trct}(w)=\min\left\{\sum_{I'}\|w-w'\|^{-\alpha}: \begin{array}{c}I'\subset I\text{ and }\\ \#(I\setminus I')=\trct\end{array}\right\}.
\]

Fix a small parameter $0<\vare<1$, and let $0<\kappa\leq \vare/10^6$.
Throughout the section, we assume 
\[
\nuni^{-\vare t/10^6}\leq \beta  \quad\text{and}\quad \ell=0.01\vare t.
\]

We will also use the following notation: 
\[
\partial_{\delta_1,\delta_2}\coneH=\Bigl(\partial_{\delta_1}\boxHs_\beta\Bigr)\cdot(\partial_{\delta_2}\{u_r: |r|\leq \eta\}),\quad\text{ for $\delta_1,\delta_2>0$};
\]
we denote $\partial_{\delta,\delta}\coneH$ simply by $\partial_\delta\coneH$.

The following is the main result of this section.

\begin{lemma}\label{lem: truncated MF main estimate}
Let $F\subset B_\rfrak(0,\beta)$ be a finite set with $\#F\geq \nuni^{9t/10}$.
Assume that $F$ satisfies~\eqref{eq: mfsc regular} 
with $\delta=\frac1{10}\inj(y)\mfsc$ and some $c\geq\nuni^{-\kappa^2t/4}$.

Let $\cone=\bigcup\coneH.\{\exp(w)y: w\in F\}$, and put
\[
\hat\cone=\bigcup\hat\coneH.\{\exp(w)y: w\in F\}
\] 
where $\hat\coneH=\overline{\coneH\setminus \partial_{10\mfsc}\coneH}$.

Assume that for some $\mfbd\geq1$ (large enough depending on $\kappa$) some $1\leq \trct\leq\nuni^{\vare t/100}$, and for $\mfsc=\nuni^{-\sqrt\vare t}$, we have 
\be\label{eq: mfm init estimate} 
\mfm_{\cone, \scmf,\trct}(e,z)\leq\mfbd,\qquad\text{for all $z\in\cone$}.
\ee
There exists ${\gdh_{\mu_\cone}}\subset [0,1]$ with 
\[
|[0,1]\setminus {\gdh_{\mu_\cone}}|\ll \nuni^{-\kappa^2t/4}
\]
and for every $r\in \gdh_{\mu_\cone}$, there exists a subset $\cone_r\subset\hat\cone$ with 
\[
\mu_{\cone}(\cone\setminus\cone_r)\ll\nuni^{-\kappa^2t/64}
\] 
so that the following holds. For every $z\in\cone_r$ we have    
\[
\mfm_{\hat\cone,\scmf,\trct_1}(a_{\ell} u_r,z)\leq 
200\nuni^{-\alpha\ell}L_1 \egbd^{1+8\kappa} +200\nuni^{2\alpha\ell}\noI_{\hat\cone,\scmf}(a_{\ell} u_r,z)
\] 
where $L_1=L\kappa^{-L}$ and $\trct_1=L_1\mfbd^{\kappa}\trct$, see Theorem~\ref{thm: proj thm}.
\end{lemma}

The proof of this lemma relies on Theorem~\ref{thm: proj thm} and will be completed in some steps. We begin with the following lemma.

\begin{lemma}\label{lem: margIz energy est}
Assume~\eqref{eq: mfm init estimate} holds. Let 
\[
\icone=\bigcup\iconeH.\{\exp(w)y: w\in F\}
\] 
where $\iconeH=\overline{\coneH\setminus \partial_{5\mfsc}\coneH}$. 
Let $m\in\bbn$. Put $z=\sfh\exp(w_z)y\in\icone$, and let $\margI_z:=\margI_{\icone, m\scmf}(e,z)$. Then  
\[
\eng_{\margI_z, \trct}(w)\leq (2+6m^4)\mfbd\qquad\text{for every $w\in\margI_z$},
\]
where $\eng$ is defined as above with $\mfsc_0=m\mfsc\,\inj(z)$.
\end{lemma}

\begin{proof}
Let $w\in\margI_z$, then $z':=\exp(w)z\in\icone$. We will estimate $\eng_{\margI_z, \trct}(w)$ in terms of $\mfm_{\cone,\scmf,\trct}(e,z')$.

Note that for every $v\in\margI_z$, there exists some $w_v\in F$ and some $\sfh_v\in\iconeH$
so that $\exp(v)z=\sfh_v\exp(w_v)y$. Thus
\begin{equation}\label{eq: commutation nth time}
\begin{aligned}
\sfh_v\exp(w_v)y&=\exp(v)z\\
&=\exp(v)\exp(-w)z'=\sfh'\exp(w'_v)z'
\end{aligned}
\end{equation}
where $\|\sfh'-I\|\ll \mfsc^2$ and $\frac12\|v-w\|\leq \|w'_v\|\leq 2\|v-w\|$, see Lemma~\ref{lem: BCH}.

Since $\sfh_v\in\coneH'$, we conclude from~\eqref{eq: commutation nth time} that 
\[
\exp(w'_v)z'=\sfh'^{-1}\sfh_v\exp(w_v)y\in\cone
\]
where we used $\sfh_v\in\iconeH$ and $\|\sfh'-I\|\ll \mfsc^2$. 
We emphasize that we can only guarantee $\exp(w'_v)z'$ belongs to $\cone$ 
and not necessarily to $\icone\subset \cone$.

Note that, $v\mapsto w'_v$ is one-to-one. Moreover, 
\be\label{eq: v-w <12 binj implies in margI}
\text{if $\|v-w\|< \frac{1}{2}\scmf\,\inj(z')$, then $w'_v\in\margI_{\cone, \scmf}(e,z')$},
\ee
since in that case we have $\|w'_v\|< \scmf\,\inj(z')$.  

Let $\{w_1=w, w_2,\ldots, w_N\}\subset I_z$ be a maximal $\mfsc/4$ separated 
subset; then $N\leq m^4$. Arguing as above with all $w_i$, we also conclude that  
\be\label{eq: margI cone cb}
\margI_z\subset \bigcup_{i=1}^N\margI_{\cone,\mfsc}(e,z_i),\quad\text{for some $\{z_1,\ldots, z_N\}\subset\cone$.}
\ee

Since $\mfsc=\nuni^{-\sqrt\vare t}$ and $\#F\geq \nuni^{0.9t}$, we have 
$\sup_{\hat z\in\cone} \#\margI_{\cone,\mfsc}(e,\hat z)\geq \nuni^{0.8t}$. 
Therefore,~\eqref{eq: mfm init estimate} and the fact that $0\leq \trct\leq\nuni^{0.01t}$ imply
\be\label{eq: Upsilon trivial lower bd}
2\mfbd\geq\sup_{\hat z\in\cone} (\scmf\,\inj(\hat z))^{-\alpha}\cdot\Bigl(\#{\margI_{\cone,\mfsc}(e,\hat z)}\Bigr)
\ee
Recall now that $0.9\,\inj(y)\leq \inj(\hat z)\leq 1.1\,\inj(y)$ for all $\hat z\in\cone$. Therefore,~\eqref{eq: margI cone cb} and~\eqref{eq: Upsilon trivial lower bd} imply that 
\be\label{eq: Upsilon trivial lower bd and margI cone cb}
\begin{aligned}
\scmf\,\inj(z')^{-\alpha}\!\cdot\!(\max\{1,\#\margI_z\})&\leq \tfrac32\sum  \scmf\,\inj(z_i)^{-\alpha}\!\cdot\!(\max\{1,\#\margI_{\cone,\mfsc}(e,z_i)\})\\
&\leq 3m^4\mfbd.
\end{aligned}
\ee
 
We now consider two cases: 
If $\#\margI_{\cone,\mfsc}(e,z')\leq \trct$, then~\eqref{eq: v-w <12 binj implies in margI} implies that $\#\{v\in\margI_z: \|v-w\|< \frac12\scmf\,\inj(z')\}\leq \trct$. Hence, using~\eqref{eq: Upsilon trivial lower bd and margI cone cb}, we get  
\[
\eng_{\margI_z, \trct}(w)\leq 2(\scmf\,\inj(z'))^{-\alpha}\cdot(\max\{1,\#\margI_z\})\leq 6m^4\mfbd
\]
This completes the proof in this case. 

Thus, let us assume $\#\margI_{\cone,\mfsc}(e,z')> \trct$, and let $\margI'\subset \margI_{\cone,\mfsc}(e,z')$ be so that 
\[
\sum_{w'\in\margI'}\|w'\|^{-\alpha}=\mfm_{\cone,\scmf,\trct}(e,z')\leq \mfbd.
\]
Let $\margI=\{v\in\margI_z: \|v-w\|<\frac{1}{2}\scmf\,\inj(z')$ and $w'_v\not\in\margI'\}$. Since $v\mapsto w'_v$  is a one-to-one map from $I$ into $\margI_{\cone,\mfsc}(e,z')\setminus \margI'$, see~\eqref{eq: v-w <12 binj implies in margI}, we have $\#\margI\leq\trct$. Therefore, 
\begin{align*}
\eng_{\margI_z, \trct}(w)&\leq\sum_{v\in \margI_z\setminus\margI}\|v-w\|^{-\alpha}
\leq 2\sum_{v\in \margI_z\setminus\margI}\|w'_v\|^{-\alpha}\\
&\leq 2\sum_{w'\in\margI'}\|w'\|^{-\alpha}+2(\scmf\,\inj(z'))^{-\alpha}\cdot(\max\{1,\#\margI_z\})\\
&\leq (2+6m^4)\mfbd,
\end{align*}
where we used $\frac12\|v-w\|\leq \|w'_v\|$ in the second inequality, the definition of $I$ in the third inequality, and~\eqref{eq: Upsilon trivial lower bd and margI cone cb} in the final inequality. 

This completes the proof of this case and of the lemma. 
\end{proof}

Let us also record the following two lemma whose proof is essentially included in 
the argument at the beginning of the proof of Lemma~\ref{lem: margIz energy est}.

\begin{lemma}\label{lem: 1-1 maps between Iz and subsets of F}
Let $\hat\cone\subset \cone'$ be as in Lemma \ref{lem: truncated MF main estimate}. Let $0<m\leq 100$, $z\in\hat\cone$, and $\delta\leq m\mfsc\,\inj(z)$. Write $z=\sfh_z\exp(w_z)y$
where $\sfh_z\in\hat\coneH$ and $w_z\in F$. Then 
\be\label{eq: 1-1 map between Iz and F}
\begin{aligned}
\#\Bigl(F\cap B_\rfrak(w_z,\delta/2)\Bigr)&\leq \#\Bigl(\margI_{\cone',m\mfsc}(e,z)\cap B_\rfrak(e,\delta)\Bigr)\\
&\leq \#\Bigl(F\cap B_\rfrak(w_z,2\delta)\Bigr).
\end{aligned}
\ee
\end{lemma}

\begin{proof}
Let us write $\margI_z=\margI_{\cone',m\mfsc}(e,z)$. 
We will first show: there is an injective map 
from $\margI_z\cap B_\rfrak(0,\delta)$ into $F\cap B_\rfrak(w_z, 2\delta)$.
For every $v\in\margI_z\cap B_\rfrak(0,\delta)$, there are $w_v\in F$ and $\sfh_v\in\iconeH$ 
so that $\exp(v)z=\sfh_v\exp(w_v)y$. Thus
\[
\begin{aligned}
\sfh_v\exp(w_v)y&=\exp(v)z\\
&=\exp(v)\sfh_z\exp(w_z)y=\sfh_z\exp(\Ad(\sfh_z^{-1})v)\exp(w_z)y\\
&=\sfh'\exp(w'_v)y
\end{aligned}
\]
where $\leq \|w'_v-w_z\|\leq \frac32\|\Ad(\sfh_z^{-1})v\|< 2\|v\|$, see Lemma~\ref{lem: BCH}. Since the map $(h,w)\mapsto h\exp(w)y$ is injective on $\boxG_{10\eta}$, we conclude that $w_v=w'_v$. Thus $v\mapsto w_v$ is an injection from $\margI_z\cap B_\rfrak(0,\delta)$ 
into $F\cap B_\rfrak(w_z, 2\delta)$.

\medskip

The other direction is similar, let $w\in F\cap B_\rfrak(w_z, \delta/2)$. Then 
\begin{align*}
    \exp(w)y&= \exp(w)\exp(-w_z)\exp(w_z)y\\
    &=\exp(w)\exp(-w_z)\sfh_z^{-1}z=\sfh'\exp(v'_w)\sfh_z^{-1}z\\
    &=\sfh'\sfh_z^{-1}\exp(\Ad(\sfh_z)v'_w)z
\end{align*}
where $\|\sfh'-I\|\ll \eta\|w-w_z\|$ and $\|\Ad(\sfh_z)v'_w\|< 2\|w-w_z\|$, see Lemma~\ref{lem: BCH}. 

Put $v_w=\Ad(\sfh_z)v'_w$. Then the above implies 
\[
\exp(v_w)z=\sfh_z\sfh'^{-1}\exp(w)y.
\]
Since $\|\sfh'-I\|\ll \eta\|w-w_z\|\ll \mfsc\eta\inj(z)$ and $\sfh_z\in\hat\coneH=\overline{\coneH\setminus\partial_{10\mfsc}\coneH}$, we conclude 
\[
\sfh_z\sfh'^{-1}\in\coneH'=\overline{\coneH\setminus\partial_{5\mfsc}\coneH}.
\]
Hence $\exp(v_w)z\in\cone'$. Moreover, we have 
$\|v_w\|\leq 2\|w-w_z\|<\delta$. These imply that $v_w\in \margI_z\cap B_\rfrak(e,\delta)$. 
Altogether, $w\mapsto v_w$ is an injection from $F\cap B_\rfrak(w_z, \delta/2)$ into $\margI_z\cap B_\rfrak(e,\delta)$. 
The proof is complete. 
\end{proof}

Let us also record the following lemma for later use

\begin{lemma}\label{lem: F energy est assuming mfm energy estimate}
Assume~\eqref{eq: mfm init estimate} holds. Let $m\in\bbn$. 
For any $w\in F$, put $F_w=B_\rfrak(w,m\mfsc\,\inj(y))\cap F$. 
Then  
\[
\eng_{F_w, \trct}(w')\leq (2+6(4m)^4)\mfbd\qquad\text{for every $w'\in F_w$}.
\]
\end{lemma}

\begin{proof}
Let $w'\in F_w$ and put $z'=\exp(w')y$. Then $z'\in\hat\cone$, and as it was done in the proof of Lemma~\ref{lem: 1-1 maps between Iz and subsets of F}, for every $w'\neq \hat w\in F_w$ we have 
\begin{align*}
    \exp(\hat w)y&= \exp(\hat w)\exp(-w')\exp(w')y\\
    &=\exp(\hat w)\exp(-w')\sfh_{w'}^{-1}z'=\bar\sfh\exp(v'_{\hat w})\sfh_{w'}^{-1}z'\\
    &=\bar\sfh\sfh_{w'}^{-1}\exp(\Ad(\sfh_{w'})v'_{\hat w})z
\end{align*}
where $\|\bar\sfh-I\|\ll \eta\|\hat w-w'\|$ and $\|\Ad(\sfh_{w'})v'_{\hat w}\|< 2\|\hat w-w'\|$, see Lemma~\ref{lem: BCH}. 

Put $v_{\hat w}=\Ad(\sfh_{w'})v'_{\hat w}$. 
Then, as in Lemma~\ref{lem: 1-1 maps between Iz and subsets of F}, we have 
$v_{\hat w}\in \margI_{\cone',4m\mfsc}(e,z')$ and the map $\hat w\mapsto v_{\hat w}$ is injective --- note that $\|\hat w-w'\|\leq 2m\mfsc\,\inj(y)$.  

This and Lemma~\ref{lem: margIz energy est}, imply that 
\[
\eng_{F_w, \trct}(w')\leq \eng_{\margI_{\cone',4m\mfsc}(e,z'), \trct}(0)\leq (2+6(4m)^4)\mfbd
\]
for every $w'\in F_w$.
\end{proof}

\subsection*{Proof of Lemma~\ref{lem: truncated MF main estimate}}
The proof will be completed in some steps. 

For every $w\in \rfrak$ and all $r\in[0,1]$, let 
\[
\xi_r(w)=(\Ad(u_r)w)_{12}=-w_{21}r^2-2w_{11}r+w_{12}.
\]

\subsection*{Applying Theorem~\ref{thm: proj thm}}
As in Lemma~\ref{lem: margIz energy est}, let 
\[
\icone=\bigcup\iconeH.\{\exp(w)y: w\in F\},
\] 
where $\iconeH=\overline{\coneH\setminus \partial_{5\mfsc}\coneH}$.
For all $z\in\icone$, put $\margI_z=\margI_{\icone, \mfsc}(e,z)$. 
In view of Lemma~\ref{lem: margIz energy est}, we have 
\be\label{eq: margIz energy est}
\eng_{\margI_z, \trct}(w)\leq 8\mfbd,\qquad\text{for all $w\in\margI_z$,}
\ee
where $\eng$ is defined with $\mfsc_0=\mfsc\,\inj(z)$.

Apply Theorem~\ref{thm: proj thm} with $\margI_z$ and $\pvare=\kappa$; let $J_z\subset [0,1]$ be the set $J'$ given by that theorem. In particular, 
\be\label{eq: meas of Jz}
|[0,1]\setminus J_z|\leq L\kappa^{-L}\mfbd^{-\kappa^2}\leq\nuni^{-\kappa^2t/2}.
\ee 
To see the last inequality, recall that $\#F\geq \nuni^{0.9t}$. 
Combining this with~\eqref{eq: Upsilon trivial lower bd} (and the discussion preceding~\eqref{eq: Upsilon trivial lower bd}), $\mfbd^{-\kappa^2}\leq \nuni^{-0.8\kappa^2t}$. The above estimate follows if we assume $t$ is large enough to account for the factor $L\vare^{-L}$.

Returning to the argument, by Theorem~\ref{thm: proj thm}, we also have that 
for every $r\in J_z$ there exists $\margI'_{z,r}\subset \margI_z$ with 
$\#(\margI_z\setminus \margI'_{z,r})\leq \nuni^{-\kappa^2t/2}\cdot(\#\margI_z)$ so that
\be\label{eq: proj thm use induct'}
\eng_{\xi_{r}(\margI_z),\trct_1}(\xi_r(w))\leq \egbd_1,\qquad\text{for every $w\in\margI'_{z,r}$,}
\ee
where $\mfbd_1=10L_1\egbd^{1+8\kappa}\geq L_1 (8\egbd)^{1+8\kappa}$.

\subsection*{The sets $\gdh_{\mu_\cone}$ and $\cone_r$}
Equip $\cone\times[0,1]$ with $\sigma:=\mu_{\cone}\times\Leb$ where 
$\Leb$ denotes the normalized Lebesgue measure on $[0,1]$. Let 
\[
Y=\left\{(z,r)\in \hat\cone\times[0,1]: \frac{\#\{w\in\margI_z:\eng_{\xi_r(\margI_z),\trct_1}(\xi_r(w))>\mfbd_1\}}{\#\margI_z}\leq \nuni^{-\kappa^2t/2}\right\}.
\]
where $\hat\cone=\bigcup\hat\coneH.\{\exp(w)y: w\in F\}$
and $\hat\coneH=\overline{\coneH\setminus \partial_{10\mfsc}\coneH}$. 
Then,~\eqref{eq: proj thm use induct'} implies
\[
\text{for all $z\in\hat\cone$, we have  $\{(z,r): r\in J_z\}\subset Y$}.
\]
Recall moreover that $\mu_\cone(\cone\setminus\hat\cone)\ll_{\adm} \mfsc$, see the definition of an admissible measure and 
in particular~\eqref{eq: adm meas again}. We thus conclude from~\eqref{eq: meas of Jz} that 
\[
\sigma(\cone\times[0,1]\setminus Y)\ll_\adm \mfsc+\nuni^{-\kappa^2t/2}\ll_{\adm} \nuni^{-\kappa^2t/2}.
\]
This and Fubini's theorem imply that there is a subset 
$\gdh_{\mu_\cone}\subset [0,1]$ with 
$|[0,1]\setminus\gdh_{\mu_\cone}|\ll_{\adm} \nuni^{-\kappa^2t/4}$ so that for all $r\in\gdh_{\mu_\cone}$, we have 
\be\label{eq: meas of Coner}
\lambda\Bigl(\cone\setminus Y_r\Bigr)\ll_\adm \nuni^{-\kappa^2t/4}
\ee 
where  $Y_r=\{z\in\hat\cone: (z,r)\in Y\}$.

For every $r\in\gdh_{\mu_\cone}$, define 
\[
\cone_r:=\Bigl\{z\in\hat\cone: \mfm_{\hat\cone,\mfsc, \trct_1}(a_\ell u_r, z)\leq 200\nuni^{-\alpha\ell}\egbd_1+200\nuni^{2\alpha\ell}\noI_{\hat\cone,\scmf}(a_{\ell} u_r,z)\Bigr\}.
\]
We will show that  
\be\label{eq: meas of cone-r}
\mu_\cone(\cone\setminus\cone_r)\leq\nuni^{-\kappa^2t/64}.
\ee
Note that the lemma follows from~\eqref{eq: meas of cone-r}. 
Thus, the rest of the argument is devoted to the proof of~\eqref{eq: meas of cone-r}.

Let $r\in {\gdh_{\mu_\cone}}$, and let $z\in Y_r$. Then $(z,r)\in Y$, and by the definition of $Y$, 
there exists a subset $\margI_{z,r}\subset \margI_z$ with 
$\frac{\#(\margI_z\setminus \margI_{z,r})}{\#\margI_z}\leq \nuni^{-\kappa^2t/2}$
so that for every $w\in\margI_{z,r}$, we have 
\be\label{eq: proj thm use induct}
\eng_{\xi_{r}(\margI_z),\trct_1}(\xi_r(w))\leq \egbd_1.
\ee

\begin{claim}
Let $\bar\eta=\inj(y)$. 
For all $ w\in\margI_{z,r}\cap B_\rfrak(0, 0.1\bar\eta\mfsc)$, we have 
\[
\mfm_{\hat\cone,\scmf, \trct_1}\bigl(a_{\ell} u_r,\exp(w)z\bigr)\leq 200\nuni^{-\alpha\ell}\egbd_1+200\nuni^{2\alpha\ell}\noI_{\hat\cone,\scmf}(a_{\ell} u_r,z).
\]
\end{claim}

\begin{proof}[Proof of the claim]
Recall that $\frac12\bar\eta\leq \inj(\bigcdot)\leq 2\bar\eta$ for all $\bigcdot\in\cone$.
Let $w\in I_{z,r}\cap B_\rfrak(0, 0.1\bar\eta\mfsc)$. For ease of notation, put $\hat z=\exp(w)z$ and $h=a_{\ell}u_r$. 

First note that if $\#\margI_{\hat\cone,\scmf}(h,\hat z)\leq \trct_1$, there is nothing to prove. Therefore, we will assume  
$\#\margI_{\hat\cone,\scmf}(h,\hat z)> \trct_1$.

Let $\margI_{h\hat z}^>=\{v\in\margI_{\hat\cone,\scmf}(h,\hat z): \|v\|\geq 0.01\nuni^{-2\ell}\scmf\,\inj(h\hat z)\}$. 
Then 
\be\label{eq: bdry vectors est}
\sum_{v\in\margI_{h\hat z}^>}\|v\|^{-\alpha}\leq100\nuni^{2\alpha\ell}(\scmf\,\inj(h\hat z))^{-\alpha}\cdot(\#\margI_{h\hat z}^>) \leq 100\nuni^{2\alpha\ell}\noI_{\hat\cone,\scmf}(h,\hat z). 
\ee

For any subset $I\subset\margI_{\hat\cone,\scmf}(h,\hat z)$, let 
\[
J_I=\{v\in\margI_{\hat\cone,\scmf}(e,\hat z): \Ad(h)v\in I\},
\] 
and put $I^{\rm new}=I\setminus \bigl(\Ad(h)\margI_{\hat\cone,\scmf}(e,\hat z)\bigr)$, i.e., $I^{\rm new}$ is the set 
of vectors in $I$ which do {\em not} equal $\Ad(h)v$ for any vector $v\in\margI_{\hat\cone,\scmf}(e,\hat z)$. 

With this notation, we have  
\be\label{eq: old and new v in I}
\sum_{v\in I}\|v\|^{-\alpha}\leq \sum_{v\in J_I}\|\Ad(h)v\|^{-\alpha}+\sum_{v\in I^{{\rm new}}}\|v\|^{-\alpha}
\ee
We first estimate the contribution of the second term on the right side of~\eqref{eq: old and new v in I}.
Recall that $\|\Ad(a_{\ell} u_r)^{\pm1}v\|\leq 3\nuni^{\ell}\|v\|$ for all $v\in\gfrak$,
in particular, we have  $\nuni^{-\ell}\inj(\hat z)/3\leq \inj(h\hat z)\leq 3\nuni^{\ell}\inj(\hat z)$.
Thus if $v\in \margI^{{\rm new}}$, then $\|v\|\geq \nuni^{-2\ell}\inj(h\hat z)\scmf/9$.
In consequence, for any $I$ we have $I^{\rm new}\subset\margI_{h\hat z}^>$, 
and the second term may be controlled using~\eqref{eq: bdry vectors est}.

We now turn to the first term on the right side of~\eqref{eq: old and new v in I}. 
The strategy is to relate this term (for an appropriate choice of $I$) to~\eqref{eq: proj thm use induct}. 

Recall that $w\in I_{z,r}\cap B_\rfrak(0, 0.1\bar\eta\mfsc)$ and $\hat z=\exp(w)z$.  
Let now 
\[
v\in \hat\margI(\hat z):=\margI_{\hat\cone,\scmf}(e,\hat z)\cap B_\rfrak(0, 0.1\bar\eta\mfsc).
\] 
Then we have 
\[
\exp(v)\hat z=\exp(v)\exp(w)z=\sfh_v\exp(w_v)z.
\]
We note that $\|w_v-(v+w)\|=\|(w_v-w)-v\|\ll\mfsc\|v\|$ and $\|\sfh_v\|\ll\mfsc^2$.
Since $\exp(v)\hat z\in\hat\cone$, this implies that $\exp(w_v)z=\sfh_v^{-1}\exp(v)\hat z\in\icone$. 
Moreover, $\|v\|,\|w\|\leq 0.1\bar\eta\mfsc$ implies that $\|w_v\|<\inj(z)\mfsc$. Altogether, we have $w_v\in\margI_z$.  

The map $v\mapsto w_v$ is on-to-one from $\hat\margI(\hat z)$ into $\margI_z$. 
Moreover, $\Ad(h^{-1})v\in \hat\margI(\hat z)$
for every $v\in\margI_{\hat\cone,\mfsc}(h,\hat z)\setminus \margI_{h\hat z}^>$.
Thus if $\#I_z\leq \trct_1$, then 
\[
\#\Bigl(\margI_{\hat\cone, \mfsc}(h,\hat z)\setminus \margI_{h\hat z}^>\Bigr)\leq \trct_1,
\] 
and the proof is complete thanks to~\eqref{eq: bdry vectors est}.

In view of this, we let $K_{w}\subset I_z$ be so that $\# (I_z\setminus K_w)\leq \trct_1$ and 
\be\label{eq: proj thm use induct''}
\sum_{w'\in K_{w}}\|\xi_r(w)-\xi_r(w')\|^{-\alpha}\leq \egbd_1,
\ee
see~\eqref{eq: proj thm use induct}. 

Let $\margI_{\rm exc}=\{v\in\hat\margI(\hat z): w_v\not\in K_{w}\}$.
Since the map $v\mapsto w_v$ is one-to-one from $\margI_{\rm exc}$ into 
$\margI_z\setminus K_{w}$, we have $\# \margI_{\rm exc}\leq \trct_1$.

As was remarked above, if $v\in\margI_{\hat\cone,\scmf}(h, \hat z)$ and $\Ad(h^{-1})v\not\in\margI_{\hat\cone,\scmf}(e, \hat z)$, then 
$\Ad(h)v\in\margI_{h\hat z}^>$. 
Therefore, using~\eqref{eq: proj thm use induct''} and~\eqref{eq: bdry vectors est}, we have   

\begin{align*}
\mfm_{\hat\cone,\scmf, \trct_1}\bigl(a_{\ell} u_r,\hat z\bigr)&\leq \sum_{v\in\hat\margI(\hat z)\setminus I_{\rm exc}}\|\Ad(h)v\|^{-\alpha}+100\nuni^{2\alpha\ell}\noI_{\hat\cone,\scmf}(h,\hat z)\\
&\leq 2 \sum_{v\in\hat\margI(\hat z)\setminus I_{\rm exc}}\|\Ad(h)(w_v-w)\|^{-\alpha}+100\nuni^{2\alpha\ell}\noI_{\hat\cone,\scmf}(h,\hat z) \\
&\leq2\sum_{v\in\hat\margI(\hat z)\setminus I_{\rm exc}} \|\nuni^{\ell}(\xi_r(w_v)-\xi_r(w))\|^{-\alpha}+100\nuni^{2\alpha\ell}\noI_{\hat\cone,\scmf}(h,\hat z)\\
&\leq2\nuni^{-\alpha\ell}\sum_{w'\in K_w} \|\xi_r(w')-\xi_r(w)\|^{-\alpha}+100\nuni^{2\alpha\ell}\noI_{\hat\cone,\scmf}(h,\hat z)\\
&\leq 2\nuni^{-\alpha\ell}\egbd_1+100\nuni^{2\alpha\ell}\noI_{\hat\cone,\scmf}(h,\hat z).
\end{align*}
We used~\eqref{eq: bdry vectors est} in the first inequality.   
For the second inequality we used the following: $\|(w_v-w)-v\|\ll\mfsc\|v\|$, moreover, 
the choice $\ell=0.01\vare n$ implies that $\nuni^{-4\ell}>\scmf$. 
Consequently, we have   
\[
\|a_\ell u_r v\|=\|a_\ell u_r(w_v-w+w')\|\geq 0.5\|a_\ell u_r(w_v-w)\|
\]
where $w'=v-(w_v-w)$ and we used  
$\|h^{\pm1}\bigcdot\|\leq 3\nuni^{\ell}\|\bigcdot\|$ for any $\bigcdot\in\gfrak$. The third inequality follows from $(\Ad(h)\bigcdot)_{12}=e^\ell\xi_r(\bigcdot)$, and the last inequality is a consequence of~\eqref{eq: proj thm use induct''}.

The above and~\eqref{eq: bdry vectors est} complete the proof of the claim. 
\end{proof}

\subsection*{Fubini's theorem and the proof of~\eqref{eq: meas of cone-r}}
In view of the claim, for every $z\in Y_r$ and every $w\in\margI_{z,r}$, we have $\exp(w)z\in\cone_r$ so long as $\exp(w)z\in\hat\cone$. We will use this to show~\eqref{eq: meas of cone-r}. That is,
\be\label{eq: meas of cone-r'}
\mu_\cone(\cone\setminus\cone_r)\leq\nuni^{-\kappa^2t/64},
\ee
which will complete the proof of the lemma.

Recall that $\bar\eta=\inj(y)$ and $\frac12\bar\eta\leq \inj(\bigcdot)\leq 2\bar\eta$ for all $\bigcdot\in\cone$. 
Set $\mfsc':=\mfsc\bar\eta/10$. 
The argument is based the following: For every $z\in Y_r$, we have 
\be\label{eq: Izr and Iz in 0.1mfsc}
\#\Bigl(\margI_{z,r}\cap B_\rfrak(0, \mfsc')\Bigr)\geq(1-\nuni^{-\kappa^2t/4})\cdot\Bigl(\#(\margI_z\cap B_\rfrak(0, \mfsc'))\Bigr),
\ee

Let us first establish~\eqref{eq: Izr and Iz in 0.1mfsc}. 
Let $z\in Y_r$. By Lemma~\ref{lem: 1-1 maps between Iz and subsets of F}, we have 
\begin{subequations}
\begin{align}
    \label{eq: use lemma Iz F}&\#\Bigl(\margI_z\cap B_\rfrak(0, \mfsc')\Bigr)\geq \#\Bigl(F\cap B_\rfrak(w_z, \mfsc'/2)\Bigr)\\
    \label{eq: use lemma Iz F again}&\#\margI_z\leq \#\Bigl(F\cap B_\rfrak(w_z, 40\mfsc')\Bigr).
\end{align}
\end{subequations}
where $z=\sfh_z\exp(w_z)y$ and in~\eqref{eq: use lemma Iz F again} we used $\frac12\bar\eta\leq \inj(z)\leq 2\bar\eta$.   

By our assumption, $F$ satisfies~\eqref{eq: mfsc regular} with $c\geq \nuni^{-\kappa^2t/4}$ and $50\mfsc'$. Thus using~\eqref{eq: use lemma Iz F} and~\eqref{eq: use lemma Iz F again}, we have  
\begin{align*}
\#\Bigl(\margI_z\cap B_\rfrak(0, \mfsc')\Bigr)&\geq \#\Bigl(F\cap B_\rfrak(w_z, \mfsc'/2)\Bigr)\\
&\geq c\cdot \Bigl(\#\Bigl(F\cap B_\rfrak(w_z, 50\mfsc'\Bigr)\Bigr)\\
&\geq c\cdot (\#\margI_z)\geq e^{-\kappa^2 t/4}\cdot (\#\margI_z)
\end{align*}
Since $\#(\margI_z\setminus \margI_{z,r})\leq \nuni^{-\kappa^2t/2}\cdot(\#\margI_z)$, the above implies that 
\[
\#(\margI_z\setminus \margI_{z,r})\leq \nuni^{-\kappa^2t/2}\cdot(\#\margI_z)\leq \nuni^{-\kappa^2t/4}\Bigl(\#\Bigl(\margI_z\cap B_\rfrak(0, \mfsc')\Bigr)\Bigr).
\]
Altogether, we conclude  
\[
\#\Bigl(\margI_{z,r}\cap B_\rfrak(0, \mfsc')\Bigr)\geq(1-\nuni^{-\kappa^2t/4})\cdot\Bigl(\#(\margI_z\cap B_\rfrak(0, \mfsc'))\Bigr),
\]
as was claimed in~\eqref{eq: Izr and Iz in 0.1mfsc}.

\medskip

Put $\cone_r^\complement=\cone\setminus\cone_r$ and assume contrary to~\eqref{eq: meas of cone-r'} that
\[
\mu_\cone(\cone_r^\complement)>\nuni^{-\kappa^2t/64}=:\delta.
\] 
We will repeatedly use properties of an admissible measure, see in particular~\eqref{eq: adm meas again}.
Recall from~\eqref{eq: meas of Coner} that
\[
\mu_\cone(\cone\setminus Y_r)\ll \nuni^{-\kappa^2t/4}\leq\delta^8.
\] 
Let $F'=\Bigl\{w\in F: \mu_{w}(Y_r\cap \coneH.\exp(w)y)\geq (1-\delta^4)\mu_w(\coneH.\exp(w)y)\Bigr\}$. 
Then by Fubini's theorem 
\[
\mu_\cone\Bigl(\bigcup_{w\not\in F'}\coneH.\exp(w)y\Bigr)\leq \delta^4.
\]

Points in $\cone$ are represented as $\sfh'\exp(v')y$, in order to utilize~\eqref{eq: Izr and Iz in 0.1mfsc}, however, it is more convenient to have a representation of points in $\cone$ in the form $\exp(v)\sfh y$. To that end, for every $w\in F'$, fix a covering $\{\boxH_{\mfsc'}.z'\}$ 
of 
\[
\Bigl(\coneH\setminus \partial_{20\mfsc}\coneH\Bigr).\exp(w)y
\]
with multiplicity $\leq K'$ (absolute constant), and let 
\[
\mathcal B'_w:=\Bigl\{\boxH_{\mfsc'}.z': \mu_w(\boxH_{\mfsc'}.z'\cap Y_r)\geq (1-\delta^2)\mu_w(\boxH_{\mfsc'}.z')\Bigr\}.
\]
Then $\mu_w\Bigl(\bigcup\Bigl\{\boxH_{\mfsc'}.z': \boxH_{\mfsc'}.z'\not\in\mathcal B'\Bigr\}\Bigr)\leq K'\delta^2$.

Let $\mathsf B=\exp(B_\rfrak(0,\mfsc'))\cdot \boxH_{\mfsc'}$, and put
\[
\hat{\mathcal B}=\{\mathsf B.z': \boxH_{\mfsc'}.z'\in\mathcal B'_w, w\in F' \}.
\]
Then there is $\mathcal B\subset\hat{\mathcal B}$ so that the multiplicity of $\mathcal B$ is 
$\leq K$ (absolute) and  
\[
\mu_\cone(\cup_{\mathcal B}\mathsf B.z')\geq 1-\adm^2KK'\delta^2-\delta^4>1-(\adm^2KK'+1)\delta^2
\]
where $\adm$ appears in the definition of $(\adl,\adm)$-admissible measure.

Recall now that $\mu_{\cone}(\hat\cone)\geq 1-O(\mfsc)>1-\delta^{16}$. 
Therefore, if we put $\mathcal B_{\rm exc}=\{\mathsf B.z'\in\mathcal B: \mu_\cone(\mathsf B.z'\cap \hat\cone)\leq (1-\delta^8)\mu_\cone(\mathsf B.z')\}$, then 
\[
\mu_\cone\Bigl(\bigcup_{\mathcal B_{\rm exc}}\mathsf B.z'\Bigr)\leq 2K\delta^8,
\]
provided that $\delta$ is small enough compared to $\adm$, $K$, and $K'$. 

Since $\mu_\cone(\cone_r^\complement)> \delta$ and the multiplicity of $\mathcal B$ is at most $K$,
there exists some $\mathsf B.z'\in\mathcal B\setminus \mathcal B_{\rm exc}$ so that  
\be\label{eq: BH Yr Er complement}
\mu_\cone\Bigl(\mathsf B.z'\cap \cone_r^\complement\Bigr)\geq \tfrac{\delta}{4K}\mu_\cone(\mathsf B.z')
\ee

Other other hand, applying the claim with 
$\sfh z'\in \boxH_{\mfsc'}.z'\cap Y_r$ we have: for every $v\in\margI_{\sfh z',r}$, $\exp(v)\sfh z'\in\cone_r$ so long as $\exp(v)\sfh z'\in\hat\cone$. 
This and the fact that every point in $\mathsf B.z'$ can be written uniquely as $\exp(v)\sfh z'$ for some 
$v\in B_\rfrak(0,\mfsc')$ and $\sfh\in \boxH_{\mfsc'}$, imply
\begin{multline*}
\mathsf B.z'\cap\cone_r^\complement\subset \Bigl(\mathsf B.z'\cap {\hat\cone}^\complement\Bigr)\bigcup\{\exp(v)\sfh z'\in \mathsf B.z': \sfh z'\not\in Y_r\} \\ \bigcup \{\exp(v)\sfh z'\in \mathsf B.z': v\not\in\margI_{\sfh z',r}\}.
\end{multline*}

We now bound the measure of the three sets appearing on the right side of the above and obtain a contradiction 
with~\eqref{eq: BH Yr Er complement}. First note that since $\mathsf B.z'\not\in\mathcal B_{\rm exc}$, we have 
\be\label{eq: not in B exc use}
\mu_\cone(\mathsf B.z'\cap {\hat\cone}^\complement)\leq \delta^8\mu_\cone(\mathsf B.z').
\ee
Moreover, since $\boxH_{\mfsc'}.z'\in\mathcal B'_w$ for some $w\in F'$, we have 
$\mu_w(\boxH_{\mfsc'}.z'\cap Y_r^\complement)\leq \delta^2\mu_w(\boxH_{\mfsc'}.z')$, hence
\be\label{eq: use def of B'}
\mu_\cone(\{\exp(v)\sfh z'\in \mathsf B.z': \sfh z'\not\in Y_r\})\leq \adm^2\delta^2\mu_\cone(\mathsf B.z'),
\ee
Finally, in view of~\eqref{eq: Izr and Iz in 0.1mfsc}, for every $\sfh z'\in \boxH_{\mfsc'}.z'\cap Y_r$, we have 
\[
\#(\margI_{\sfh z',r}\cap B_\rfrak(0,\mfsc'))\geq (1-\delta^8)\cdot\Bigl(\#(\margI_{\sfh z'}\cap B_\rfrak(0,\mfsc'))\Bigr). 
\]
This and the definition of admissible measure again imply
\be\label{eq: Izr and Iz in 0.1mfsc for hz'}
\mu_\cone(\{\exp(v)\sfh z'\in \mathsf B.z': v\not\in\margI_{\sfh z',r}\})\leq \adm^2\delta^8\mu_\cone(\mathsf B.z').
\ee

Now~\eqref{eq: not in B exc use},~\eqref{eq: use def of B'} and~\eqref{eq: Izr and Iz in 0.1mfsc for hz'}, imply that 
\[
\mu_\cone(\mathsf B.z'\cap\cone_r^\complement)\leq \Bigl(\adm^2\delta^2+(\adm^2+1)\delta^8\Bigr)\mu_\cone(\mathsf B.z'),
\] 
which contradicts~\eqref{eq: BH Yr Er complement} provided that $\delta$ is small enough. 

The proof is complete. 
\qed

\section{Improving the dimension}\label{sec: improve dim}
In this section, we will state and begin the proof of Proposition~\ref{propos: imp dim main}.
The proof is based on an inductive scheme, and relies on results in \S\ref{sec: convex comb} 
and \S\ref{sec: MF and IG}; it will occupy this section as well as~\S\ref{sec: induction} and \S\ref{sec: proof main prop}.

Fix a small parameter $0<\vare<1$ and a large parameter $t$ for the rest of this section as well as~\S\ref{sec: induction} and \S\ref{sec: proof main prop} --- in our applications, $\vare$ will depend on $\mixexp$ in~\eqref{eq:actual-mixing-intro} and $t$ will be chosen $\asymp\log R$ where $R$ is as in Theorem~\ref{thm:main}. 

Put $\ell=\vare t/100$. 
We will also fix a parameter $0<\kappa\leq\vare/10^6$, and put $\beta=e^{-\kappa t}$ and $\eta^2=\beta$, see Proposition~\ref{propos: imp dim main}. We also recall that  $0.9<\alpha<1$.

Let $\sigma$ denote the uniform measure on $\boxHs_{\beta+100\beta^2}$, where for any $\delta>0$, 
\[
\boxHs_{\delta}=\{u_s^-:|s|\leq {\delta}\}\cdot\{a_\tau: |\tau|\leq \delta\}.
\] 
 
For all $d>0$, define $\rwm_d$ by 
$\int\varphi\diff\!\rwm_d=\ave\varphi(a_{d}\uvk)\uvkd$ for any $\varphi\in C_c(H)$.
Recall from~\eqref{eq: def mu tn...t1} that 
\[
\mu_{t,\ell, n}=\rwm_{\ell}\conv\cdots\conv\nu_{\ell}\conv\sigma\conv\rwm_{t}
\]
where $\rwm_\ell$ appears $n$ times in the above expression.

\begin{propos}\label{propos: imp dim main}
Let $x_1\in X$, and assume that Proposition~\ref{prop:closing lemma intro}(2) does not hold for the point $x_1$, and parameters $D\geq 10$ and $t$. Let  
\[
d_{1}=100\lceil\tfrac{4D-3}{2\vare}\rceil,\quad d_2=d_{1}-\lceil\tfrac{10^4}{\sqrt\vare}\rceil,\quad\text{and}\quad\kappa = 10^{-6}d_1^{-2};
\] 
as before, we put $\beta=e^{-\kappa t}$ and $\eta^2=\beta$. 

Let $r_1\in I(x_1)$ and put $x_2=a_{8t}u_{r_1}x_1$, see Proposition~\ref{prop:closing lemma intro}(1). 
For every $d_2\leq d\leq d_1$, there is a collection $\Xi_d=\{\cone_{d,i}: 1\leq i\leq N_d\}$
of sets 
\[
\cone_{d,i}=\coneH.\{\exp(w)y_{d,i}: w\in F_{d,i}\}\subset X_\eta,
\] 
with $F_{d,i}\subset B_\rfrak(0,\beta)$, and $(\adl_{d,i},\adm_{d,i})$-admissible measures $\mu_{\cone_{d,i}}$, see~\S\ref{sec: cone and mu cone}, where $\adm_{d,i}$ depend on $d_1$ and $X$, so that both of the following hold: 
\begin{enumerate}
\item Let $\mfsc=\nuni^{-\sqrt\vare t}$. 
Let $d_2\leq d\leq d_1$, and let $1\leq i\leq N_d$. Then for all $w\in F_{d,i}$ and all $z=h\exp(w)y_{d,i}\in\cone_{d,i}$ with $h\in\overline{\coneH\setminus\partial_{10\mfsc}\coneH}$, both of the following hold:
\begin{align}
\label{eq: improve dim regularity}&\#\Bigl(B_\rfrak(w,4\mfsc\,\inj(y_{d,i}))\!\cap\! F_{d,i}\Bigr)\geq \nuni^{-\vare t}\!\!\sup_{w'\in F_{d,i}}\!\!\#\Bigl(B_\rfrak(w',4\mfsc\,\inj(y_{d,i}))\!\cap\! F_{d,i}\Bigr)\\
\label{eq:improve dim energy est}&\mfm_{\cone_{d,i}, \mfsc, \trct}(e,z)\leq \nuni^{\vare t}\noI_{\cone_{d,i},\mfsc}(e,z)\quad\text{where $\trct\leq \nuni^{0.01\vare t}$}
\end{align}

\medskip
\item For every $\varphi\in C_c^{\infty}(X)$, all $\tau\leq d_1\ell$ and $|s|\leq 2$, we have
\begin{multline}\label{eq: nud1 and nud1-d}
\biggl|\int \varphi(a_\tau u_shx_2)\diff\!\mu_{t,\ell, d_1}(h)-\sum_{d,i}c_{d,i}\int \varphi(a_\tau u_sz)\diff\!\rwm_\ell^{(d_{1}-d)}\conv\mu_{\cone_{d,i}}(z)\biggr|\\
\ll \Lip(\varphi)\beta^{\ref{k: bootstrap beta exp}}
\end{multline}
where $c_{d,i}\geq0$ and $\sum_{d,i}c_{d,i}=1-O(\beta^{\ref{k: bootstrap beta exp}})$, $\Lip(\varphi)$ 
is the Lipschitz norm of $\varphi$, and $\ref{k: bootstrap beta exp}$ and the implied constants depend on $X$.
\end{enumerate}
\end{propos}

As it was mentioned, the proof is based on an inductive scheme. 
The base case relies on Proposition~\ref{prop:closing lemma intro}(1) and Lemma~\ref{lem: convex comb 1}. 
Indeed, combining Proposition~\ref{prop:closing lemma intro}(1) and Lemma~\ref{lem: convex comb 1}, the measure 
$(\sigma\conv\rwm_t).x_2$ (up to an exponentially small error) can be written as $\sum c_i\mu_{\cone_i}$ 
where $\mu_{\cone_i}$ is an admissible measure for all $i$, and 
\[
\mfm_{\cone_{i}, \mfsc, 1}(e,z)\leq e^{Dt}\qquad\text{for all $i$ and all $z\in\cone_{i}$}.
\]  
This will serve as the base case of the induction.
We will then combine Lemma~\ref{lem: convex comb 2} and Lemma~\ref{lem: truncated MF main estimate} to inductively improve this dimension while obtaining convex combinations similar to the expressions appearing in~\eqref{eq: nud1 and nud1-d}. 
For technical reasons, Lemma~\ref{lem: regular tree decomposition} will be applied 
after every step to ensure regularity of the sets $F$ which are used to define sets $\cone$ (again, we are allowed to drop subsets of $F$ with exponentially small density). 

We now turn to the details of the argument, beginning with some general facts. 
In the next three lemmas, let 
\[
\cone=\coneH.\{\exp(w)y: w\in F\}\subset X_\eta
\]
where $F\subset B_\rfrak(0,\beta)$. 

\begin{lemma}\label{lem: rfrak is invariant}
Let $z\in\cone$, and write $z=\sfh\exp(w)y$ for some $w\in F$ and $\sfh\in\coneH$. Then 
\be\label{eq: rfrak is invariant}
4\noI_{\cone, 2\scmf}(e,\exp(w)y)\geq \noI_{\cone,\scmf}(e,z).
\ee
In particular, there exists some $w_0\in F$ so that 
\be\label{eq: rfrak is invariant sup}
4\noI_{\cone, 2\scmf}(e,\exp(w_0)y)\geq \sup_z\noI_{\cone,\scmf}(e,z). 
\ee
\end{lemma}

\begin{proof} 
The proof is similar to the proof of Lemma~\ref{lem: 1-1 maps between Iz and subsets of F}.
Let us write $z'=\exp(w)y$, i.e., $z=\sfh z'$. 
Let $v\in\margI_{\cone,\scmf}(e,z)$. Then $\exp(v)z\in\cone$, hence, there exist $\hat w_{v}\in F$ and $\hat\sfh\in\coneH$ so that 
\be\label{eq:zhat zbar w vw}
\begin{aligned}
\exp(v)z&=\hat\sfh\exp(\hat w_v)y=\hat\sfh\exp(\hat w_v)\exp(-w)\exp(w)y\\
&=\hat\sfh\exp(\hat w_v)\exp(-w)z'=\hat\sfh \sfh_v\exp(w_v)z';
\end{aligned}
\ee
for some $\sfh_v\in H$ and $w_v\in\rfrak$ so that 
\be\label{eq: size v-w h-w}
\text{$0.5\|\hat w_v-w\|\leq \|w_v\|\leq 2\|\hat w_v-w\|\quad$ and $\quad\|\sfh_v-I\|\leq \ref{E:BCH}\beta\|w_v\|$},
\ee 
see Lemma~\ref{lem: BCH}. 

Using Lemma~\ref{lem:dist-sheet}, recall that $\scmf\,\inj(z)\leq 0.01\eta$, we conclude that
\be\label{eq:hw-vw-yet-again}
\|w_v\|\leq 2\|v\|\leq  2\scmf\,\inj(z). 
\ee 
This and~\eqref{eq: size v-w h-w} imply that $\|\sfh_v-I\|\ll \scmf\,\inj(z)\leq \beta^2$ where the implied constant is absolute; 
hence, $\sfh_v^{\pm1}\in\coneH$.  
Moreover, comparing the second and the last term in~\eqref{eq:zhat zbar w vw}, it follows that
$\sfh_v\exp(w_v) z'=\exp(\hat w_v)y$. Since $\hat w_v\in F$,  
\[
\exp(w_v)z'=\sfh_v^{-1}\exp(\hat w_v)y\in\cone.
\]
We deduce that $w_v\in \margI_{\cone,2\scmf}(e,z')$. 
Furthermore, note that the map $v\mapsto w_v$ is injective. Hence, 
\be\label{eq: rfrak is invariant'}
\#\margI_{\cone,2\scmf}(e,z')\geq \#\margI_{\cone,\scmf}(e,z).
\ee

Recall now that $0.5\inj(z')\leq \inj( z)\leq 2\inj(z')$, and 
\[
\noI_{\cone,\scmf}(h,z)=\Bigl(\#\margI_{\cone,\scmf}(h,z)\Bigr)\cdot (\scmf\,\inj(hz))^{-\alpha},
\] 
see~\eqref{eq: def noI h,z}. Therefore,~\eqref{eq: rfrak is invariant} follows from~\eqref{eq: rfrak is invariant'}. 

\medskip

To see the second claim, let $\hat z$ be so that $\sup_z\noI_{\cone,\scmf}(e,z)=\noI_{\cone,\scmf}(e,\hat z)$. 
By the definition of $\cone$, there exists some $w\in F$ and $\sfh\in\coneH$ so that $\hat z=\sfh\exp(w)y$. The claim thus follows 
from~\eqref{eq: rfrak is invariant}. 
\end{proof}

\subsection*{Cubes and the function $\noI$}\label{sec: use regular tree}
Recall that $\cone=\{\exp(w)y: w\in F\}\subset X_\eta$. 
For a parameter $\mathsf M$ and every $k\in\bbn$, we let $\mathcal Q_{\mathsf Mk}$ denote the collection of $2^{-\mathsf Mk}$-cubes, see~\S\ref{sec: regularization lemmas}. Let $k_0\in\bbn$ be so that  
\[
2^{-k_0-1}\leq \mfsc\,\inj(y)<2^{-k_0}.
\]

\begin{lemma}\label{lem: regular sets initial dim}
Let $k_1>k_0$ be an integer, and assume that for every integer 
$k_0-10\leq k\leq k_1$, there exists $\tau_{k}>0$ so that, for all $Q\in\mathcal Q_{\mathsf Mk}$ 
\be\label{eq: regular tree'}
\text{either}\quad2^{\mathsf M(\tau_{k}-2)}\leq \#(F\cap Q)\leq 2^{\mathsf M\tau_{k}}\quad\text{or}\quad F\cap Q=\emptyset.
\ee
Let $z=h\exp(w)y\in\cone$ where $h\in \overline{\coneH\setminus\partial_{10\mfsc}\coneH}$.
Then 
\[
\ref{C: M and dimension}^{-1}\!\!\sup_{w'\in F}\noI_{\cone,\mfsc}(e,\exp(w')y)\leq \noI_{\cone,\mfsc}(e,z)\\
\leq\ref{C: M and dimension}\!\!\sup_{w'\in F}\noI_{\cone,\mfsc}(e,\exp(w')y) 
\]
where $\constE\label{C: M and dimension}$ depends on $\mathsf M$.

Furthermore, 
\[
\noI_{\cone,\mfsc}(e,z)
\leq\ref{C: M and dimension}\!\!\sup_{w'\in F}\noI_{\cone,\mfsc}(e,\exp(w')y) 
\]
holds true for all $z\in\cone$. 
\end{lemma}

\begin{proof}
The upper bound is a consequence of Lemma~\ref{lem: rfrak is invariant}. Indeed by that lemma, we have 
\[
\noI_{\cone,\scmf}(e,z)\leq 4\sup_{w'}\noI_{\cone,2\scmf}(e,\exp(w')y).
\]
To replace $2\mfsc$ with $\mfsc$, note that~\eqref{eq: regular tree'} and the definition of $\noI$ imply 
\[
\sup_{w'}\noI_{\cone,2\scmf}(e,\exp(w')y)\ll\sup_{w'}\noI_{\cone,\scmf}(e,\exp(w')y)
\] 
where the implied constant depends on $\mathsf M$. 
The upper bound estimate for $\noI_{\cone,\scmf}(e,z)$ follows. 

As the proof shows, we did not use the condition on $h$ for this bound, thus the final claim follows.  

We now turn to the proof of the lower bound. 
Since $h\in \overline{\coneH\setminus\partial_{10\mfsc}\coneH}$, Lemma~\ref{lem: 1-1 maps between Iz and subsets of F} applied with $z$, $w$ and $\delta=\scmf\inj(z)$, implies 
\[
\#\Bigl(F\cap B_\rfrak(w,\scmf\inj(z)/2)\Bigr)\leq \#\margI_{\cone,\scmf}(e,z)
\]
This and the definition of $\noI$ yield the following: 
\begin{equation}\label{eq:mfht margI large Fi proof}
\begin{aligned}
\noI_{\cone,\scmf}(e,z)&=(\#\margI_{\cone,\scmf}(e,z))\cdot(\scmf\,\inj(z))^{-\alpha}\\
&\geq \#(F\cap B_\rfrak(w',\scmf\,\inj(z)/2))\cdot(\scmf\,\inj(z))^{-\alpha}\\
&\gg \sup_{w'}(\#F\cap B_\rfrak(w',4\scmf\,\inj(z)))\cdot(\scmf\,\inj(z))^{-\alpha},
\end{aligned}
\end{equation}
where we used~\eqref{eq: regular tree'} in the last inequality.

Note that for all $w'\in F$, we have $\inj(z)/2\leq \inj(\exp(w')y)\leq 2\inj(z)$. Moreover, 
$\margI_{\cone,\mfsc}(e,\exp(w')y)=\margI_{\cone',\mfsc}(e,\exp(w')y)$ where 
\[
\cone'=\Bigl(\overline{\coneH\setminus\partial_{5\mfsc}\coneH}\Bigr)\cdot\{\exp(w'')y: w''\in F\}.
\]
Thus~\eqref{eq:mfht margI large Fi proof} and Lemma~\ref{lem: 1-1 maps between Iz and subsets of F}, applied with $\delta=\mfsc\,\inj(\exp(w')y)$, imply 
\[
\noI_{\cone,\mfsc}(e,z)\gg\sup_{w'}\noI_{\cone,\mfsc}(e,\exp(w')y).
\]
The proof is complete. 
\end{proof}

We also record the following lemma which is similar to Lemma~\ref{eq: numb Fj 1}.

\begin{lemma}\label{lem:noI-tri-bd}
There exists $\constE\label{E:noI}>0$ so that the following holds. 
Let $0<\scmf\leq \beta^6$. Then for every $m\in\bbn$ with $\nuni^{m}\leq \scmf^{-1/2}$, 
every $|\rel|\leq 2$, and every $z\in\cone\subset X_\eta$, we have 
\[
\noI_{\cone,\scmf}(a_mu_\rel,z)\leq \ref{E:noI}\eta^{-3}\nuni^{4m}\cdot \Bigl(\sup_{z'}\noI_{\cone,\scmf}(e,z')\Bigr).
\]
\end{lemma}

\begin{proof}
Let $z\in\cone$, and let $w\in\margI_{\cone,\scmf}(a_mu_r,z)$. Then $\exp(w)a_mu_rz\in a_mu_r\cone$
which implies $\exp(\Ad(a_{-m}u_{-r})w)z\in\cone$. Moreover, we have 
\[
\|\Ad(a_{-m}u_{-r})w\|\leq 100\nuni^m \inj(a_mu_rz) \scmf\leq 100\nuni^m \scmf=:\scmf'.
\] 
Since $\inj(z)\geq \eta$, we get that $\inj(z)\scmf'/\eta\geq \scmf'$, hence 
\[
\Ad(a_{-m}u_{-r})w\in \margI_{\cone, b'/\eta}(e,z).
\] 
This and the fact that $\nuni^m\mfsc\leq \mfsc^{1/2}\leq \beta^3$ imply: 
$w\mapsto \Ad(a_{-m}u_{-r})w$ is an injection map from $\margI_{\cone,\scmf}(a_mu_r,z)$ into $\margI_{\cone, \mfsc'/\eta}(e,z)$.

Now arguing as in the proof of Lemma~\ref{lem: rfrak is invariant}, with $\mfsc$ replaced by $\mfsc'/\eta\leq \beta^2$, we conclude that 
\[
 \#\margI_{\cone, \mfsc'/\eta}(e,z)\leq \# \Bigl(F\cap B_\rfrak(w_z,2\mfsc'/\eta)\Bigr),
\]
for some $w_z\in F$. 
Note moreover that $B_\rfrak(w_z, \mfsc'/\eta)$ may be covered with $\ll \eta^{-3}\nuni^{3m}$ boxes of the form $B_{\rfrak}(w_i,\mfsc/2)$; thus
\begin{align*}
\#\margI_{\cone,\scmf}(a_mu_r,z)\leq \#\margI_{\cone, \mfsc'/\eta}(e,z)&\leq \# \Bigl(F\cap B_\rfrak(w_z,2\mfsc'/\eta)\Bigr)\\
&\ll\eta^{-3}\nuni^{3m}\cdot \sup_{w'}\#\Bigl(F\cap B_\rfrak(w',\mfsc/2)\Bigr)\\
&\ll\eta^{-3}\nuni^{3m}\cdot \Bigl(\sup_{z'}\#\margI_{\cone, \mfsc}(e,z')\Bigr),
\end{align*}
see also Lemma~\ref{lem: 1-1 maps between Iz and subsets of F} for the last inequality. 

Since $\inj(a_mu_rz)\gg e^{-m}\inj(z)$, 
\[
\noI_{\cone,\mfsc}(h,z)=(\inj(hz)\mfsc)^{-\alpha}\cdot\Bigl(\max\{\#\margI_{\cone,\mfsc}(h,z),1\}\Bigr),
\] 
and $0<\alpha\leq1$, the lemma follows. 
\end{proof}

\subsection{The dimension improvement lemma}\label{sec: new cone dim}
As it was done before, let $\kappa=10^{-6}d_1^{-2}\leq \vare/10^6$.
Suppose 
\[
\cone_{\rm old}=\coneH.\{\exp(w)y_0: w\in F_{\rm old}\}
\] 
satisfies the conditions in Lemma~\ref{lem: truncated MF main estimate}. 
That is, $F_{\rm old}\subset B_\rfrak(0,\beta)$ is finite with $\#F_{\rm old}\geq \nuni^{9t/10}$, and
\begin{multline}\label{eq: mfsc regular recall in dim impr}
\#\Bigl(F_{\rm old}\cap B_\rfrak(w, \mfsc\,\inj(y_0)/10^3)\Bigr)\geq \\
e^{-\kappa^2t/4}\cdot \Bigl(\#(F_{\rm old}\cap B_\rfrak(w, \mfsc\,\inj(y_0)/10))\Bigr).
\end{multline} 
Moreover, for all $z\in\cone_{\rm old}$, we have 
\be\label{eq: mfm init estimate'} 
\mfm_{\cone_{\rm old}, \scmf,\trct}(e,z)\leq\mfbd,
\ee
where $\mfbd\geq1$, $1\leq \trct\leq\nuni^{\vare t/100}$ and $\mfsc=\nuni^{-\sqrt\vare t}$. 

Let $\mu_{\cone_{\rm old}}$ be an admissible measure on $\cone_{\rm old}$. 
By Lemma~\ref{lem: truncated MF main estimate}, there exists ${\gdh_{\mu_{\cone_{\rm old}}}}\subset [0,1]$ with 
\[
\Bigl|[0,1]\setminus {\gdh_{\mu_{\cone_{\rm old}}}}\Bigr|\ll \nuni^{-\kappa^2t/4},
\]
and for every $r\in {\gdh_{\mu_{\cone_{\rm old}}}}$, there exists a subset 
\be\label{eq: cone Old,r is in hat cone Old}
\cone_{{\rm old},r}\subset\hat\cone_{\rm old}=\bigcup\hat\coneH.\{\exp(w)y_0: w\in F\}, \qquad (\hat\coneH=\overline{\coneH\setminus \partial_{10\mfsc}\coneH})
\ee
satisfying $\mu_{\cone_{\rm old}}(\cone_{\rm old}\setminus\cone_{{\rm old},r})\ll\nuni^{-\kappa^2t/64}$ 
and the following: for all $z'\in\cone_{{\rm old},r}$,   
\be\label{eq: mfm estimate use}
\mfm_{\hat\cone_{\rm old},\scmf,\trct_1}(a_{\ell} u_r,z')\leq 
200\nuni^{-\alpha\ell}L_1 \egbd^{1+8\kappa}+200\nuni^{2\alpha\ell}\noI_{\hat\cone_{\rm old},\scmf}(a_{\ell} u_r,z');
\ee 
where $L_1=L\kappa^{-L}$ and $\trct_1=\trct+L_1\mfbd^{\kappa}$, and we assume $\egbd$ is large enough compared to $\kappa$, see also Theorem~\ref{thm: proj thm}. 

\medskip

Let us put $\hat{\hat\coneH}=\overline{\coneH\setminus \partial_{10^3\beta^2}\coneH}$, and define 
\be\label{eq: define hat hat 1}
\hat{\hat\cone}_{\rm old}=\hat{\hat\coneH}.\{\exp(w)y_0: w\in F_{\rm old}\}.
\ee

The following lemma is an important ingredient in the proof of Lemma~\ref{lem: dimension improve 2}; the latter will be applied in every step of our inductive argument. 
Roughly speaking, in view of \eqref{eq: mfm estimate use}, Lemma~\ref{lem: dimension improves for new cone} implies
that for $r\in \gdh_{\mu_{\cone_{\rm old}}}$ offsprings of 
$a_\ell u_r\cone_{\rm old}$ (see \S\ref{sec: conv comb induction})
have improved coarse dimension, possibly after slight trimming.

Let us recall the notation 
\[
\umt_\ell^H=\{u^-_s:|s|\leq e^{-\ell}\beta^2\}\cdot\{a_\tau:|\tau|\leq \beta^2\}\cdot U_\eta.
\] 

\begin{lemma}\label{lem: dimension improves for new cone}
With the above notation, let $r\in \gdh_{\mu_{\cone_{\rm old}}}$. 
Let $(\cone',\mu_{\cone'})$, 
\[
\cone'=\coneH.\{\exp(w)y: w\in F'\}\subset X_\eta,
\]
be an offspring of $a_\ell u_r\mu_{\cone_{\rm old}}$, see \eqref{eq: def conej 2} and~\eqref{eq: def mu conj 2}. Recall from~\eqref{eq: interior mu r-i} that 
\[
\text{$\umt^H_\ell.\exp(w)y\subset a_\ell u_r.\cone_{\rm old}\quad$ for all $w\in F'$},
\]
Let $F\subset F'$ satisfy that for all $w\in F$, we have 
\be\label{eq: leaves meet cone-r}
\umt^H_\ell.\exp(w)y\cap \Bigl(a_\ell u_r.(\cone_{{\rm old},r}\cap \hat{\hat\cone}_{\rm old})\Bigr)\neq\emptyset,
\ee
and put $\cone=\coneH.\{\exp(w)y: w\in F\}$ and $\mu_\cone=\frac{1}{\mu_{\cone'}(\cone)}\mu_{\cone'}|_{\cone}$.

Then for every $z=h\exp(w)y\in\cone$ (where $h\in\coneH$ and $w\in F$), we have 
\be\label{eq: used to be sublemma}
\mfm_{\cone,\scmf,\trct_1}(e,z)\leq 2\mfm_{\hat\cone_{\rm old},\scmf, \trct_1}(a_\ell u_r,z_0)+10\noI_{\cone,\scmf}(e,z)
\ee 
where $z_0\in\cone_{{\rm old},r}\cap \hat{\hat\cone}_{\rm old}$ is so that $a_\ell u_r z_0=\sfh_0\exp(w)y$ for some $\sfh_0\in\umt_\ell^H$. 
\end{lemma}

\begin{proof}
Note that  
\be\label{eq: estimate mfm ind}
\mfm_{\cone,\scmf,\trct_1}(e,z)\leq \sum_{I}\|v\|^{-\alpha}+10\noI_{\cone,\scmf}(e,z)
\ee
for every $I\subset\bigl\{v\in\margI_{\cone, \mfsc}(e,z): \|v\|\leq 0.1\mfsc\,\inj(z)\bigr\}$ with 
$\#(\margI_{\cone,\mfsc}(e,z)\setminus I)\leq \trct_1$.
We will relate the first term on the right side of~\eqref{eq: estimate mfm ind} to 
\[
\mfm_{\hat\cone_{\rm old},\scmf, \trct_1}(a_\ell u_r,z_0).
\]

Let us begin with the following computation. 
Let $w\neq w_1\in F$, and let $z_1\in\cone_{{\rm old},r}\cap \hat{\hat\cone}_{\rm old}$ and 
$\sfh_1\in\umt_\ell^H$ be so that $\sfh_1\exp(w_1)y=a_\ell u_r z_1$. Then 
\be\label{eq: zi hi wi} 
\begin{aligned}
a_\ell u_rz_1&=\sfh_1\exp(w_1)y=\sfh_1\exp(w_1)\exp(-w)\sfh_0^{-1}a_\ell u_rz_0\\
&=\sfh_1\sfh_0^{-1}\exp(\Ad(\sfh_0)w_1)\exp(-\Ad(\sfh_0)w)a_\ell u_rz_0\\
&=\sfh_1\sfh_0^{-1}\hat\sfh\exp(\hat w)a_\ell u_rz_0
\end{aligned}
\ee
where $\hat\sfh\in H$ and $\hat w\in\rfrak$, moreover, by Lemma~\ref{lem: BCH}, we have
\begin{subequations}
\begin{align}
\label{eq:hi-hat-wi-1}&\|\hat\sfh-I\|\leq \ref{E:BCH}\beta\|\hat w\|\qquad\text{and}\\
\label{eq:hi-hat-wi-2}&0.5\|\Ad(\sfh_0)(w-w_1)\|\leq \|\hat w\|\leq2\|\Ad(\sfh_0)(w-w_1)\|.
\end{align}
\end{subequations}

Let $v\in I_{\cone,\scmf}(e,z)$. Then $z,\exp(v)z\in\cone$, and we have 
\[
z=h\exp(w)y=h\sfh_0^{-1}a_\ell u_rz_0=\bar ha_\ell u_rz_0,
\]
where $\bar h\in\boxH_{1.1\eta}$, recall that $z_0\in \cone_{{\rm old},r}\cap \hat{\hat\cone}_{\rm old}$.
Similarly, since $\exp(v)z\in\cone$, there exist $w_v\in F$ and $z_v\in \cone_{{\rm old},r}\cap \hat{\hat\cone}_{\rm old}$ 
so that 
\[
\text{$\exp(v)z=h'\exp(w_v)y\quad$ and $\quad\sfh_v\exp(w_v)y=a_\ell u_rz_v$.}
\] 
Thus, $\exp(v)z= \bar h_va_\ell u_rz_v$ where $z_v\in \cone_{{\rm old},r}\cap \hat{\hat\cone}_{\rm old}$ and $\bar h_v\in\boxH_{1.1\eta}$. Hence  
\be\label{eq: z-v and z0 relation 0}
\begin{aligned}
a_\ell u_rz_v=\bar h_v^{-1}\exp(v)z&=\bar h_v^{-1}\exp(v)\bar ha_\ell u_rz_0\\
&=\bar h_v^{-1}\bar h\exp(\Ad(\bar h^{-1})v)a_\ell u_rz_0
\end{aligned}
\ee

Applying~\eqref{eq: zi hi wi} with $w_1=w_v$ and $\sfh_1=\sfh_v$, we get that  
\be\label{eq: z-v and z0 relation}
a_\ell u_rz_v=\sfh_{v}\sfh_0^{-1}\hat\sfh\exp(\hat w_v)a_\ell u_rz_0
\ee
where $\hat \sfh$ and $\hat w_v$ satisfy~\eqref{eq:hi-hat-wi-1} and~\eqref{eq:hi-hat-wi-2}, and $\sfh_0,\sfh_v\in\umt_\ell^H$. 

Since $(\hat h,\hat w)\mapsto\hat h\exp(\hat w)a_\ell u_rz_0$ is injective over $\boxH_{10\eta}\times B_\rfrak(0,10\eta)$, we conclude from~\eqref{eq: z-v and z0 relation} and~\eqref{eq: z-v and z0 relation 0} that $\hat w_v=\Ad(\bar h^{-1})v$. In particular, 
\be\label{eq:wij-sublemma}
\|\hat w_v\|\leq 2\|v\|.
\ee

Moreover, the elements $\{z_v: v\in\margI_{\cone,\mfsc}(e,z)\}$ belong to different local $H$-orbits, thus $v\mapsto \hat w_v$ is well-defined and one-to-one.

Recall that $\cone\subset X_{\eta}$. Assume now that 
$\|v\|\leq \scmf\,\inj(z)/10$, then $\|\hat w_v\|\leq \scmf\,\inj(z)/5$.
This estimate and~\eqref{eq:hi-hat-wi-1} imply that
\[
\|\hat\sfh-I\|\leq \ref{E:BCH}\beta\|\hat w_{v}\| \ll \scmf\beta\leq \beta^2\nuni^{-\ell}; 
\]
recall that $\scmf\leq e^{-\sqrt\vare t}$ and $e^{-\ell},\beta\geq e^{-0.01\vare t}$.
 
In view of the definition of $\hat{\hat\cone}_{\rm old}$ in~\eqref{eq: define hat hat 1}, we have 
\[
\text{$z_v\in \cone_{{\rm old},r}\cap \hat{\hat\cone}_{\rm old}\quad$ implies $\quad\boxH_{100\beta^2}.z_v\subset \hat\cone_{\rm old}$.}
\]
Moreover, $\sfh_0,\sfh_v\in\umt^H_\ell$ and $\|\hat\sfh-I\|\leq \beta^2\nuni^{-\ell}$. 
Therefore, 
\[
\hat\sfh^{-1}\sfh_0\sfh_v^{-1}a_\ell u_rz_v\in a_\ell u_r\hat\cone_{\rm old},
\]
see~\eqref{eq:well-rd-tau-1}. This and~\eqref{eq: z-v and z0 relation} yield 
\[
\exp(\hat w_v)a_\ell u_rz_0=\hat\sfh^{-1}\sfh_0\sfh_v^{-1}a_\ell u_rz_v\in a_\ell u_r\hat\cone_{\rm old}.
\]
This and $\|\hat w_{v}\|\leq \scmf\,\inj(z)/5<\scmf\,\inj(a_\ell u_rz_0)$ imply $\hat w_{v}\in \margI_{\hat\cone_{\rm old},\scmf}(a_\ell u_r,z_0)$. 

Let now $J\subset \margI_{\hat\cone_{\rm old},\scmf}(a_\ell u_r,z_0)$ be a subset so that  
\[
\text{$\#\margI_{\hat\cone_{\rm old},\scmf}(a_\ell u_r,z_0)\setminus J=\trct_1\quad$ and }\quad\mfht_{\hat\cone_{\rm old},\scmf,\trct_1}(a_\ell u_r,z_0)=\sum_{\hat w\in J}\|\hat w\|^{-\alpha}.
\]
Put $I_J=\{v\in\margI_{\cone, \mfsc}(e,z): \|v\|\leq 0.1\mfsc\,\inj(z), \hat w_v\not\in J\}$. 
Since $v\mapsto \hat w_{v}$ is a one-to-one map from $I_J$ into 
$\margI_{\hat\cone_{\rm old},\scmf}(a_\ell u_r,z_0)\setminus J$, we have $\#I_J\leq \trct_1$. 
Applying~\eqref{eq: estimate mfm ind} with 
\[
I=\{v\in\margI_{\cone, \mfsc}(e,z): \|v\|\leq 0.1\mfsc\}\setminus I_J,
\] 
and using~\eqref{eq:wij-sublemma}, we conclude 
\[
\mfm_{\cone,\scmf,\trct_1}(e,z)\leq 2\mfht_{\hat\cone_{\rm old},\scmf,\trct_1}(a_\ell u_r,z_0)+10\noI_{\cone,\scmf}(e,z),
\]
as it was claimed in the lemma. 
\end{proof}

Recall that $d_1=100\lceil (4D-3)/2\vare\rceil$, $\kappa=10^{-6}d_1^{-2}$, and $\ell=0.01\vare t$, see Proposition~\ref{propos: imp dim main}. From this point to the end of this section, we will assume 
\be\label{eq: kappa is small}
\mfbd^{\kappa d_1}\leq \nuni^{\ell/100}.
\ee
Moreover, we assume that $t$ is large enough so that 
\be\label{eq: compare t and kappa}
L_1=L\kappa^{-L}<e^{\ell/100}
\ee 
--- this amounts to $t\gg |\log\vare|/\vare$, later we will choose $\vare$ to depend only on $\mixexp$ in~\eqref{eq: exp mixing}. 
We will also assume that $0.9<\alpha<1$. 

The following lemma combines the results in this section, 
and will be applied in every step of our inductive proof of Proposition~\ref{propos: imp dim main}. 

\begin{lemma}\label{lem: dimension improve 2}
Let the notation be as in Lemma~\ref{lem: dimension improves for new cone}. In particular,
\[
\text{$L_1=L\kappa^{-L}\quad$ and $\quad\trct_1=L_1\mfbd^{\kappa}\trct$.}
\] 
Assume further that~\eqref{eq: regular tree'} (with some parameter $\mathsf M$) holds true for $F_{\rm old}$. 

Let $w_0\in F_{\rm old}$ be so that 
\[
\noI_{\cone_{\rm old},\mfsc}(e,\exp(w_0)y_0)=\sup_{w'}\noI_{\cone_{\rm old},\mfsc}(e,\exp(w')y_0).
\]
Then we have the following.
\begin{enumerate}
\item If $\mfbd\geq \nuni^{\vare t/2}\noI_{\cone_{\rm old},\mfsc}(e,\exp(w_0)y_0)$, then 
\be\label{eq: moderate density}
\mfm_{\cone,\scmf,\trct_1}(e,z)\leq \nuni^{-0.6\ell}\mfbd+10\noI_{\cone,\scmf}(e,z)\quad\text{for all $z\in\cone$}.
\ee
\item If $\mfbd<\nuni^{\vare t/2}\noI_{\cone_{\rm old},\mfsc}(e,\exp(w_0)y_0)$, 
then both of the following hold
\begin{enumerate}
\item For every $\hat z=\hat h\exp(\hat w)y_0\in\cone_{\rm old}$ with $\hat h\in\overline{\coneH\setminus\partial_{10\mfsc}\coneH}$, we have
\be\label{eq: high density 1}
\mfm_{\cone_{\rm old},\scmf,\trct}(e,\hat z)\leq \nuni^{\vare t/2}\noI_{\cone_{\rm old},\mfsc}(e,\exp(w_0)y_0)\leq\ref{C: M and dimension}\nuni^{\vare t/2}\noI_{\cone_{\rm old},\mfsc}(e,\hat z)
\ee
where $\ref{C: M and dimension}$ is as in Lemma~\ref{lem: regular sets initial dim} (which depends on $\mathsf M$). 
\item For every $z\in\cone$, we have 
\be\label{eq: high density 2}
\mfm_{\cone,\scmf,\trct_1}(e,z)\leq \nuni^{-0.6\ell}\Bigl(\nuni^{{\vare t}/{2}}\!\cdot\noI_{\cone_{{\rm old},\mfsc}}(e,\exp(w_0)y_0)\Bigr)\!+10\noI_{\cone,\scmf}(e,z).
\ee
\end{enumerate}
\end{enumerate}
\end{lemma}

\begin{proof}
Since~\eqref{eq: regular tree'} holds true for $F_{\rm old}$, Lemma \ref{lem: regular sets initial dim} is applicable with $\cone_{\rm old}$; we will utilize that lemma several times in the course of the proof.  

Let $z=h\exp(w)y\in\cone$, and let $z'\in\cone_{{\rm old},r}\cap \hat{\hat\cone}_{\rm old}$ be so that $a_\ell u_r z'=\sfh\exp(w)y$ for some $\sfh\in\umt_\ell^H$. 
By Lemma~\ref{lem: dimension improves for new cone}, we have 
\be\label{eq: dim improves for new cone use}
\mfm_{\cone,\scmf,\trct_1}(e,z)\leq 2\mfm_{\hat\cone_{\rm old},\scmf,\trct_1}(a_{\ell} u_r,z')+10\noI_{\cone,\scmf}(e,z).
\ee
Moreover, since $z'\in \cone_{{\rm old},r}$, we conclude from~\eqref{eq: mfm estimate use} that 
\be\label{eq: mfm estimate use to improve}
\mfm_{\hat\cone_{\rm old},\scmf,\trct_1}(a_{\ell} u_r,z')\leq 
200\nuni^{-\alpha\ell}L_1 \egbd^{1+8\kappa}+200\nuni^{2\alpha\ell}\noI_{\hat\cone_{\rm old},\scmf}(a_{\ell} u_r,z').
\ee

We give initial bounds for the two terms on the right side of~\eqref{eq: mfm estimate use to improve}. In view of~\eqref{eq: kappa is small} and~\eqref{eq: compare t and kappa}, we have  
\be\label{eq: 1st term in mfm use improve}
200\nuni^{-\alpha\ell}L_1 \egbd^{1+8\kappa}\leq e^{-0.7\ell}\egbd,
\ee
where we also used $0.9<\alpha<1$ and assumed $\ell=\vare t/100$ is large enough to account for the factor $200$. 
 
As for the second term, 
using the fact that $\hat\cone_{\rm old}\subset\cone_{\rm old}$, we obtain 
\be\label{eq: 2nd term in mfm use improve 1}
\begin{aligned}
200\nuni^{2\alpha\ell}\noI_{\hat\cone_{\rm old},\scmf}(a_{\ell} u_r,z')&\leq200\nuni^{2\alpha\ell}\noI_{\cone_{\rm old},\scmf}(a_{\ell} u_r,z')\\
&\leq200\ref{E:noI}\eta^{-3}\nuni^{6\ell}\cdot \sup_{z''}\noI_{\cone_{\rm old},\scmf}(e,z'')\\
&\leq e^{\vare t/10}\cdot \sup_{w'}\noI_{\cone_{\rm old},\scmf}(e,\exp(w')y_0);
\end{aligned}
\ee
we used Lemma~\ref{lem:noI-tri-bd} in the second inequality and 
used (the final claim in) Lemma~\ref{lem: regular sets initial dim} to replace $\sup_{z''}$ by $\sup_w$, we also used $\eta>e^{-0.01\ell}$ and assumed $t$ is large to account for the constants
$\ref{C: M and dimension}$ and $200\ref{E:noI}$.

We now begin the proof of the estimates in the lemma. Let us first assume 
\be\label{eq: moderate density assump}
\mfbd\geq \nuni^{\vare t/2}\cdot\noI_{\cone_{\rm old},\mfsc}(e,\exp(w_0)y_0),
\ee
where $\noI_{\cone_{\rm old},\mfsc}(e,\exp(w_0)y_0)=\sup_{w'}\noI_{\cone_{\rm old},\mfsc}(e,\exp(w')y_0)$, as in the statement of the lemma. 
Then~\eqref{eq: 2nd term in mfm use improve 1} and~\eqref{eq: moderate density assump} imply that 
\be\label{eq: 2nd term in mfm use improve}
\begin{aligned}
200\nuni^{2\alpha\ell}\noI_{\hat\cone_{\rm old},\scmf}(a_{\ell} u_r,z')&\leq e^{\vare t/10}\cdot \sup_{w'}\noI_{\cone_{\rm old},\scmf}(e,\exp(w')y_0) \\
&\leq  e^{\vare t/10}\cdot (e^{-\vare t/2}\mfbd) \leq \nuni^{-\ell}\mfbd
\end{aligned}
\ee
where we used $\ell=\vare t/100$.

Thus, combining~\eqref{eq: dim improves for new cone use},~\eqref{eq: mfm estimate use to improve},~\eqref{eq: 1st term in mfm use improve}, and~\eqref{eq: 2nd term in mfm use improve}, one gets    
\begin{align*}
\mfm_{\cone,\scmf,\trct_1}(e,z)&\leq 2\mfm_{\hat\cone_{\rm old},\scmf,\trct_1}(a_{\ell} u_r,z')+10\noI_{\cone,\scmf}(e,z)\\
&\leq e^{-0.7\ell}\egbd+\nuni^{-\ell}\mfbd+10\noI_{\cone,\scmf}(e,z)\\
&\leq e^{-0.6\ell}\mfbd+10\noI_{\cone,\scmf}(e,z).
\end{align*}
This establishes part~(1).

Let us now turn to the proof of part~(2). Therefore, we assume   
\be\label{eq: high density assump}
\mfbd< \nuni^{\vare t/2}\cdot\noI_{\cone_{\rm old},\mfsc}(e,\exp(w_0)y_0).
\ee
First note that by Lemma~\ref{lem: regular sets initial dim}, if $\hat z=\hat h\exp(\hat w)y_0\in\cone_{\rm old}$ 
where $\hat h\in\overline{\coneH\setminus\partial_{10\mfsc}\coneH}$, 
\be\label{eq: high density assump'}
\ref{C: M and dimension}^{-1}\noI_{\cone_{\rm old},\mfsc}(e,\exp(w_0)y_0)\leq \noI_{\cone_{\rm old},\mfsc}(e,\hat z)\leq \ref{C: M and dimension}\noI_{\cone_{\rm old},\mfsc}(e,\exp(w_0)y_0).
\ee

We conclude that  
\begin{align*}
\mfm_{\cone_{\rm old}, \scmf,\trct}(e,\hat z)\leq\mfbd&\leq \nuni^{\vare t/2}\cdot\noI_{\cone_{\rm old},\mfsc}(e,\exp(w_0)y_0)\\
&\leq \ref{C: M and dimension} \nuni^{\vare t/2}\cdot\noI_{\cone_{\rm old},\mfsc}(e,\hat z),
\end{align*}
where we used~\eqref{eq: mfm init estimate'} in the first inequality, used~\eqref{eq: high density assump} in the second inequality, and used~\eqref{eq: high density assump'} in the final inequality. This gives~\eqref{eq: high density 1}.

We now turn to the proof of~\eqref{eq: high density 2}. Recall from~\eqref{eq: 1st term in mfm use improve} and \eqref{eq: 2nd term in mfm use improve 1}, 
\begin{align*}
\mfm_{\hat\cone_{\rm old},\scmf,\trct_1}(a_{\ell} u_r,z')&\leq 
200\nuni^{-\alpha\ell}L_1 \egbd^{1+8\kappa}+200\nuni^{2\alpha\ell}\noI_{\hat\cone_{\rm old},\scmf}(a_{\ell} u_r,z')\\
&\leq e^{-0.7\ell}\mfbd+ e^{\vare t/10}\noI_{\hat\cone_{\rm old},\scmf}(e,\exp(w_0)y_0).
\end{align*}

In view of~\eqref{eq: high density assump} and since $\ell=\vare t/100$, we have 
\[
e^{-0.7\ell}\mfbd+ e^{\vare t/10}\noI_{\hat\cone_{\rm old},\scmf}(e,\exp(w_0)y)\leq e^{-0.6\ell}\Bigl(\nuni^{{\vare t}/{2}}\!\cdot\noI_{\cone_{{\rm old},\mfsc}}(e,\exp(w_0)y_0)\Bigr).
\]
Finally, using~\eqref{eq: dim improves for new cone use} and the above, we conclude that  
\begin{align*}
\mfm_{\cone,\scmf,\trct_1}(e,z)&\leq 2\mfm_{\hat\cone_{\rm old},\scmf,\trct_1}(a_{\ell} u_r,z')+10\noI_{\cone,\scmf}(e,z)\\
&\leq e^{-0.6\ell}\Bigl(\nuni^{{\vare t}/{2}}\!\cdot\noI_{\cone_{{\rm old},\mfsc}}(e,\exp(w_0)y_0)\Bigr)+10\noI_{\cone,\scmf}(e,z).
\end{align*}
The proof is complete. 
\end{proof}

\section{An inductive construction}\label{sec: induction}
As it was mentioned, the proof of Proposition~\ref{propos: imp dim main} is based on an inductive construction.
We will carry out this construction in this section and complete the proof of Proposition~\ref{propos: imp dim main} 
in the next section.

Recall that $0<\vare<1$ is a small parameter (in our application, $\vare$ will depend on $\ref{k:mixing}$, see~\eqref{eq: choose theta equi sec}) and $t>1$ is a large parameter (which will be chosen to be $\asymp\log R$ where $R$ is as in Theorem~\ref{thm:main}). Recall also that 
\be\label{eq: choose kappa 1}
\kappa = 10^{-6}d_1^{-1}\leq 10^{-6}\vare,
\ee 
where $d_1=100\lceil{(4D-3)}/(2\vare)\rceil$, see Proposition~\ref{propos: imp dim main}. 

Set $\mfsc=e^{-\sqrt\vare t}$, $\beta=e^{-\kappa t}$, and $\eta^2=\beta$.

From now until the end of \S\ref{sec: proof main prop}, we fix some $\mathsf M$ 
so that
\be\label{eq: condition on M 1}
\text{$2^{-{\mathsf M}}(D+1)< \kappa/100\quad$ and $\quad 6{\mathsf M}<2^{\kappa {\mathsf M}/100}$}.
\ee
That is, conditions in~\eqref{eq: condition on M} are satisfied with $\kappa=10^{-6}d_1^{-1}$ and $m_0=D$; note that $\kappa(D+1)\leq 10^{-6}\vare$. 
In particular, Lemma~\ref{lem: regular tree decomposition} is applicable with $\mathsf M$ and any $F\subset B_\rfrak(0,\beta)$ satisfying $e^{t/2}\leq\#F\leq e^{2t}$ and~\eqref{eq: tree dec eng bd} with $\egbd\leq e^{(D+1)t}$. This lemma will be applied, several times, in this section.

\subsection{Consequences of Proposition~\ref{prop:closing lemma intro}}\label{sec: initial dim} 
Let $x_1$, $t$, and $D$ be as in Proposition~\ref{propos: imp dim main}. 
By our assumption, Proposition~\ref{prop:closing lemma intro}(1) holds for these choices. Recall that $x_2=a_{8t}u_{r_1}x_1$
where $r_1\in I(x_1)$. 
Then the map $h\mapsto hx_2$ is injective over $\boxHs_\beta\cdot a_t\cdot U_1$, see Proposition~\ref{prop:closing lemma intro}(1). In particular, Lemma~\ref{lem: convex comb 1} may be applied with $x_2$, and yields the following: 
for every $\varphi\in C_c^\infty(X)$, every $\tau>0$, and all $|s|\leq 2$, 
\begin{multline}\label{eq: initial dim convex comb}
\biggl|\int\varphi(a_\tau u_shx_2)\diff(\sigma\conv\rwm_t)(h)-\sum_i c_i\int\varphi(a_\tau u_sz)\diff\!\mu_{\cone_i}(z)\biggr|\\\ll \beta\Lip(\varphi)
\end{multline}
where the implied constant depends only on $X$. 

Recall from~\eqref{eq: def conej} that $\cone_i=\coneH.\{\exp(w)y_i: w\in F_i\}$ where $y_i\in X_{3\eta/2}$. In particular, $\cone_i\subset X_\eta$. Recall also from Lemma~\ref{eq: numb Fj 1} and Lemma~\ref{lem: popular conej's 1} that 
\be\label{eq: num of Fi initial dim}
\beta^9e^t\leq \#F_i \leq \beta^{-3}e^{t}.
\ee

Moreover, in view of the definition of $\cone_i$ and Proposition~\ref{prop:closing lemma intro}(1), we have
\be\label{eq: initial dim'}
\mfm_{\cone_{i}, \mfsc, 1}(e,z)\leq e^{Dt}
\ee 
for all $z\in \cone_i$. 

\subsection{Regular tree decomposition of $F_i$}
We will decompose $F_i$ into subsets which are homogeneous in all {\em relevant} scales. 
First note that in view of~\eqref{eq: initial dim'} and Lemma~\ref{lem: F energy est assuming mfm energy estimate} 
applied with $m=4$, we have 
\be\label{eq: eng bound for Fw use}
\eng_{F_{i,w}, \trct}(w')\leq 10^6e^{Dt}\qquad\text{for every $w'\in F_{i,w}$}
\ee
where for all $w\in F_i$, we put $F_{i,w}=F_i\cap B_\rfrak(w, 4\mfsc\inj(y_i))$.
 
Let $k_1> k_{i,0}$ be positive integers defined as follows: 
\be\label{eq: def k0i and k1}
2^{k_{i,0}}< (\mfsc\,\inj(y_i))^{-1}\leq 2^{k_{i,0}+1}\quad\text{and}\quad 2^{k_1}< 10^6e^{Dt}\leq 2^{k_1+1}.
\ee

Let $\mathsf M$ be as above, see~\eqref{eq: condition on M 1}. 
For every $i$ as above, apply Lemma~\ref{lem: regular tree decomposition} to $F_i$. Then we can write 
\be\label{eq: Fi as union ini dim}
F_i=F_i'\bigcup (\textstyle \bigcup_\varsigma F_i^\varsigma)
\ee
where $\#F_i'\leq \beta^{1/4}\cdot (\#F_i)$. 
Furthermore, for every $i$ and $\varsigma$ we have 
\be\label{eq: num F i varsigma}
\beta^{11}e^t\leq \beta^2\cdot(\#F_i)\leq \#F_i^\varsigma\leq \#F_i\leq \beta^{-3}e^{t},
\ee
(where we used~\eqref{eq: num of Fi initial dim}), and for every $k_{i,0}-10\leq k\leq k_1$, 
there exists some $\tau_{ik}^\varsigma$ so that 
for all $Q\in\mathcal Q_{\mathsf Mk}$ we have 
\be\label{eq: regular tree'' ???}
\text{either}\quad2^{\mathsf M(\tau_{i k}^\varsigma-2)}\leq \#F_i^\varsigma\cap Q\leq 2^{\mathsf M\tau_{i k}^\varsigma}\quad\text{or}\quad F_i^\varsigma\cap Q=\emptyset.
\ee

\subsection{Initial dimension}\label{sec: initial dim F-i-varsigma}
Put $\cone_i^\varsigma=\coneH.\{\exp(w)y_i: w\in F_i^\varsigma\}$ for all $i$ and $\varsigma$.
Then both of the following hold 
\begin{enumerate}
\item Let $z=h\exp(w)\in\cone_i^\varsigma$ where $h\in\overline{\coneH\setminus\partial_{10\mfsc}\coneH}$, then 
\be\label{eq: noI constant initial dim'}
\begin{aligned}
\ref{C: M and dimension}^{-1}\!\!\sup_{w'\in F_i^\varsigma}\noI_{\cone_i^\varsigma,\mfsc}(e,\exp(w')y)&\leq \noI_{\cone_i^\varsigma,\mfsc}(e,z)\\
&\leq\ref{C: M and dimension}\!\!\sup_{w'\in F_i^\varsigma}\noI_{\cone_i^\varsigma,\mfsc}(e,\exp(w')y).
\end{aligned}
\ee
\item For all $z\in\cone_i^\varsigma$, we have 
\be\label{eq: energy bound cone-i-j'}
\mfm_{\cone_{i}^\varsigma, \mfsc, 0}(e,z)\leq e^{Dt}.
\ee
\end{enumerate}
Note that~\eqref{eq: noI constant initial dim'} is a consequence of Lemma~\ref{lem: regular sets initial dim}, 
and~\eqref{eq: energy bound cone-i-j'} follows from \eqref{eq: initial dim'} since $\cone_i^\varsigma\subset\cone_i$. We also note that the second inequality in~\eqref{eq: noI constant initial dim'} holds true for all $z\in\cone_i^\varsigma$, see Lemma~\ref{lem: regular sets initial dim}.

With this notation,~\eqref{eq: initial dim convex comb} may be rewritten as follows: for all $\tau>0$ and $|s|\leq 2$, we have   
\begin{multline}\label{eq: initial dim convex comb'}
\biggl|\int\varphi(a_\tau u_shx_2)\diff\mu_{t,\ell,0}(h)-\sum_i \sum_\varsigma c_{i,\varsigma} \int\varphi(a_\tau u_sz)\diff\!\mu_{\cone_i^\varsigma}(z)\biggr|\\
\ll \beta\Lip(\varphi),
\end{multline}
here $c_{i,\varsigma}=c_i\mu_{\cone_i}(\cone_i^\varsigma)$; $\mu_{\cone_i^\varsigma}$ denotes $\mu_{\cone_i}|_{\cone_i^\varsigma}$
normalized to be a probability measure; for any integer $n\geq 0$, we put 
$\mu_{t,\ell, n}=\rwm_{\ell}\conv\cdots\conv\nu_{\ell}\conv\sigma\conv\rwm_{t}$
where $\rwm_\ell$ appears $n$-times; and the implied constant depends only on $X$.

For notational convenience, let us write 
\be\label{eq: def mathcal Z}
\{(\cone_i^\varsigma,\mu_{\cone_i^\varsigma}): i, \varsigma\}=\{(\cone_\zeta,\mu_{\cone_\zeta}): \zeta\in \mathcal Z\},
\ee 
for an index set $\mathcal Z$.

\subsection{Random walk trajectories: one step}\label{sec: admissible}
Beginning with $\cone_{\zeta_0}$ for some $\zeta_0\in\mathcal Z$
as above, we will use Lemma~\ref{lem: convex comb 2} to construct sets $\cone$.  
Then Lemma~\ref{lem: dimension improves for new cone} implies that the estimate on the corresponding 
Margulis function exponentially improves after each step. 

Let us begin by fixing some notation. Let $\zeta_0\in\mathcal Z$ be as above. Put
\[
\mathsf A_0^{\zeta_0}=\{\zeta_0\},
\]   
and recall $\Bigl(\cone_{\zeta_0},\mu_{\cone_{\zeta_0}}\Bigr)$ from above. 
Using an inductive construction,  we will define 
$\mathsf A_n^{\zeta_0}$ and $(\cone_{\Xi},\mu_{\cone_{\Xi}})$ 
for all $n\geq 1$ and all $\Xi\in\mathsf A_n^{\zeta_0}$.

Let us begin with the definition in the case $n=1$. 
Put 
\[
(\cone_{\rm old}, \mu_{\cone_{\rm old}})=(\cone_{\zeta_0},\mu_{\cone_{\zeta_0}}).
\] 
In view of~\eqref{eq: energy bound cone-i-j'} and~\eqref{eq: regular tree'' ???},  
$(\cone_{\rm old}, \mu_{\cone_{\rm old}})$ satisfies the conditions in Lemma \ref{lem: truncated MF main estimate} 
with $\egbd=e^{Dt}$, $\trct=0$, and $c$ depending only on $\mathsf M$. Recall also that $0<\kappa\leq \vare/10^6$. 
By Lemma \ref{lem: truncated MF main estimate}, thus,
there exists $\gdh_{\mu_{\cone_{\rm old}}}\subset [0,1]$ with 
\[
\Bigl|[0,1]\setminus \gdh_{\mu_{\cone_{\rm old}}}\Bigr|\ll \nuni^{-\kappa^2t/4},
\]
and for every $r\in \gdh_{\mu_{\cone_{\rm old}}}$, there exists a subset 
\[
\cone_{{\rm old},r}\subset\hat\cone_{\rm old}=\bigcup\hat\coneH.\{\exp(w)y_0: w\in F_{\rm old}\}, \qquad (\hat\coneH=\overline{\coneH\setminus \partial_{10\mfsc}\coneH})
\]
satisfying $\mu_{\cone_{\rm old}}(\cone_{\rm old}\setminus\cone_{{\rm old},r})\ll\nuni^{-\kappa^2t/64}$ 
and the following: for all $z\in\cone_{{\rm old},r}$,   
\be\label{eq: mfm estimate one step use}
\mfm_{\hat\cone_{\rm old},\scmf,\trct_1}(a_{\ell} u_r,z)\leq 
200L_1 \nuni^{-\alpha\ell}\egbd^{1+8\kappa}+200\nuni^{2\alpha\ell}\noI_{\hat\cone_{\rm old},\scmf}(a_{\ell} u_r,z);
\ee 
where $L_1=L\kappa^{-L}$ and $\trct_1=1+L_1\mfbd^{\kappa}$. 
We assumed $\egbd$ is large (depending on $\kappa$) and the fact that $\trct=1$ in the above bound, see also Theorem~\ref{thm: proj thm}.  

\medskip

Recall that $d_1=100\lceil\frac{4D-3}{2\vare}\rceil$, and fix a maximal $e^{-6d_1\ell}$-separated subset 
\[
\mathcal L_{\cone_{\rm old}}\subset L_{\mu_{\cone_{\rm old}}}.
\] 
For every $r_0\in\mathcal L_{\cone_{\rm old}}$, let 
\[
\{(\cone_{\zeta}, \mu_{\cone_{\zeta}}): \zeta\in \mathcal Z_{\zeta_0,r_0}''\}
\]
be the set of offsprings of $a_\ell u_{r_0}\cone_{\rm old}$, see~\eqref{eq: def conej 2} and~\eqref{eq: def mu conj 2}. 
In particular, $\cone_{\zeta}=\coneH.\{\exp(w)y_\zeta: w\in F_{\zeta}\}$ where 
\[
F_{\zeta}\subset \Bigl\{w\in B_\rfrak(0,\beta): \umt^H_\ell.\exp(w)y_\zeta\subset a_\ell u_{r_0}\mu_{\cone_{\rm old}}\Bigr\}, 
\]
and $y_\zeta\in X_{3\eta/2}$. Moreover,~\eqref{eq: num in F-j-m-i-r almost const'} implies that for every $\zeta\in \mathcal Z_{\zeta_0,r_0}''$, 
\be\label{eq: num F i varsigma 1}
\beta^{9}\cdot(\#F_{\rm old})\leq\#F_{\zeta}\leq \beta^8\cdot(\#F_{\rm old}). 
\ee

Let us put $\hat{\hat\coneH}=\overline{\coneH\setminus \partial_{100\beta^2}\coneH}$, and define 
\[
\hat{\hat\cone}_{\rm old}=\hat{\hat\coneH}.\{\exp(v)y_0: v\in F_{\rm old}\}.
\]
Then, we have  
\be\label{eq: measure of hat hat cone 1}
\mu_{\cone_{\rm old}}\Bigl(\cone_{\rm old}\setminus (\cone_{{\rm old},r_0}\cap \hat{\hat\cone}_{\rm old})\Bigr)\ll \beta+e^{-\kappa^2t/64}.
\ee

Let $F_{\zeta,r_0}=\Bigl\{w\in F_{\zeta}: \umt^H_\ell.\exp(w)y_\zeta\cap a_\ell u_{r_0}\Bigl(\cone_{{\rm old},r_0}\cap \hat{\hat\cone}_{\rm old}\Bigr)=\emptyset\Bigr\}$. If $\#F_{\zeta, r_0}\leq 10^{-6}\cdot(\#F_\zeta)$, replace $\cone_{\zeta}$ with 
\[
\coneH.\{\exp(w)y_\zeta: w\in F_\zeta\setminus F_{\zeta, r_0}\}
\]
otherwise, discard the set $\cone_{\zeta}$ entirely. Such replacements will increase the set $a_\ell u_{r_0}\cone_{\rm old}\backslash\bigcup_\zeta\cone_{\zeta}$. But thanks to \eqref{eq: measure of hat hat cone 1}, this doesn't affect the properties that we will need later, or more precisely the inequality \eqref{eq: convex comb 2 use step 1 1} in Lemma \ref{lem: bd mfm cone zeta and conv comp} below. 

Let $\mathcal Z'_{\zeta_0, r_0}\subset \mathcal Z''_{\zeta_0, r_0}$ be the set of indices which survive the above process. Abusing the notation, for every $\zeta\in\mathcal Z'_{\zeta_0, r_0}$, we denote $F_\zeta\setminus F_{\zeta, r_0}$ by $F_\zeta$ and denote 
$\coneH.\{\exp(w)y_\zeta: w\in F_\zeta\setminus F_{\zeta, r_0}\}$
by $\cone_\zeta$. 

Thus, we obtain a collection 
$\{(\cone_{\zeta},\mu_{\cone_{\zeta}}): \zeta\in \mathcal Z'_{\zeta_0,r_0}\}$ 
satisfying the following: If $\zeta\in \mathcal Z_{\zeta_0,r_0}'$ and $w\in F_{\zeta}$, then 
\[
\umt^H_\ell.\exp(w)y_\zeta\cap a_\ell u_{r_0}\Bigl(\cone_{{\rm old},r_0}\cap \hat{\hat\cone}_{\rm old}\Bigr)\neq\emptyset;
\]
moreover, the following analogue of~\eqref{eq: num F i varsigma 1} holds 
\be\label{eq: control over card of fiber'}
0.5\beta^9\cdot(\#F_{\rm old})\leq \#F_{\zeta}\leq 2\beta^8\cdot(\#F_{\rm old}).
\ee

With this notation, define 
\be\label{eq: def B1 i varsigma}
\mathsf B_1^{\zeta_0}=\Bigl\{(\zeta_0, r_0, \zeta): r_0\in \mathcal L_{\cone_{\zeta_0}}, \zeta\in \mathcal Z_{\zeta_0, r_0}'\Bigr\},
\ee
and for every $\Xi=(\zeta_0, r_0, \zeta)\in\mathsf B_1^{\zeta_0}$, put 
\[
\cone_\Xi=\coneH.\{\exp(w)y_\Xi: w\in F_\Xi\},
\]
where $y_\Xi=y_\zeta$ and $F_\Xi=F_\zeta$.

\medskip

\begin{lemma}\label{lem: step one of ind initial leaves}
Let $\Xi=(\zeta_0, r_0, \zeta)\in\mathsf B_1^{\zeta_0}$, and write $F=F_\Xi$, $y=y_\Xi$, and $\cone=\cone_\Xi$. Let $w_0\in F_{\zeta_0}$ be so that  
\[
\noI_{\cone_{\zeta_0},\mfsc}(e,\exp(w_0)y_0)=\sup_{w'}\noI_{\cone_{\zeta_0},\mfsc}(e,\exp(w')y_0).
\]
Then one of the following properties holds: 
\begin{enumerate}
\item If $e^{Dt}\geq \nuni^{\vare t/2}\noI_{\cone_{\zeta_0},\mfsc}(e,\exp(w_0)y_0)$, then 
\be\label{eq: moderate density 1'}
\mfm_{\cone,\scmf,\trct_1}(e,z)\leq \nuni^{-0.6\ell}e^{Dt}+10\noI_{\cone,\scmf}(e,z)\quad\text{for all $z\in\cone$},
\ee
where $\trct_1=1+L\kappa^{-L}e^{\kappa Dt}$.
\item If $e^{Dt}<\nuni^{\vare t/2}\noI_{\cone_{\zeta_0},\mfsc}(e,\exp(w_0)y_0)$, then both of the following hold
\begin{enumerate}
\item Let $z=h\exp(w)y_0\in\cone_{\zeta_0}$ where $h\in\overline{\coneH\setminus\partial_{10\mfsc}\coneH}$, then 
\be\label{eq: high density 1'}
\mfm_{\cone_{\zeta_0},\scmf,\trct}(e,z)\leq \nuni^{\vare t/2}\noI_{\cone_{\zeta_0},\mfsc}(e,\exp(w_0)y_0)\leq\ref{C: M and dimension}\nuni^{\vare t/2}\noI_{\cone_{\zeta_0},\mfsc}(e,z),
\ee
(indeed the first inequality above holds for {\em every} $z\in\cone_{\zeta_0}$).
\item For all $z\in\cone$, we have 
\be\label{eq: high density 2 1'}
\mfm_{\cone,\scmf,\trct_1}(e,z)\leq \nuni^{-0.6\ell}\Bigl(\nuni^{{\vare t}/{2}}\!\cdot\noI_{\cone_{\zeta_0},\mfsc}(e,\exp(w_0)y_0)\Bigr)\!+10\noI_{\cone,\scmf}(e,z).
\ee
\end{enumerate}
\end{enumerate} 
Indeed case~(2) does not hold and we are always in case~(1).
\end{lemma}

\begin{proof}
Note that $e^{\kappa Dt}\leq e^{\ell t/100}$. 
Moreover, in view of~\eqref{eq: regular tree'' ???} and the fact that for every $w\in F_{\zeta_0}$, we have 
\[
\umt^H_\ell.\exp(w)y\cap a_\ell u_{r_0}\Bigl(\cone_{{\rm old},r}\cap \hat{\hat\cone}_{\rm old}\Bigr)\neq\emptyset,
\]
Lemma~\ref{lem: dimension improve 2} is applicable with $\cone_{\zeta_0}$ and $\cone$. Applying loc.\ cit.\ with $\cone_{\zeta_0}$ and $\cone$ thus implies all but the final claim in this lemma.

To see the final claim, note that by~\eqref{eq: num of Fi initial dim}, we have 
\[
e^{\vare t/2}\noI_{\cone_{\zeta_0},\mfsc}(e,\exp(w)y_0)\leq e^{\vare t/2}\cdot (2\eta\mfsc)^{-\alpha}\cdot (\beta^{-3}\nuni^{t})\leq \nuni^{2t}.
\] 
Moreover, $D\geq 10$, see Proposition~\ref{prop:closing lemma intro}, hence, case~(2) cannot hold. 
\end{proof}

Let $\mfbd_0=e^{Dt}$. For every $\Xi=(\zeta_0, r_0, \zeta)\in\mathsf B_1^{\zeta_0}$, define $\mfbd_{\Xi,1}$ as follows: if 
\[
\nuni^{-0.6\ell}e^{Dt}\geq 10\sup_{z\in \cone_\Xi}\noI_{\cone_\Xi,\scmf}(e,z),
\] 
then we put 
\be\label{eq: define mfbd0}
\mfbd_{\Xi,1}=e^{-\ell/2} e^{Dt}.
\ee 
Otherwise, i.e., if $
\nuni^{-0.6\ell}e^{Dt}<10\sup_{z\in \cone_{\Xi}}\noI_{\cone_{\Xi},\scmf}(e,z)$,
then we put
\be\label{eq: define mfbd0 1}
\mfbd_{\Xi,1}=20\sup_{z\in \cone_\Xi}\noI_{\cone_\Xi,\scmf}(e,z).
\ee

\begin{lemma}\label{lem: bd mfm cone zeta and conv comp} 
The following three statements hold: 
\begin{enumerate}
\item For every $\Xi=(\zeta_0, r_0, \zeta)\in\mathsf B_1^{\zeta_0}$, we have $\mfbd_{\Xi,1}\leq e^{Dt}$.
\item Let $\Xi=(\zeta_0, r_0, \zeta)\in\mathsf B_1^{\zeta_0}$, then 
\be\label{eq: mfbd 1}
\mfm_{\cone_\Xi,\scmf,\trct_{1}}(e,z)\leq \mfbd_{\Xi,1},
\ee
where $\trct_1=1+L\kappa^{-L}e^{\kappa Dt}$.
\item Let $r_0\in \mathcal L_{\cone_{\zeta_0}}$. Then
\be\label{eq: convex comb 2 use step 1 1}
\begin{aligned}
\biggl|\int\varphi(a_\tau u_s.z)\diff(a_\ell u_{r_0}\mu_{\cone_{\zeta_0}})(z)-&\sum_{\mathsf B_1^{\zeta_0}} c_{\Xi'}\int\varphi(a_\tau u_sz)\diff\!\mu_{\cone_{\Xi'}}(z)\biggr|\\
&\ll\max\Bigl\{\eta^{1/2}, e^{-\kappa^2t/64}\Bigr\}\Lip(\varphi),
\end{aligned}
\ee
for every $\varphi\in C_c^\infty(X)$, every $0<\tau\leq 2d_1\ell$, and all $|s|\leq 2$.
\end{enumerate}
\end{lemma}

\begin{proof}
The claim in part~(1) is clear if $\mfbd_{\Xi,1}=e^{-\ell/2}e^{Dt}$. 
Assume thus that 
\[
\mfbd_{\Xi,1}=20\sup_{z\in \cone_\Xi}\noI_{\cone_\Xi,\scmf}(e,z).
\]
Then by the definition of $\noI$,~\eqref{eq: num of Fi initial dim} and~\eqref{eq: control over card of fiber'}, we have 
\[
\mfbd_{\Xi,1}\ll \mfsc^{-\alpha}\eta^{-\alpha}\cdot (\#F_\Xi)\leq e^{2t},
\]
where we also used $\mfsc=e^{-\sqrt\vare t}$ and $\eta\geq e^{-0.01\vare t}$. The claim follows as $D\geq 10$. 

Part~(2) follows from the definition of $\mfbd_{\Xi,1}$ and Lemma~\ref{lem: step one of ind initial leaves}.

To see part~(3), apply Lemma~\ref{lem: convex comb 2}, with $\mathsf d_0=3d_1\ell$ (note that $\tau+\ell\leq \mathsf d_0$) 
and $r_0$. By that lemma thus  
\[
\begin{aligned}
\biggl|\int\varphi(a_du_s.z)\diff(a_\ell u_{r_0}\mu_{\cone_{\zeta_0}})(z)-&\sum c_{\zeta}\int\varphi(a_d u_sz)\diff\!\mu_{\cone_{\zeta}}(z)\biggr|\\
&\ll\max\Bigl\{\eta^{1/2}, e^{-\kappa^2t/64}\Bigr\}\Lip(\varphi),
\end{aligned}
\]
where the sum is over $\zeta\in\mathcal Z''_{\zeta_0,r_0}$. 

We can replace the summation over $\mathcal Z''_{\zeta_0,r_0}$ by summation over $\mathcal Z'_{\zeta_0,r_0}$ (hence over $\mathsf B_1^{\zeta_0}$) in view of~\eqref{eq: measure of hat hat cone 1} and the definition of $\mathcal Z'_{\zeta_0,r_0}$. 
\end{proof}

\subsection{Regularizing $F_\Xi$}\label{sec: regular F zeta step 1}
In preparation for the next step of the inductive construction, we will refine the set $\mathsf B_1^{\zeta_0}$ 
by decomposing $F_\Xi$ (for $\Xi\in\mathsf B_1^{\zeta_0}$) into sets satisfying estimates similar to those in~\eqref{eq: regular tree}.

To that end, let $\Xi=(\zeta_0, r_0,\zeta_1)\in\mathsf B_1^{\zeta_0}$, and let 
$F=F_\Xi$, $y=y_\Xi$, and $\cone=\cone_\Xi$.
In view of Lemma~\ref{lem: bd mfm cone zeta and conv comp}(2) and Lemma~\ref{lem: F energy est assuming mfm energy estimate}, 
\[
\eng_{F_w, \trct_1}(w')\leq 10^6\mfbd_{\Xi,1}\qquad\text{for every $w'\in F_w$},
\]  
where $F_w=F\cap B_\rfrak(w,4\mfsc\inj(y))$.  

Let $k_1> k_{0}$ be positive integers defined as follows: 
\be\label{eq: def k0 and k1 step 1}
2^{k_{0}}< (\mfsc\,\inj(y))^{-1}\leq 2^{k_{0}+1}\quad\text{and}\quad 2^{k_1}< 10^6\mfbd_{\Xi,1}\leq 2^{k_1+1}.
\ee
Let $\mathsf M$ be as above, see~\eqref{eq: condition on M 1}. 
Applying Lemma~\ref{lem: regular tree decomposition}, we can write 
\be\label{eq: Fi as union ini dim ??}
F=F'\bigcup \Bigl(\textstyle \bigcup_l F_l\Bigr)
\ee
where $\#F'\leq \beta^{1/4}\cdot (\#F)$ and $\#F_l\geq \beta^2\cdot(\#F)$.
In view of~\eqref{eq: control over card of fiber'}, we have  
\be\label{eq: num F ir jm varsigma 0}
\begin{aligned}
0.5\beta^{11}\cdot(\#F_{\zeta_0})\leq\beta^2\cdot(\#F)&\leq\#F_l\\
&\leq \#F \leq 2\beta^8\cdot(\#F_{\zeta_0}), 
\end{aligned}
\ee
and for every $k_0-10\leq k\leq k_1$, 
there exists some $\tau_k=\tau_{k}^l$ so that 
\be\label{eq: used to be regular tree''}
\text{either}\quad2^{\mathsf M(\tau_{k}-2)}\leq \#F_l\cap Q\leq 2^{\mathsf M\tau_{k}}\quad\text{or}\quad F_l\cap Q=\emptyset,
\ee
for all $Q\in\mathcal Q_{\mathsf Mk}$.

Let us also note that combining~\eqref{eq: num F ir jm varsigma 0} and~\eqref{eq: num F i varsigma}, we conclude
\be\label{eq: num F ir jm varsigma}
\frac12\beta^{22}e^t\leq \#F_l\leq 2\beta^{5}e^t.
\ee 

Let $\mathcal Z_{\zeta_0,r_0}$ be an enumeration of 
$\{(\zeta',l):\zeta'\in \mathcal Z_{0, r_0}', l\in \mathcal K_{(\zeta_0,r_0,\zeta')}\}$ where 
for every $\Xi=(\zeta_0,r_0,\zeta')\in\mathsf B_1^{\zeta_0}$, we let 
\[
\text{$\mathcal K_\Xi=\{l: F_l$ as in $\eqref{eq: Fi as union ini dim ??}\}$}.
\]
If $\zeta\in \mathcal Z_{\zeta_0,r_0}$ corresponds to $(\zeta',l)$, put $y_\zeta=y_{\zeta'}$ 
and $F_\zeta=(F_{\zeta'})_l$, see~\eqref{eq: Fi as union ini dim ??}. 

Define 
\be\label{eq: def A1 i varsigma}
\mathsf A_1^{\zeta_0}=\Bigl\{(\zeta_0, r_0, \zeta_1): r_0\in \mathcal L_{\cone_{\zeta_0}}, \zeta_1\in \mathcal Z_{\zeta_0, r_0}\Bigr\},
\ee
and for every $\Xi=(\zeta_0, r_0, \zeta_1)\in\mathsf A_1^{\zeta_0}$, put 
\[
\cone_\Xi=\coneH.\{\exp(w)y_\Xi: w\in F_{\Xi}\},
\]
where $y_\Xi=y_{\zeta_1}$ and $F_\Xi=F_{\zeta_1}$.

\begin{lemma}\label{lem: bd mfm cone zeta and conv comp A1} 
Let $\Xi=(\zeta_0, r_0, \zeta_1)\in\mathsf A_1^{\zeta_0}$, and suppose $\zeta_1$ correspond to $(\zeta',l)$ as above. Put $\egbd_{\Xi,1}=\egbd_{\Xi',1}$ where $\Xi'=(\zeta_0, r_0, \zeta')\in\mathsf B_1^{\zeta_0}$. Then both of the following hold: 
\begin{enumerate}
\item We have 
\be\label{eq: mfbd 1 2}
\mfm_{\cone_\Xi,\scmf,\trct_{1}}(e,z)\leq \mfbd_{\Xi,1},
\ee
where $\trct_1=1+L\kappa^{-L}e^{\kappa Dt}$.
\item Let $r_0\in \mathcal L_{\cone_{\zeta_0}}$. Then
\be\label{eq: convex comb 2 use step 1 2}
\begin{aligned}
\biggl|\int\varphi(a_\tau u_s.z)\diff(a_\ell u_{r_0}\mu_{\cone_{\zeta_0}})(z)-&\sum_{\mathsf A_1^{\zeta_0}} c_{\Xi}\int\varphi(a_\tau u_sz)\diff\!\mu_{\cone_\Xi}(z)\biggr|\\
&\ll\max\Bigl\{\eta^{1/2}, e^{-\kappa^2t/64}\Bigr\}\Lip(\varphi),
\end{aligned}
\ee
for every $\varphi\in C_c^\infty(X)$, every $0<\tau\leq 2d_1\ell$, and all $|s|\leq 2$,
\end{enumerate}
\end{lemma}

\begin{proof}
Part~(1) follows from Lemma~\ref{lem: bd mfm cone zeta and conv comp}(2) and the fact that $\cone_\Xi\subset\cone_{\Xi'}$. 
Part~(2) follows from Lemma~\ref{lem: bd mfm cone zeta and conv comp}(3) in view of~\eqref{eq: Fi as union ini dim ??} if we put 
\[
c_{\Xi}=c_{\Xi'}\mu_{\cone_{\Xi'}}(\cone_{\Xi})
\]
and use the fact that $\mu_{\cone_{\Xi'}}$ is admissible, see Lemma~\ref{lem: mu cone is admissible 2}. 
\end{proof}

\subsection{Random walk trajectories: $n$-steps}\label{sec: A zeta0 n}
We now assume that $\mathsf A_{n}^{\zeta_0}$ is defined for some $n\geq 1$, 
and will define $\mathsf A_{n+1}^{\zeta_0}$. The construction is similar to the case $n=0$ completed in previous sections. Indeed, as it was done in that case, we will define $\mathsf A_{n+1}^{\zeta_0}$ using the collection of $2n+3$ tuples  
\[
(\zeta_0,r_0,\ldots, \zeta_n, r_n, \zeta_{n+1})
\]
satisfying the following properties  
\begin{itemize}
\item $\hat\Xi:=(\zeta_0,r_0,\ldots, \zeta_n)\in \mathsf A_{n}^{\zeta_0}$, 
\item $r_n\in\mathcal L_{\cone_{\hat\Xi}}$, and 
\item $\zeta_{n+1}\in \mathcal Z_{n,r_n}'$,
\end{itemize}
where $\mathcal L_{\cone_{\hat\Xi}}\subset \gdh_{\mu_{\cone_{\hat\Xi}}}$ is a maximal $e^{-6d_1}$-separated subset, 
see Lemma~\ref{lem: truncated MF main estimate} for $\gdh_{\cone_{\hat\Xi}}$, and 
\[
\mathcal Z_{\zeta_n,r_n}'\subset\mathcal Z''_{\zeta_n,r_n}
\]
where $\mathcal Z''_{\zeta_n,r_n}$ is the index set enumerating the offsprings of $a_\ell u_{r_n}\cone_{\hat\Xi}$, see~\eqref{eq: def conej 2} and~\eqref{eq: def mu conj 2} for offsprings. 

{\em We now turn to the details:} Recall that $0<\kappa\leq \vare/10^6$, for all $m\in\bbn$ put
\be\label{eq: def trct n+1}
\trct_{m}=1+mL\kappa^{-L}e^{\kappa Dt},
\ee 
see Lemma~\ref{lem: bd mfm cone zeta and conv comp A1} for $\trct_1$.

Let $\hat\Xi=(\zeta_0,r_0,\ldots, \zeta_n)\in \mathsf A_{n}^{\zeta_0}$, and put
\[
(\cone_{\rm old}, \mu_{\cone_{\rm old}})=(\cone_{\hat\Xi},\mu_{\cone_{\hat\Xi}});
\]
note that $\cone_{\hat\Xi}=\coneH.\{\exp(w)y_{\hat\Xi}: w\in F_{\hat\Xi}\}$, where 
\be\label{eq: ind hyp on num F hat Xi}
\frac{1}{2^n}\beta^{11(n+1)}e^t\leq \#F_{\hat\Xi}\leq 2^n\beta^{8n-3}e^t,
\ee
see~\eqref{eq: num of Fi initial dim} and~\eqref{eq: num F ir jm varsigma}.

Then, by inductive hypothesis, we have  
\be\label{eq: Xi' initial cond}
\mfm_{\cone_{\hat\Xi},\mfsc, \trct_n}(e,z)\leq \mfbd_{\hat\Xi,n}\quad\text{for all $z\in\cone_{\hat\Xi}$,}
\ee
where $\mfbd_{\hat\Xi,n}$ is defined inductively. Recall that $\mfbd_0=\nuni^{Dt}$ also see~\eqref{eq: define mfbd0} 
and~\eqref{eq: define mfbd0 1} for the definition of $\mfbd_{\hat\Xi, 1}$.  In particular, we have 
\be\label{eq: egbd m < eDt}
\mfbd_{\hat\Xi,n}\leq e^{Dt},
\ee
see Lemma~\ref{lem: bd mfm cone zeta and conv comp}(1).

Recall that $d_1=100\lceil\frac{4D-3}{2\vare}\rceil$. Fix a maximal $e^{-6d_1\ell}$-separated subset 
\[
\mathcal L_{\cone_{\rm old}}\subset L_{\mu_{\cone_{\rm old}}}.
\] 
For every $r_n\in \mathcal L_{\cone_{\rm old}}$, let 
\[
\{(\cone_{\zeta}, \mu_{\cone_{\zeta}}): \zeta\in \mathcal Z_{\hat\Xi,r_n}''\}
\]
be the set of all offsprings of $a_\ell u_{r_n}\cone_{\rm old}=a_\ell u_{r_n}\cone_{\hat\Xi}$, see~\eqref{eq: def conej 2} and~\eqref{eq: def mu conj 2}. 
In particular, $\cone_{\zeta}=\coneH.\{\exp(w)y_\zeta: w\in F_{\zeta}\}$ where 
\[
F_{\zeta}\subset \Bigl\{w\in B_\rfrak(0,\beta): \umt^H_\ell.\exp(w)y_\zeta\subset a_\ell u_{r_n}\mu_{\cone_{\rm old}}\Bigr\}
\]
for some $y_\zeta\in X_{3\eta/2}$.

Moreover,~\eqref{eq: num in F-j-m-i-r almost const'} implies that
for every $\zeta\in \mathcal Z_{\hat\Xi,r_n}''$, we have 
\be\label{eq: num F i varsigma n}
\beta^{9}\cdot(\#F_{\rm old})\leq\#F_{\zeta}\leq \beta^8\cdot(\#F_{\rm old}). 
\ee

Let us put $\hat{\hat\coneH}=\overline{\coneH\setminus \partial_{100\beta^2}\coneH}$, and define 
\[
\hat{\hat\cone}_{\rm old}=\hat{\hat\coneH}.\{\exp(v)y_{\rm old}: v\in F_{\rm old}\}.
\]
Then, we have  
\be\label{eq: measure of hat hat cone n}
\mu_{\cone_{\rm old}}\Bigl(\cone_{\rm old}\setminus (\cone_{{\rm old},r_n}\cap \hat{\hat\cone}_{\rm old})\Bigr)\ll \beta+e^{-\kappa^2t/64}.
\ee

Let $F_{\zeta,r_n}=\Bigl\{w\in F_{\zeta}: \umt^H_\ell.\exp(w)y_\zeta\cap a_\ell u_{r_n}\Bigl(\cone_{{\rm old},r_n}\cap \hat{\hat\cone}_{\rm old}\Bigr)=\emptyset\Bigr\}$. If $\#F_{\zeta, r_n}\leq 10^{-6}\cdot(\#F_\zeta)$, replace $\cone_{\zeta}$ with 
\[
\coneH.\{\exp(w)y_\zeta: w\in F_\zeta\setminus F_{\zeta, r_n}\}
\]
otherwise, discard the set $\cone_{\zeta}$ entirely. As in how \eqref{eq: measure of hat hat cone 1} was used, the inequality \eqref{eq: measure of hat hat cone n} assures that such replacements causes no damage later.

Let $\mathcal Z'_{\hat\Xi, r_n}\subset \mathcal Z''_{\hat\Xi, r_n}$ be the set of indices which survive the above process. Abusing the notation, for every $\zeta\in\mathcal Z'_{\hat\Xi, r_n}$, we denote $F_\zeta\setminus F_{\zeta, r_n}$ by $F_\zeta$ and denote 
$\coneH.\{\exp(w)y_\zeta: w\in F_\zeta\setminus F_{\zeta, r_n}\}$
by $\cone_\zeta$. 

Thus, we obtain a collection 
$\Bigl\{(\cone_{\zeta},\mu_{\cone_{\zeta}}): \zeta\in \mathcal Z'_{\hat\Xi,r_n}\Bigr\}$
satisfying the following: If $\zeta\in \mathcal Z_{\hat\Xi,r_n}'$ and $w\in F_{\zeta}$, then   
\[
\umt^H_\ell.\exp(w)y_\zeta\cap a_\ell u_{r_n}\Bigl(\cone_{{\rm old},r_n}\cap \hat{\hat\cone}_{\rm old}\Bigr)\neq\emptyset;
\]
moreover, the following analogue of~\eqref{eq: num F i varsigma n} holds 
\be\label{eq: control over card of fiber' n}
0.5\beta^9\cdot(\#F_{\rm old})\leq \#F_{\zeta}\leq 2\beta^8\cdot(\#F_{\rm old}).
\ee

With this notation, define 
\be\label{eq: def Bn i varsigma}
\mathsf B_{n+1}^{\zeta_0}=\left\{(\zeta_0, r_0,\ldots, \zeta_n, r_n, \zeta):\begin{array}{c} \hat\Xi=(\zeta_0, r_0,\ldots, \zeta_n)\in\mathsf A_{n}^{\zeta_0},\\ 
r_n\in \mathcal L_{\cone_{\hat\Xi}}, \zeta\in \mathcal Z_{\hat\Xi, r_n}'\end{array}\right\}.
\ee

For every $\Xi=(\zeta_0, \ldots, \zeta_n, r_n, \zeta)\in\mathsf B_{n+1}^{\zeta_0}$, put 
\[
\cone_\Xi=\coneH.\{\exp(w)y_\Xi: w\in F_\Xi\},
\]
where $y_\Xi=y_\zeta$ and $F_\Xi=F_\zeta$.

\medskip

\begin{lemma}\label{lem: step n of ind initial leaves}
Let $\Xi=(\zeta_0, \ldots, \zeta_n, r_n, \zeta)\in\mathsf B_{n+1}^{\zeta_0}$, and write 
\[
\text{$\hat\Xi=(\zeta_0, \ldots, \zeta_n)$, $\;F=F_\Xi$, $\;y=y_\Xi$, and $\;\cone=\cone_\Xi$}.
\]
Let $w_0\in F_{\hat\Xi}$ be so that 
\[
\noI_{\cone_{\hat\Xi},\mfsc}(e,\exp(w_0)y_{\hat\Xi})=\sup_{w'}\noI_{\cone_{\hat\Xi},\mfsc}(e,\exp(w')y_{\hat\Xi}).
\]
Then one of the following properties holds: 
\begin{enumerate}[label={(A-\arabic*)}]
\item\label{(A-1)}  If $\mfbd_{\hat\Xi,n}\geq \nuni^{\vare t/2}\noI_{\cone_{\hat\Xi},\mfsc}(e,\exp(w_0)y_{\hat\Xi})$, then 
\be\label{eq: moderate density n'}
\mfm_{\cone,\scmf,\trct_{n+1}}(e,z)\leq \nuni^{-0.6\ell}\mfbd_{\hat\Xi,n}+10\noI_{\cone,\scmf}(e,z)\quad\text{for all $z\in\cone$},
\ee
where $\trct_{n+1}=1+(n+1)L\kappa^{-L}e^{\kappa Dt}$, see~\eqref{eq: def trct n+1}. 
\medskip
\item\label{(A-2)} If $\mfbd_{\hat\Xi,n}<\nuni^{\vare t/2}\noI_{\cone_{\hat\Xi},\mfsc}(e,\exp(w_0)y_{\hat\Xi})$, 
then both of the following hold
\begin{enumerate}
\item Let $z=h\exp(w)y_{\hat\Xi}\in\cone_{\hat\Xi}$ where $h\in\overline{\coneH\setminus\partial_{10\mfsc}\coneH}$, then 
\be\label{eq: high density n'}
\mfm_{\cone_{\hat\Xi},\scmf,\trct_n}(e,z)\leq \nuni^{\vare t/2}\noI_{\cone_{\hat\Xi},\mfsc}(e,\exp(w_0)y_{\hat\Xi})\leq\ref{C: M and dimension}\nuni^{\vare t/2}\noI_{\cone_{\hat\Xi},\mfsc}(e,z),
\ee
(indeed the first inequality above holds for {\em every} $z\in\cone_{\hat\Xi}$).
\item For all $z\in\cone$, we have 
\be\label{eq: high density 2 n'}
\mfm_{\cone,\scmf,\trct_{n+1}}(e,z)\leq \nuni^{-0.6\ell}\Bigl(\nuni^{{\vare t}/{2}}\!\cdot\noI_{\cone_{\hat\Xi},\mfsc}(e,\exp(w_0)y_{\hat\Xi})\Bigr)\!+10\noI_{\cone,\scmf}(e,z).
\ee
\end{enumerate}
\end{enumerate} 
\end{lemma}

\begin{proof}
Recall that $\egbd_{\hat\Xi,n}\leq e^{Dt}$, see~\eqref{eq: egbd m < eDt}; we have $e^{\kappa Dt}\leq e^{\ell t/100}$. Moreover, note that for every $w\in F_{\hat\Xi}$, we have 
\[
\umt^H_\ell.\exp(w)y\cap a_\ell u_{r_n}\Bigl(\cone_{{\rm old},r_n}\cap \hat{\hat\cone}_{\rm old}\Bigr)\neq\emptyset.
\]
Moreover, using $\egbd_{\hat\Xi,n}\leq e^{Dt}$ again, we have 
\[
\trct_n+L\kappa^{-L}\egbd_{\hat\Xi,n}^{\kappa}\leq \trct_{n+1}.
\]
The claims in the lemma thus follow from Lemma~\ref{lem: dimension improve 2} applied with $\cone_{\hat\Xi}$, $\cone$ and $\trct=\trct_n$.
\end{proof}

Let $\Xi=(\zeta_0,\ldots, \zeta_n, r_n,\zeta)\in\mathsf B_{n+1}^{\zeta_0}$ and put $\hat\Xi=(\zeta_0,\ldots, \zeta_n)$. We define $\mfbd_{\Xi, n+1}$ as follows: If case~\ref{(A-1)} holds and 
\[
\nuni^{-0.6\ell}\mfbd_{\hat\Xi,n}\geq 10\sup_{z\in \cone_\Xi}\noI_{\cone_\Xi,\scmf}(e,z),
\] 
then we put 
\be\label{eq: define mfbd n+1 1}
\mfbd_{\Xi,n+1}=e^{-\ell/2} \mfbd_{\hat\Xi,n}.
\ee 
If case~\ref{(A-1)} holds and $\nuni^{-0.6\ell}\mfbd_{\hat\Xi,n}< 10\sup_{z\in \cone_\Xi}\noI_{\cone_\Xi,\scmf}(e,z)$,
then we put
\be\label{eq: define mfbd n+1 1 1}
\mfbd_{\Xi, n+1}=20\sup_{z\in \cone_\Xi}\noI_{\cone_\Xi,\scmf}(e,z).
\ee

If case~\ref{(A-2)} holds and 
\[
\nuni^{-0.6\ell}\Bigl(\nuni^{{\vare t}/{2}}\!\cdot\noI_{\cone_{\hat\Xi},\mfsc}(e,\exp(w_0)y_{\hat\Xi})\Bigr)\geq 10\sup_{z\in \cone_\Xi}\noI_{\cone_\Xi,\scmf}(e,z),
\] 
we put 
\be\label{eq: define mfbd n+1 2}
\mfbd_{\Xi,n+1}=\nuni^{-\ell/2}\Bigl(\nuni^{{\vare t}/{2}}\!\cdot\noI_{\cone_{\hat\Xi},\mfsc}(e,\exp(w_0)y_{\hat\Xi})\Bigr).
\ee 
If case~\ref{(A-2)} holds and
\[
\nuni^{-0.6\ell}\Bigl(\nuni^{{\vare t}/{2}}\!\cdot\noI_{\cone_{\hat\Xi},\mfsc}(e,\exp(w_0)y_{\hat\Xi})\Bigr)< 10\sup_{z\in \cone_\Xi}\noI_{\cone_\Xi,\scmf}(e,z),
\] 
then we put
\be\label{eq: define mfbd n+1 2 1}
\mfbd_{\Xi,n+1}=20\sup_{z\in \cone_\Xi}\noI_{\cone_\Xi,\scmf}(e,z).
\ee

\begin{lemma}\label{lem: bd mfm cone zeta and conv comp n} 
The following three statements hold: 
\begin{enumerate}
\item For every $\Xi=(\zeta_0, \ldots,\zeta_n, r_n, \zeta)\in\mathsf B_{n+1}^{\zeta_0}$, we have $\mfbd_{\Xi,n+1}\leq e^{Dt}$.
\item Let $\Xi=(\zeta_0, \ldots,\zeta_n, r_n, \zeta)\in\mathsf B_{n+1}^{\zeta_0}$, then 
\be\label{eq: mfbd n+1}
\mfm_{\cone_\Xi,\scmf,\trct_{n+1}}(e,z)\leq \mfbd_{\Xi,n+1},
\ee
where $\trct_{n+1}=1+(n+1)L\kappa^{-L}e^{\kappa Dt}$.
\item Let $\hat\Xi\in\mathsf A_n^{\zeta_0}$ and let 
$r_{n}\in \mathcal L_{\cone_{\hat\Xi}}$. Then for every $\varphi\in C_c^\infty(X)$, every $0<\tau\leq 2d_1\ell$, and all $|s|\leq 2$, we have 
\be\label{eq: convex comb 2 use step n+1 1}
\begin{aligned}
\biggl|\int\varphi(a_\tau u_s.z)\diff(a_\ell u_{r_n}\mu_{\cone_{\hat\Xi}})(z)-&\sum c_{\Xi}\int\varphi(a_\tau u_sz)\diff\!\mu_{\cone_{\Xi}}(z)\biggr|\\
&\ll\max\Bigl\{\eta^{1/2}, e^{-\kappa^2t/64}\Bigr\}\Lip(\varphi),
\end{aligned}
\ee
where the sum is over $\mathcal Z_{\hat\Xi, r_n}'$, and for every $\zeta\in\mathcal Z_{\hat\Xi, r_n}'$, we let 
\[
\Xi=(\zeta_0, r_0,\ldots, \zeta_n, r_n,\zeta).
\]
\end{enumerate}
\end{lemma}

\begin{proof}
Let $\Xi=(\zeta_0, \ldots,\zeta_n, r_n, \zeta)$ and put $\hat\Xi=(\zeta_0, \ldots,\zeta_n)$. The claim in part~(1) follows from~\eqref{eq: egbd m < eDt} if  $\mfbd_{\Xi,n+1}=e^{-\ell/2}\egbd_{\hat\Xi,n}$. 

We now consider the other two possibilities. First suppose that  
\[
\mfbd_{\Xi,n+1}=20\sup_{z\in \cone_\Xi}\noI_{\cone_\Xi,\scmf}(e,z).
\]
Then by the definition of $\noI$,~\eqref{eq: ind hyp on num F hat Xi} and~\eqref{eq: control over card of fiber' n}, we have 
\[
\mfbd_{\Xi,n+1}\ll \mfsc^{-\alpha}\eta^{-\alpha}\cdot (\#F_\Xi)\leq e^{2t},
\]
where we also used $\mfsc=e^{-\sqrt\vare t}$ and $\eta\geq e^{0.01\vare t}$. The claim in this case also follows as $D\geq 10$. 

Finally, let us assume
\[
\mfbd_{\Xi,n+1}=\nuni^{-\ell/2}\Bigl(\nuni^{{\vare t}/{2}}\!\cdot\sup_{w'}\noI_{\cone_{\hat\Xi},\mfsc}(e,\exp(w')y_{\hat\Xi})\Bigr).
\]
Then again using the definition of $\noI$, and~\eqref{eq: ind hyp on num F hat Xi}, we have 
\[
\mfbd_{\Xi,n+1}\ll e^{\vare t/2}\mfsc^{-\alpha}\eta^{-\alpha}\cdot (\#F_{\hat\Xi})\leq e^{2t},
\]
which completes the proof of part~(1).

Part~(2) follows from the definition of $\mfbd_{\Xi,n+1}$ and Lemma~\ref{lem: step n of ind initial leaves}.

To see part~(3), apply Lemma~\ref{lem: convex comb 2}, with $\mathsf d_0=3d_1\ell$ (note that $\tau+\ell\leq \mathsf d_0$) 
and $r_n$. 
The claim then follows from Lemma~\ref{lem: convex comb 2} and~\eqref{eq: measure of hat hat cone n}. 
\end{proof}

\subsection{Regularizing $F_\Xi$}\label{sec: regular F zeta step n}
Similar to what was done in \S\ref{sec: regular F zeta step 1}, we will define the set $\mathsf A_{n+1}^{\zeta_0}$ 
by decomposing $F_\Xi$ (for $\Xi\in\mathsf B_{n+1}^{\zeta_0}$) into sets satisfying estimates similar to those in~\eqref{eq: regular tree}.

To that end, let $\Xi=(\zeta_0, \ldots,\zeta_n,r_n,\zeta_{n+1})\in\mathsf B_{n+1}^{\zeta_0}$, and let 
$F=F_\Xi$, $y=y_\Xi$, $\cone=\cone_\Xi$.
In view of Lemma~\ref{lem: bd mfm cone zeta and conv comp n}(2) and Lemma~\ref{lem: F energy est assuming mfm energy estimate}, 
\[
\eng_{F_w, \trct_1}(w')\leq 10^6\mfbd_{\Xi,n+1}\qquad\text{for every $w'\in F_w$},
\]  
where $F_w=F\cap B_\rfrak(w,4\mfsc\inj(y))$.  

Let $k_1> k_{0}$ be positive integers defined as follows: 
\be\label{eq: def k0 and k1 step n}
2^{k_{0}}< (\mfsc\,\inj(y))^{-1}\leq 2^{k_{0}+1}\quad\text{and}\quad 2^{k_1}< 10^6\mfbd_{\Xi,n+1}\leq 2^{k_1+1}
\ee
Let $\mathsf M$ be as above, see~\eqref{eq: condition on M 1}.
Applying Lemma~\ref{lem: regular tree decomposition}, we can write 
\be\label{eq: Fi as union ini dim n}
F=F'\bigcup \Bigl(\textstyle \bigcup_l F_l\Bigr)
\ee
where $\#F'\leq \beta^{1/4}\cdot (\#F)$ and $\#F_l\geq \beta^2\cdot(\#F)$.
In view of~\eqref{eq: control over card of fiber' n}, we have  
\be\label{eq: control over card of fiber' ??}
\begin{aligned}
0.5\beta^{11}\cdot(\#F_{\hat\Xi})\leq\beta^2\cdot(\#F)&\leq\#F_l\\
&\leq \#F \leq 2\beta^8\cdot(\#F_{\hat\Xi}), 
\end{aligned}
\ee
and for every $k_0-10\leq k\leq k_1$, 
there exists some $\tau_k=\tau_{k}^l$ so that 
\be\label{eq: regular tree'''}
\text{either}\quad2^{\mathsf M(\tau_{k}-2)}\leq \#F_l\cap Q\leq 2^{\mathsf M\tau_{k}}\quad\text{or}\quad F_l\cap Q=\emptyset,
\ee
for all $Q\in\mathcal Q_{\mathsf Mk}$.

Let us also note that combining~\eqref{eq: control over card of fiber' ??} and~\eqref{eq: ind hyp on num F hat Xi}, we conclude
\be\label{eq: num F ir jm varsigma n}
\frac{1}{2^{n+1}}\beta^{11(n+2)}e^t\leq \#F_l\leq 2^{n+1}\beta^{8(n+1)-3}e^t.
\ee 

Let $\hat\Xi=(\zeta_0,\ldots,\zeta_n)\in\mathsf A_n^{\zeta_0}$, and let $r_n\in\mathcal L_{\cone_{\hat\Xi}}$. We let $\mathcal Z_{\hat\Xi,r_n}$ denote an enumeration of 
\[
\{(\zeta',l):\zeta'\in \mathcal Z_{\hat\Xi, r_n}', l\in \mathcal K_{\Xi}\}
\]
where for $\Xi=(\zeta_0,\ldots,\zeta_n,r_n,\zeta')\in\mathsf B_{n+1}^{\zeta_0}$, we let $\mathcal K_\Xi=\{l: F_l$ as in $\eqref{eq: Fi as union ini dim n}\}$. 
If $\zeta\in \mathcal Z_{\hat\Xi,r_n}$ corresponds to $(\zeta',l)$, then we put $y_\zeta=y_{\zeta'}$ 
and $F_\zeta=(F_{\zeta'})_l$, see~\eqref{eq: Fi as union ini dim n} and the discussion leading to Lemma~\ref{lem: step n of ind initial leaves}. 

Define 
\be\label{eq: define A n+1 zeta}
\mathsf A_{n+1}^{\zeta_0}=\left\{(\zeta_0, r_0,\ldots, \zeta_n, r_n, \zeta_{n+1}):\begin{array}{c} \hat\Xi=(\zeta_0,r_0,\ldots, \zeta_n)\in\mathsf A_{n}^{\zeta_0},\\ 
r_n\in \mathcal L_{\cone_{\hat\Xi}}, \zeta_{n+1}\in \mathcal Z_{\hat\Xi, r_n}\end{array}\right\},
\ee
and for every $\Xi=(\zeta_0, r_0, \ldots, \zeta_n, r_n,\zeta_{n+1})\in\mathsf A_{n+1}^{\zeta_0}$, put 
\[
\cone_\Xi=\coneH.\{\exp(w)y_\Xi: w\in F_{\Xi}\},
\]
where $y_\Xi=y_{\zeta_{n+1}}$ and $F_\Xi=F_{\zeta_{n+1}}$.

\begin{lemma}\label{lem: bd mfm cone zeta and conv comp An} 
Let $\Xi=(\zeta_0, r_0, ,\ldots, \zeta_n, r_n,\zeta_{n+1})\in\mathsf A_{n+1}^{\zeta_0}$. Suppose $\zeta_{n+1}$ corresponds to $(\zeta',l)$ as above, i.e., $\Xi'=(\zeta_0,r_0,\ldots,\zeta_n, r_n,\zeta')\in\mathsf B_{n+1}^{\zeta_0}$ and $l\in\mathcal K_{\Xi'}$. Put $\mfbd_{\Xi,n+1}=\mfbd_{\Xi',n+1}$.
Both of the following hold: 
\begin{enumerate}
\item Let  $\trct_{n+1}=1+(n+1)L\kappa^{-L}e^{\kappa Dt}$. Then
\be\label{eq: mfbd n 2}
\mfm_{\cone_\Xi,\scmf,\trct_{n+1}}(e,z)\leq \mfbd_{\Xi,n+1}.
\ee
\item Let $\hat\Xi=(\zeta_0, r_0,\ldots, \zeta_n)\in\mathsf A_n^{\zeta_0}$ and let 
$r_{n}\in \mathcal L_{\cone_{\hat\Xi}}$. Then for every $\varphi\in C_c^\infty(X)$, every $0<\tau\leq 2d_1\ell$, and all $|s|\leq 2$, we have
\be\label{eq: convex comb 2 use step n 2}
\begin{aligned}
\biggl|\int\varphi(a_\tau u_s.z)\diff(a_\ell u_{r_0}\mu_{\cone_{\hat\Xi}})(z)-&\sum c_{\Xi}\int\varphi(a_\tau u_sz)\diff\!\mu_{\cone_\Xi}(z)\biggr|\\
&\ll\max\Bigl\{\eta^{1/2}, e^{-\kappa^2t/64}\Bigr\}\Lip(\varphi),
\end{aligned}
\ee
where the sum is over $\mathcal Z_{\hat\Xi, r_n}$, and for every $\zeta\in\mathcal Z_{\hat\Xi, r_n}$, we let 
\[
\Xi=(\zeta_0, r_0,\ldots, \zeta_n, r_n,\zeta).
\]
\end{enumerate}
\end{lemma}

\begin{proof}
Part~(1) follows from Lemma~\ref{lem: bd mfm cone zeta and conv comp n}(2) and the fact that $\cone_\Xi\subset\cone_{\Xi'}$.

As for part~(2), we again use the above notation, i.e., 
\[
\Xi=(\zeta_0, r_0,\ldots, \zeta_n, r_n,\zeta).
\]
where $\hat\Xi=(\zeta_0, r_0,\ldots, \zeta_n)$. Suppose $\zeta$ corresponds to $(\zeta',l)$ as above, that is, $\Xi'=(\zeta_0,r_0,\ldots,\zeta_n, r_n,\zeta')\in\mathsf B_{n+1}^{\zeta_0}$ and $l\in\mathcal K_{\Xi'}$.
Then part~(2) in the lemma follows from Lemma~\ref{lem: bd mfm cone zeta and conv comp n}(3) in view of~\eqref{eq: Fi as union ini dim n} if we put 
\[
c_{\Xi}=c_{\Xi'}\mu_{\cone_{\Xi'}}(\cone_{\Xi})
\]
and use the fact that $\mu_{\cone_{\Xi'}}$ is admissible, see Lemma~\ref{lem: mu cone is admissible 2}. 
\end{proof}

\section{Final sets and the proof of Proposition~\ref{propos: imp dim main}}\label{sec: proof main prop}
We will complete the proof of Proposition~\ref{propos: imp dim main} in this section. 
Let $\zeta_0\in\mathcal Z$, see \S\ref{sec: initial dim} in particular~\eqref{eq: def mathcal Z}, 
and let $\mathsf A_{n}^{\zeta_0}$ be defined as in~\eqref{eq: define A n+1 zeta}. 

Recall that $0<\vare<1$ is a small parameter (in our application, $\vare$ will depend on $\ref{k:mixing}$, see~\eqref{eq: choose theta equi sec}) and $t>1$ is a large parameter (which will be chosen to be $\asymp\log R$ where $R$ is as in Theorem~\ref{thm:main}); let $\mfsc=e^{-\sqrt\vare t}$. Recall also from Proposition~\ref{propos: imp dim main} that we fixed 
\be\label{eq: choose kappa}
\kappa = 10^{-6}d_1^{-1}\leq 10^{-6}\vare;
\ee 
where $d_1=100\lceil{(4D-3)}/(2\vare)\rceil$, see Proposition~\ref{propos: imp dim main}.

Set $\beta=e^{-\kappa t}$ and $\eta^2=\beta$. 
Recall from~\eqref{eq: def trct n+1} that 
\[
\trct_{n}=1+nL\kappa^{-L}e^{\kappa Dt}.
\] 
In particular, so long as $t$ is large enough, we have 
\be\label{eq: trct d1}
\trct_{d_1}=1+d_1L\kappa^{-L}e^{\kappa Dt}\leq e^{0.01\vare t}.
\ee

Recall also our assumption that Proposition~\ref{prop:closing lemma intro}(1) holds, and that 
\[
x_2=a_{8t}u_{r_1}x_1
\] 
where $r_1\in I(x_1)$. 
Then $x_2\in X_\eta$, and the map $h\mapsto hx_2$ is injective over $\boxHs_\beta\cdot a_t\cdot U_1$, see Proposition~\ref{prop:closing lemma intro}(1).

Motivated by the conditions in~\ref{(A-1)} and~\ref{(A-2)} of Lemma~\ref{lem: step n of ind initial leaves}, we make the following definition.

\begin{defin}\label{def: final sets}
Let $d_2:=d_1-\Bigl\lceil\frac{10^4}{\sqrt\vare}\Bigr\rceil$ where $d_1=100\Bigl\lceil\frac{4D-3}{2\vare}\Bigr\rceil$, and let 
\[
d_2\leq d\leq d_1.
\]

Let $\zeta_0\in\mathcal Z$. An element $\Xi\in\mathsf A_d^{\zeta_0}$ 
is said to be {\em final} if 
\be\label{eq: final dim 1 1}
\egbd_{\Xi,d}< \nuni^{\vare t/2}\sup_{w\in F_\Xi}\noI_{\cone_\Xi,\mfsc}(e,\exp(w)y_{\Xi}),
\ee
where $\cone_\Xi=\coneH.\{\exp(w)y_\Xi: w\in F_\Xi\}$. 
\end{defin}

It will be more convenient to distinguish elements of $\mathsf A_d^{\zeta_0}$ satisfying~\eqref{eq: final dim 1 1} for $d<d_2$ as well. Thus, for every $0\leq d\leq d_1$, let 
\[
\hat{\mathsf A}_d^{\zeta_0}=\Bigl\{\Xi\in \mathsf A_d^{\zeta_0}: \Xi\text{ satisfies~\eqref{eq: final dim 1 1}}\Bigr\}.
\]
Note that if $d_2\leq d\leq d_1$, then $\Xi\in\hat{\mathsf A}_d^{\zeta_0}$ if and only if it is final. 

\begin{lemma}\label{lem: final is has dim 1}
If $\Xi\in\hat{\mathsf A}_d^{\zeta_0}$, then
\[
\mfm_{\cone_\Xi,\mfsc, \trct_d}(e,z)\leq \ref{C: M and dimension}\nuni^{\vare t/2}\noI_{\cone_\Xi,\mfsc}(e,z)
\]
for all $z=h\exp(w)y_\Xi\in\cone_{\Xi}$ with $h\in\overline{\coneH\setminus\partial_{10\mfsc}\coneH}$. 
\end{lemma}

\begin{proof}
Let $z$ be as in the statement. Then by 
Lemma~\ref{lem: regular sets initial dim}, we have 
\[
\sup_w\noI_{\cone_\Xi,\scmf}(e,\exp(w)y)\leq \ref{C: M and dimension}\noI_{\cone_\Xi,\scmf}(e,z).
\]
Moreover, by~\eqref{eq: mfbd 1}, we have 
\[
\mfm_{\cone_\Xi,\mfsc, \trct_d}(e,z')\leq\egbd_{\Xi,d},\quad\text{for all $z'\in\cone_\Xi$}.
\]
The claim in the lemma follows from these, in view of~\eqref{eq: final dim 1 1}.   
\end{proof}

We fix the following notation: Let $0\leq d\leq d_1$, for any
\[
\Xi=(\zeta_0, r_0,\ldots,\zeta_{d-1}, r_{d-1}, \zeta_{d})\in\mathsf A_{d}^{\zeta_0}, 
\] 
and $0\leq n\leq d$, put $\Xi_n:=(\zeta_0, r_0,\ldots,\zeta_n)$. 

\begin{lemma}\label{lem: every Xi has future}
Let $\Xi\in\mathsf A_{d_2}^{\zeta_0}$, and let $d_2\leq d\leq d_1$. Let $\Xi'\in\mathsf A_{d}^{\zeta_0}$ 
be so that $\Xi'_{d_2}=\Xi$. Then at least one of the following holds. 
\begin{enumerate}
\item \label{every Xi first}
There exists $d_2\leq n\leq d$ so that $\Xi'_n\in\hat{\mathsf A}_{n}^{\zeta_0}$.
\item \label{every Xi second} There exists $d< d'\leq d_1$ and $\Xi''\in \hat{\mathsf A}_{d'}^{\zeta_0}$ so that $\Xi''_{d}=\Xi'$.
\end{enumerate} 
In particular,
\begin{enumerate}[resume*]
    \item \label{in particular item}
    For every $\Xi\in\mathsf A_{d_2}^{\zeta_0}$ and every $\Xi'\in\mathsf A_{d_1}^{\zeta_0}$ with $\Xi'_{d_2}=\Xi$, there exists $d_2\leq d\leq d_1$ so that $\Xi'_d\in\hat{\mathsf A}_{d}^{\zeta_0}$.
\end{enumerate}
\end{lemma}

\begin{proof}
First note that \eqref{in particular item} is a direct consequence of \eqref{every Xi first}--\eqref{every Xi second}. Thus it is enough to prove the latter.

For every $\Xi\in\mathsf A_{d_2}^{\zeta_0}$, put  
\be\label{eq: def past}
{\rm past}(\Xi)=\Bigl\{n_i\leq d_2: \Xi_{n_i}\in\hat{\mathsf A}_{n_i}^{\zeta_0}\Bigr\}
\ee 
if such $n_i$ exists, otherwise put ${\rm past}(\Xi)=\emptyset$; 
in the former case, we will write ${\rm past}(\Xi)=\{n_1< \cdots < n_{m_{\Xi}}\}$. 
It follows from the definition (see~\eqref{eq: final dim 1 1}) 
that if $n\in {\rm past}(\Xi)$, then 
\[
\egbd_{\Xi_{n},n}< e^{\vare t/2}\sup_w\noI_{\cone_{\Xi_{n}},\mfsc}(e,\exp(w)y_{\Xi_{n}}).
\]

Let $d$ and $\Xi'\in\mathsf A_{d}^{\zeta_0}$ be as in the statement; 
note that for every $d\leq d'\leq d_1$, we have 
\[
\Bigl\{\Xi''\in \mathsf A_{d'}^{\xi_0}:\Xi''_{d}=\Xi'\Bigr\}\neq\emptyset;
\] 
see the discussion leading to~\eqref{eq: define A n+1 zeta}. 

We will consider two cases, ${\rm past}(\Xi)=\emptyset$ and ${\rm past}(\Xi)\neq\emptyset$, separately
(though the argument in both cases is similar).

\medskip

\noindent
{\bf Case 1.}\ Assume that ${\rm past}(\Xi)=\emptyset$.

Suppose that the claim in the lemma fails. 
Then for every $\Xi''\in \mathsf A_{d_1}^{\xi_0}$ with $\Xi''_{d}=\Xi'$ 
and all $0\leq n\leq d_1$ we have 
\be\label{eq: A-1 holds for all n}
\egbd_{\Xi_{n}'',n}\geq e^{\vare t/2}\sup_w\noI_{\cone_{\Xi_{n}''},\mfsc}(e,\exp(w)y_{\Xi_{n}''}).
\ee
For $0\leq n\leq d_2$,~\eqref{eq: A-1 holds for all n} follows from ${\rm past}(\Xi)=\emptyset$ and $\Xi''_{d_2}=\Xi$; for $d_2\leq n\leq d_1$, it follows from the fact that $\Xi''_n\not\in\hat{\mathsf A}_{n}^{\zeta_0}$, see~\eqref{eq: final dim 1 1}. 

We will show that~\eqref{eq: A-1 holds for all n} leads to a contradiction. 
To that end, put 
\[
\cone''=\cone_{\Xi''}=\coneH.\{\exp(w)y: w\in F''\}.
\]
Recall that $\ell=0.01\vare t$ and $d_1=100\lceil\frac{4D-3}{2\vare}\rceil$. 
Thus $\frac{\ell d_1}{2}\geq \frac{(4D-3)t}{4}$ and   
\be\label{eq: d max is chosen to beat M}
\nuni^{-\ell d_1/2}\nuni^{Dt}\leq e^{-(4D-3)t/4} \nuni^{Dt}\leq e^{3t/4}.
\ee
In view of~\eqref{eq: A-1 holds for all n}, we have~\ref{(A-1)} and~\eqref{eq: define mfbd n+1 1} hold for all $0\leq n\leq d_1$. That is $\mfbd_{\Xi''_n,n}=e^{-\ell/2} \mfbd_{\Xi''_{n-1},n-1}$ for all $0<n\leq d_1$. 
Since $\mfbd_0=e^{Dt}$, we conclude from~\eqref{eq: d max is chosen to beat M} that 
\be\label{eq: mfbd d1 case 1}
\mfbd_{\Xi'',d_1} = e^{-d_1\ell/2}e^{Dt}\leq e^{3t/4}.
\ee

We will compare~\eqref{eq: mfbd d1 case 1} with a lower bound for $\noI_{\cone'',\mfsc}$ which we now obtain. In view of~\eqref{eq: ind hyp on num F hat Xi}, we have 
\[
\#F''\geq (0.5)^{d_1}\beta^{11(d_{1}+1)}\nuni^{t}.
\] 
This and~\eqref{eq: trivial lower bd} imply that for all $w\in F''$, 
\be\label{eq: size of noI for final sets'}
\noI_{\cone'',\scmf}(e,\exp(w)y)\geq \nuni^{-4\sqrt\vare t}(\#F'')\geq  \nuni^{-4\sqrt\vare t}\beta^{11d_{1}+12}\nuni^{t}\geq \nuni^{0.9t} 
\ee
where in the last inequality we used $\beta=\nuni^{-\kappa t}$ and $100d_1\kappa\leq 0.01$, see~\eqref{eq: choose kappa}.

We conclude from~\eqref{eq: mfbd d1 case 1} and~\eqref{eq: size of noI for final sets'} that
\[
\mfbd_{\Xi'',d_1}\leq \sup_w\noI_{\cone'',\scmf}(e,\exp(w)y).
\]
This contradicts $\Xi''\not\in\hat{\mathsf A}_{d_1}^{\zeta_0}$, 
and completes the proof in this case.

\medskip

\noindent
{\bf Case 2.}\ Assume that ${\rm past}(\Xi)\neq\emptyset$. 

Let us write ${\rm past}(\Xi)=\{n_1< \cdots< n_{m_\Xi}\}$, and let $\Xi'$ be as in the statement. 
We will write $n_m=n_{m_\Xi}$ for simplicity in the notation.  
Assume again that the claim in the lemma fails. 
First note that $n_m<d_2$ otherwise part~\eqref{every Xi first} would hold with $n=d_2$, 
which contradicts our assumption. 
Similar to~\eqref{eq: A-1 holds for all n}, for every 
$\Xi''\in \mathsf A_{d_1}^{\xi_0}$ with $\Xi''_{d}=\Xi'$ and all $n_m< n\leq d_1$ we have 
\be\label{eq: A-1 holds for all n > nm}
\egbd_{\Xi_{n}'',n}\geq e^{\vare t/2}\sup_w\noI_{\cone_{\Xi_{n}''},\mfsc}(e,\exp(w)y_{\Xi_{n}''}).
\ee
For $n_m< n\leq d_2$, this follows from ${\rm past}(\Xi)=\{n_1,\ldots, n_m\}$ and $\Xi''_{d_2}=\Xi$; for $d_2\leq n\leq d_1$, it follows from our assumption that $\Xi''_n\not\in\hat{\mathsf A}_{n}^{\zeta_0}$.

As in Case 1, we will show that~\eqref{eq: A-1 holds for all n > nm} leads to a contradiction. Put 
\[
\cone''=\cone_{\Xi''}=\coneH.\{\exp(w)y: w\in F''\}.
\]
We will now inductively estimate $\mfbd_{\Xi_{n}'',n}$ for $n_m<n\leq d_1$. 
Since $\Xi_{n_m}=\Xi_{n_m}''\in\hat{\mathsf A}_{n_m}^{\zeta_0}$ and $\Xi_{n_m+1}''\not\in\hat{\mathsf A}_{n_m+1}^{\zeta_0}$ (see~\eqref{eq: def past}),
we conclude that~\ref{(A-2)} and~\eqref{eq: define mfbd n+1 2} are used to define $\mfbd_{\Xi_{n_m+1}'',n_m+1}$. Thus there exists some $w_0\in F_{\Xi_{n_m}}$ so that  
\be\label{eq: Xi n2 high density}
\begin{aligned}
\mfbd_{\Xi_{n_m+1}'',n_m+1}&= \nuni^{-\ell/2} \nuni^{\vare t/2}\noI_{\cone_{\Xi_{n_m}},\mfsc}\Bigl(e,\exp(w_0)y_{\Xi_{n_m}}\Bigr)\\
&\leq 2\nuni^{-\ell/2} \nuni^{\vare t/2}\eta^{-\alpha}\mfsc^{-\alpha}\cdot \Bigl(\#F_{\Xi_{n_m}}\Bigr)
\end{aligned}
\ee
where we used the definition of $\noI$ in the last inequality.

We now turn to $\egbd_{\Xi_n'',n}$ for $n>n_m+1$. In view of~\eqref{eq: A-1 holds for all n > nm} applied for $n$ and $n-1$, we have \ref{(A-1)} and~\eqref{eq: define mfbd n+1 1} hold. Thus
\[
\mfbd_{\Xi_n'',n}=\nuni^{-\ell/2}\mfbd_{\Xi_{n-1}'',n-1}\quad\text{ for all $n_m+1< n\leq d_1$}.
\]
This and~\eqref{eq: Xi n2 high density}, imply that 
\be\label{eq: mfbd d1 case 2}
\mfbd_{\Xi'',d_1}\leq \nuni^{-\ell(d_1-n_m)/2}\cdot\Bigl(2\nuni^{\vare t/2}\eta^{-\alpha}\mfsc^{-\alpha}\Bigr)\cdot \Bigl(\#F_{\Xi_{n_m}}\Bigr).
\ee
We will compare~\eqref{eq: mfbd d1 case 2} with a lower bound for $\noI_{\cone'',\mfsc}$ which we now obtain. In view of~\eqref{eq: control over card of fiber' ??}, we have 
\[
\#F''\geq (0.5)^{d_1}\beta^{11(d_{1}-n_m)}\cdot(\#F_{\Xi_{n_m}}).
\] 
This and~\eqref{eq: trivial lower bd} imply that for all $w\in F''$,
\be\label{eq: size of noI for final sets' ?}
\begin{aligned}
\noI_{\cone'',\scmf}(e,\exp(w)y)&\geq \nuni^{-4\sqrt\vare t}(\#F'')\\
&\geq  \nuni^{-4\sqrt\vare t}(0.5)^{d_1}\beta^{11(d_{1}-n_m)}\cdot(\#F_{\Xi_{n_m}}).
\end{aligned}
\ee

Since $\Xi''\not\in\hat{\mathsf A}_{d_1}^{\zeta_0}$, we have 
\[
\mfbd_{\Xi'',d_1}\geq e^{\vare t/2}\sup_w\noI_{\cone'',\scmf}(e,\exp(w)y).
\]
Combining this with~\eqref{eq: mfbd d1 case 2} and~\eqref{eq: size of noI for final sets' ?}, we conclude that 
\begin{align*}
\nuni^{-\frac{\ell(d_1-n_m)}{2}}\cdot\Bigl(2\nuni^{\frac{\vare t}{2}}\eta^{-\alpha}\mfsc^{-\alpha}\Bigr)&\cdot \Bigl(\#F_{\Xi_{n_m}}\Bigr)\geq \mfbd_{\Xi'',d_1}\\
&\geq e^{\frac{\vare t}{2}}\sup_w\noI_{\cone'',\scmf}(e,\exp(w)y)\\
&\geq e^{\frac{\vare t}{2}}\nuni^{-4\sqrt\vare t}(0.5)^{d_1}\beta^{11(d_{1}-n_m)}\cdot(\#F_{\Xi_{n_m}}).
\end{align*}
Comparing the first and last terms, cancelling $\#F_{\Xi_{n_m}}$ and $e^{\vare t/2}$ from both sides, and multiplying by
$\beta^{-11(d_{1}-n_m)}$ and replacing $2^{d_1+1}$ by $\beta^{-1}$, 
\[
\nuni^{-\ell(d_1-n_m)/2}\beta^{-11(d_{1}-n_m)-1}\cdot\Bigl(\eta^{-\alpha}\mfsc^{-\alpha}\Bigr)\geq \nuni^{-4\sqrt\vare t}.
\]
Recall now that $\beta=\nuni^{-\kappa t}$, $0<\kappa\leq \vare/10^6$, see~\eqref{eq: choose kappa}, and that $\ell=0.01\vare t$. Therefore, 
\[
\nuni^{-\ell(d_1-n_m)/2}\beta^{-11(d_{1}-n_m)-1}\leq \nuni^{-\ell(d_1-n_m)/3}
\]
This and the above thus imply that
\be\label{eq: final est case 2 n2}
\nuni^{-\ell(d_1-n_m)/3}\cdot\Bigl(\eta^{-\alpha}\mfsc^{-\alpha}\Bigr)\geq \nuni^{-4\sqrt\vare t}.
\ee
However, $\ell=0.01\vare t$ and $d_1-n_m\geq d_1-d_2\geq 10^4/\sqrt\vare$. 
Therefore, we have $\ell(d_1-n_m)\geq 100\sqrt\vare t$. This, together with $\eta\geq e^{-\vare t}$ and $\mfsc=\nuni^{-\sqrt\vare t}$, implies 
\[
\nuni^{-\ell(d_1-n_m)/3}\cdot\Bigl(\eta^{-\alpha}\mfsc^{-\alpha}\Bigr)\leq \nuni^{-30\sqrt\vare t}
\]
which contradicts~\eqref{eq: final est case 2 n2} and finishes the proof in Case 2 as well.  
\end{proof}

In view of this lemma, let $\hat{\mathsf A}_{d_2,d_2}^{\zeta_0}=\hat{\mathsf A}^{\zeta_0}_{d_2}$, and for every $d_2< d\leq d_1$, let
\[
\hat{\mathsf A}_{d_2,d}^{\zeta_0}=\Bigl\{\Xi\in \hat{\mathsf A}_{d}^{\zeta_0}: \Xi_n\not\in \hat{\mathsf A}_{n}^{\zeta_0} \text{ for any $d_2\leq n< d$}\Bigr\}.
\]
Let $N_d^{\zeta_0}=\#\hat{\mathsf A}_{d_2,d}^{\zeta_0}$ and enumerate the elements of $\hat{\mathsf A}_{d_2,d}^{\zeta_0}$ as $\{\cone_{d}^i\}$. For all $d$ as above and all $1\leq i\leq N_d^{\zeta_0}$,
$\cone_{d}^i$ and $\mu_{\cone_{d}^{i}}$ denote $\cone_{\Xi_d^i}$ and $\mu_{\cone_{\Xi_d^i}}$, respectively --- we note that $\cone_{d}^i$
and $\mu_{\cone_{d}^i}$ also depend on $\zeta_0$, however, this abuse of notation will not cause confusion in what follows. 

\begin{lemma}\label{lem: app with final sets}
For every $\varphi\in C_c^{\infty}(X)$, all $0<\tau\leq d_1\ell$ and $|s|\leq 2$ we have
\begin{multline*}
\biggl|\int \varphi(a_\tau u_shx_2)\diff\!\mu_{t,\ell, d_1}(h)-\sum_{\mathcal Z}\sum_{d,i}c_{d,i}\int \varphi(a_\tau u_sz)\diff\!\rwm_\ell^{(d_{1}-d)}\conv\mu_{\cone_d^i}(z)\biggr|\\
\ll \Lip(\varphi)\beta^{\star}
\end{multline*}
where for every $\zeta_0\in\mathcal Z$, the inner sum is over $d_2\leq d\leq d_1$ and $1\leq i\leq N_d^{\zeta_0}$, 
$c_{d,i}\geq0$ with $\sum_{d,i}c_{d,i}=1-O(\beta^{\star})$, $\Lip(\varphi)$ 
is the Lipschitz norm of $\varphi$, and the implied constants depend on $X$.
\end{lemma}

\begin{proof}
We will use the above notation also the notation from \S\ref{sec: induction}.
Let 
\[
\{(\cone_{\zeta_0},\mu_{\zeta_0}): \zeta_0\in \mathcal Z\}
\] 
be as in~\eqref{eq: def mathcal Z}.
For every $\zeta_0\in\mathcal Z$, let $\mathsf A_{d_2}^{\zeta_0}$ 
be as in~\eqref{eq: define A n+1 zeta}. Then by part~(2) in Lemma~\ref{lem: bd mfm cone zeta and conv comp An}, for $0<\tau'\leq 2d_1\ell$, we have  
\begin{multline}\label{eq: convex comb 2 use'}
\biggl|\int \varphi(a_{\tau'}u_shx_1)\diff\!\mu_{t,\ell, d_2}(h)-\sum_{\zeta_0\in\mathcal Z}\sum_{\Xi\in\mathsf A_{d_2}^{\zeta_0}} c_{\Xi}\int\varphi(a_{\tau'} u_sz)\diff\!\mu_{\cone_\Xi}(z)\biggr|\\
\ll\max\Bigl\{\eta^{1/2}, e^{-\kappa^2t/64}\Bigr\}\Lip(\varphi).
\end{multline}
Recall that $a_{\ell_1}u_{r}a_{\ell_2}=a_{\ell_1+\ell_2}u_{e^{-\ell_2}r}$ for all $\ell_1,\ell_2,r\in\bbr$. Arguing as in Lemma~\ref{lem: thickening stable},~\eqref{eq: convex comb 2 use'} (applied with $\tau'=\tau+(d_1-d_2)\ell\leq 2d_1\ell$) implies that 
\begin{multline}\label{eq: convex comb 2 use''}
\biggl|\int \!\!\varphi(a_\tau u_shx_1)\diff\!\mu_{t,\ell, d_1}(h)-\sum c_{\Xi}\int\!\!\varphi(a_\tau u_sz)\diff\!\rwm^{(d_1-d_2)}\!\conv\mu_{\cone_\Xi}(z)\biggr|\\
\ll\max\Bigl\{\eta^{1/2}, e^{-\kappa^2t/64}\Bigr\}\Lip(\varphi)
\end{multline}
where $\sum=\sum_{\zeta_0\in\mathcal Z}\sum_{\Xi\in\mathsf A_{d_2}^{\zeta_0}}$.  
 
Let $\zeta_0\in\mathcal Z$ and let $\Xi\in\mathsf A_{d_2}^{\zeta_0}$. For every $d_2\leq d\leq d_1$, put 
 \[
\hat{\mathsf A}_{d_2,d}^{\zeta_0}(\Xi)=\{\Xi'\in\hat{\mathsf A}_{d_2,d}^{\zeta_0}: \Xi'_{d_2}=\Xi\};
\]
note in particular that if $\Xi\in\hat{\mathsf A}_{d_2}^{\zeta_0}$, then $\hat{\mathsf A}_{d_2,d}^{\zeta_0}(\Xi)=\emptyset$ for all $d> d_2$. 

We claim that
\begin{multline}\label{eq: sum Xi as dim 1}
\bigg|\int\!\!\varphi(a_\tau u_sz)\diff\!\rwm^{(d_1-d_2)}\!\conv\mu_{\cone_\Xi}-\sum\! c_{\Xi'}\!\!\int\!\!\varphi(a_\tau u_sz)\diff\!\rwm^{(d_1-d)}\!\conv\mu_{\cone_{\Xi'}}\biggr|\\
\ll \max\Bigl\{\eta^{1/2}, e^{-\kappa^2t/64}\Bigr\}\Lip(\varphi)
\end{multline}
where now $\sum=\sum_{d_2\leq d\leq d_1}\sum_{\hat{\mathsf A}_{d_2,d}^{\zeta_0}(\Xi)}$ and again $\sum c_{\Xi'}>1-O(\beta^\star)$. 

Note that~\eqref{eq: sum Xi as dim 1} and~\eqref{eq: convex comb 2 use''} finish the proof of the lemma. Thus, we need to prove~\eqref{eq: sum Xi as dim 1}. 

As it was mentioned, if 
$\Xi\in\hat{\mathsf A}_{d_2}^{\zeta_0}$, then $\hat{\mathsf A}_{d_2,d}^{\zeta_0}(\Xi)=\emptyset$ 
for all $d> d_2$, and there is nothing to prove. Let now $\Xi\in\mathsf A_{d_2}^{\zeta_0}\setminus\hat{\mathsf A}_{d_2}^{\zeta_0}$. 
Then we have 
\[
\int\varphi(a_\tau u_sz)\diff\!\rwm^{(d_1-d_2)}\!\conv\mu_{\cone_\Xi}=\int_0^1\int\varphi(a_\tau u_sz)\diff(\rwm^{d_1-d_2-1}\conv (a_\ell u_{r}\mu_{\cone_\Xi}))\diff\!r.
\]
Thus by Lemma~\ref{lem: convex comb 2} applied to the right side of the above, see also Lemma \ref{lem: bd mfm cone zeta and conv comp An}, we have 
\begin{multline*}
\biggl|\int\varphi(a_\tau u_sz)\diff\!\rwm^{(d_1-d_2)}\!\conv\mu_{\cone_\Xi}-\sum c_{\Xi'}\int\varphi((a_\tau u_sz))\diff(\rwm^{d_1-d_2-1}\conv\mu_{\cone_{\Xi'}})\biggr|\\
\ll\eta^{1/2}\Lip(\varphi),
\end{multline*}
where the sum is over $\Xi'\in\mathsf A_{d_2+1}^{\zeta_0}$ with $\Xi'_{d_2}=\Xi$. 

We now continue inductively, i.e., write 
\begin{multline*}
\Bigl\{\Xi'\in\mathsf A_{d_2+1}^{\zeta_0}: \Xi'_{d_2}=\Xi\Bigr\}=\\\hat{\mathsf A}_{d_2, d_2+1}^{\zeta_0}(\Xi)\cup\Bigl\{\Xi'\in\mathsf A_{d_2+1}^{\zeta_0}:\Xi'_{d_2}=\Xi, \Xi'\not\in\hat{\mathsf A}_{d_2, d_2+1}^{\zeta_0}(\Xi)\Bigr\}
\end{multline*}
and decompose the sum $\sum_{\Xi'}$ accordingly.  Repeat the above for all $\Xi'\in\mathsf A_{d_2+1}^{\zeta_0}$ with $\Xi'_{d_2}=\Xi$ but $\Xi'\not\in\hat{\mathsf A}_{d_2, d_2+1}^{\zeta_0}(\Xi)$. 
In view of Lemma~\ref{lem: every Xi has future}, this process terminates at some $d\leq d_1$, 
and the claim in~\eqref{eq: sum Xi as dim 1} follows. 
\end{proof}

\begin{proof}[Proof of Proposition~\ref{propos: imp dim main}]
Proposition~\ref{propos: imp dim main} follows from Lemma~\ref{lem: app with final sets}, as we now explicate. 
The decomposition in Lemma~\ref{lem: app with final sets} is of the form claimed in~\eqref{eq: nud1 and nud1-d}. 

Moreover, the sets provided by Lemma~\ref{lem: app with final sets} satisfy~\eqref{eq: improve dim regularity} in view of~\eqref{eq: regular tree'''} as $t$ is sufficiently large and $M$ is fixed. They also satisfy~\eqref{eq:improve dim energy est} thanks to Lemma~\ref{lem: final is has dim 1}. In view of Lemma~\ref{lem: mu cone is admissible} and Lemma~\ref{lem: mu cone is admissible 2}, the measures are $(\adl_{{\bigcdot}},\adm_{{\bigcdot}})$-admissible with $\adm_{{\bigcdot}}$ depending only on $X$ and the number of steps, which is $\leq d_1$. Finally, in view of~\eqref{eq: trct d1},  
\[
\trct_{d}\leq \trct_{d_1}\leq e^{0.01\vare t}. 
\]
The proof is complete.
\end{proof}

\section{From large dimension to equidistribution}\label{sec: equidistribution}
Let $0<\ref{k:mixing}\leq 1$ be the constant given by Proposition~\ref{prop: 1-epsilon N}; recall that this constant is closely related to the spectral gap (or mixing rate) in $G/\Gamma$, c.f.~\eqref{eq: exp mixing}.
Throughout this section, we fix $\vare$ as follows 
\be\label{eq: choose theta equi sec}
0<\sqrt\vare \leq 10^{-8}\ref{k:mixing}.
\ee
We also recall that $\beta=e^{-\kappa t}$ and $\eta^2=\beta$ where $0<\kappa\leq\vare/10^6$. 

The following is the main result of this section.

\begin{propos}\label{prop: high dim to equid}
The following holds for all large enough $t$. 
Let $F\subset B_\rfrak(0,\beta)$ be a finite set with $\#F\geq \nuni^{0.9t}$. 
Let 
\[
\cone=\coneH.\{\exp(w)y: w\in F\}\subset X_\eta
\]
be equipped with an admissible measure $\mu_\cone$ (the definition is recalled below). 
Assume further that the following two properties are satisfied:
\begin{enumerate}
\item For all $w\in F$, we have 
\be\label{eq: psi almos cst equi prop} 
\#\Bigl(B_\rfrak(w,4\mfsc\,\inj(y))\cap F\Bigr)\geq \nuni^{-\vare t}\sup_{w'\in F}\#\Bigl(B_\rfrak(w',4\mfsc\,\inj(y))\cap F\Bigr).
\ee
\item For all $z=h\exp(w)y$ with $h\in\overline{\coneH\setminus\partial_{10\mfsc}\coneH}$, we have 
\be\label{eq: dim 1 equi prop} 
\mfm_{\cone, \scmf,\trct}(e,z)\leq \nuni^{\vare t}\noI_{\cone,\scmf}(e,z)
\ee
where $\trct\leq \nuni^{0.01\vare t}$, $\nuni^{-\sqrt\vare t}\leq \scmf\leq \nuni^{-\sqrt\vare t/2}$, 
and $\alpha=1-\sqrt\vare$, see~\S\ref{sec: MF and IG}. 
\end{enumerate}

Let $2\sqrt\vare t\leq \tau\leq 0.01\ref{k:mixing} t$. Then 
\[
\bigg|\int_0^1 \int\varphi(a_\tau u_rz)\diff\!\mu_{\cone}(z)\diff\!r-\int \varphi\diff\! m_X\biggr|\ll \Sob(\varphi)\nuni^{-\vare^2 t}
\]
for all $\varphi\in C_c^\infty(X)$.
\end{propos}

The proof, which is based on 
Proposition~\ref{prop: 1-epsilon N} and Theorem~\ref{thm: proj thm}, or more precisely Theorem~\ref{thm: proj thm 2}, will be completed in several steps.

Let us first
recall from \S\ref{sec: cone and mu cone} that a probability measure $\mu_\cone$ on 
$\cone$ is said to be $(\adl, \adm)$-admissible if 
\[
\mu_\cone=\frac1{\sum_{w\in F}\mu_w(X)}\sum_{w\in F}\mu_w
\]
where for every $w\in F$, $\mu_w$ is a measure on $\coneH.\exp(w)y$ satisfying that  
\be\label{eq: adm meas again 2}
\diff\!\mu_w(\sfh\exp(w)y)=\adl\ddensity_w(\sfh)\diff\!m_H(\sfh)\quad\text{where $1/\adm\leq \ddensity_w(\bigcdot)\leq \adm$;}
\ee
moreover, there is a subset $\coneH_w=\bigcup_{p=1}^{\adm}\coneH_{w,p}\subset \coneH$
so that 
\begin{enumerate}
\item $\mu_w\Bigl((\coneH\setminus \coneH_w).\exp(w)y\Bigr)\leq \adm\beta \mu_w(\coneH.\exp(w)y)$,
\item The complexity of $\coneH_{w,p}$ is bounded by $\adm$ for all $p$, and 
\item $\Lip(\ddensity_w|_{\coneH_{w,p}})\leq \adm$ for all $p$.
\end{enumerate}

\subsection{Localizing the set $F$}\label{sec: local F}
Recall that $F\subset B_\rfrak(0,\beta)$, and the set
\[
\cone=\coneH.\{\exp(w)y: w\in F\}
\] 
is equipped with a $(\adl, \adm)$-admissible measure $\mu_\cone$. In order to use Proposition~\ref{prop: 1-epsilon N}, we need to {\em move} $F$ to the direction of $\Lie(V)\subset\rfrak$, while controlling the errors in other directions. To facilitate this, we cover $F$ with subsets contained in cubes of size $\asymp\mfsc\,\inj(y)$ ---  {\em localized} Margulis functions were considered in the improving the dimension phase, precisely for this reason. 

Let $\bar\eta>0$ be so that $\bar\eta/2\leq \inj(z)\leq 2\bar\eta$ for all $z\in\cone$, and that $\bar\eta\mfsc$ is a dyadic number.
For every $v\in B_{\rfrak}(0,\beta)$, let $Q(v)$ be a cube with center $v$ and size $4\bar\eta\mfsc$. Fix a covering $\{Q(v_i): v_i\in F\}$ of $F$ with multiplicity bounded by $K$ (absolute). 

Since $\#\{Q(v_i): v_i\in F\}\ll (\bar\eta\mfsc)^{-3}$,~\eqref{eq: psi almos cst equi prop} implies that for all $i$ and $j$,
\begin{subequations}
\begin{align}
\label{eq: Qvi and Qvj}&\nuni^{-\vare t}\cdot(\#(Q(v_j)\cap F))\leq \#(Q(v_i)\cap F)\leq \nuni^{\vare t}\cdot(\#(Q(v_j)\cap F))\;\;\text{and}\\
\label{eq: no low density balls}&\#(Q(v_i)\cap F)\geq (\bar\eta\mfsc)^{4}\cdot(\#F)
\end{align}
\end{subequations}
where we used $e^{-\sqrt\vare t}\leq\mfsc\leq e^{-\sqrt\vare t/2}$ and $\bar\eta\geq e^{-0.001\vare t}$, and assumed $t$ is large to account for implied multiplicative constants.

For every $i$, define $\density_i:Q(v_i)\to\{1/j:j=1,\ldots, K\}$ by
\[
\density_i(w)=(\#\{Q(v_j): w\in Q(v_j)\})^{-1};
\]
we extend $\density_i$ to $\rfrak$ by defining it to be zero outside $Q(v_i)$.
 
For every $i$, let $\cone_i=\coneH.\{\exp(w)y:w\in Q(v_i)\}$. Let 
\[
\diff\!\mu_{\cone_i}(\sfh\exp(w)y)=\density_i(w)\diff\!\mu_{\cone}(\sfh\exp(w)y).
\] 
Then $\mu_\cone=\sum_i\mu_{\cone_i}$.

\subsection{A decomposition of the integral}\label{sec: decomp integral}
Recall that $\tau\geq 2\sqrt\vare t$. Let $\ell_2=|\log128\bar\eta\mfsc|$ (then $\sqrt\vare t/2\leq \ell_2\leq \sqrt\vare t+\vare t$) and let $\ell_1=\tau-\ell_2$. 
Let $0<\delta\leq 1$, and let $\varphi\in C_c^\infty(X)$. Then  
\begin{multline}\label{eq: int decomp 1}
\int_0^1\int\varphi(a_\tau u_rz)\diff\!\mu_{\cone}(z)\diff\!r=\\
\delta^{-1}\int_{0}^\delta\int_{0}^1\int \varphi(a_{\ell_1}u_{r_1}a_{\ell_2}u_{r_2}z)\diff\!\mu_{\cone}(z)\diff\!r_2\diff\!r_1
+O(e^{-\ell_2}\Lip(\varphi))
\end{multline}
where the implied constant depends on $X$. Note that in the integral above $r_1$ runs over $[0,\delta]$ and $r_2$ over $[0,1]$.

Thus we will investigate the first term on the right side of~\eqref{eq: int decomp 1}. 
Using the decomposition $\mu_\cone=\sum\mu_{\cone_i}$ and Fubini's theorem we have 
\begin{multline}\label{eq: int decomp 2}
\delta^{-1}\int_{0}^\delta\int_{0}^1\int \varphi(a_{\ell_1}u_{r_1}a_{\ell_2}u_{r_2}z)\diff\!\mu_{\cone}(z)\diff\!r_2\diff\!r_1=\\
\delta^{-1}\int_{0}^\delta\int_{0}^1\sum_i\int \varphi(a_{\ell_1}u_{r_1}a_{\ell_2}u_{r_2}z)\diff\!\mu_{\cone_i}(z)\diff\!r_2\diff\!r_1=\\
\sum_i\delta^{-1}\int_{0}^\delta\int_{0}^1\int \varphi(a_{\ell_1}u_{r_1}a_{\ell_2}u_{r_2}z)\diff\!\mu_{\cone_i}(z)\diff\!r_2\diff\!r_1.
\end{multline}

The following lemma will complete the proof of Proposition~\ref{prop: high dim to equid}. 

\begin{lemma}\label{lem: equidist mu cone-i}
Fix some $i$, and let $\bar\mu_{\cone_i}=\frac{1}{\mu_{\cone_i}(\cone_i)}\mu_{\cone_i}$, i.e., 
the probability measure proportional to $\mu_{\cone_i}$. Then 
\[
\biggl|\delta^{-1}\int_{0}^\delta\int_{0}^1\int \varphi(a_{\ell_1}u_{r_1}a_{\ell_2}u_{r_2}z)\diff\!\bar\mu_{\cone_i}(z)\diff\!r_2\diff\!r_1-\int\varphi\diff\!m_X\biggr|\ll\nuni^{-\vare^2t}\Sob(\varphi).
\]
\end{lemma}

\begin{proof}
Recall that $\cone_i=\coneH.\{\exp(w)y:w\in Q(v_i)\}$. Let $z_i=\exp(v_i)y$. It will be more convenient to replace $y$ in the definition of $\cone_i$ by $z_i$: Note that 
\be\label{eq: chage y to zi}
\begin{aligned}
\sfh\exp(w)y&=\sfh\exp(w)\exp(-v_i)\exp(v_i)y\\
&=\sfh\sfh_w\exp(v_w)z_i
\end{aligned}
\ee
where $\|\sfh_w-I\|\ll\mfsc\beta$ and $\frac12\|w-v_i\|\leq \|v_w\|\leq 2\|w-v_i\|$, see Lemma~\ref{lem: BCH}.

Note also that the map $w\mapsto v_w$ is one-to-one. Let $F_i=\{v_w: w\in Q(v_i)\}$ 
and let $\cconeH=\overline{\coneH\setminus\partial_{20\mfsc}\coneH}$. Put 
\[
\ccone_i:=\cconeH.\{\exp(v)z_i: v\in F_i\}.
\]
Then by~\eqref{eq: chage y to zi} and since $\|\sfh_w-I\|\ll\mfsc\beta$, we have $\ccone_i\subset\cone_i$; moreover, $\bar\mu_{\cone_i}(\cone_i\setminus\ccone_i)\ll \mfsc$. Thus it suffices to show the claim in the lemma with $\bar\mu_{\cone_i}$ replaced by $\hat\mu_i:=\frac{1}{\bar\mu_{\cone_i}(\ccone_i)}\bar\mu_{\cone_i}|_{\ccone_i}$. 

For later reference, let us also record that~\eqref{eq: chage y to zi} and  $\|\sfh_w-I\|\ll\mfsc\beta$ implies also that in fact
\be\label{eq: ccone in cone'}
\ccone_i\subset \cone':=\coneH'.\{\exp(w)y: w\in F\}
\ee
where $\coneH'=\overline{\coneH\setminus\partial_{10\mfsc}\coneH}$. In particular,~\eqref{eq: dim 1 equi prop} holds true for all $z\in\ccone_i$.

Recall that $\hat\mu_i$ is the probability measure proportional to 
$\sum_w\hat\mu_{i,w}$ where $\diff\!\hat\mu_{i,w}=\hat{\density}_{i,w}\diff\! m_H$ and 
$(K\adm)^{-1}\leq\hat{\density}_{i,w}\leq \adm$. We will use Fubini's theorem to change the order of disintegration of $\hat\mu_i$ as follows. 
Let $z\in\ccone_i$, then  
\[
z=\sfh\exp(v)z_i=\exp(\Ad(\sfh)v)\sfh z_i\in\ccone_i.
\] 
Moreover, $\Ad(\sfh)v\in B_\rfrak(0,8\bar\eta\mfsc)$. Since $\bar\eta/2\leq \inj(z')\leq 2\bar\eta$ for every $z'\in\cone_i$, we conclude that 
\[
\Ad(\sfh)v\in\margI_{\cone_i,32\mfsc}(e,\sfh z_i).
\] 
Let $\pi:\ccone_i\to\coneH.z_i$ denote the projection $z=\sfh\exp(v)z_i\mapsto \sfh z_i$.
Using Fubini's theorem, we have  
\[
\hat\mu_i=\int\hat\mu_i^\sfh\diff\!\pi_*\hat\mu_i(\sfh.z_i),
\]
where $\hat\mu_{i}^\sfh$ denotes the conditional measure of $\hat\mu_i$ for the factor map $\pi$. Note that $\hat\mu_{i}^\sfh$ is supported on $\{\exp(w)\sfh z_i:w\in\margI_{\cone_i,32\mfsc}(e,\sfh z_i)\}$.  
In view of the above discussion, $\diff\!\pi_*\hat\mu_i$ is proportional to $\hat{\density}\diff\! m_H$ restricted to the support of $\pi_*\hat\mu_i$ where $1\ll\hat{\density}\ll 1$, moreover, 
for every $i$, and every $w\in\supp(\hat\mu_i^\sfh)$,  
\be\label{eq: weight of atoms of cmui}
\hat\mu_i^\sfh(w)\asymp (\#F_i)^{-1}
\ee  
where the implied constant depends on $K$ and $\adm$. 

Now, using Fubini's theorem we have 
\begin{multline*}
\delta^{-1}\int_{0}^\delta\!\int_{0}^1\!\int \varphi(a_{\ell_1}u_{r_1}a_{\ell_2}u_{r_2}z)\diff\!\hat\mu_{i}(z)\diff\!r_2\diff\!r_1=\\
\delta^{-1}\int_{\cconeH.z_i}\!\int_{0}^\delta\!\int_{0}^1\!\int \varphi(a_{\ell_1}u_{r_1}a_{\ell_2}u_{r_2}\exp(w)\sfh z_i)\diff\!\hat\mu_i^\sfh(w)\diff\!r_2\diff\!r_1\diff\!\pi_*\cmu_i(\sfh.z_i).
\end{multline*}

Fix some $i$. We will investigate
\be\label{eq: fix i and h reduces to this}
\delta^{-1}\!\int_{\cconeH.z_i}\!\int_0^\delta\!\int_0^1\!\!\int\!\varphi(a_{\ell_1}u_{r_1}a_{\ell_2}u_{r_2}\exp(w)\sfh z_i)\diff\!\hat\mu_i^\sfh(w)\diff\!r_2\diff\!r_1\diff\!\pi_*\cmu_i(\sfh.z_i).
\ee

\subsection*{Discretized dimension of $\cmu_i^\sfh$}
Let us put
\[
F_{i}^\sfh:=\supp(\cmu_i^\sfh)=\{\Ad(\sfh)v: v\in F_i\}.
\] 
Moreover, recall from~\eqref{eq: ccone in cone'} that $\exp(\Ad(\sfh)v)\sfh z_i=\sfh\exp(v)z_i\in\ccone_i\subset\cone'$.
Since $\|v\|\leq 4\bar\eta\mfsc$, for every $v\in F_i$, we conclude that 
\be\label{eq: supp in margi 8b}
F_{i}^\sfh\subset\margI_{\cone',32\mfsc}(e,\sfh z_i). 
\ee
Furthermore, by~\eqref{eq: no low density balls} and since $\#F\geq e^{0.9t}$, we have 
\be\label{eq: supp has many elements}
\#F_{i}^\sfh=\#F_i=\#(Q(v_i)\cap F)\geq (\bar\eta\mfsc)^{4}\cdot(\#F)\geq e^{0.8t}.
\ee 

Recall now that 
\be\label{eq: dim 1 equi prop'}
\mfm_{\cone, \scmf,\trct}(e,z')\leq \nuni^{\vare t}\noI_{\cone,\scmf}(e,z')\leq \nuni^{\vare t}\sup_{z''\in\cone}\noI_{\cone,\scmf}(e,z'')
\ee
for all $z'\in\cone'$, where we used~\eqref{eq: dim 1 equi prop} to get the first bound.

Apply Lemma~\ref{lem: margIz energy est} with $\egbd=\nuni^{\vare t}\sup_{z'\in\cone}\noI_{\cone,\scmf}(e,z')$, $z=\sfh z_i$, 
and $\margI_{\sfh z_i}=\margI_{\cone',32\mfsc}(e,\sfh z_i)$.
We thus conclude that
\be\label{eq: margIz energy est use equi}
\eng_{\margI_{\sfh z_i}, \trct}(w)\ll\mfbd\qquad\text{for every $w\in\margI_{\sfh z_i}$.}
\ee

Moreover, by~\eqref{eq: Qvi and Qvj} and Lemma~\ref{lem: rfrak is invariant}, we have 
\begin{align*}
\#F_{i}^\sfh=\#F_i=\#(Q(v_i)\cap F)&\gg \nuni^{-\vare t}\sup_{z'}\#\margI_{\cone,\mfsc}(e,z')\\
&=\nuni^{-\vare t} \sup_{z'}\Bigl((\inj(z')\mfsc)^{\alpha}\noI_{\cone,\mfsc}(e,z')\Bigr)\\
&\gg \nuni^{-2\vare t}  (\bar\eta \mfsc)^{\alpha}\egbd,
\end{align*}
where we also used the definition of $\egbd$ in the last inequality.

Recall that $\trct\leq \nuni^{0.01\vare t}$. Therefore,~\eqref{eq: supp in margi 8b},~\eqref{eq: supp has many elements}, and~\eqref{eq: margIz energy est use equi}, in view of the above, imply that 
\[
\eng_{F_{i}^\sfh, \trct}(w)\ll\mfbd\ll \nuni^{2\vare t}(\bar\eta\mfsc)^{-\alpha}\cdot (\#F_{i}^\sfh)\qquad\text{for every $w\in F_{i}^\sfh$}.
\]
Using $\trct\leq \nuni^{0.01\vare t}$ and~\eqref{eq: supp has many elements} again, we conclude that 
\[
\sigma_i^\sfh(B_\rfrak(w,\mfsc'))\ll \nuni^{2\vare t}(\mfsc'/\bar\eta\mfsc)^{\alpha}\quad\text{for all $\mfsc'\geq(\#F_{i}^\sfh)^{-1}$},
\]
where $\sigma_i^\sfh$ is the uniform measure on $F_{i}^\sfh$. 
This and~\eqref{eq: weight of atoms of cmui} imply that
\be\label{eq: uniform meas on F-i-h is regular}
\cmu_i^\sfh(B_\rfrak(w,\mfsc'))\ll\nuni^{2\vare t}(\mfsc'/\bar\eta\mfsc)^{\alpha}\quad\text{for all $\mfsc'\geq(\#F_{i}^\sfh)^{-1}$,}
\ee
where the implied constant depends only on $\adm$ and $K$.

\subsection*{Projecting the dimension} 
Recall that $0<\ref{k:mixing}\leq 1$, we have 
\[
2\sqrt\vare t\leq \tau\leq 0.01\ref{k:mixing} t\leq 0.01t.
\] 

For every $r\in[0,1]$ and $w\in B_\rfrak(0,128\bar\eta\mfsc)$, write 
\be\label{eq: Iwasawa decomposition Ad(ur)w}
\exp(\Ad(u_r)w)=\begin{pmatrix}d_{r,w} &0\\ c_{r,w} & 1/d_{r,w}\end{pmatrix}\begin{pmatrix}1 &\xi_r(w)\\ 0& 1\end{pmatrix} u_{\hat r}
\ee
where $|d_{r,w}-1|, |c_{r,w}|\ll \nuni^{-\ell_2}$, $\xi_r(w)\in\rfrak^+$, and $\hat r=\hat r(w,r)$ satisfies $|\hat r|\ll e^{-2\ell_2}$. Note that 
$\hat r=0$ if $G=\SL_2(\R)\times \SL_2(\R)$. 

In view of~\eqref{eq: uniform meas on F-i-h is regular}, we may apply Theorem~\ref{thm: proj thm 2} with $F_i^\sfh$, $\mfsc_1=\nuni^{-\ell_2}=128 \bar\eta\mfsc$, $\mfsc_0=(\#F_i^\sfh)^{-1}$, $\cmu_i^\sfh$, $\vare$, and 
\[
\mfsc'=\nuni^{-3\ell_1-\ell_2}\geq \nuni^{-4\tau}\geq \nuni^{-0.04 t}\geq (\#F_i^\sfh)^{-1},
\] 
where we used $\tau\leq 0.01 t$ and~\eqref{eq: supp has many elements}.

Let $J_{\mfsc'}\subset[0,1]$ and  
$\Theta_{\mfsc',r_2}\subset F_i^\sfh$ (for every ${r_2}\in J_{\mfsc'}$) be as in Theorem~\ref{thm: proj thm 2}. Set $J^\sfh:=J_{\mfsc'}$.
Let $\bar\mu_{i,r_2}^\sfh$ denote the projection of $\cmu^\sfh_i|_{\Theta_{\mfsc',r_2}}$
under the map $w\mapsto\xi_{r_2}(w)$. 
Then, by Theorem~\ref{thm: proj thm 2}, we have 
\be\label{eq: regularity of rho i r2}
\bar\mu_{i,r_2}^{\sfh}(I)\leq L\vare^{-L}\nuni^{2\vare n}(\mfsc'/\bar\eta\scmf)^{\alpha-7\vare}
\ee
for every interval $I$ of length $\mfsc'$ where $L$ is absolute.

Moreover,  
$\left|[0,1]\setminus J^\sfh\right|\leq L\vare^{-L}\mfsc'^{\vare}$ which is $\leq L\vare^{-L} \nuni^{-\vare^{3/2} t}$ since $\mfsc'<e^{-2\sqrt\vare t}$.
Thus by Fubini's theorem there exists some $J\subset [0,1]$ with $\left|[0,1]\setminus J\right|\ll\vare^{-2L} \nuni^{-\vare^{8/5} t}$ 
and for every $r_2\in J$, a subset $\hat{\mathsf E}(r_2)\subset \hat{\mathsf E}$ with $|\hat{\mathsf E}\setminus\hat{\mathsf E}(r_2)|\ll\vare^{-2L} \nuni^{-\vare^{8/5} t}$ for that the~\eqref{eq: regularity of rho i r2} holds with $r_2$ and any $\sfh\in\hat{\mathsf E}(r_2)$.

This and the fact that for all $w$, the Jacobian of the map $r_2\to \hat r(w,r_2)$ is $1+O(e^{-2\ell_2})$ imply that for any $r_1 \in [0,\delta]$
\begin{multline}\label{eq: conv dmax d0 s 1}
\int_{\cconeH.z_i}\!\int_0^1\!\int\varphi(a_{\ell_1}u_{r_1}a_{\ell_2}u_{r_2}\exp(w)\sfh z_{i})\diff\!\cmu_i^\sfh(w)\diff\!r_2\diff\!\pi_*\cmu_i(\sfh.z_i)=\\
\int_{J}\int_{\cconeH.z_i}\!\!\int\varphi\Bigl(a_{\ell_1}u_{r_1}a_{\ell_2}f(w,r_2)u_{r_2}\sfh z_i\Bigr)\diff\!\cmu_i^\sfh(w)\diff\!r_2\diff\!\pi_*\cmu_i(\sfh.z_i)+ \\O\Bigl(\Sob(\varphi)L\vare^{-2L}\nuni^{-\vare^{8/5} t}\Bigr),
\end{multline}
where $f(w,r_2)=\begin{pmatrix}d_{r,w} &0\\ c_{r,w} & 1/d_{r,w}\end{pmatrix}\begin{pmatrix}1 &\xi_r(w)\\ 0& 1\end{pmatrix}$.

\subsection*{Approximating orbits using the projection $\xi_{r_2}$}
In view of~\eqref{eq: conv dmax d0 s 1}, we need to investigate the contribution of the first term on the right side of~\eqref{eq: conv dmax d0 s 1} to~\eqref{eq: fix i and h reduces to this}.  
We begin by fixing the size of $\delta$ and some algebraic considerations. 

Recall that $\sqrt\vare t/2\leq \ell_2\leq \sqrt\vare t+\vare t$ and $\ell_1=\tau-\ell_2\geq \sqrt\vare t-\vare t$. Define $0<\delta\leq 1$ by the following equation 
\be\label{eq:choose delta}
\nuni^{\ell_1}\delta = \nuni^{\sqrt\vare t/4}\leq\nuni^{\ell_2/2}.
\ee

For any $r_2\in[0,1]$, put $z_{i,r_2}^\sfh=a_{\ell_2}u_{r_2}\sfh z_{i}$. 
Using~\eqref{eq: Iwasawa decomposition Ad(ur)w} and~\eqref{eq: supp in margi 8b}, for any $w\in F_i^\sfh$ and all $r_1\in[0,\delta]$, we have 
\begin{multline*}
a_{\ell_1}u_{r_1}a_{\ell_2}f(w,r_2)a_{-\ell_2}z^\sfh_{i,r_2}
=\\a_{\ell_1}u_{r_1}\begin{pmatrix}d_{r_2,w} &0\\ \nuni^{-\ell_2}c_{r_2,w} & 1/d_{r_2,w}\end{pmatrix}\begin{pmatrix}1 &\nuni^{\ell_2}\xi_{r_2}(w)\\ 0& 1\end{pmatrix}z^\sfh_{i,r_2}
\end{multline*}
where $|c_{r_2,w}|, |d_{r_2,w}-1|\ll \nuni^{-\ell_2}$. From this, we conclude that  
\be\label{eq: applying a ell-1 ur-1}
a_{\ell_1}u_{r_1}\exp(\Ad(a_{\ell_2}u_{r_2})w)z^\sfh_{i,r_2}= g a_{\ell_1}u_{r_1}\begin{pmatrix}1 &\nuni^{\ell_2}\xi_{r_2}(w)\\ 0 & 1\end{pmatrix}z^\sfh_{i,r_2}
\ee
where $\|g-I\|\ll \nuni^{\ell_1}\delta\nuni^{-\ell_2}\ll \nuni^{-\ell_2/2}\leq \nuni^{-\sqrt\vare t/4}$, see~\eqref{eq:choose delta}.

\subsection*{Applying Proposition~\ref{prop: 1-epsilon N}}
Fix $r_2 \in J$ and $\sfh\in\cconeH(r_2)$. 
Let $\hat\mu_{i,r_2}^\sfh$ denote the image of $\bar\mu_{i,r_2}^\sfh$ under the map $s\mapsto\nuni^{\ell_2}s$.
In view of~\eqref{eq: applying a ell-1 ur-1} and the fact that     
$\cmu^\sfh_i(F_i^\sfh\setminus \Theta_{\mfsc',r_2})\leq L\vare^{-L}\nuni^{-\vare^{3/2} t}$ we have 
\begin{multline*}
\delta^{-1}\int_{0}^\delta \!\int \varphi\Bigl(a_{\ell_1}u_{r_1}a_{\ell_2}f(w_2)a_{-\ell_2}z_{i,r_2}^\sfh\Bigr)\diff\!\cmu_i^\sfh(w)\diff\!r_1=\\
\delta^{-1}\int_{0}^\delta \!\int \varphi\Bigl(a_{\ell_1}u_{r_1}v_sz_{i,r_2}^\sfh\Bigr)\diff\!\hat\mu_{i,r_2}^\sfh(s)\diff\!r_1+ O\Bigl(\Sob(\varphi) L\vare^{-2L}\nuni^{-\vare^{8/5} t}\Bigr).
\end{multline*}

Recall that $\alpha=1-\sqrt\vare$. By~\eqref{eq: regularity of rho i r2}, the measure $\hat\mu_{i,r_2}^\sfh$ satisfies the condition~\eqref{eq:C-rho-reg-N} 
in Proposition~\ref{prop: 1-epsilon N} for
\[
\theta=\sqrt\vare+7\vare, \qquad \bpz=\nuni^{-3\ell_1}, \quad\text{and}\quad  C=L\vare^{-L}\nuni^{2\vare t}.
\]
Apply Proposition~\ref{prop: 1-epsilon N} for $t=\ell_1$, and the above chosen $\delta$; note that $|\log\bpz|/4\leq t=\ell_1\leq |\log\bpz|/2$ so that in particular~\eqref{eq: comapring bpz and eta} holds. Then as $\bpz^{1/2}\leq \nuni^{-\ell_1}$ the first term in the right hand side of \eqref{eq: in Prop 5.2} dominates and   
\begin{multline}\label{eq: tau small almost final}
\biggl|\delta^{-1}\int_{0}^\delta\!\int \varphi\Bigl(a_{\ell_1}u_{r_1}v_sz_{i,r_2}^\sfh\Bigr)\diff\!\hat\mu_{i,r_2}^\sfh(s)\diff\!r_1-\int\varphi\diff\!m_X\biggr|\\\ll \Sob(\varphi)\Bigl(L\vare^{-L}\nuni^{2\vare t}\nuni^{3(\sqrt\vare+7\vare)\ell_1}\Bigr)^{1/2} (\nuni^{\ell_1}\delta)^{-\ref{k:mixing}}.
\end{multline}

Recall that $\ell_1\leq \tau\leq 0.01\ref{k:mixing} t$. Therefore,
\[
\nuni^{3\sqrt\vare\ell_1}\leq \nuni^{0.03\ref{k:mixing}\sqrt\vare t}.
\]
Moreover, $\ell_1\leq \tau\leq 0.01t$, hence $21\vare\ell_1\leq \vare t$, and using~\eqref{eq: choose theta equi sec} we get 
\[
3\vare=3(\sqrt\vare)^2\leq 0.01\ref{k:mixing}\sqrt\vare.
\] 
Thus, $\nuni^{2\vare t}\cdot \nuni^{21\vare\ell_1}\leq \nuni^{3\vare t}\leq \nuni^{0.01\ref{k:mixing}\sqrt\vare t}$. 
Altogether, we conclude that  
\[
\nuni^{2\vare t}\nuni^{3(\sqrt\vare+7\vare)\ell_1}\leq \nuni^{0.04\ref{k:mixing}\sqrt\vare t}.
\] 
Since $e^{\ell_1}\delta=\nuni^{\sqrt\vare t/4}$. The above implies that the right side of~\eqref{eq: tau small almost final} is 
\[
\ll\Sob(\varphi) L\vare^{-L}\nuni^{-\ref{k:mixing}\sqrt \vare t/5}\ll \Sob(\varphi) L\vare^{-L}\nuni^{-\vare t}
\] 
where in the second inequality is a consequence of~\eqref{eq: choose theta equi sec}.

Choosing $t$ large enough so that $L\vare^{-2L}\nuni^{-\vare^{8/5} t}\leq \nuni^{-\vare^2 t}$, we conclude that 
\begin{multline*}
\delta^{-1}\!\int_{\cconeH.z_i}\!\int_0^\delta\!\int_0^1\!\!\int\!\varphi(a_{\ell_1}u_{r_1}a_{\ell_2}u_{r_2}\exp(w)\sfh z_i)\diff\!\hat\mu_i^\sfh(w)\diff\!r_2\diff\!r_1\diff\!\pi_*\cmu_i(\sfh.z_i)=\\
\int \varphi\diff\!m_X+ O(\Sob(\varphi)\nuni^{-\vare^2 t}).
\end{multline*}
The proof is complete. 
\end{proof}

\begin{proof}[Proof of Proposition~\ref{prop: high dim to equid}]
In view of~\eqref{eq: int decomp 1} and~\eqref{eq: int decomp 2}, 
the proposition follows from Lemma~\ref{lem: equidist mu cone-i}.
\end{proof}

\section{Proof of Theorem~\ref{thm:main}}\label{sec: proof of main thm}

The proof will be completed in some steps 
and it is based on various propositions which were discussed so far.

\subsection*{Fixing the parameters}
Fix $\vare$ as follows 
\be\label{eq: choose theta equi sec final}
0<\sqrt\vare < 10^{-8}\ref{k:mixing}
\ee
where $\ref{k:mixing}$ is as in Proposition~\ref{prop: 1-epsilon N}.

Let $D=D_0D_1+2D_1$ where $D_0$ is as in Proposition~\ref{prop: linearization translates} and $D_1$ is as in Proposition~\ref{prop:closing lemma intro}; we will always assume $D_1, D_0\geq 10$. We will show the claim holds with
\[
A=15+2D_0.
\]

Let us assume (as we may) that  
\be\label{eq: how big is R xi}
R\geq \max\{(10\ref{E:non-div-main})^3\inj(x_0)^{-2}, e^{\ref{E:non-div-main}}, e^{s_0}, \ref{c: linear trans}\},
\ee
see Proposition~\ref{prop:non-div} and Proposition~\ref{prop: linearization translates}. Let $T\geq R^A$, and
suppose that Theorem~\ref{thm:main}(2) does not hold with this $A$. 
That is, for every $x\in X$ so that $Hx$ is periodic with $\vol(Hx)\leq R$, 
\be\label{eq: main thm 2 fails}
\dist_X(x,x_0)> R^{A}(\log T)^AT^{-1}\geq (\log S)^{D_0}S^{-1}
\ee
where $S:=R^{-A}T$. 

Since $D_0, D_1\geq 10$, we have 
\[
A=15+2D_0\geq 10+(10+2D)D_1^{-1}\geq 10+(\tfrac{5}{2}\sqrt\vare+9 +\tfrac{4D-3}{2})D_1^{-1}. 
\]
Therefore, 
\be\label{eq: computation for t0}
\begin{aligned}
\log T-\Bigl((\tfrac{5}{2}\sqrt\vare+9 +\tfrac{4D-3}{2})D_1^{-1}\Bigr)\log R\geq \log T-A\log R+10\log R&\\
\geq\log S+2|\log\inj(x_0)|+8\log R&
\end{aligned}
\ee
we used $R\geq \inj(x_0)^{-2}$ and $\log S=\log T-A\log R$ in the last inequality.  

Let $t=\frac{1}{D_1}\log R$, $\ell = \vare t/100$, and $d_1=100\lceil\frac{4D-3}{2\vare}\rceil$. Then 
\be\label{eq: range of d1ell}
\tfrac{4D-3}{2}t\leq d_1\ell\leq \tfrac{4D-3}{2}t+\vare t.
\ee
As it was done in~\eqref{eq: choose kappa}, fix 
\[
0<\kappa < \min\{10^{-6}d_1^{-1}, 10^{-6}\vare\}.
\]
Let $\beta=e^{-\kappa t}$ and let $\eta=\beta^{1/2}$; note that $\eta\geq e^{-0.1\ell}$.

Let us write $\log T=t_3+t_2+t_1+t_0$ where 
\be\label{eq: def t0 t1 t2}
\begin{aligned}
&t_0=\log T-((\tfrac{5}{2}\sqrt\vare+9 +\tfrac{4D-3}{2})D_1^{-1})\log R\\
&t_1=8t,\;\text{ and }\; t_2=t+d_1\ell.
\end{aligned}
\ee
Note that $t_0, t_1, t_2\geq t$ (see~\eqref{eq: computation for t0} for $t_0>t$). We now estimate $t_3$; indeed
\begin{align*}
t_3&= \log T-(t_0+t_1+t_2)\\
&=(\tfrac{5}{2}\sqrt\vare + 9 +\tfrac{4D-3}{2})D_1^{-1}\log R-9t-d_1\ell\\
&=(\tfrac{5}{2}\sqrt\vare + 9 +\tfrac{4D-3}{2})t-9t-d_1\ell
\end{align*}
where we used $t=\frac{1}{D_1}\log R$ in the last equation. 
This and~\eqref{eq: range of d1ell} imply 
\be\label{eq: range of t3}
2\sqrt\vare t\leq t_3\leq 3\sqrt\vare t.
\ee

Recall that $a_{\ell_1}u_ra_{\ell_2}=a_{\ell_1+\ell_2}u_{e^{-\ell_2}r}$. Thus, for any $\varphi\in C_c^\infty(X)$, we have
\begin{multline}\label{eq: proof integral 1}
\int_0^1\varphi(a_{\log T} u_rx_0)\diff\!r=\quad O(\|\varphi\|_\infty e^{-t})\quad +\\
\int_0^1\int_0^1\int_0^1\int_0^1\varphi(a_{t_3} u_{r_3}a_{t_2}u_{r_2}a_{t_1}u_{r_1}a_{t_0}u_{r_0}x_0)\diff\!r_3\diff\!r_2\diff\!r_1\diff\!r_0 
\end{multline}
where the implied constant is absolute and we used $t_0, t_1, t_2\geq t$.

\subsection*{Improving the Diophantine condition}\label{sec:improving diophantine}
Apply Proposition~\ref{prop: linearization translates} with $S=R^{-A}T$, then for all
\[
\tau\geq \max\{\log S,2|\log\inj(x_0)|\}+s_0,
\] 
we have the following  
\be\label{eq: apply linearization proof}
\biggl|\biggl\{r\in [0,1]: \begin{array}{c}\text{$a_{\tau}u_rx_0\not\in X_\eta$ or $\exists\, x$ with $\vol(Hx)\leq R$} \\ 
\text{so that }\dist_X(x, a_{\tau}u_rx_0)\leq R^{-D_0-1}\end{array}\biggr\}\biggr|\ll\eta^{1/2},
\ee
where we also used $\eta^{1/2}\geq R^{-1}$ and $R\geq \ref{c: linear trans}$. 

Let $J_0\subset[0,1]$ be the set of those $r_0\in[0,1]$ so that $a_{t_0}u_{r_0}x_0\in X_{\eta}$ and 
\[
\dist_X(x, a_{t_0}u_{r_0}x_0)> R^{-D_0-1}=e^{-D_1(D_0+1)t}
\] 
for all $x$ with $\vol(Hx)\leq R=e^{D_1t}$. 
Then since by~\eqref{eq: computation for t0} and~\eqref{eq: how big is R xi} we have  
\[
t_0\geq \log S+2|\log\inj(x_0)|+8\log R\geq \max\{\log S,2|\log\inj(x_0)|\}+s_0,
\] 
the assertion in~\eqref{eq: apply linearization proof} implies that $|[0,1]\setminus J_0|\ll\eta^{1/2}$. In consequence, 
\begin{multline}\label{eq: proof integral 2}
\int_{0}^1\varphi(a_{\log T} u_rx_0)\diff\!r=\quad O(\|\varphi\|_\infty \eta^{1/2})\quad +\\
\int_{J_0}\int_0^1\int_0^1\int_0^1\varphi(a_{t_3} u_{r_3}a_{t_2}u_{r_2}a_{t_1}u_{r_1}x(r_0))\diff\!r_3\diff\!r_2\diff\!r_1\diff\!r_0 
\end{multline}
where $x(r_0)=a_{t_0}u_{r_0}x_0$ and the implied constant depends on $X$.

\subsection*{Applying the closing lemma}
For every $r_0\in J_0$, we now apply Proposition~\ref{prop:closing lemma intro} 
with $x(r_0)$, $D=D_0D_1+2D_1$ and the parameter $t$. 
For any such $r_0$, we have
\[
\dist_X(x, x(r_0))> e^{-D_1(D_0+1)t} =e^{(-D+D_1)t}
\] 
for all $x$ with $\vol(Hx)\leq e^{D_1t}$. Thus Proposition~\ref{prop:closing lemma intro}(1) holds.
Let 
\[
J_1(r_0)=I(x(r_0))=I(a_{t_0}u_{r_0}x_0)
\] 
Then 
\begin{multline}\label{eq: proof integral 3}
\int_0^1\varphi(a_{\log T} u_rx_0)\diff\!r=\quad O(\|\varphi\|_\infty \eta^{1/2})\quad +\\
\int_{J_0}\int_{J_1(r_0)}\int_0^1\int_0^1\varphi(a_{t_3} u_{r_3}a_{t_2}u_{r_2}x(r_0,r_1))\diff\!r_3\diff\!r_2\diff\!r_1\diff\!r_0 
\end{multline}
where $x(r_0, r_1)=a_{t_1}u_{r_1}a_{t_0}u_{r_0}x_0$ and the implied constant is absolute.

\subsection*{Improving the dimension phase}
Fix some $r_0\in J_0$, and let $r_1\in J(r_0)$. Put $x_1=x(r_0, r_1)$. 
Recall from~\eqref{eq: def mu tn...t1} that 
\[
\mu_{t,\ell, d_1}=\rwm_{\ell}\conv\cdots\conv\nu_{\ell}\conv\sigma\conv\rwm_{t}
\]
where $\rwm_\ell$ appears $d_1$ times in the above expression. 
In view of Lemma~\ref{lem: thickening stable}, 
\begin{multline}\label{eq: proof integral 4}
\biggl|\int_0^1\!\!\int_0^1\varphi(a_{t_3} u_{r_3}a_{t_2}u_{r_2}x_1)\diff\!r_3\diff\!r_2-\int\!\!\int_0^1\!\! \varphi(a_{t_3} u_{r_3}hx_1)\diff\!r_3\mu_{t,\ell, d_1}(h)\biggr|\\
\ll \Lip(\varphi)e^{-\ell}\ll \Lip(\varphi)\eta^{1/2}.
\end{multline}

We now apply Proposition~\ref{propos: imp dim main} with $x_1$, $t_3$ and $r_3\in[0,1]$. Then 
 \begin{multline}\label{eq: proof integral 5}
\int_0^1\!\!\int \varphi(a_{t_3} u_{r_3}hx_1)\diff\!\mu_{t,\ell, d_1}(h)\diff\!r_3=\\
\sum_{d,i}c_{d,i}\int_0^1\!\!\int \varphi(a_{t_3} u_{r_3}z)\diff\!\rwm_\ell^{(d_{1}-d)}\conv\mu_{\cone_{d,i}}(z)\diff\!r_3 + O(\Lip(\varphi)\beta^{\ref{k: bootstrap beta exp}})
\end{multline}
where the sum is over 
\[
d_{1}-\lceil{10^4}\vare^{-1/2}\rceil=d_2\leq d\leq  d_{1},
\]
$c_{d,i}\geq0$ and $\sum_{d,i}c_{d,i}=1-O(\beta^{\ref{k: bootstrap beta exp}})$ and the implied constants depend on $X$.
Moreover, for all $d,i$ both of the following hold 
\begin{subequations}
\begin{align}
\label{eq: proof final regularity}&\#\Bigl(B_\rfrak(w,4\mfsc\,\inj(y))\cap F_{d,i}\Bigr)\geq \nuni^{-\vare t}\!\!\sup_{w'\in F_{d,i}}\#\Bigl(B_\rfrak(w',4\mfsc\,\inj(y))\cap F_{d,i}\Bigr)\\
\label{eq: proof final energy est}&\mfm_{\cone_{d,i}, \mfsc, \trct}(e,z)\leq \nuni^{\vare t}\noI_{\cone_{d,i},\mfsc}(e,z)\quad\text{where $\trct\leq \nuni^{0.01\vare t}$}
\end{align}
\end{subequations}
for all $w\in F_{d,i}$ and all $z=h\exp(w)y_{d,i}\in\cone_{d,i}$ with $h\in\overline{\coneH\setminus\partial_{10\mfsc}\coneH}$.

\subsection*{From large dimension to equidistribution}
For every $d_2\leq d\leq d_1$, set
\[
\tau_d:=t_3+(d_1-d)\ell.
\]
Since $0\leq d_1-d\leq \lceil{10^4}\vare^{-1/2}\rceil$, $\ell=0.01\vare t$, and $2\sqrt\vare t\leq t_3\leq 3\sqrt\vare t$, see~\eqref{eq: range of t3},  
\be\label{eq: range of tau-d}
2\sqrt\vare t\leq\tau_d\leq (4+10^2)\sqrt\vare t\leq 0.01\ref{k:mixing}t
\ee
where in the last inequality we used $0<\sqrt\vare < 10^{-8}\ref{k:mixing}$, see~\eqref{eq: choose theta equi sec final}.

In view of Lemma~\ref{lem: thickening stable}, for all $d,i$ as above, we have 
\begin{multline}\label{eq: proof integral 6}
\int_0^1\!\!\int \varphi(a_{t_3} u_{r_3}z)\diff\!\rwm_\ell^{(d_{1}-d)}\conv\mu_{\cone_{d,i}}(z)\diff\!r_3=\\
\int_0^1\!\!\int \varphi(a_{\tau_d} u_rz)\diff\!\mu_{\cone_{d,i}}(z)\diff\!r+O(\Lip(\varphi)e^{-\ell})
\end{multline}
where the implied constant depends on $X$. 

We now apply Proposition~\ref{prop: high dim to equid} with $\cone_{d,i}$ (in view of~\eqref{eq: proof final regularity} and~\eqref{eq: proof final energy est} the conditions in that proposition are satisfied) and $\tau_d$ which is in the admissible range thanks to~\eqref{eq: range of tau-d}. Hence, for all $d,i$ as above, we have 
\be\label{eq: proof integral 7}
\bigg|\int_0^1 \int\varphi(a_{\tau_d} u_rz)\diff\!\mu_{\cone_{d,i}}(z)\diff\!r-\int \varphi\diff\! m_X\biggr|\ll \Sob(\varphi)\nuni^{-\vare^2 t}
\ee
where the implied constant depends on $X$. 

Let $\ref{k:main-1}=\min\{\vare^2, \ref{k: bootstrap beta exp}\kappa, \kappa/4\}$. Then 
~\eqref{eq: proof integral 7},~\eqref{eq: proof integral 6},~\eqref{eq: proof integral 5},~\eqref{eq: proof integral 4},~\eqref{eq: proof integral 3}, \eqref{eq: proof integral 2}, and~\eqref{eq: proof integral 1}, imply that 
\[
\bigg|\int_0^1\varphi(a_{\log T} u_rx_0)\diff\!r-\int \varphi\diff\! m_X\biggr|\ll \Sob(\varphi)\nuni^{-\ref{k:main-1} t}\ll \Sob(\varphi)R^{-\ref{k:main-1} /D_1}
\]
where the implied constant depends on $X$. The proof is complete.
\qed


\section{Proof of Theorem~\ref{thm: main arithmeticity relaxed}}\label{sec: proof arith relaxed}

The argument is similar to the proof of Theorem~\ref{thm:main}, the main difference here is that even though Proposition~\ref{prop: linearization translates} holds without the arithmeticity assumption on $\Gamma$, its output, i.e., points which are not near periodic $H$-orbits, is too weak for our closing lemma, in the absence of arithmeticity. Indeed the assertion (2') in~\S\ref{sec: closing lemma} only guarantees that if Proposition~\ref{prop:closing lemma intro}(1) fails, then we can find a nearby point $x$ whose stabilizer contains a non-elementary Fuchsian subgroup which is generated by {\em small} elements; 
without the arithmeticity assumption on $\Gamma$, however, the orbit $Hx$ need not be periodic, see e.g.,~\cite[\S12]{BO-GF}, in contrast to what happens in the arithmetic case (cf.\ Lemma~\ref{lem:non-elementary}). Therefore, the proof of Theorem~\ref{thm: main arithmeticity relaxed} will not include the improving Diophantine condition step which was present in the proof of Theorem~\ref{thm:main} (see p.~\pageref{sec:improving diophantine}). To remedy this issue, we will choose the parameter $D$ in the proof to be $O(1/\delta)$; this is responsible for the error rate $T^{-\delta^2\ref{k:main-1}}$ in Theorem~\ref{thm: main arithmeticity relaxed}(1). Let us now turn to the details.   
\subsection*{Fixing the parameters}
Fix $\vare$ as follows 
\be\label{eq: choose theta equi sec final 1}
0<\sqrt\vare < 10^{-8}\ref{k:mixing}
\ee
where $\ref{k:mixing}$ is as in Proposition~\ref{prop: 1-epsilon N}.

Let $0<\delta<1/4$ be as in the statement of Theorem~\ref{thm: main arithmeticity relaxed}, and let $D_1$ be as in Proposition~\ref{prop:closing lemma intro}.
Put $t=\frac{\delta}{D_1}\log T$, and define $D$ by 
\be\label{eq: define D thm 1.3}
\tfrac{4D-3}{2}+ 9 +\tfrac52\sqrt\vare=D_1/\delta
\ee
Since $\delta<1/4$, we have $D\geq 2D_1$. Let 
\be\label{eq: choose A thm 1.3}
A'=\Bigl(\tfrac{4D-3}{2}+ 9 +\tfrac52\sqrt\vare\Bigr)/(D-D_1);
\ee
note that $A'\ll1$ where the implied constant is absolute. 

We assume $T$ is large enough so that 
\[
e^t> (10\ref{E:non-div-main})^3\inj(x_0)^{-2}.
\]
Suppose that Theorem~\ref{thm: main arithmeticity relaxed}(2) fails for this choice of $A'$. That is for all $x\in X$ such that ${\rm Stab}_H(x)$ contains elements $\gamma_1$ and $\gamma_2$ 
so that
\begin{itemize}
    \item $\|\gamma_1\|,\|\gamma_2\|\leq T^{\delta}$  
    \item $\langle\gamma_1,\gamma_2\rangle$ is Zariski dense in $H$
\end{itemize} 
we have    
\be\label{eq: thm 1.3 2 fails}
\dist_X(x,x_0)> T^{-1/A'}=e^{(-D+D_1)t}.
\ee
We will show that Theorem~\ref{thm: main arithmeticity relaxed}(1) holds.

Put $\ell = \vare t/100$, and $d_1=100\lceil\frac{4D-3}{2\vare}\rceil$. Then 
\be\label{eq: range of d1ell 1}
\tfrac{4D-3}{2}t\leq d_1\ell\leq \tfrac{4D-3}{2}t+\vare t.
\ee
We define the parameter $\kappa$ as follows:  
\be\label{eq: def kappa proof thm 1.3}
\kappa=\tfrac12\min\{10^{-6}d_1^{-1}, 10^{-6}\vare\},
\ee
and let $\beta=e^{-\kappa t}$ and let $\eta=\beta^{1/2}$; note that $\eta\geq e^{-0.1\ell}$ and that $\kappa\asymp\delta$.

Let us write $\log T=t_3+t_2+t_1$ where 
\be\label{eq: def t0 t1 t2 1}
t_1=8t\quad\text{ and }\quad t_2=t+d_1\ell.
\ee
Note that $t_1, t_2\geq t$. We now estimate $t_3$; indeed
\begin{align*}
t_3&= \log T-(t_1+t_2)\\
&=tD_1/\delta-9t-d_1\ell\\
&=(\tfrac{4D-3}{2}+9+\tfrac{5}{2}\sqrt\vare)t-9t-d_1\ell
\end{align*}
where we used $tD_1/\delta=\log T$ in the second equation and~\eqref{eq: define D thm 1.3} in the last equation. 
This and~\eqref{eq: range of d1ell 1} imply 
\be\label{eq: range of t3 1}
2\sqrt\vare t\leq t_3\leq 3\sqrt\vare t.
\ee

Recall that $a_{\ell_1}u_ra_{\ell_2}=a_{\ell_1+\ell_2}u_{e^{-\ell_2}r}$. Thus, for any $\varphi\in C_c^\infty(X)$, we have
\begin{multline}\label{eq: proof integral 1 1}
\int_0^1\varphi(a_{\log T} u_rx_0)\diff\!r=\quad O(\|\varphi\|_\infty e^{-t})\quad +\\
\int_0^1\int_0^1\int_0^1\varphi(a_{t_3} u_{r_3}a_{t_2}u_{r_2}a_{t_1}u_{r_1}x_0)\diff\!r_3\diff\!r_2\diff\!r_1\diff\!r_0 
\end{multline}
where the implied constant is absolute and we used $t_1, t_2\geq t$.

The rest of the argument follows,  mutatis mutandis, the same steps as in the proof of Theorem~\ref{thm:main}, as we now explicate. 

\subsection*{Applying the closing lemma}
We now apply Proposition~\ref{prop:closing lemma intro} 
with $x_0$, $D$ as in~\eqref{eq: define D thm 1.3} and the parameter $t$ (which is assumed to be large). 
In view of~\eqref{eq: thm 1.3 2 fails}, Proposition~\ref{prop:closing lemma intro}(1) holds.
Let $J_1=I(x_0)$. Then 
\begin{multline}\label{eq: proof integral 3 1}
\int_0^1\varphi(a_{\log T} u_rx_0)\diff\!r=\quad O(\|\varphi\|_\infty \eta^{1/2})\quad +\\
\int_{J_1}\int_0^1\int_0^1\varphi(a_{t_3} u_{r_3}a_{t_2}u_{r_2}x(r_1))\diff\!r_3\diff\!r_2\diff\!r_1
\end{multline}
where $x(r_1)=a_{t_1}u_{r_1}x_0$ and the implied constant is absolute.

\subsection*{Improving the dimension phase}
Fix some $r_1\in J_1$, and put $x_1=x(r_1)$. 
Recall from~\eqref{eq: def mu tn...t1} that 
\[
\mu_{t,\ell, d_1}=\rwm_{\ell}\conv\cdots\conv\nu_{\ell}\conv\sigma\conv\rwm_{t}
\]
where $\rwm_\ell$ appears $d_1$ times in the above expression. 
In view of Lemma~\ref{lem: thickening stable}, 
\begin{multline}\label{eq: proof integral 4 1}
\biggl|\int_0^1\!\!\int_0^1\varphi(a_{t_3} u_{r_3}a_{t_2}u_{r_2}x_1)\diff\!r_3\diff\!r_2-\int\!\!\int_0^1\!\! \varphi(a_{t_3} u_{r_3}hx_1)\diff\!r_3\mu_{t,\ell, d_1}(h)\biggr|\\
\ll \Lip(\varphi)e^{-\ell}\ll \Lip(\varphi)\eta^{1/2}.
\end{multline}

We now apply Proposition~\ref{propos: imp dim main} with $x_1$, $t_3$ and $r_3\in[0,1]$. Then 
 \begin{multline}\label{eq: proof integral 5 1}
\int_0^1\!\!\int \varphi(a_{t_3} u_{r_3}hx_1)\diff\!\mu_{t,\ell, d_1}(h)\diff\!r_3=\\
\sum_{d,i}c_{d,i}\int_0^1\!\!\int \varphi(a_{t_3} u_{r_3}z)\diff\!\rwm_\ell^{(d_{1}-d)}\conv\mu_{\cone_{d,i}}(z)\diff\!r_3 + O(\Lip(\varphi)\beta^{\ref{k: bootstrap beta exp}})
\end{multline}
where the sum is over 
\[
d_{1}-\lceil{10^4}\vare^{-1/2}\rceil=d_2\leq d\leq  d_{1},
\]
$c_{d,i}\geq0$ and $\sum_{d,i}c_{d,i}=1-O(\beta^{\ref{k: bootstrap beta exp}})$ and the implied constants depend on $X$.
Moreover, for all $d,i$ both of the following hold 
\begin{subequations}
\begin{align}
\label{eq: proof final regularity 1}&\#\Bigl(B_\rfrak(w,4\mfsc\,\inj(y))\cap F_{d,i}\Bigr)\geq \nuni^{-\vare t}\!\!\sup_{w'\in F_{d,i}}\#\Bigl(B_\rfrak(w',4\mfsc\,\inj(y))\cap F_{d,i}\Bigr)\\
\label{eq: proof final energy est 1}&\mfm_{\cone_{d,i}, \mfsc, \trct}(e,z)\leq \nuni^{\vare t}\noI_{\cone_{d,i},\mfsc}(e,z)\quad\text{where $\trct\leq \nuni^{0.01\vare t}$}
\end{align}
\end{subequations}
for all $w\in F_{d,i}$ and all $z=h\exp(w)y_{d,i}\in\cone_{d,i}$ with $h\in\overline{\coneH\setminus\partial_{10\mfsc}\coneH}$.

\subsection*{From large dimension to equidistribution}
For every $d_2\leq d\leq d_1$, set
\[
\tau_d:=t_3+d_1-d.
\]
Since $0\leq d_1-d\leq \lceil{10^4}\vare^{-1/2}\rceil$, $\ell=0.01\vare t$, and $2\sqrt\vare t\leq t_3\leq 3\sqrt\vare t$, see~\eqref{eq: range of t3 1}, 
\be\label{eq: range of tau-d 1}
2\sqrt\vare t\leq\tau_d\leq (4+10^2)\sqrt\vare t\leq 0.01\ref{k:mixing}t
\ee
where in the last inequality we used $0<\sqrt\vare < 10^{-8}\ref{k:mixing}$, see~\eqref{eq: choose theta equi sec final 1}.

In view of Lemma~\ref{lem: thickening stable}, for all $d,i$ as above, we have 
\begin{multline}\label{eq: proof integral 6 1}
\int_0^1\!\!\int \varphi(a_{t_3} u_{r_3}z)\diff\!\rwm_\ell^{(d_{1}-d)}\conv\mu_{\cone_{d,i}}(z)\diff\!r_3=\\
\int_0^1\!\!\int \varphi(a_{\tau_d} u_rz)\diff\!\mu_{\cone_{d,i}}(z)\diff\!r+O(\Lip(\varphi)e^{-\ell})
\end{multline}
where the implied constant depends on $X$. 

We now apply Proposition~\ref{prop: high dim to equid} with $\cone_{d,i}$. In view of~\eqref{eq: proof final regularity 1} and~\eqref{eq: proof final energy est 1} the conditions in that proposition are satisfied, and $\tau_d$ is in the admissible range thanks to~\eqref{eq: range of tau-d 1}. Hence, for all $d,i$ as above, we have 
\be\label{eq: proof integral 7 1}
\bigg|\int_0^1 \int\varphi(a_{\tau_d} u_rz)\diff\!\mu_{\cone_{d,i}}(z)\diff\!r-\int \varphi\diff\! m_X\biggr|\ll \Sob(\varphi)\nuni^{-\vare^2 t}
\ee
where the implied constant depends on $X$. 

Let $\hat\kappa=\min\{\vare^2, \ref{k: bootstrap beta exp}\kappa, \kappa/4\}$. Then 
~\eqref{eq: proof integral 7 1},~\eqref{eq: proof integral 6 1},~\eqref{eq: proof integral 5 1},~\eqref{eq: proof integral 4 1},~\eqref{eq: proof integral 3 1}, and~\eqref{eq: proof integral 1 1}, imply that 
\[
\bigg|\int_0^1\varphi(a_{\log T} u_rx_0)\diff\!r-\int \varphi\diff\! m_X\biggr|\ll \Sob(\varphi)e^{-\hat\kappa t}=\Sob(\varphi)T^{-\delta\hat\kappa/D_1}
\]
where the implied constant depends on $X$. 

In view of the definition of $\hat\kappa$ and~\eqref{eq: def kappa proof thm 1.3}, we have $\hat\kappa\gg \delta$ where the implied constant depends only on $X$. The proof is complete.
\qed

\section{Proof of Theorem \ref{thm:main unipotent}} \label{sec: proof main unip}

The proof is based on Theorem~\ref{thm:main} and the following lemma, which is a special case of \cite[Thm.~1.4]{LMMS} tailored to our application here.  

\begin{lemma}\label{lem: unipotent linearization}
There exist $A_3$, $D_3$, and $\constE\label{E: unip lin}$ (depending on $X$) 
so that the following holds. Let $S,M>0$, and $0<\eta<1/2$ satisfy
\[
S\geq M^{A_3}\quad\text{and}\quad M\geq \ref{E: unip lin}\eta^{-A_3}.
\]
Let $x_1\in X_\eta$, and suppose there exists $\exceptional\subset \{r\in [-S,S]: u_rx_1\in X_\eta\}$
with 
\[
|\exceptional|>\ref{E: unip lin}\eta^{1/D_3}S
\]
so that for every $r\in\exceptional$, 
there exists $y_r\in X$ with 
\[
\text{$\vol(H.y_r)\leq M\quad$ and $\quad\dist(u_{r}x_1, y_r)\leq M^{-A_3}$}.
\]
Then one of the following holds  
\begin{enumerate}
    \item There exists $x \in G/\Gamma$ with $\vol(H.x)\leq M^{A_3}$, 
and for every $r\in [-S,S]$ there exists $g\in G$ with $\|g\|\leq M^{A_3}$ so that  
\[
\dist_X(u_{s}x_1, gH.x)\leq M^{A_3}\left(\frac{|s-r|}{S}\right)^{1/D_3}\quad\text{for all $s\in[-S,S]$.}
\] 
\item For every $r\in[-S,S]$ and $t \in [\log M, \log S]$, the injectivity radius at $a_{-t}u_r x_1$ is at most $M^{A_3}e^{-t}$.
\end{enumerate}
\end{lemma}

The lemma will be proved using \cite[Thm.~1.4]{LMMS} or more precisely \cite[Cor.~7.2]{LMMS}. 
The statements in~\cite{LMMS} use a slightly different language than the one 
we used in this paper, thus we begin by recalling some terminology 
to relate Lemma~\ref{lem: unipotent linearization} to \cite[Thm.~1.4]{LMMS}.  

\subsection*{Arithmetic groups}
Let ${\bf G}=\SL_2\times \SL_2$ if $G=\SL_2(\lf)\times\SL_2(\lf)$, and 
${\bf G}={\rm Res}_{\qlf/\lf}(\SL_2)$ if $G=\SL_2(\qlf)$.
Then ${\bf G}$ is defined over $\lf$ and $G={\bf G}(\lf)$; moreover, $H=\H(\bbr)$ where 
$\H\subset\G$ is an algebraic subgroup.

Recall that $\Gamma$ is assumed to be arithmetic. 
Therefore, there exists a semisimple simply connected $\bbq$-group $\tilde\G\subset\SL_N$, for some $N$, and an epimorphism 
\[
\rho:\tilde\G(\bbr)\to \G(\bbr)=G
\]
of $\bbr$-groups with compact kernel so that $\Gamma$ is commensurable with $\rho(\tilde\G(\bbz))$.  Note that $\tilde\G$ can be chosen to be $\bbq$-almost simple unless $\Gamma\subset\SL_2(\bbr)\times\SL_2(\bbr)$ is a reducible lattice, in which case $\tilde\G$ can be chosen to have two $\bbq$-almost simple factors. We assume $\tilde\G$ is thus chosen. 

Moreover, since $\tilde\G$ is simply connected, we can identify $\tilde G(\bbr)$ with $G\times G'$ where $G'=\ker(\rho)$ is compact.

We are allowed to choose the parameter $M$ in the lemma to be large depending on $\Gamma$, therefore, by passing to a finite index subgroup, we will assume that both of the following hold: 
\begin{itemize}
    \item $\Gamma\subset \tilde\Gamma:=\rho(\tilde\G(\bbz))$, where $\tilde\G(\bbz)=\tilde\G(\bbr)\cap\SL_N(\bbz)$, and    
    \item if $\Gamma\subset\SL_2(\bbr)\times\SL_2(\bbr)$ is reducible, then $\Gamma=\Gamma_1\times\Gamma_2$.
\end{itemize}

With this notation, every $\gamma\in \Gamma$ lifts uniquely to $(\gamma,\sigma(\gamma))\in\tilde\Gamma$, where $\sigma$ is (a collection of) Galois automorphisms.  For every $g\in G$, we put 
\[
\hat g=(g,1)\in G\times G'.
\]

Suppose now that $g\in G$ is so that $Hg\Gamma$ is periodic. 
Let $\Delta_g=\Gamma\cap g^{-1}Hg$, and let $\tilde\Delta_g=\rho^{-1}(\Delta_g)\cap\tilde \Gamma$. Let $\tilde\H_g$ 
be the Zariski closure of $\tilde\Delta_g$. 
Then $\tilde\H_g$ is a semisimple $\bbq$-subgroup, 
and the restriction of $\rho$ to $\tilde\H_g$ surjects onto $g^{-1}\H g$. 
Let $\tilde H_g=\tilde\H_g(\bbr)$, then  
\[
\overline{\hat g^{-1}\hat H\hat g\tilde\Gamma}=\tilde H_g\tilde\Gamma
\]

\subsection*{Lie algebras and the adjoint representation}
We continue to write $\Lie(G)=\gfrak$ and $\Lie(H)=\hfrak$; 
these are considered as $6$-dimensional (resp.\ $3$-dimensional) $\lf$-vector spaces. 

Let $v_H$ be a unit vector on the line $\wedge^3\hfrak$. Note that 
\[
N_G(H)=\{g\in G: gv_H=v_H\}
\] 
which contains $H$ as a subgroup of index two. 

Let $\tilde\gfrak=\Lie(\tilde\G(\bbr))$, this Lie algebra has a natural $\bbq$-structure. 
Moreover, $\tilde\gfrak_\bbz:=\tilde\gfrak\cap\sl_N(\bbz)$ is a $\tilde\G(\bbz)$-stable lattice in $\tilde\gfrak$.

If there exists $g\in G$ so that $Hg\Gamma$ is periodic, fix 
${\mathsf g}_1,\ldots,{\mathsf g}_m$ so that $\vol(H{\mathsf g}_i\Gamma)\ll 1$ 
(the implied constant and $m$ depend on $\Gamma$) 
and that every $\tilde H_g$ is conjugate to some $\tilde H_i=\tilde H_{{\mathsf g}_i}$ in $\tilde G$. 
Let ${\bf v}_{i}$ be a primitive integral vector on the line 
\[
\wedge^{\dim \tilde H_i}(\Lie(\tilde H_i))\subset\wedge^{\dim \tilde H_i}\tilde\gfrak.
\]
Then $N_{\tilde G}(\tilde H_i)=\{g\in \tilde G: gv_{i}=v_i\}$,
and $\tilde H_i\subset N_{\tilde G}(\tilde H_i)$ has finite index. 
For all $i$, ${\bf v}_i=c_i \cdot \Bigl((g_i^{-1}v_H)\wedge v'_i\Bigr)$ where $v'_i\in\wedge\Lie(G')$ and $|c_i|\asymp 1$.

More generally, if ${\bf L}\subset \tilde\G$ is a $\bbq$-algebraic group, 
we let ${\bf v}_L$ be a primitive integral vector on the line 
$\wedge^{\dim L}\Lie(L)\subset\wedge^{\dim L}\tilde\gfrak$ where $L=\Lbf(\bbr)$.

\subsection*{Volume and height of periodic orbits}
Let ${\bf L}\subset \tilde\G$ be a $\bbq$-algebraic group.
Recall the definition of the height of ${\bf L}$ from~\cite{LMMS}
\[
{\rm ht}(\Lbf)=\|{\bf v}_L\|.
\]  

Recall that $\tilde G=G\times G'$. 
We fix a right invariant metric on $\tilde G$ defined using 
the killing form and the maximal compact subgroup $\tilde K=K\times G'$ 
where $K=\SO(2)\times \SO(2)$ if $G=\SL_2(\bbr)\times\SL_2(\bbr)$ and $K={\rm SU}(2)$ if $G=\SL_2(\bbc)$; 
this metric induces the right invariant metric on $G$ which we fixed on p.~\pageref{d definition page}.

\begin{lemma}\label{lem: volume and height}
Let $Hg\Gamma$ be a periodic orbit, and let $\tilde\H_g$ be as above.  
Both of the following properties hold:
\[
{\rm ht}(\tilde\H_g)^\star\ll\vol(\tilde H_g\tilde\Gamma/\tilde\Gamma)\ll {\rm ht}(\tilde\H_g)^\star
\]
\[
\|g\|^{-\star}\vol(Hg\Gamma)\ll\vol(\tilde H_g\tilde\Gamma/\tilde\Gamma)\ll\|g\|^\star\vol(Hg\Gamma)
\]
\end{lemma}

\begin{proof}
For the first claim see~\cite[\S17]{EMV} or~\cite[App.~B]{EMMV} 
(for the upper bound, see also~\cite[\S2]{ELMV-1}, which treats the case of tori but the proof there works for the semisimple case as well).

To see the second claim, note that
$\tilde H_g\tilde\Gamma$ projects onto $g^{-1}Hg\Gamma$ and the fiber is compact which volume $\asymp 1$. 
Therefore, 
\[
\vol(\tilde H_g\tilde\Gamma)\asymp \vol(g^{-1}Hg\Gamma). 
\] 
Moreover, left multiplication by $g$ changes the volume by $\|g\|^\star$. 

The claim follows.   
\end{proof}

\begin{proof}[Proof of Lemma~\ref{lem: unipotent linearization}]
In view of our assumption in the lemma, periodic $H$ orbits exists. Let $\tilde H_1,\ldots, \tilde H_m$ be as above.  
Let $A_3$ and $D_3$ be large constants which will be explicated later, in particular, 
we will let $A_3>\max(A, D_2)$, $D_3>D$ and $\ref{E: unip lin}> \max\{mE_1, \ref{c: red th}\}$ 
where $A$, $D$, and $E_1$ are as in \cite[Thm.~1.4]{LMMS} applied with 
$\{\hat u_r\}\subset\tilde G$, and $D_2$ and $\ref{c: red th}$ are as in Lemma~\ref{lem: reduction theory}. 

We first interpret the condition in the lemma as a condition about the action of 
$\{\hat u_r\}$ on $\tilde G/\tilde \Gamma$. Let us write $x_1=g_1\Gamma$, where 
$\|g_1\|\leq\ref{c: red th} \eta^{-D_2}\leq M$, see Lemma~\ref{lem: reduction theory} and our assumption in this lemma. 
Similarly, for every $r\in\exceptional$, let us write $y_r=g(r)\Gamma$ where 
$\|g(r)\|\leq M$ and for every such $r$, there exists $\gamma_r\in\Gamma$ so that 
\be\label{eq: unip lin exceptial gamma r}
\|u_rg_1\gamma_r\|\leq M+1\quad\text{and}\quad  u_rg_1\gamma_r=\epsilon(r)g(r), 
\ee
where $\|\epsilon(r)\|\ll M^{-A_3}$. 

For every $1\leq i\leq m$, let 
\[
\exceptional_i=\{r\in\exceptional: \text{$\tilde H^r:=\tilde H_{g(r)}$ is a conjugate of $\tilde H_i$}\}.
\]
Then, there exists some $i$ so that $|\exceptional_i|\geq |\exceptional|/m$. 
Replacing $\exceptional$ by $\exceptional_i$, we assume that $\tilde H^r$ is a conjugate of $\tilde H_i$ 
for all $r\in\exceptional$. Let us write $\tilde H^r=\tilde g(r)^{-1}\tilde H_i\tilde g(r)$. Then 
\[
\tilde g(r)=({\mathsf g}_i^{-1}g(r), \tilde g'(r))\in G\times G',
\]
and ${\bf v}^r:=\frac{\|{\bf v}_{\tilde H^r}\|}{\|\tilde g(r)^{-1}{\bf v}_i\|}\tilde g(r)^{-1}{\bf v}_i=\pm {\bf v}_{\tilde H^r}$. 
Moreover, we have  
\be\label{eq: vr and the norm of gr}
{\bf v}^r=c_r\cdot\Bigl((g(r)^{-1}v_H)\wedge (\tilde g'(r)^{-1}v'_i)\Bigr)\quad\text{where $|c_r|\ll M{\rm ht}(\tilde\H_g)\ll M^\star$}
\ee
where we used Lemma~\ref{lem: volume and height} to conclude $M{\rm ht}(\tilde\H_g)\ll M^\star$.

Recall that $\hat g=(g,1)$ for all $g\in G$. In view of~\eqref{eq: unip lin exceptial gamma r}, we have 
\be\label{eq: unip lin exceptional gamma r 1}
\hat u_r\hat g_1(\gamma_r,\sigma(\gamma_r)).{\bf v}^r=
c_r\cdot \biggl((\epsilon(r) v_H)\wedge \Bigl((\sigma(\gamma_r)\tilde g'(r)^{-1})v'_i\Bigr)\biggr).
\ee
Since $G'$ is compact, we conclude from~\eqref{eq: unip lin exceptional gamma r 1} that 
\be\label{eq: unip lin exceptional gamma r 2}
\|\hat u_r\hat g_1(\gamma_r,\sigma(\gamma_r)).{\bf v}^r\|\leq M^{A'_3},
\ee
for some $A'_3$.

Let $z\in\gfrak$ be a vector so that $u_r=\exp(rz)$. 
Using~\eqref{eq: unip lin exceptional gamma r 1} and associativity of the exterior algebra, we have 
\begin{align}
\notag\|z\wedge \Bigl(\hat u_r\hat g_1(\gamma_r,\sigma(\gamma_r)).{\bf v}^r\Bigr)\|&=|c_r|\Bigl\|\Bigl(z\wedge \epsilon(r) v_H\Bigr)\wedge \Bigl((\sigma(\gamma_r)\tilde g'(r)^{-1})v'_i\Bigr)\Bigr\|\\
\label{eq: unip lin exceptional gamma r 3}&\ll M^\star M^{-A_3}<\eta^{A}M^{-AA'_3}/E_1. 
\end{align}
where we used $\|\epsilon(r)\|\ll M^{-A_3}$ in the second to last inequality, $A$ and $E_1$ are as in \cite[Thm.~1.4]{LMMS}, and we choose $A_3$ large enough so that the last estimate holds.

In view of~\eqref{eq: unip lin exceptional gamma r 2} and~\eqref{eq: unip lin exceptional gamma r 3}, 
conditions in \cite[Cor.~7.2]{LMMS} are satisfied. 
Hence, there exist $\tilde\gamma=(\gamma, \sigma(\gamma))\in \tilde\Gamma$, 
$r\in\exceptional$, and a subgroup 
\[
\tilde\H'\subset \tilde\gamma^{-1}\tilde\H^r\tilde\gamma\cap\tilde\H^r
\]
satisfying that $\tilde\H'(\bbc)$ is generated by unipotent subgroups (see~\cite[p.~3]{LMMS}) 
so that both of the following hold for all $r\in[-S,S]$
\begin{subequations}
\begin{align}
\label{eq: tube throughout 1}&\|u_rg_1{\bf v}_{\tilde H'}\|\ll M^\star\\
\label{eq: tube throughout 2}&\|z\wedge (u_rg_1{\bf v}_{\tilde H'})\|\ll S^{-1/D}M^\star.
\end{align}
\end{subequations}

Let $\tilde H'=\tilde\H'(\bbr)$. Since $\|g_1\|\leq M$, we conclude from~\eqref{eq: tube throughout 1}, applied with $r=0$, that 
\be\label{eq: height of tilde H'}
\|{\bf v}_{\tilde H'}\|\ll M^\star.
\ee
Let us consider two possibilities:

\subsection*{Case 1} $\rho(\tilde H')$ is a conjugate of $H$.  

First note that this implies 
\[
\rho(\tilde H')=g(r_0)^{-1}Hg(r_0)\quad \text{where $r_0\in\exceptional$ is as above.}
\] 
Let us write $g'=g(r_0)$. Then $\|g'\|\leq M$, and we have 
\be\label{eq: volume of Hg'Gamma}
\begin{aligned}
\vol(Hg'\Gamma/\Gamma)&\ll \|g'\|^\star\vol(g'^{-1}Hg'\Gamma/\Gamma)\\
&\ll M^\star{\rm ht}(\tilde\H')\ll M^\star
\end{aligned}
\ee
where we used Lemma~\ref{lem: volume and height} in the second 
and~\eqref{eq: height of tilde H'} in the last inequality.

Recall that $H$ is a symmetric subgroup of $G$, i.e., there exists an involution $\tau: G\to G$ so that $H$ 
is the connected component of the identity in ${\rm Fix}(\tau)$. 
In particular, $G=KA'H$ for an $\bbr$-diagonalizable subgroup $A'$.
For every $r\in[-S,S]$, let us write 
\[
u_rg_1=g'^{-1}k_rb_rg' g'^{-1}h_rg'\in g'^{-1}KA'g' g'^{-1}Hg',
\]
and put $g'_r=g'^{-1}k_rb_rg'$. 
Then~\eqref{eq: tube throughout 1} and~\eqref{eq: height of tilde H'} imply that 
\[
\|g'_r\|\ll \|g'_r{\bf v}_{\tilde H'}\|^\star\|{\bf v}_{\tilde H'}\|^\star\|g'\|^\star\ll\|u_rg_1{\bf v}_{\tilde H'}\|^\star M^\star\ll M^\star.
\]
Since the map $r\mapsto u_rg_1{\bf v}_{\tilde H'}$ is a polynomial map whose coefficients are $\ll M^\star$, we conclude that
\[
g'_s=\epsilon'(s,r)g'_r\quad\text{where $\|\epsilon'(s,r)\|\ll M^\star\Bigl({|s-r|}/{S}\Bigr)^{\star}$}.
\]   
Since $u_sg_1=g'_sg'^{-1}h_sg'$ and $d$ is right invariant, the above implies  
\[
\dist(u_sg_1, g'_rg'^{-1}Hg')\ll M^\star\Bigl({|s-r|}/{S}\Bigr)^{\star};
\]
hence part~(1) in the lemma holds if for every $r\in[-S,S]$ we let $g=g'_rg'^{-1}$.  

\subsection*{Case 2}
$\rho(\tilde H')=g'^{-1}Ug'$ where $U=\{u_r\}$. 

First note that if this holds, then $\tilde\G=\G$ (as $\bbr$-groups). Indeed in this case $\Gamma$ is a non-uniform arithmetic lattice, thus 
$\tilde\G=R_{k/\bbq}(\SL_2)$ for a quadratic extension $k/\bbq$ if $G=\SL_2(\bbc)$ or $G=\SL_2(\bbr)\times \SL_2(\bbr)$ and $\Gamma$ is irreducible. If $\Gamma=\Gamma_1\times\Gamma_2$ in $G=\SL_2(\bbr)\times \SL_2(\bbr)$, then since the projection of $g'^{-1}Ug'$ to both factors is a nontrivial unipotent subgroup, $\Gamma_1$ and $\Gamma_2$ are both non-uniform arithmetic lattices; hence, $\tilde\G=\SL_2\times\SL_2$. 

Moreover, note that in this case ${\bf v}_{\tilde H'}\in\Lie(G)$, and we have 
\[
\exp({\bf v}_{\tilde H'})\in \tilde H'\cap\Gamma.
\]
Let us consider the case of $G=\SL_2(\bbc)$, the computations in the other case is similar by considering each component. Put 
\[
g_1.{\bf v}_{\tilde H'}= \begin{pmatrix}a & b\\ c & -a\end{pmatrix}.
\]
Then~\eqref{eq: tube throughout 1} implies that for every $r\in[-S,S]$ we have 
\[
\biggl\|u_r\begin{pmatrix}a & b\\ c & -a\end{pmatrix}u_{-r}\biggr\|=\biggl\|\begin{pmatrix}a+cr & -cr^2-2ar+b\\ c & -a-cr\end{pmatrix}\biggr\|\ll M^\star
\]
Hence $|c|S^2\ll M^\star$ and $|a|S\ll M^\star$, which implies $|a+cr|\ll M^\star S^{-1}$.

Let now $t\in[\log M,\log S]$, then 
\[
\biggl\|a_{-t}u_r\begin{pmatrix}a & b\\ c & -a\end{pmatrix}u_{-r}a_t\biggr\|=\biggl\|\begin{pmatrix}a+cr & e^{-t}(-cr^2-2ar+b)\\ e^{t}c & -a-cr\end{pmatrix}\biggr\|\ll M^\star e^{-t}, 
\]
where we used $e^t|c|, |a+cr|\ll M^\star S^{-1}\leq M^\star e^{-t}$.  

Since $\exp({\bf v}_{\tilde H'})\in \tilde H'\cap\Gamma$, the above implies the claim in part~(2).
\end{proof}

\subsection{Proof of Theorem~\ref{thm:main unipotent}} 
Let $A$ be as Theorem~\ref{thm:main}, and let $A_3$, $D_3$ and $\ref{E: unip lin}$ 
be as in Lemma~\ref{lem: unipotent linearization}. 
Increasing $A_3$ and $D_3$ if necessary, we may assume $A_3,D_3\geq 10A$. 
We will show the theorem holds with 
\[
A_1=\star A_3\geq 4A_3\quad\text{and}\quad A_2=D_3 
\]

Let $C=\max\{(10\ref{E:non-div-main})^3, e^{\ref{E:non-div-main}}, e^{s_0}, \ref{c: linear trans}, \ref{E: unip lin}\}$, 
see~\eqref{eq: how big is R xi}. Let $R\geq C^{2}$, and put
\[
d=3A_3\log R \quad\text{and}\quad \eta=(C/R)^{1/A_3}.
\] 
Let $T>R^{A_1}$, and put $T_1=e^{-d}T\geq R^{A_3}$. Then 
\begin{multline}\label{eq: proof main unip 1}
\frac{1}{T}\int_0^T\varphi(u_rx_0)\diff\!r=\frac{1}{T_1}\int_0^{T_1}\varphi(a_d u_{r_1}a_{-d}x_0)\diff\!r_1\\
=\frac{1}{T_1}\int_0^{T_1}\!\!\int_0^1\varphi(a_d u_{r}u_{r_1}a_{-d}x_0)\diff\!r\diff\!r_1+O(\|\varphi\|_\infty T_1^{-1})
\end{multline}
where the implied constant is $\leq 2$. 

Put $x_1=a_{-d}x_0$, and define 
\begin{subequations}
\begin{align}
\label{eq: main unip exceptional cusp}\exceptional_1&=\{r_1\in[0,T_1]: u_{r_1}x_1\not\in X_{\eta}\}\\
\label{eq: main unip exceptional periodic}\exceptional_2&=\biggl\{r_1\in[0,T_1]: \begin{array}{c}\text{there exists $x$ with $\vol(Hx)\leq R$}\\\text{and }\dist(u_{r_1}x_1,x)\leq R^Ad^{A}e^{-d}\end{array}\biggr\}.
\end{align}
\end{subequations}

Let us first assume that 
\be\label{eq: exceptional sets are small}
|\exceptional_1|\leq C\eta^{1/2}T_1 \quad\text{and}\quad |\exceptional_2|\leq 2C^2R^{-\kappa}T_1,
\ee
where $\kappa=\min\{1/(2A_3), 1/(2D_3)\}$. 
 
For every 
\[
r_1\in[0,T_1]\setminus\Bigl(\exceptional_1\cup\exceptional_2\Bigr),
\]
put $x(r_1)=u_{r_1}x_1$. Then  
\[
R\geq C\eta^{-A_3}\geq C\inj(x(r_1))^{-2},
\]
see~\eqref{eq: how big is R xi}; moreover, $e^d>R^A$. 
Thus conditions of Theorem~\ref{thm:main} hold true with $e^d$, $R$, and $x(r_1)$. 
Moreover, in view of the definition of $\exceptional_2$, part~(2) in Theorem~\ref{thm:main} does not hold with these choices. 
Altogether, we conclude that for every $r_1$ as above, 
\[
\biggl|\int_0^1\varphi(a_d u_{r}x(r_1))\diff\!r-\int\varphi\diff\!m_X\biggr|\leq \Sob(\varphi)R^{-\ref{k:main-1}}
\]  
This,~\eqref{eq: exceptional sets are small} and~\eqref{eq: proof main unip 1} 
imply that 
\[
\biggl|\frac{1}{T}\int_0^T\varphi(u_rx_0)\diff\!r-\int\varphi\diff\!m_X\biggr|\leq (R^{-\ref{k:main-1}}+3C^2R^{-\kappa}+2T_1^{-1})\Sob(\varphi),
\]
where we used $C\eta^{1/2}\leq C^2R^{-\kappa}$. 

Hence, part~(1) in Theorem~\ref{thm:main unipotent} holds with 
$\ref{k:main uni}=\min(\ref{k:main-1},\kappa)/2$ if we assume $R$ is large enough.

\medskip

We now assume to the contrary that~\eqref{eq: exceptional sets are small} fails:

\subsection*{Assume that $|\exceptional_1|> C\eta^{1/2}T_1$.}  
We will show that part~(3) in the theorem holds under this condition; 
the argument is similar to Case 2 in Lemma~\ref{lem: unipotent linearization}. 

Let us write $x_0=g_0\Gamma$. Then
\begin{align*}
\{u_{r_1}x_1: r_1\in[0,T_1]\}&=\{a_{\log T_1}u_ra_{-d-\log T_1}x_0: r\in [0,1]\}\\
&=\{a_{\log T_1}u_ra_{-\log T}g_0\Gamma: r\in[0,1]\}.
\end{align*}
Our assumption $|\exceptional_1|> C\eta^{1/2}T_1$ and the change of variables thus imply 
\[
|\{r\in[0,1]: a_{\log T_1}u_ra_{-\log T}g_0\Gamma\not\in X_{\eta}\}|>\ref{E:non-div-main} \eta^{1/2},
\]
where we used $C\geq \ref{E:non-div-main}$, see Proposition~\ref{prop:Non-div-main} for $\ref{E:non-div-main}$.
 
This and Proposition~\ref{prop:Non-div-main}, applied with $a_{-\log T}g_0\Gamma$, the interval $[0,1]$, $\log T_1$, 
and $\vare=\eta$, imply that 
\[
\inj(a_{-\log T}g_0\Gamma)\ll T_1^{-1};
\] 
the implied constant depends on $X$. Hence, there is some $\gamma\in\Gamma$ so that 
\[
a_{-\log T}g_0\gamma g_0^{-1}a_{\log T}\in\boxG_{C'T_1^{-1}}
\]
where $C'$ depends on $X$. 
Assuming $R$ and hence $T_1$ is large enough, the above implies that $\gamma$ is a unipotent element. 
In particular, we have 
\[
a_{-\log T}g_0\gamma g_0^{-1}a_{\log T}=\exp\left(\begin{pmatrix}a & b\\ c & -a\end{pmatrix}\right)
\]
where $|a|, |b|, |c|\ll T_1^{-1}=e^{d}T^{-1}=R^{3A_3}T^{-1}$. Hence,
\[
g_0\gamma g_0^{-1}=\exp\left(\begin{pmatrix}a & Tb\\ T^{-1}c & -a\end{pmatrix}\right).
\]
Let $b'=Tb$ and $c'=c/T$. Then 
\[
|b'|\ll R^{3A_3}\quad\text{and}\quad |c'|\ll R^{3A_3}T^{-2},
\] 
which implies that $|a+c'r|\ll R^{3A_3} T^{-1}$ for every $r\in[0,T]$. 
Therefore, for every $r\in[0,T]$ and every $t\in[\log R,\log T]$ we have 
\[
a_{-t}u_rg_0\gamma g_0^{-1}u_{-r}a_{t}=\begin{pmatrix}a+c'r & e^{-t}(-c'r^2-2ar+b')\\ e^{t}c' & -a-c'r\end{pmatrix}.
\]
Note that $|a+c'r|\ll R^{3A_3} T^{-1}$, $e^{t}|c'|\ll R^{3A_3}T^{-1}$, and 
\[
e^{-t}|-c'r^2-2ar+b'|\ll R^{3A_3} e^{-t}.
\]
In consequence, part~(3) in the theorem holds with $A_1=3A_3+1$ if we assume $R$ is large enough. 

\subsection*{Assume that $|\exceptional_2|> 2C^2R^{-\kappa}T_1$}
If $|\exceptional_1|> C\eta^{1/2}T_1$, then part~(3) in the theorem holds as we just discussed. 
Thus, we may assume that 
\[
|\exceptional_2|> 2C^2R^{-\kappa}T_1\quad\text{and}\quad |\exceptional_1|\leq  C\eta^{1/2}T_1.
\] 
Put $\exceptional':=\exceptional_2\setminus \exceptional_1$. Then 
\[
\exceptional'=\biggl\{r_1\in[0,T_1]: \begin{array}{c}\text{$u_{r_1}x_1\in X_\eta$ and there exists $x$ with}\\ \text{$\vol(Hx)\leq R$ 
and $\dist(u_{r_1}x_1,x)\leq R^Ad^{A}e^{-d}$}\end{array}\biggr\},
\]
and $|\exceptional'|\geq C^2R^{-\kappa}T_1\geq \ref{E: unip lin}R^{-1/D_3}T_1$.
Moreover, assuming $R$ is large enough, we have 
\[
R^Ad^{A}e^{-d}=R^A(3A_3\log R)^AR^{-3A_3}\leq R^{-A_3}.
\]

Fix some $r_1\in\exceptional'$ for the rest of the argument. Put 
\[
\text{$x_2=u_{r_1}x_1=u_{r_0}a_{-d}x_0\;\;$ and $\;\;\exceptional=\exceptional'-r_1 \subset [-T_1,T_1]$}.
\]
Then the conditions in Lemma~\ref{lem: unipotent linearization} are satisfied with $x_2$,
$\exceptional$, $\eta$, $M=R$, and $S=T_1=R^{3A_3}T^{-1}$. 

Assume first that part~(2) in Lemma~\ref{lem: unipotent linearization} holds. Then 
there exists $x \in G/\Gamma$ with $\vol(H.x)\leq R^{A_3}$, 
and for every $r\in [-T_1,T_1]$ there exists $g\in G$ with $\|g\|\leq R^{A_3}$ so that  
\[
\dist_X(u_{s}x_2, gHx)\leq R^{A_3}\left(\frac{|s-r|}{T_1}\right)^{1/D_3}\quad\text{for all $s\in[-T_1,T_1]$.}
\] 
Since $s-r_1, r-r_1\in[-T_1,T_1]$ for all $s,r\in[0,T_1]$, the above implies 
\begin{align*}
\dist_X(u_{e^ds}x_0, a_dgHx)&=\dist_X(a_du_sa_{-d}x_0, a_dgHx)\\
&=\dist_X(a_du_{s-r_1}u_{r_1}a_{-d}x_0, a_dgHx)\\
&=\dist_X(a_du_{s-r_1}x_2, a_dgHx)\\
&\ll e^{\star d}\dist_X(u_{s-r_1}x_2, gHx)\leq R^{\star A_3}\biggl(\frac{|e^ds-e^dr|}{T}\biggr)^{1/D_3}.
\end{align*}
That is part~(1) holds with $A_1=\star A_3$ and $A_2=D_3$ for all large enough $R$.  

Assume now that part~(2) in Lemma~\ref{lem: unipotent linearization} holds.
Therefore, for every $r\in[-T_1,T_1]$ and every $t_1 \in [\log R, \log T_1]$, the injectivity radius of $a_{-t_1}u_r x_2$ is at most 
$R^{A_3}e^{-t_1}$.

Let $t_1\in[\log R, \log T_1]$ and $r\in[0,T_1]$, then 
\begin{align*}
\inj(a_{-t_1}u_{e^dr}x_0)&=\inj(a_{-t_1}a_du_{r-r_1}u_{r_1}a_{-d}x_0)\\
&\ll e^{\star d}\inj(a_{-t_1}u_{r-r_1}x_2)\leq R^{\star A_3}e^{-t_1}. 
\end{align*}
This implies part~(3) of the theorem for all $t\in[\log R, \log T_1]$ and large enough $R$.  

Let now $t\in[\log T_1,\log T]$. Then $t=s+\log T_1$ where $0\leq s\leq 3A_3\log R$,
and we have 
\[
\inj(a_{-t}u_{e^dr}x_0)=\inj(a_{-s}a_{-\log T_1}u_{e^dr}x_0)\leq R^{\star A_3}T_1^{-1}\leq R^{\star A_3}e^{-t}.
\]
Altogether, part~(3) in the theorem holds, again with $A_1=\star A_3$ and assuming 
$R$ is large enough depending on $X$.
\qed

\appendix


\section{Proof of Proposition~\ref{prop: linearization translates}}\label{sec: proof linearization}
In this section we prove Proposition~\ref{prop: linearization translates}. 
The proof is based on the study of a certain Margulis function whose definition will be recalled in~\eqref{eq: define f_Y app A}.   

For every $d>0$, define the probability measure $\sigma_d$ on $H$ by
\[
\int \varphi(h)\diff\!\sigma_d(h)=\frac{1}{3}\int_{-1}^2\varphi(a_du_r)\diff\!r.
\]
Let us first remark our choice of the interval $[-1,2]$: 
We will define a function $f_Y$ in~\eqref{eq: define f_Y app A} below. In Lemmas~\ref{lem: linear algebra}--\ref{lem: Margulis func periodic}, certain estimates for 
\[
\int f_Y(h\,\bigcdot)\diff(\sigma_{d_1}\!\conv\!\cdots\!\conv\!\sigma_{d_n})(h)
\]
will be obtained, then in Lemma~\ref{lem: average over [0,1]}, we will convert these estimates to similar estimates for 
\[
\int_0^1 f_Y(a_{d_1+\cdots+d_n}u_r\,\bigcdot)\diff\!r.
\]
The argument in Lemma~\ref{lem: average over [0,1]} is based on commutation relations between $a_d$ and $u_r$. Similar arguments have been used several times throughout the paper, however, since the function $f_Y$ can have a rather large Lipschitz constant, we will not appeal to continuity properties of $f_Y$ in Lemma~\ref{lem: average over [0,1]}. Instead, we will use the fact that $[0,1]\subset[-1,2]+r$ for any $|r|\leq 1/2$.

\medskip

We begin with the following linear algebra lemma.

\begin{lemma}[cf.~Lemma 5.2,~\cite{EMM-Upp}]
\label{lem: linear algebra}
For all $0\neq w\in\rfrak$, we have 
\[
\int\|\Ad(h)w\|^{-1/3}\diff\!\sigma_d(h)\leq C' e^{-d/3}\|w\|^{-1/3}
\]
where $C'$ is an absolute constant.  
\end{lemma}

\begin{proof}
We may assume $\|w\|=1$. Let us write $w=\begin{pmatrix} x & y\\ z & -x\end{pmatrix}$. Then 
\[
\Ad(a_tu_r)w=\begin{pmatrix} x+zr & e^t(-zr^2-2xr+y)\\ e^{-t}z & -x-zr\end{pmatrix}
\]
For every $\vare>0$, let 
\[
I(\vare)=\{r\in [-1,2]: \vare/2\leq |-zr^2-2xr+y|\leq \vare\},
\]
then $|I(\vare)|\leq C''\vare^{1/2}$ where $C''$ is absolute, see~e.g.~\cite[Prop.~3.2]{KM-Nondiv}. (This estimate is responsible for our choice of exponent $1/3$ which is $<1/2$.)

Moreover, for every $r\in I(\vare)$, we have $\|\Ad(a_tu_r)w\|\geq e^{t}\vare/2$. 
Note also that $\sup_{[-1,2]}|-zr^2-2xr+y|\leq 10$. Altogether, we have 
\begin{align*}
\int\|\Ad(h)w\|^{-1/3}&\diff\!\sigma_d\leq \sum_{-4}^{\infty}\int_{I(2^{k})} \|\Ad(a_tu_r)w\|^{-1/3}\diff\!r\\
&\leq C''\sum_{k=-4}^\infty 2^{-k/2}\Bigl(e^{-t/3}2^{(k+1)/3}\Bigr)\leq 2C'' e^{-t/3}\sum_{k=-4}^\infty 2^{{-k}/{6}}.
\end{align*} 
The claim follows. 
\end{proof}

We also need the following 

\begin{propos}\label{prop: average of inj}
There exists $C\geq C'$ (absolute) so that 
\[
\int \inj(hx)^{-1/3}\diff\!\sigma_d^{(\ell)}(h)\leq C^\ell e^{-\ell d/3}\inj(x)^{-1/3}+\bar B e^{2d/3}
\]
where $\sigma_d^{(\ell)}$ denotes the $\ell$-fold convolution and 
$\bar B\geq 1$ depends only of $X$.
\end{propos}

\begin{proof}
This follows from~\cite[Prop.~A.3]{LM-PolyDensity} if one replaces the use of Equation (2.12) in that proof 
by Lemma~\ref{lem: linear algebra}, see also~\cite[Lemma 2.4]{LM-PolyDensity}.
\end{proof}

Let $Y=Hy$ be a periodic orbit. 
For every $x\in X\setminus Y$, define 
\[
I_Y(x)=\{w\in \rfrak: 0<\|w\|< \inj(x), \exp(w)x\in Y\}.
\]  
Recall from~\cite[\S9]{LM-PolyDensity}, that 
\be\label{eq: number of sheets}
\#I_Y(x)\leq E\vol(Y)
\ee
for a constant $E$ depending only on $X$.

For every $h=a_du_r$ with $d\geq 0$ and $r\in[-1,2]$, and all $w\in\mathfrak g$, 
we have  
\be\label{eq: cont Ad app B}
\|\Ad(h^{\pm1})w\|\leq 10e^d\|w\|.
\ee
Replacing $10$ by a bigger constant $c$, if necessary, 
we also assume that
\be\label{eq: cont inj app B}
c^{-1}e^{-d}\inj(x)\leq \inj(h^{\pm1}x)\leq ce^d\inj(x)
\ee
for all such $h$ and all $x\in X$.

Define 
\be\label{eq: define f_Y app A}
f_{Y}(x)=\begin{cases}\sum_{w\in I_Y(x)}\|w\|^{-1/3}& I_Y(x)\neq \emptyset\\
\inj(x)^{-1/3}&\text{otherwise}\end{cases}.
\ee

\begin{lemma}\label{lem: Margulis func periodic 1}
Let $C$ be as in Proposition~\ref{prop: average of inj}, and let  
$d\geq 3\log(4C)$. Then 
\begin{multline*}
\int f_Y(hx)\diff\!\sigma_d(h)\leq \\
Ce^{-d/3}f_Y(x)+ce^{d}E\vol(Y)\cdot (Ce^{-d/3}\inj(x)^{-1/3}+\bar Be^{d})
\end{multline*}
where $\bar B$ is as in Proposition~\ref{prop: average of inj}.  
\end{lemma}

\begin{proof}
Since $Y$ is fixed throughout the argument, 
we drop it from the index in the notation, e.g., we will denote $f_{Y}$ by $f$ etc.

Let $d\geq 0$ and let $h=a_du_r$ for some $r\in[-1,2]$. 
Let $x\in X$. First, let us assume that there exists some $w\in I(hx)$ with 
\[
\|w\|<c^{-2}e^{-2d}\inj(hx)=:\Upsilon.
\]
This in particular implies that both $I(hx)$ and $I(z)$ are non-empty. 
Hence, we have  
\begin{align}
\notag f_Y(hx)&=\sum_{w\in \margI(hx)}\|w\|^{-1/3}\\
\notag&=\sum_{\|w\|< \Upsilon}\|w\|^{-1/3}+\sum_{\|w\|\geq \Upsilon}\|w\|^{-1/3}\\
\label{eq: MF-periodic-1}&\leq \sum_{w\in I(x)}\|\Ad(h)w\|^{-1/3}+c^{2/3}e^{2d/3}\Bigl(\# \margI(hx)\Bigr)\cdot\inj(hx)^{-1/3}.
\end{align}

Note also that if $\|w\|\geq \Upsilon=c^{-2}e^{-2d}\inj(hx)$ for all $w\in \margI(hx)$ 
(which in view of the choice of $c$ includes the case $I(x)=\emptyset$) 
or if $\margI(hx)=\emptyset$, then
\be\label{eq: MF-periodic-2}
f_Y(hx)\leq c^{2/3}e^{2d/3}\Bigl(\# \margI(hx)\Bigr)\cdot\inj(hx)^{-1/3}.   
\ee

Averaging~\eqref{eq: MF-periodic-1} and~\eqref{eq: MF-periodic-2} over $[-1,2]$ and using~\eqref{eq: number of sheets}, we conclude that 
\begin{multline*}
\int f_Y(hx)\diff\!\sigma_d(h)\leq 
\sum_{w\in I(x)}\int \|h w\|^{-1/3}\diff\!\sigma_d(h)\quad+\\ c^{2/3}e^{2d/3}E\vol(Y)\cdot\int \inj(hx)^{-1/3}\diff\!\sigma_d(h);
\end{multline*}
we replace the summation on the right by $0$ if $I(x)=\emptyset$. 

Thus by Lemma~\ref{lem: linear algebra} and Proposition~\ref{prop: average of inj}, we conclude that  
\begin{multline*}
\int f_Y(hx)\diff\!\sigma_d(h)\leq Ce^{-d/3}\cdot \sum_{w\in I(x)}\|w\|^{-1/3}\quad+\\
ce^dE\vol(Y)\cdot (Ce^{-d/3}\inj(x)^{-1/3}+\bar Be^{d})
\end{multline*}
where we replaced $2d/3$ by $d$. This may be rewritten as  
\[
\int f_Y(hx)\diff\!\sigma_d(h)\leq 
Ce^{-d/3}f_Y(x)+ce^dE\vol(Y)\cdot (Ce^{-d/3}\inj(x)^{-1/3}+\bar Be^{d}).
\]

The proof is complete.
\end{proof}

\begin{lemma}\label{lem: Margulis func periodic}
There is an absolute constant $T_0$ so that the following holds. 
Let $T\geq T_0$ and define 
\[
d_i=10^{-2}\cdot (2^{-i}\log T)
\]
for all $i=1,\ldots, k$ where $k$ is the largest integer so that $d_k\geq 3\log(4C)$ and $C$ is as in Proposition~\ref{prop: average of inj} --- note that $\frac{1}{2}\log\log T\leq k\leq 2\log \log T$. 

Then 
\begin{multline*}
\int f_{Y}(hx)\diff\!\sigma_{d_1}^{(100)}\conv\cdots\conv\sigma_{d_k}^{(100)}(h)\leq \\
(\log T)^{D'_0}T^{-1/3} \biggl(f(x)+B'\vol(Y)\inj(x)^{-1/3}\textstyle\sum_{i=1}^k e^{2d_i}\biggr)+
 B'\vol(Y)
\end{multline*}
where $D'_0, B'\geq 1$ are absolute.   
\end{lemma}

\begin{proof}
Again since $Y$ is fixed throughout the argument, 
we drop it from the index in the notation, e.g., we will denote $f_{Y}$ by $f$ etc. 

Let us make the following two observations:
\be\label{eq: Di and di}
5\sum_{j=i+1}^{k} d_{j}\geq 0.05\times 2^{-i-1}\log T\geq 0.01\times 2^{-i}\log T=d_{i}
\ee
There is an absolute constant $M\geq 1$ so that the following holds    
\be\label{eq: C ell di and stuff}
\sum_{j=1}^{i} C^{100 (i-j)}e^{-d_j}\leq \sum_{j=1}^kC^{100 (k-j)} e^{-d_j}\leq M
\ee
for all $1\leq i\leq k$. 

By Lemma~\ref{lem: Margulis func periodic}, for all $d\geq 3\log(4C)$, we have 
\begin{multline}\label{eq: main MF periodic app}
\int f(hx)\diff\!\sigma_d(h)\leq \\
Ce^{-d/3}f(x)+cEe^d\vol(Y)\cdot (Ce^{-d/3}\inj(x)^{-1/3}+\bar Be^{d}).
\end{multline}

Let $\lambda=cE\bar B$ and $\ell=100$. Iterating~\eqref{eq: main MF periodic app}, $\ell$-times, we conclude that
\begin{multline*}
\int f(h_k\cdots h_1x)\diff\!\sigma_{d_1}^{(\ell)}(h_1)\cdots\diff\!\sigma_{d_k}^{(\ell)}(h_k)\leq \\
C^\ell e^{-\ell d_k/3}\int\! f(h_{k-1}\cdots h_1x)\diff\!\sigma_{d_1}^{(\ell)}(h_1)\cdots\diff\!\sigma_{d_{k-1}}^{(\ell)}(h_{k-1})\quad +\\
cEe^{d_k}\vol(Y)(\Xi_k+2\bar Be^{d_k})
\end{multline*}
we used $Ce^{-d_k/3}\leq 1/4$ to bound the $\ell$-terms geometric sum by $2\bar Be^{d_k}$, and 
\[
\Xi_k=\sum_{j=0}^{\ell-1} (Ce^{-d_k/3})^{\ell-j}\!\!\int\inj(h_kh_{k-1}\cdot\cdot h_1x)^{-\frac{1}3}\diff\!\sigma_{d_1}^{(\ell)}(h_1)\cdot\cdot\diff\!\sigma_{d_{k-1}}^{(\ell)}(h_{k-1})\diff\!\sigma_{d_k}^{(j)}(h_k).
\]
Note that $cEe^{d_k}\vol(Y)(\Xi_k+2\bar Be^{d_k})\leq \lambda \vol(Y) e^{2d_k}(\Xi_k+2)$, therefore,  
\begin{multline}\label{eq: main estimate MF periodic diff di}
\int f(h_k\cdots h_1x)\diff\!\sigma_{d_1}^{(\ell)}(h_1)\cdots\diff\!\sigma_{d_k}^{(\ell)}(h_k)\leq \\
C^\ell e^{-\ell d_k/3}\int\! f(h_{k-1}\cdots h_1x)\diff\!\sigma_{d_1}^{(\ell)}(h_1)\cdots\diff\!\sigma_{d_{k-1}}^{(\ell)}(h_{k-1})\quad +\\
\lambda \vol(Y) e^{2d_k}(\Xi_k+2).
\end{multline}

We will apply Proposition~\ref{prop: average of inj}, to bound $\Xi_k$ from above.  
Let us begin by applying Proposition~\ref{prop: average of inj}, $j$-times with $d_k$, then 
\[
\Xi_k\leq C^\ell e^{-\ell d_k/3}\int\inj(h_{k-1}\cdot\cdot h_1x)^{-1/3}\diff\!\sigma_{d_1}^{(\ell)}(h_1)\cdot\cdot\diff\!\sigma_{d_{k-1}}^{(\ell)}(h_{k-1})+\lambda e^{d_k}
\]
where we used $Ce^{-d_k/3}\leq 1/4$  and $\lambda=cE\bar B \geq 2\bar B$ to estimate the $\ell$-terms geometric sum. 

The goal now is to inductively apply Proposition~\ref{prop: average of inj}, $\ell$ times with $d_i$ for all $1\leq i\leq k-1$, in order to simplify the above estimate. 
Applying Proposition~\ref{prop: average of inj}, $\ell$-times with $d_{k-1}$, we obtain from the above that
\begin{multline*}
\Xi_k\leq C^{2\ell}e^{-\ell(d_k+d_{k-1})/3}\int\inj(h_{k-2}\cdot\cdot h_1x)^{-1/3}\diff\!\sigma_{d_1}^{(\ell)}(h_1)\cdot\cdot\diff\!\sigma_{d_{k-2}}^{(\ell)}(h_{k-2})\quad +\\ 
C^\ell e^{-\ell d_k/3}\cdot (\lambda e^{d_{k-1}})+\lambda e^{d_k}.
\end{multline*}
Put $\Theta_k=0$, and for every $1\leq i<k$, let $\Theta_{i}=\sum_{j=i+1}^{k}d_{j}$. Continuing the above inequalities inductively, we conclude 
\begin{align*}
\Xi_k&\leq C^{\ell k}e^{-\ell(\sum_{i=1}^k d_i)/3}\inj(x)^{-1/3}+\lambda(e^{d_k}+\sum_{i=1}^{k-1} C^{\ell (k-i)}e^{-\ell \Theta_{i}/3}e^{d_{i}})\\
&\leq C^{\ell k}e^{-\ell(\sum_{i=1}^k d_i)/3}\inj(x)^{-1/3}+\lambda(e^{d_k}+\sum_{i=1}^{k-1} C^{\ell (k-i)}e^{-d_i})\\
&\leq C^{\ell k}e^{-\ell(\sum_{i=1}^k d_i)/3}\inj(x)^{-1/3}+\lambda(e^{d_k}+M)
\end{align*}
where we used $\ell \Theta_i/3 =100 \Theta_i/3\geq 100d_i/15$, see~\eqref{eq: Di and di}, in the second to last inequality and~\eqref{eq: C ell di and stuff} in the last inequality.  

Iterating~\eqref{eq: main estimate MF periodic diff di} and the above analysis, we conclude 
\begin{multline*}
\int f(h_k\cdots h_1x)\diff\!\sigma_{d_1}^{(\ell)}(h_1)\cdots\diff\!\sigma_{d_k}^{(\ell)}(h_k)\leq\\
C^{\ell k}e^{-\ell(\sum_{i=1}^k d_i)/3} f(x)+\lambda\vol(Y)\sum_{i=1}^k C^{\ell(k-i)}e^{-\ell \Theta_i/3}e^{2d_i}\Bigl(\Xi_i+ 2\Bigr)
\end{multline*}
where for every $1\leq i\leq k$, we have 
\[
\Xi_i=\sum_{j=0}^{\ell-1} (Ce^{-d_i/3})^{\ell-j}\!\!\int\inj(h_ih_{i-1}\cdot\cdot h_1x)^{-\frac{1}3}\diff\!\sigma_{d_1}^{(\ell)}(h_1)\cdot\cdot\diff\!\sigma_{d_{i-1}}^{(\ell)}(h_{i-1})\diff\!\sigma_{d_i}^{(j)}(h_i).
\]
Arguing as above, we have 
\[
\Xi_i\leq C^{\ell i}e^{-\ell(\sum_{j=1}^i d_j)/3}\inj(x)^{-1/3}+\lambda(e^{d_i}+M).
\] 
Recall that $\Theta_i=\sum_{j=i+1}^k d_j$; therefore, we conclude that 
\begin{multline*}
\int f(h_k\cdots h_1x)\diff\!\sigma_{d_1}^{(\ell)}(h_1)\cdots\diff\!\sigma_{d_k}^{(\ell)}(h_k)\leq\\
C^{\ell k}e^{-\ell(\sum_{i=1}^k d_i)/3} \biggl(f(x)+\lambda\vol(Y)\inj(x)^{-1/3}\textstyle\sum_{i=1}^k e^{2d_i}\biggr) + \\
(M+2)\lambda^2\vol(Y)\sum_{i=1}^k C^{\ell(k-i)}e^{-\ell \Theta_i/3}e^{3d_i}
\end{multline*}
In view of~\eqref{eq: Di and di},
$\ell \Theta_i/3 =100 \Theta_i/3\geq 100d_i/15$. Hence, using~\eqref{eq: C ell di and stuff}, 
the last term above is $\leq B'\vol(Y)$ for an absolute constant $B'\geq \lambda$. 

Moreover, $\ell\sum d_i=100\sum d_i=\log T - O(1)$ where the implied constant is absolute, and $k\leq 2\log\log T$. Hence,
\[
C^{\ell k}e^{-\ell(\sum_{i=1}^k d_i)/3}\leq (\log T)^{1+200\log C}T^{-1/3}
\]  
so long as $T$ is large enough. The proof of the lemma is complete.  
\end{proof}

\begin{lemma}\label{lem: average over [0,1]}
Let the notation be as in Lemma~\ref{lem: Margulis func periodic}, in particular for every $T\geq T_0$ define 
$d_1,\ldots, d_k$ as in that lemma. Put $d(T)=100\sum d_i$, then 
\begin{multline*}
\int_0^1 f_{Y}(a_{d(T)}u_rx)\diff\!r\leq \\
3(\log T)^{D'_0} T^{-1/3}\biggl(f_{Y}(x)+B\vol(Y)\inj(x)^{-1/3}\sum e^{2d_i}\biggr)+B\vol(Y)
\end{multline*}
where $B\geq 1$ is absolute.
\end{lemma}

\begin{proof}
Again, since $Y$ is fixed throughout the argument, 
we drop it from the index in the notation, e.g., we will denote $f_{Y}$ by $f$ etc. 

By Lemma~\ref{lem: Margulis func periodic}, we have 
\begin{multline}\label{eq: use Margulis func periodic [0,1]}
\frac{1}{3^{100k}}\int_{-1}^2\cdots\int_{-1}^2 f(a_{d_k}u_{r_{k,100}}\cdots a_{d_k}u_{r_{k,1}}\cdots a_{d_1}u_{r_{1,1}}x)\diff\!r_{1,1}\cdots\diff\!r_{k,100} \leq\\
(\log T)^{D'_0} T^{-1/3}\biggl(f_{Y}(x)+B'\vol(Y)\inj(x)^{-1/3}\sum e^{2d_i}\biggr)+B'\vol(Y).
\end{multline}

Now, for every $(r_{k,100},\ldots, r_{1,2}, r_{1,1})\in [-1,2]^{100k}$, we have 
\[
a_{d_k}u_{r_{k,100}}\cdots a_{d_k}u_{r_{k,1}}\cdots a_{d_1}u_{r_{1,1}}=a_{d(T)}u_{\varphi(\hat r)+r_{1,1}}
\] 
where $\hat r=(r_{k,100},\ldots, r_{1,2})$ and $|\varphi(\hat r)|\leq 0.2$. 

In view of~\eqref{eq: use Margulis func periodic [0,1]}, there is $\hat r= (r_{k,100},\ldots, r_{1,2})\in [-1,2]^{100k-1}$ so that 
\begin{multline}\label{eq: hat r and r1}
\frac{1}{3}\int_{-1+\varphi(\hat r)}^{2+\varphi(\hat r)} f(a_{d(T)}u_{r}x)\diff\!r\leq \\
(\log T)^{D'_0} T^{-1/3}\biggl(f_{Y}(x)+B'\vol(Y)\inj(x)^{-1/3}\sum e^{2d_i}\biggr)+B'\vol(Y).
\end{multline}
Since $|\varphi(\hat r)|\leq 0.2$, we have $[0,1]\subset [-1,2]+\varphi(\hat r)$. 
Therefore,~\eqref{eq: hat r and r1} and the fact that $f\geq 0$ imply that
\begin{multline*}
\frac{1}{3}\int_{0}^1 f(a_{d(T)}u_{r}x)\diff\!x\leq \\
(\log T)^{D'_0} T^{-1/3}\biggl(f_{Y}(x)+B'\vol(Y)\inj(x)^{-1/3}\sum e^{2d_i}\biggr)+B'\vol(Y).
\end{multline*}
The lemma follows with $B=3B'$.
\end{proof}

\begin{proof}[Proof of Proposition~\ref{prop: linearization translates}]
Let $R\geq 1$ be a parameter and assume that $\vol(Y)\leq R$.
Recall that for a periodic orbit $Y$, we put
\[
f_{Y}(x)=\begin{cases}\sum_{w\in I_Y(x)}\|w\|^{-1/3}& I_Y(x)\neq \emptyset\\
\inj(x)^{-1/3}&\text{otherwise}\end{cases}.
\]
Let $\psi(x_0)=\max\{\dist(x_0,Y)^{-1/3},\inj(x_0)^{-1/3}\}$. Then 
\be\label{eq: fYd and dist}
\psi(x_0)\ll f_{Y,d}(x_0)\ll \vol(Y)\psi(x_0),
\ee
where the implied constant depends only on $X$, see~\eqref{eq: number of sheets}. 

With the notation of Lemma~\ref{lem: Margulis func periodic}: let $T\geq T_0$ 
and $d_i=0.01\times 2^{-i}\log T$ for $1\leq i\leq k$. Then
\be\label{eq: d(T) is almost T}
\log T-\bar b\leq d(T)\leq \log T
\ee
where $\bar b$ is absolute. 

There exists $T_1\geq T_0$ so that for all $T\geq T_1$ we have 
\[
(\log T)^{D'_0} T^{-1/3}\sum e^{2d_i}\leq T^{-1/4}.
\]
Let $T'_1=\max\{T_1, 3D'_0\}$, then $(\log T)^{D'_0}T^{-1/3}$ is decreasing on $[T'_1,\infty)$. Let
\be\label{eq: choose T_2 periodic}
T_2=\inf\{ T\geq \max\{T_1', \inj(x_0)^{-2}\}: (\log T)^{D'_0}T^{-1/3}\leq d(x_0, Y)^{1/3}\}.
\ee 
In view of~\eqref{eq: fYd and dist} and since $\vol(Y)\leq R$, thus for all $T\geq T_2$, we have 
\[
(\log T)^{D'_0}T^{-1/3} f_Y(x_0)\ll R(\log T)^{D'_0}T^{-1/3}\psi(x_0)\\
\]
By the definition of $T_2$, we have $(\log T)^{D'_0}T^{-1/3} d(x_0, Y)^{-1/3}\leq 1$, and 
\[
(\log T)^{D'_0} T^{-1/3}\inj(x_0)^{-1/3}\sum e^{2d_i}\leq T^{-1/4}\inj(x_0)^{-1/3}\leq 1.
\]
In particular, using~\eqref{eq: fYd and dist} again, we have $(\log T)^{D'_0}T^{-1/3} f_Y(x_0)\ll R$. 

Altogether, we conclude that for all $T\geq T_2$, we have 
\be\label{eq: length of the flow in linearization}
\log(T)^{D'_0}T^{-1/3}\Bigl(f_Y(x_0)+B\vol(Y)\inj(x_0)^{-1/3}\sum e^{2d_i}\Bigr)\leq B_2'R
\ee
where $B_2'$ is absolute. 

Let $T\geq T_2$, and let $d(T)=100\sum d_i$ where $d_i$'s are as above.
Using~\eqref{eq: length of the flow in linearization} and Lemma~\ref{lem: average over [0,1]}, 
\be\label{lem: use average over [0,1] lemma}
\int_0^1 f_{Y}(a_{d(T)}u_rx)\diff\!r\leq B_2 R
\ee
where $B_2=3B_2'+B$. 

Let $D\geq 10$. Then by~\eqref{lem: use average over [0,1] lemma} we have 
\[
|\{r\in[0,1]: f_{Y}(a_{d(T)}u_rx_0)>B_2R^{D}\}|\leq B_2R/B_2R^{D}\leq R^{-D+1}.
\]
In view of~\eqref{eq: fYd and dist}, there is an absolute constant $B_1$ so that 
$\dist_X(a_{s}u_rx_0, Y)\leq B_1^{-1}R^{-3D}$ implies $f_{Y}(a_{s}u_rx_0)>B_2R^{D}$ for all $s\geq 0$ and $r\in[0,1]$.
Therefore, we conclude from the above that 
\be\label{eq: control dist x0 and Y app B}
\Bigl|\Bigl\{r\in [0,1]: \dist_X(a_{d(T)}u_rx_0, Y)\leq B_1^{-1}R^{-3D}\Bigr\}\Bigr|\leq R^{-D+1}.
\ee

Let now $s\geq \log T_2$, then by~\eqref{eq: d(T) is almost T} there exists some $T\geq T_2$ so that 
\[
d(T)-2\bar b\leq s\leq d(T)+2\bar b
\] 
For every $s\geq \log T_2$ let $T_s$ denote the minimum such $T$. 
Then~\eqref{eq: cont Ad app B} implies that is $\hat B\geq 1$ (absolute)
so that if $s\geq \log T_2$ and $r\in[0,1]$ are so that
\[
\dist_X(a_{s}u_rx_0, Y)\leq \hat B^{-1}R^{-3D},
\] 
then $\dist_X(a_{d(T_s)}u_rx_0, Y)\leq B_1^{-3}R^{-3D}$.
This and~\eqref{eq: control dist x0 and Y app B}, imply that
\be\label{eq: control dist x0 and Y app B s}
\Bigl|\Bigl\{r\in [0,1]: \dist_X(a_{s}u_rx_0, Y)\leq \hat B^{-1}R^{-3D}\Bigr\}\Bigr|\leq R^{-D+1}
\ee

Let $\ref{E:non-div-main}$ be as in Proposition~\ref{prop:Non-div-main}, increasing $T_1$ if necessary, we  
will assume $\log T_2\geq |\log (\inj(x_0))|+\ref{E:non-div-main}$. 
Using Proposition~\ref{prop:Non-div-main}, thus, we conclude that  
\be\label{eq: non div app B}
\Bigl|\Bigl\{r\in [0,1]:\inj(a_s\uvk x)< \eta\Bigr\}\Bigr|<\ref{E:non-div-main}\eta^{1/2}
\ee
for any $\eta>0$ and all $s\geq \log T_2$.

Altogether, from~\eqref{eq: control dist x0 and Y app B s} and~\eqref{eq: non div app B} it follows that 
for any $s\geq \log T_2$, we have 
\be\label{eq: fYd and inj app B}
\biggl|\biggl\{r\in[0,1]: \begin{array}{c}\inj(a_s\uvk x)< \eta\quad\text{ or }\\
\dist_X(a_{s}u_rx_0, Y)\leq \hat B^{-1}R^{-3D}\end{array}\biggr\}\biggr|\leq \ref{E:non-div-main}\eta^{1/2}+R^{-D+1}.
\ee

In view of~\cite[Cor.~10.7]{MO-MargFun}, the number of periodic $H$-orbits with volume $\leq R$ 
in $X$ is $\leq \hat E R^6$ where $\hat E$ depends on $X$. 
Let $D=8$ and $\ref{c: linear trans}=\max\{\hat E, \hat B, \ref{E:non-div-main}\}$. Then~\eqref{eq: fYd and inj app B} implies 
\begin{multline}\label{eq: fYd and inj app B final}
\biggl|\biggl\{r\in[0,1]: \!\!\begin{array}{c}\inj(a_s\uvk x)< \eta\text{ or there exists $x$ with}\\
\vol(Hx)\leq R\text{ s.t. }\dist_X(a_{s}u_rx_0, x)\leq \frac{1}{\ref{c: linear trans}R^{24}}\end{array}\!\!\biggr\}\biggr|\\ 
\leq \ref{c: linear trans}(\eta^{1/2}+R^{-1}).
\end{multline} 

We now show that ~\eqref{eq: fYd and inj app B final} implies the proposition. Suppose 
\[
\dist_X(x_0,x)\geq S^{-1}(\log S)^{3D'_0}
\]
for every $x$ with $\vol(Hx)\leq R$. Then by~\eqref{eq: choose T_2 periodic}, we have 
\[
T_2\leq \max\{S, \inj(x_0)^{-2}, T_1'\}.
\]
Therefore, the proposition follows from~\eqref{eq: fYd and inj app B final} if we let $D_0=\max\{24,3D'_0\}$ and put $s_0=\log T_1'$. 
\end{proof}

\section{Proof of Proposition~\ref{prop:closing lemma intro}}\label{sec: proof closing}

In this section, we will give a detailed proof of Proposition~\ref{prop:closing lemma intro}. 
As it was mentioned, the proof is a slight modification of~\cite[Prop. 6.1]{LM-PolyDensity}. 

\begin{proof}[Proof of Proposition~\ref{prop:closing lemma intro}]
In what follows all the implied multiplicative constants depend only on~$X$.

We begin by recalling Proposition~\ref{prop: Non-div main}: for all positive $\vare$, every interval $J\subset[0,1]$, 
and every $x\in X$, we have 
\be\label{eq:cpct-return}
\Bigl|\Bigl\{r\in J:\inj(a_du_r x)< \vare^2\Bigr\}\Bigr|<\ref{E:non-div-main}\vare|J|, 
\ee
so long as $d\geq |\log(|J|^2\inj(x))|+\ref{E:non-div-main}$.  

We also recall Lemma~\ref{lem: reduction theory}: 
Let $0<\eta\leq \eta_X$ and let $g\in G$ be so that $g\Gamma\in X_\eta$. 
Then there exists some $\gamma\in\Gamma$ so that 
\be\label{eq: red theory pf closing}
\|g\gamma\|\leq \ref{c: red th}\eta^{-D_2}.
\ee

For the rest of the argument, let  
\be\label{eq: t D2 and eta clsoing}
\rws\geq 100 D_2|\log(\eta\,\inj(x_1))|+\ref{E:non-div-main}
\ee
Let $r_1\in[0,1]$ be so that $x_2=a_tu_{r_1}x_1\in X_\eta$. 
Write $x_2=g_2\Gamma$ where $|g_2|\ll \eta^{-D_2}$, see~\eqref{eq: red theory pf closing}.  

We will show that unless part~(2) in the proposition holds, we have the following: for every $x_2$, there exists $J(x_2)\subset [0,1]$ with $|[0,1]\setminus J(x_2)|\leq 200\ref{E:non-div-main}\eta^{1/2}$ so that for all $r\in J(x_2)$, we have: 
\begin{itemize}
\item[(a)] $a_{7t}u_rx_2\in X_\eta$,  
\item[(b)] the map $\sfh\mapsto \sfh a_{7t}u_rx_2$ is injective on $\coneH_{\rws}$, and 
\item[(c)] for all $z\in \coneH_\rws.a_{7t}u_rx_2$ we have $f_{\rws,\alpha}(z)\leq\nuni^{D\rws}$.
\end{itemize}
This will imply that part~(1) in the proposition holds as 
\[
a_{7t}u_ra_tu_{r'}x_1=a_{8t}u_{r'+e^{-t}r}x_1.
\]

Assume contrary to the above claim that for some $x_2$ as above, there exists a subset $I'_{\rm bad}\subset [0,1]$ with  
$|I'_{\rm bad}|>200\ref{E:non-div-main}\eta^{1/2}$ so that one of (a), (b), or (c) above fails. 
Then in view of~\eqref{eq:cpct-return} applied with $x_2$ and $7t$, there is a subset 
$I_{\rm bad}\subset [0,1]$ with $|I_{\rm bad}|\geq 100\ref{E:non-div-main}\eta^{1/2}$ 
so that for all $r\in I_{\rm bad}$ we have $a_{7t}u_rx_2\in X_{\eta}$, but  
\begin{itemize}
\item either the map $\sfh\mapsto \sfh a_{7t}u_rx_2$ is not injective on $\coneH_{\rws}$,
\item or there exists $z\in \coneH_\rws.a_{7t}u_rx_2$ so that $f_{\rws,\alpha}(z)>\nuni^{D\rws}$.
\end{itemize}
We will show that this implies part~(2) in the proposition holds. 

\subsection*{Finding lattice elements $\gamma_r$}
We introduce the shorthand notation $h_r:=a_{7t}u_r$, for any $r\in[0,1]$. 
Let us first investigate the latter situation. That is: for $r\in I_{\rm bad}$ (recall that $h_rx_2\in X_{\eta}$) 
there exists some $z=\sfh_1h_rx_2\in \coneH_\rws.h_rx_2$, so that $f_{\rws,\alpha}(z)>\nuni^{D\rws}$. 
Since $h_rx_2\in X_{\eta}$, we have  
\be\label{eq:Ct-cusp}
\inj(\sfh h_rx_2)\gg \eta\nuni^{-\rws}, \quad\text{for all $\sfh\in\coneH_\rws$}.
\ee
Using the definition of $f_{\rws,\alpha}$, thus, we conclude that 
if $\margI_\rws(z)=\{0\}$, then $f_{t,\alpha}(z)\ll \eta^{-1}\nuni^\rws$. Since $t\geq 100D_2|\log\eta|$, assuming $t$ 
is large enough, we conclude that $\margI_\rws(z)\neq\{0\}$. 
Recall also that by virtue of Lemma~\ref{eq: numb Fj 1} 
we have $\#\margI_\rws(z)\ll \eta^{-4}\nuni^{4\rws}$, see also~\cite[Lemma 6.4]{LM-PolyDensity}.

Altogether, if $D\geq6$ and $t$ is large enough, there exists some $w\in I_\rws(z)$ with 
\[
0<\|w\|\leq \nuni^{(-D+5)\rws}.
\]

The above implies that for some $w\in \rfrak$ with $\|w\|\leq \nuni^{(-D+5)\rws}$ 
and $\sfh_1\neq \sfh_2\in\coneH_\rws$, we have $\exp(w)\sfh_1h_rx_2=\sfh_2h_rx_2$.
Thus 
\be\label{eq:wh-sh}
\exp(w_r)h_r^{-1}\sfs_rh_rx_2=x_2
\ee 
where $\sfs_r=\sfh_2^{-1}\sfh_1$, $w_r=\Ad(h_r^{-1}\sfh_2^{-1})w$. 
In particular, $\|w_r\|\ll \nuni^{(-D+13)\rws}$.
Assuming $t$ is large enough compared to the implied multiplicative constant,
\be\label{eq:wh-est}
0<\|w_r\|\leq \nuni^{(-D+14)\rws}.
\ee
Recall that $x_2=g_2\Gamma$ where $|g_2|\ll \eta^{-D_2}$, thus,~\eqref{eq:wh-sh} implies 
\be\label{eq:gamma-h}
\exp(w_r)h_r^{-1}\sfs_rh_r=g_2\gamma_rg_2^{-1}
\ee
where $1\neq\sfs_r\in H$ with $\|\sfs_r\|\ll \nuni^{\rws}$ and $e\neq\gamma_r\in \Gamma$.

Similarly, if for some $r\in I_{\rm bad}$, $\sfh\mapsto \sfh h_rx_2$ is not injective, then
\[
h_r^{-1}\sfs_r h_r=g_2\gamma_rg_2^{-1}\neq e.
\]
In this case we actually have $e\neq \gamma_r\in g_2^{-1}Hg_2$ --- we will not use this extra information in what follows.  

\subsection*{Some properties of the elements $\gamma_r$}
Recall that $\|g_2\|\ll \eta^{-D_2}$ and that $t\geq 100D_2|\log\eta|$. Therefore, 
\be\label{eq:size-gammah}
\|\gamma_r^{\pm1}\|\leq \nuni^{9\rws}
\ee
again we assumed $\rws$ is large compared to $\|g_2\|$ hence the estimate $\ll \nuni^{8.5\rws}$ 
is replaced by $\leq \nuni^{9t}$.

Let $\xi>0$ be so that $\|g\gamma g^{-1}-I\|\geq 20\xi\eta^{2D_2}$ for all $\gamma\in\Gamma\setminus \{1\}$ 
and $\|g\|\leq \ref{c: red th}\eta^{-D_2}$, see~\eqref{eq: red theory pf closing}.   
Write $\sfs_r=\begin{pmatrix} a_1 & a_2 \\ a_3 & a_4 \end{pmatrix}\in H$ where $|a_i|\leq 10\nuni^{t}$.
Then by~\eqref{eq:gamma-h}, we have 
\[
\|h_r^{-1}\sfs_r h_r-I\|=\biggl\|u_{-r}\begin{pmatrix} a_1 & e^{-7t}a_2 \\ e^{7t}a_3 & a_4 \end{pmatrix}u_r-I\biggr\|\geq 10\xi\eta^{2D_2}
\]
which implies that 
\be\label{eq: closing not unipotent}
\max\{e^{7t}|a_3|, |a_1-1|, |a_4-1|\}\geq \xi\eta^{2D_2}.
\ee
Note also that if $e^{7t}|a_3|<\xi\eta^{2D_2}$, then $|a_2a_3|\leq10\xi\eta^{2D_2} e^{-6t}$, thus $|a_1a_4-1|\ll \eta^{\star}e^{-6t}$. We conclude from~\eqref{eq: closing not unipotent} that $|a_1-a_4|\gg \eta^{2D_2}$. Altogether, 
\be\label{eq: a1-a4}
\max\{e^{7t}|a_3|, |a_1-a_4|\}\gg \eta^{2D_2}.
\ee
 
 \medskip
 
Since $|I_{\rm bad}|\geq 100\ref{E:non-div-main}\eta^{1/2}$, there are two intervals 
$J,J'\subset [0,1]$ with $\dist(J,J')\geq \eta^{1/2}$, $|J|, |J'|\geq \eta^{1/2}$, and
\be\label{eq: Ji cap I bad has many points}
|J\cap I_{\rm bad}|\geq \eta\quad\text{and}\quad |J'\cap I_{\rm bad}|\geq \eta.
\ee 
Put $J_\eta=J\cap I_{\rm bad}$.

\subsection*{Claim:} There are $\gg e^{29t/10}$ distinct elements in $\{\gamma_r: r\in J_{\eta}\}$.

Fix $r\in J_{\eta}$ as above, and consider the set of $r' \in J_{\eta}$ so that and $\gamma_r=\gamma_{r'}$.
Then for each such $r'$, 
\begin{align*}
h_r^{-1}\sfs_r h_r&=\exp(-w_r)g_2\gamma_rg_2^{-1}=\exp(-w_r)\exp(w_{r'})h_{r'}^{-1}\sfs_{r'} h_{r'}\\
&=\exp(w_{rr'})h_{r'}^{-1}\sfs_{r'} h_{r'}
\end{align*}
where $w_{rr'}\in \gfrak$ and $\|w_{rr'}\|\ll \nuni^{(-D+14)\rws}$. 

Set $\tau=e^{7t}(r'-r)$. Assuming $D\geq32$, we conclude that
\begin{equation}\label{eq: u_tau equation}
  u_{\tau}\sfs_r u_{-\tau}=h_{r'}h_r^{-1}\,\sfs_r\, h_rh_{r'}^{-1}=\exp(\hat w_{rr'})\sfs_{r'}
\end{equation}
where $\|\hat w_{rr'}\|=\|\Ad(h_{r'})w_{rr'}\|\ll \nuni^{(-D+21)}$.

Finally, we compute 
\[
u_{\tau}\sfs_r u_{-\tau}=\begin{pmatrix} a_1+a_3\tau& a_2+(a_4-a_1)\tau-a_3\tau^2  \\  a_3& a_4-a_3\tau\end{pmatrix}.
\]

In view of~\eqref{eq: a1-a4}, for every $r\in J_{\eta}$ the set of $r'\in J_{\eta}$ so that 
\begin{equation}\label{eq: define J_r}
  |a_2e^{-7t}+(a_4-a_1)(r'-r)-a_3e^{7t}(r'-r)^2|\leq 10^{4} e^{-6t}  
\end{equation}
has measure $\ll \eta^{-4D_2}e^{-3t}$ since at least one of the coefficients of this quadratic polynomial is of size $\gg\eta^{2D_2}$. 
Let $J_{\eta,r}$ be the set of $r'\in J_{\eta}$ for which \eqref{eq: define J_r} holds.

If $r'\in J_{\eta} \setminus J_{\eta,r}$,
then $|a_2+(a_4-a_1)\tau-a_3\tau^2|>10^{4}e^{t}$ (recall that $\tau=e^{7t}(r'-r)$), thus for all 
$r'\in J_{\eta} \setminus J_{\eta,r}$, we have 
\[
\|u_{\tau}\sfs_r u_{-\tau}\|> 10^4e^{t}> \|\exp(\hat w_{rr'})\sfs_{r'}\|,
\]
in contradiction to \eqref{eq: u_tau equation}.

In other words, for each $\gamma \in \Gamma$ the set of $r \in J_{\eta}$ for which $\gamma_r = \gamma$ has measure 
$\ll \eta^{-4D_2}e^{-3t}$ and so the set $\{\gamma_r : r \in J_{\eta}\}$ has at least 
$\gg \eta^{4D_1+1}e^{3t}\gg e^{29t/10}$ distinct elements (recall from~\eqref{eq: t D2 and eta clsoing} that $t\geq 100D_2|\log\eta|$); 
this establishes the claim.

\subsection*{Zariski closure of the group generated by $\{\gamma_r : r \in I_{\rm bad}\}$} {$\,$}
\\

\noindent
We now consider two possibilities for the elements $\{\gamma_r  : r \in I_{\rm bad}\}$. 

\subsection*{Case 1} The family $\{\gamma_r: r \in I_{\rm bad}\}$ is commutative. \\

Let ${\bf L}$ denote the Zariski closure of $\langle \gamma_{r}: r\in I_{\rm bad}\rangle$. Since $\langle \gamma_{r}\rangle$ is commutative, so is~${\bf L}$. Let $C_{\bf G}$ denote the center of $\bf G$.
We claim that ${\bf L}={\bf L}'{\bf C'}$ where ${\bf C}'\subset C_{\bf G}$ and ${\bf L}'$ is either a unipotent group or a torus. Indeed since ${\bf L}$ is commutative, we have ${\bf L}={\bf T}{\bf V}$ where $\bf T$ is a (possibly finite) algebraic subgroup of a torus, $\bf V$ is a unipotent group and ${\bf T}$ and $\bf V$ commute. Therefore, if both $\bf T$ and ${\bf V}$ are non-central, then  $G=\SL_2(\bbr)\times\SL_2(\bbr)$ and $\Gamma=\Gamma_1\times\Gamma_2$ is reducible.
Moreover, ${\bf T}\subset {\bf T}' C_{\bf G}$ where ${\bf T}'$ is an algebraic subgroup of a torus, and ${\bf T}'$ and ${\bf V}$ belong to different $\SL_2(\bbr)$ factors in $G$. Let us assume $\bf V$ belongs to the second factor. Recall from~\eqref{eq:wh-sh} that 
\be\label{eq:wh-sh'}
\exp(w_r)h_r^{-1}\sfs_rh_r=g_2\gamma_rg_2^{-1}
\ee
where $\|w_r\|\leq e^{(-D+14)t}$ with $D\geq 32$ and 
$h_r^{-1}\sfs_rh_r\in H=\{(h,h): h\in \SL_2(\bbr)\}$. 
Now if $\gamma_r=(\gamma_r^1,\gamma_r^2)$, then~\eqref{eq:wh-sh'} together with the bound $\|h_r^{-1}\sfs_rh_r\|\ll e^{8t}$ implies that  $|{\rm tr}(\gamma_r^1)-{\rm tr}(\gamma_r^2)|\ll e^{(-D+22)t}$; moreover, since $\gamma_r^2\in{\bf V} C_{\bf G}$, we have $|{\rm tr}(\gamma_r^2)|=2$. This and the fact that the length of closed geodesics in (finite volume) hyperbolic surfaces is bounded away from zero imply that $|{\rm tr}(\gamma_r^1)|=2$ if $t$ is large enough. This contradicts the fact that $\bf T$ is a non-central subgroup of a torus. Hence, the claim holds.

We now show that ${\bf L}'$ is indeed a unipotent group.
In view of the above discussion, $\#\{\gamma_r: r\in J_{\eta}\}\geq e^{29t/10}$.
Note also that that for every torus $T\subset G$, we have  
\[
\#(B_T(e,R)\cap \Gamma)\ll (\log R)^2,
\]
where the implied constant is absolute. These, in view of the bound $\|\gamma_{r}\|\leq e^{9t}$, see~\eqref{eq:size-gammah}, 
imply that ${\bf L}'$ is unipotent.

Since ${\bf L}'$ is a unipotent subgroup of $\bf G$, we have that
\[\#\{\gamma_r: \|\gamma_r\|\leq e^{4t/3}\}\ll e^{8t/3}.\] 
Furthermore, there are $\gg e^{29t/10}$ distinct elements $\gamma_r$ with $r\in J_{\eta}$.
Thus  
\[
\#\{\gamma_r : \|\gamma_r\|>100e^{4t/3} \text{ and } r\in J_{\eta}\}\gg e^{29t/10}.
\]

For every $r\in I_{\rm bad}$, write 
\[
\sfs_{r}=\begin{pmatrix} a_{1,r} & a_{2,r} \\ a_{3,r} & a_{4,r} \end{pmatrix}\in H
\] 
where $|a_{j,r}|\leq 10\nuni^{t}$. 

We will obtain an improvement of~\eqref{eq: closing not unipotent}.  
Let $\xi\eta^{2D_2}\leq \Upsilon\leq e^{4t/3}$ and assume that $\|g_2\gamma_rg_2^{-1}-I\|\geq 20 \Upsilon$ --- by definition of $\xi$, this holds with $\Upsilon=\xi\eta^{2D_2}$ for all $r\in I_{\rm bad}$ and as we have just seen this also holds for with $\Upsilon=e^{4t/3}$ for many choices of $r\in J_{\rm bad}$.
We claim 
\begin{equation}\label{eq: lower bound a(3,r)}
    |a_{3,r}|\geq \Upsilon e^{-7t}.
\end{equation}
Indeed by~\eqref{eq:gamma-h}, we have 
\[
\|h_r^{-1}\sfs_r h_r-I\|=\biggl\|u_{-r}\begin{pmatrix} a_{1,r} & e^{-7t}a_{2,r} \\ e^{7t}a_{3,r} & a_{4,r} \end{pmatrix}u_r-I\biggr\|\geq 10 \Upsilon.
\]
This implies that $\max\{e^{7t}|a_{3,r}|, |a_{1,r}-1|, |a_{3,r}-1|\}\geq \Upsilon$. Assume contrary to our claim that $|a_{3,r}|<\Upsilon e^{-7t}$. 
Then 
\begin{equation}\label{eq:max1,4}
    \max\{|a_{1,r}-1|, |a_{4,r}-1|\}\geq \Upsilon;
\end{equation}
furthermore, we get $|a_{2,r}a_{3,r}|\ll \Upsilon e^{-6t}$. 
Thus, 
\begin{equation}\label{eq:max1*4}
    |a_{1,r}a_{4,r}-1|\ll \Upsilon e^{-6t}\ll e^{-14t/3}.
\end{equation}
Moreover, since $h_r^{-1}\sfs_r h_r$ is very nearly $g_2\gamma_{r}g_2^{-1}$, and the latter is either a unipotent element or its minus, we conclude that 
\begin{equation}\label{eq:trace min}
    \min(|a_{1,r}+a_{4,r}-2|, |a_{1,r}+a_{4,r}+2|)\ll e^{(-D+22)t}.
\end{equation}
Equations~\eqref{eq:max1*4} and~\eqref{eq:trace min} contradict~\eqref{eq:max1,4} 
if $t$ is large enough (recall again from~\eqref{eq: t D2 and eta clsoing} that $t\geq 100D_2|\log\eta|$). 
Hence necessarily $|a_{3,r}|\geq \Upsilon e^{-7t}$.

Using this, we now show that Case 1 cannot occur. 
Since ${\bf L}'$ is unipotent, there exists some $g$  so that 
${\bf L}'(\bbr)\subset gNg^{-1}$; moreover $g$ can be chosen to be in the maximal compact subgroup of $G$ --- for our purposes, we only need to know that the size of $g$ can be bounded by an absolute constant. 

It follows that
\be\label{eq: us sr and N}
u_{-r}\begin{pmatrix} a_{1,r} & e^{-7t}a_{2,r} \\ e^{7t}a_{3,r} & a_{4,r} \end{pmatrix}u_r\in \exp(-w_r) (gNg^{-1})\cdot C_{\bf G}
\ee
for all $r\in I_{\rm bad}$.  
We show that this leads to a contradiction when $G=\SL_2(\bbc)$, the proof in the other case is similar by considering first and second coordinates. 

Recall the intervals $J$ and $J'$ from~\eqref{eq: Ji cap I bad has many points}, and let $r_0\in J'\cap I_{\rm bad}$.
then $|r_0-r|\geq \eta^{1/2}$ for all $r\in J_\eta$. 
Then,~\eqref{eq: us sr and N}, yields that
\be\label{eq: us sr and N'}
u_{-r+r_0}\begin{pmatrix} a_{1,r} & e^{-7t}a_{2,r} \\ e^{7t}a_{3,r} & a_{4,r} \end{pmatrix}u_{r-r_0}\in \exp(-w_r') (u_{r_0}gNg^{-1}u_{-r_0})\cdot C_{\bf G}
\ee
for all $r\in I_{\rm bad}$.

Let us write $u_{r_0}g=\begin{pmatrix} a & b \\ c & d \end{pmatrix}$, 
then for all $z\in\bbc$ we have  
\[
u_{r_0}g\begin{pmatrix} 1 & z \\ 0 & 1 \end{pmatrix}g^{-1}u_{-r_0}= \begin{pmatrix}1-acz& a^2z\\ -c^2z & 1+acz\end{pmatrix}.
\]
Let $z_0\in\bbc$ be so that 
\[
\begin{pmatrix} a_{1,r_0} & e^{-7t}a_{2,r_0} \\ e^{7t}a_{3,r_0} & a_{4,r_0}\end{pmatrix}=\pm\exp(-w_r)\begin{pmatrix}1-acz_0& a^2z_0\\ -c^2z_0 & 1+acz_0\end{pmatrix}.
\]
By \eqref{eq: lower bound a(3,r)} applied with $\Upsilon=\xi\eta^{2D_2}$, 
$|{a_{3,r_0}}|\geq \xi \eta^{2D_2}e^{-7t}$. Since $|a|, |b|, |c|, |d|\ll1$, comparing the bottom left entries of the matrices, we get $|z_0|\gg \eta^{2D_2}$. 
Now, since $|a_{2,r_0}|\leq 10e^{t}$, comparing the top right entries we conclude that $|a|\ll \eta^{-2D_2}e^{-3t}\ll e^{-29t/10}$. 
Since $\det(g)=1$, it follows that  $|c|$ is also $\gg 1$.  

Let now $r\in J_{\eta}$ be so that $\|\gamma_r\|\geq 100e^{4t/3}$. We write $r_1=r-r_0$, $a'_{2,r}=e^{-7t}a_{2,r}$ and $a'_{3,r}=e^{7t}a_{3,r}$. By~\eqref{eq: lower bound a(3,r)}, applied this time with $\Upsilon=e^{4t/3}$, we have that $|a'_{3,r}|\geq e^{4t/3}$; note also that $|a'_{2,r}|\ll e^{-6t}$. In view of~\eqref{eq: us sr and N'}, there exists $z_r\in\bbc$ so that 
\begin{align*}
u_{-r_1}\begin{pmatrix} a_{1,r} & a'_{2,r} \\ a'_{3,r} & a_{4,r} \end{pmatrix}u_{r_1}&=\begin{pmatrix} a_{1,r}-r_1a'_{3,r}  & a'_{2,r}+(a_{4,r}-a_{1,r})r_1-a'_{3,r}r_1^{2} \\ a'_{3,r} & a_{4,r}+r_1a'_{3,r} \end{pmatrix}\\
&=\pm\exp(-w_r')\begin{pmatrix}1-acz_r& a^2z_r\\ -c^2z_r & 1+acz_r\end{pmatrix}.
\end{align*}
Recall that $|a'_{3,r}|\geq e^{4t/3}$, $|a_{1,r}|$ and $|a_{4,r}|$ are $\ll e^{t}$, and $|a'_{2,r}|\ll e^{-6t}$;
moreover $\eta^{1/2}\leq |r_1|\leq 1 $ and by~\eqref{eq: t D2 and eta clsoing} $e^{t/10}\geq\eta^{-1}$. We cocnlude
\[
|a'_{3,r}|\eta/10\leq |a'_{2,r}+(a_{4,r}-a_{1,r})r-a'_{3,r}r^{2}|\leq 2|a'_{3,r}|.
\]
Hence, since $w_r'$ is small, $|c^2z_r|\eta\ll |a^2z_r|\ll |c^2z_r|$. 
On the other hand, using $r=r_0$, we already established $|a| \ll e^{-29t/10}$ and $|c|\gg1$, thus $|a^2z_r|\ll e^{-5t}|c^2z_r|$, which is a contradiction, see~\eqref{eq: t D2 and eta clsoing} again.

Altogether, we conclude that Case~1 cannot occur.

\subsection*{Case 2} There are $r,r'\in I_{\rm bad}$ so that $\gamma_r$ and $\gamma_{r'}$ do not commute. \\ 

We first recall versions of \cite[Lemma 6.2]{LM-PolyDensity} and  \cite[Lemma 6.3]{LM-PolyDensity}. 
The statements in those lemmas assume $g_2\in \mathfrak S_{\rm cpt}$. 
However, the arguments work without any changes and one has the following.

Let $v_H$ be a unit vector on the line $\wedge^3\hfrak\subset \wedge^3\gfrak$. 

\begin{lemma}\label{lem:non-elementary}
Assume $\Gamma$ is arithmetic. There exist $\constE\label{E:non-el-1}$ and $\constk\label{k:non-el-2}$ depending on $\Gamma$, and $\constE\label{E:non-el-2}$ (absolute)
so that the following holds.  
Let $\gamma_1,\gamma_2\in\Gamma$ be two non-commuting elements. 
If $g\in G$ is so that $\gamma_i g^{-1}v_H=g^{-1}v_H$ for $i=1,2$, then $Hg\Gamma$ is a closed orbit with 
\[
\vol(Hg\Gamma)\leq\ref{E:non-el-1}  \|g\|^{\ref{E:non-el-2}}\Bigl(\max\{\|\gamma_1^{\pm1}\|,\|\gamma_2^{\pm1}\|\}\Bigr)^{\ref{k:non-el-2}}.
\] 
\end{lemma} 

\begin{lemma}\label{lem:almost-inv}
Assume $\Gamma$ has algebraic entries. 
There exist $\constk\label{k:Eq-proj}$, $\constk\label{k:Eq-proj-2}$, $\constE\label{E:Eq-proj-mul}$ and $\constE\label{E:Eq-proj-mul-2}$
so that the following holds. Let $\gamma_1,\gamma_2\in\Gamma$ be two non-commuting elements,
and let 
\[
\delta\leq  \ref{E:Eq-proj-mul}^{-1}\Big(\max\{\|\gamma_1^{\pm1}\|, \|\gamma_2^{\pm1}\|\}\Big)^{-\ref{k:Eq-proj}}.
\]
Suppose there exists some $g\in G$ so that $\gamma_i g^{-1}v_H=\epsilon_ig^{-1}v_H$ for $i=1,2$ 
where $\|\epsilon_i-I\|\leq \delta$. 
Then, there is some $g'\in G$ such that 
\[
\|g'-g^{-1}\|\leq \ref{E:Eq-proj-mul}  \|g\|^{\ref{E:Eq-proj-mul-2}}\delta \Big(\max\{\|\gamma_1^{\pm1}\|, \|\gamma_2^{\pm1}\|\}\Big)^{\ref{k:Eq-proj-2}}
\] 
and $\gamma_ig'v_H=g'v_H$ for $i=1,2$.
\end{lemma}

Let us now return to the analysis in Case 2. 
Recall that $\|g_2\|\leq \eta^{-D_1}$, we will assume $t$ is large enough so that 
\[
e^{t}\geq \eta^{-2D_1\max\{\ref{E:non-el-2},\ref{E:Eq-proj-mul-2}\}}.
\]
Recall that $\exp(w_r)h_r^{-1}\sfs_rh_r=g_2\gamma_rg_2^{-1}$, thus 
\[
\gamma_r .g_2^{-1}v_H=\exp(\Ad(g_2^{-1})w_r).g_2^{-1}v_H.
\]
Moreover, since $\|w_r\|\leq \nuni^{(-D+16)\rws}$,  
\[
\|\Ad(g_2^{-1})w_r\|\ll \eta^{-2D_1}\nuni^{(-D+14)\rws}\ll e^{(-D+15)t}
\]
similar statements also hold for $r'$.

Recall that $\|\gamma_r^{\pm1}\|,\|\gamma_{r'}^{\pm1}\|\leq e^{9t}$.
If $D$ is large enough, we may apply Lemma~\ref{lem:almost-inv}
and conclude that there exists some $g_3\in G$ with 
\[
\|g_2-g_3\|\leq \ref{E:Eq-proj-mul}\eta^{-D_1\ref{E:Eq-proj-mul-2}}\nuni^{(-D+15+9\ref{k:Eq-proj-2})\rws}\leq \ref{E:Eq-proj-mul}\nuni^{(-D+16+9\ref{k:Eq-proj-2})\rws},
\] 
so that $\gamma_r .g_3^{-1}v_H=g_3^{-1}v_H$ and 
$\gamma_{r'} .g_2^{-1}v_H=g_2^{-1}v_H$.

In view of Lemma~\ref{lem:non-elementary}, thus, we have $Hg_3\Gamma$ is periodic and 
\[
\vol(Hg_3\Gamma)\leq \ref{E:non-el-1}\eta^{-D_2\ref{E:non-el-2}}\Bigl(\max\{\|\gamma_r^{\pm1}\|,\|\gamma_{r'}^{\pm1}\|\}\Bigr)^{\ref{k:non-el-2}}\leq \ref{E:non-el-1}\nuni^{1+9\ref{k:non-el-2}\rws}.
\]

Then for $t$ large enough,  $\vol(Hg_2\Gamma)\leq\nuni^{D'_0\rws}$ and $d_X(g_2\Gamma,g_2\Gamma)\ll e^{(-D+D_0')t}$ for $D'_0=9\max\{\ref{k:non-el-2}, \ref{k:Eq-proj-2}\}+16$.

Since $g_2\Gamma=x_2=a_tu_{r_1}x_1$, part~(2) in the proposition holds with $x'=(a_tu_{r_1})^{-1}g_3\Gamma$ and 
$D_0=\max\{D'_0+2, 32\}$ if $t$ is large enough (recall that we already assumed in several places that $D \geq 32$). 
\end{proof}

We note that the only place we used the arithmeticity of $\Gamma$ is Lemma~\ref{lem:non-elementary}. If we instead assume $\Gamma$ has algebraic entries, the argument above goes through and yields {\em (2')} in \S\ref{sec: closing lemma}.


\section{Proof of Theorem~\ref{thm: proj thm}}\label{sec: proof proj}

Theorem~\ref{thm: proj thm} will be proved using the following theorem which is~\cite[Thm.~B.1]{LM-PolyDensity}. As noted there, \cite[Thm.~B.1]{LM-PolyDensity} is an adaptation of works of K\"{a}enm\"{a}ki, Orponen, and Venieri~\cite{kenmki2017marstrandtype} and of Zahl \cite{Zahl,Zahl-Smooth} tailored to our needs. We refer to ~\cite[App.\ B]{LM-PolyDensity} for a brief history of this problem and more references.

\begin{thm}\label{thm: proj thm app}
Let $0<\alpha\leq1$ and let $0<\mfsc_1<1$. 
Let $\Theta\subset B_\rfrak(0,1)$ be a finite set satisfying 
\be\label{eq: energy bd proj thm}
\#(B_\rfrak(w,\mfsc)\cap \Theta)\leq \egbd \mfsc^\alpha \quad\text{for every $w\in\Theta$ and all $\mfsc\geq \mfsc_1$}
\ee
where $\egbd\geq 1$. 

Let $0<\pvare<0.01\alpha$.
For every $\rhsc\geq \egbd^{-1/\alpha}$, there exists a subset $J_\rhsc\subset [0,1]$ with $|[0,1]\setminus J_\rhsc|\leq  L\pvare^{-L}\rhsc^{\pvare}$ 
so that the following holds. 
Let $r\in J_\rhsc$, then there exists a subset $\Theta_{\rhsc,r}\subset \Theta$ with 
\[
\#(\Theta\setminus \Theta_{\rhsc,r})\leq L\pvare^{-L} \rhsc^{\pvare}\cdot (\#\Theta)
\]
such that for all $w\in \Theta_{\rhsc,r}$, we have 
\[
\#\{w'\in\Theta: |\xi_r(w)-\xi_r(w')|\leq \rhsc\}\leq L\pvare^{-L} \egbd^{1+7\pvare} \rhsc^\alpha\cdot (\#\Theta)
\]
where $L$ is an absolute constant and 
\[
\xi_r(w)=(\Ad(u_r)w)_{12}=-w_{21}r^2-2w_{11}r+w_{12}.
\]
\end{thm}

\begin{proof}[Proof of Theorem~\ref{thm: proj thm}]
First note that replacing $\Theta$ by $\frac{1}{\mfsc_0}\Theta$ and $\egbd$ by $b_0^{\alpha}\egbd$, we may assume $\mfsc_0=1$.  
Note that~\eqref{eq: energy bd proj thm main} implies that  
\[
\#(B(w,\mfsc)\cap \Theta)\leq \egbd\mfsc^\alpha+\trct, 
\]
for all $w$ and all $\mfsc$. Thus we have 
\[
\#(B(w,\mfsc)\cap \Theta)\leq 2\egbd\mfsc^\alpha\quad\text{for all $\mfsc\geq (\tfrac{\trct}{\egbd})^{1/\alpha}=:\mfsc_1$}
\]
and all $w$. Thus $\Theta$ satisfies~\eqref{eq: energy bd proj thm} in Theorem~\ref{thm: proj thm app}. 

We will work with dyadic scales. Let $\ell_1=\lfloor-\log\mfsc_1\rfloor$.
Let $L$ be as in Theorem~\ref{thm: proj thm app}; put $C=L\pvare^{-L}$. 

Let $\ell_2=20+\lfloor\pvare\log\egbd\rfloor$. Then
\[
\sum_{\ell=\ell_2}^{\infty} 2^{-\pvare\ell}< 10^{-6}\egbd^{-\pvare^2}.
\]
Let $J=\bigcap_{\ell=\ell_2}^{\ell_1} J_{2^{-\ell}}$. 
Then the choice of $\ell_2$ and Theorem~\ref{thm: proj thm app} imply that 
\[
|[0,1]\setminus J|\leq C\egbd^{-\pvare^2}.
\]

For every $r\in J$, let $\Theta_r=\bigcap_{\ell=\ell_2}^{\ell_1}\Theta_{2^{-\ell},r}$.
Then by Theorem~\ref{thm: proj thm app}, 
\[
\rho(\Theta\setminus\Theta_r)\leq C\egbd^{-\pvare^2}.
\]
Moreover, for all $w\in \Theta_r$ and all $\ell_2\leq \ell\leq \ell_1$
we have
\be\label{eq: regularity proj proof}
\#(\{w'\in \Theta: |\xi_{r}(w')-\xi_r(w)|\leq 2^{-\ell}\})\leq  C\egbd^{1+7\pvare} 2^{-\alpha\ell}.
\ee

Let $w\in \Theta_r$, and put $\Theta(w)=\Theta\setminus \{w'\in \Theta: |\xi_{r}(w')-\xi_r(w)|\leq 2^{-\ell_1}\}$. 
In view of~\eqref{eq: regularity proj proof}, applied with $\ell=\ell_1$, we have 
\[
\begin{aligned}
\#(\{w'\in \Theta: |\xi_{r}(w')-\xi_r(w)|\leq 2^{-\ell_1}\})&\leq  C\egbd^{1+7\pvare} 2^{-\alpha\ell_1}\leq C\egbd^{1+7\pvare} \mfsc_1^{\alpha}\\
&=C\egbd^{1+7\pvare} \tfrac{\trct}{\egbd}=C\egbd^{7\pvare}\trct. 
\end{aligned}
\]
In other words we have 
\be\label{eq: size of Theta(w)}
\#(\Theta\setminus \Theta(w))\leq 2C\egbd^{7\pvare}\trct.
\ee
Moreover,~\eqref{eq: regularity proj proof} applied with $\ell_2\leq \ell\leq \ell_1$, implies that  
\be\label{eq:modified energy cont pf}
\begin{aligned}
\sum_{w'\in\Theta(w)}\|\xi_r(w)-\xi_r(w')\|^{-\alpha}&\leq\sum_{\ell=\ell_2}^{\ell_1}C\egbd^{1+7\pvare} 2^{-\alpha\ell}2^{\alpha\ell}+2^{\alpha\ell_2}\\
&=\ell_1C\egbd^{1+7\pvare}+2^{\alpha\ell_2}\cdot(\#\Theta).
\end{aligned}
\ee

Recall that $\#\Theta\leq \egbd$ and that $2^{\alpha\ell_2}\leq 2^{20}\egbd^\pvare$. 
The claim in the theorem thus follows from~\eqref{eq: size of Theta(w)} and~\eqref{eq:modified energy cont pf}.
\end{proof}

We also need the following theorem which was used in 
\S\ref{sec: equidistribution}, in particular in the proof of Lemma~\ref{lem: equidist mu cone-i}. We will reduce this to the results proved in~\cite[App.~B]{LM-PolyDensity}, these results have now been obtained in greater generality, see~\cite{PYZ}. 

\begin{thm}\label{thm: proj thm 2}
Let $0<\alpha\leq1$, and let  $0<\rhsc_1<\rhsc_0\leq1$. 
Let $\Theta\subset B_\rfrak(0,\rhsc_0)$ be a finite set, and let $\theta$ 
denote a probability measure on $\Theta$. Assume further that the following two properties hold
\begin{subequations}
\begin{align}
\label{eq: lambda counting meas thm 2}&K^{-1}\leq \theta(w)\leq K\\
\label{eq:appendix-regularity-F}
&\theta(B_\rfrak(w, \rhsc))\leq \bar\egbd\cdot (\rhsc/\rhsc_0)^\alpha\qquad\text{for all $w$ and all $\rhsc\geq \rhsc_1$}
\end{align}
\end{subequations}
where $\bar\egbd\geq 1$ and $K$ is absolute.

Let $0<\pvare<0.01\alpha$, and let $J\subset [0,1]$ be an interval with $|J|\geq 10^{-4}$. 
For every $\rhsc\geq \rhsc_1$, there exists a subset $J_{\rhsc}\subset J$ with $|J\setminus J_{\rhsc}|\ll \rhsc^{\pvare}$ 
so that the following holds. 
Let $r\in J_{\rhsc}$, then there exists a subset $\Theta_{\rhsc,r}\subset \Theta$ with 
\[
\theta(\Theta\setminus \Theta_{\rhsc,r})\ll \rhsc^{\pvare}
\]
such that for all $w\in\Theta_{\rhsc,r}$, we have 
\[
\theta\Bigl(\{w'\in\Theta: |\zeta_{r}(w')-\zeta_r(w)|\leq \rhsc\}\Bigr)\leq C(\rhsc/\rhsc_0)^{\alpha-7\pvare}
\] 
where $C\ll\pvare^{-\star}\bar\egbd$, the implied constants are absolute and $\zeta_r(w)\in\mathfrak r^+$ is defined as follows:
\[
u_r\exp(w)u_{-r}=\begin{pmatrix}d_{r,w} &0\\ c_{r,w} & 1/d_{r,w}\end{pmatrix}\begin{pmatrix}1 &\zeta_r(w)\\ 0& 1\end{pmatrix}u_{\hat r},
\]
for some $\hat r=\hat r(w, r)$.
\end{thm}

\begin{proof} 
In view of the assumption~\eqref{eq: lambda counting meas thm 2}, it suffices to prove the claim when $\theta$ is the uniform measure on $\Theta$. 

Define $f: B_{\rfrak}(0, 0.01)\to G$ by
\[
f\left(\!\begin{pmatrix}w_{11} & w_{12}\\ w_{21} & -w_{11}\end{pmatrix}\!\right)=\begin{pmatrix}1+w_{11} &w_{12}\\ w_{21} & \frac{1+w_{12}w_{21}}{1+w_{11}}\end{pmatrix}.
\] 

There exists an absolute constant $\delta_0$ so that the map
$g=f^{-1}\circ\exp$ is a diffeomorphism from $B_\rfrak(0,\delta_0)$ onto its image and 
\be\label{eq: derivative of Psi}
\|Dg-I\|\leq 0.01.
\ee

We may, without loss of generality, assume that $\Theta\subset B_\rfrak(0,\delta_0)$. 
Let $\Theta'=g(\Theta)$. Then, in view of~\eqref{eq:appendix-regularity-F} and~\eqref{eq: derivative of Psi}, we have 
\be\label{eq:appendix-regularity-F'}
\frac{\#B_\rfrak(w, \rhsc)\cap \Theta'}{\#\Theta'}\leq 2\bar\egbd\cdot (\rhsc/\rhsc_0)^\alpha\qquad\text{for all $w$ and all $\rhsc\geq \rhsc_1$}.
\ee

Moreover, for any $w\in B_\rfrak(0,\delta_0)$, we have 
\[
u_{r}\exp(w)u_{-r}= u_{r}f(g(w))u_{-r}.
\]
Therefore, it suffices to prove the theorem with $\exp$ replaced by $g$. 

Altogether, it suffices to prove the theorem for $\czeta_r$ defined as follows 
\[
u_r\begin{pmatrix}1+w_{11} &w_{12}\\ w_{21} & \frac{1+w_{12}w_{21}}{1+w_{11}}\end{pmatrix}u_{-r}=\begin{pmatrix}d'_{r,w} &0\\ c'_{r,w} & 1/d'_{r,w}\end{pmatrix}\begin{pmatrix}1 &\czeta_r(w)\\ 0& 1\end{pmatrix},
\] 
and when $\theta$ is the counting measure.

The above definition implies that 
\[
\czeta_r(w)=\frac{w_{12}+\frac{w_{12}w_{21}-2w_{11}-w_{11}^2}{1+w_{11}}r-w_{21}r^2}{1+w_{11}+w_{21}r};
\]
define $\cZ(w)=\{(r,\czeta_r(w)): r\in[0,1]\}$ if $G=\SL_2(\R)\times\SL_2(\R)$ and $\cZ(w)=\{(r,\Im(\czeta_r(w))): r\in[0,1]\}$ if $G=\SL_2(\C)$. 

First one argues as in~\cite[Prop.\  2.1]{PYZ} to establish the {\em cinematic} curvature conditions~\cite[Eq.~(1.5) and~(1.6)]{Zahl} for this family. This can alternatively be checked directly, as we now explicate in the first case above. 
Defined $\Phi:\bbr^2\times\bbr^2\to \bbr$ by 
\[
\Phi(x,y)=y_2(1+x_1)+\frac{(2x_1+x_1^2)y_1+(x_2+x_1x_2)y_1^2}{1+x_1+x_2y_1}.
\]
Note that $\Phi(0,y)=y_2$ and that 
\[
\cZ(w)=\{y\in\bbr^2: y_1\in[0,1], \Phi(w_{11}, w_{21}, y)=w_{12}\}.
\]
Assuming $|x_i|\leq 0.1$ and $|y_i|\leq 1$, a direct calculation shows that 

\begin{align*}
\frac{\partial\Phi}{\partial y_1}&=\frac{(1+x_1)(x_1^2+2x_1+2x_2(1+x_1)y_1+x_2^2y_1^2)}{(1+x_1+x_2y_1)^2}\\
\frac{\partial^2\Phi}{\partial y_1^2}&=\frac{2(1+x_1)x_2}{(1+x_1+x_2y_1)^3}.
\end{align*}
In particular, there exists some absolute constant $C$ so that 
\be\label{eq: cinematic curvature}
\tfrac{1}{C}\max\{|x_1|,|x_2|\}\leq|\tfrac{\partial \Phi}{\partial y_1}|+|\tfrac{\partial^2 \Phi}{\partial y_1^2}|\leq C\max\{|x_1|, |x_2|\}. 
\ee 
In view of~\cite[Eq.~(21)]{Kolasa-Wolff}, thus, the family $\cZ$ satisfies the {\em cinematic} 
curvature conditions~\cite[Eq.~(1.5) and~(1.6)]{Zahl}.

For two curves $\cZ=\{y\in\bbr^2: y_1\in[0,1], \Phi(w_{11}, w_{21}, y)=w_{12}\}$ and 
$\cZ'=\{y'\in\bbr^2: y_1'\in[0,1], \Phi(w_{11}', w_{21}', y')=w'_{12}\}$, define 
\[
\Delta(\cZ,\cZ')=\inf_{y\in \cZ,y'\in\cZ'}\|y-y'\|+\biggl|\frac{d_y\Phi(w_{11}, w_{21}, y)}{\|d_y\Phi(w_{11}, w_{21}, y)\|}-\frac{d_y\Phi(w_{11}', w_{21}', y)}{\|d_y\Phi(w_{11}', w_{21}', y')\|}\biggr|;
\]
this provides a quantitative tool to study incidence of $\cZ$ and $\cZ'$. 

In view of~\eqref{eq: cinematic curvature} and the fact that the level curves $\cZ$ here are algebraic, we may apply~\cite[Lemma 5.18]{Zahl}, see also~\cite{Zahl-Smooth}. Therefore, the proof of the theorem goes through the same lines as the proof of~\cite[Thm.~B.1]{LM-PolyDensity} (see also the proof of Theorem~\ref{thm: proj thm app})
if we replace the family $\Xi$ there by the family $\cZ$ and $\Delta$ there by $\Delta$ above. 
\end{proof}

\bibliographystyle{halpha}
\bibliography{papers}

\end{document}